\pdfoutput=1
\documentclass[a4paper,10pt]{article}
\usepackage{fullpage}
 \usepackage{ae,lmodern}
\usepackage[english,greek,frenchb,french]{babel}
\usepackage[utf8]{inputenc}  
\usepackage[T1]{fontenc}
\usepackage{url,csquotes}
\usepackage[hidelinks,hyperfootnotes=false]{hyperref}
\usepackage{float}
\usepackage{amsfonts,amssymb,enumerate}
\usepackage[shortlabels]{enumitem}
\usepackage{amsmath}
\usepackage{authblk}
\allowdisplaybreaks[1]
\usepackage{amsthm}
\usepackage{textgreek}
\usepackage{graphicx}
\usepackage{bbm}
\usepackage{listings}
\usepackage{titling}
\usepackage{dsfont}
\usepackage{color}
\usepackage{bmpsize}
\usepackage{algorithm}
\usepackage{algorithmic}
\usepackage{nomencl}
\makenomenclature
\newcommand{\E}{\mathbb{E}}
    \newcommand{\Prb}{\mathbb{P}}
	\newcommand{\cP}{\mathcal{P}}
	\newcommand{\cA}{\mathcal{A}}
	\newcommand{\an}{\overrightarrow{n}}
		
				\newcommand{\cS}{\mathcal{S}}
								
									\newcommand{\ogam}{\overrightarrow{\gamma}}
									\newcommand{\oGam}{\overrightarrow{\Gamma}}

		\newcommand{\cE}{\mathcal{E}}
			
		\newcommand{\amu}{\overrightarrow{\mu}}
		\newcommand{\cF}{\mathcal{F}}
				
				\newcommand{\fT}{\mathfrak{T}}
				\newcommand{\fM}{\mathfrak{M}}
				\newcommand{\fB}{\mathfrak{B}}
		\newcommand{\cM}{\mathcal{M}}
		\newcommand{\cG}{\mathcal{G}}
		\newcommand{\cL}{\mathcal{L}}
		\newcommand{\cU}{\mathcal{U}}
	\newcommand{\cB}{\mathcal{B}}
		\newcommand{\B}{B(n)}
		\newcommand{\cV}{\mathcal{V}}

							\newcommand{\cN}{\mathcal{N}}
							
		\newcommand{\cH}{\mathcal{H}}
	\newcommand{\sR}{\mathbb{R}}
	\newcommand{\sN}{\mathbb{N}}
	\newcommand{\vv}{\overrightarrow{v}}
	\newcommand{\nn}{\overrightarrow{n}}
	\newcommand{\ssigma}{\overrightarrow{\sigma}}
		\newcommand{\ff}{\overrightarrow{f}}
		\newcommand{\sS}{\mathbb{S}}
	\DeclareMathOperator{\diam}{diam}

				\DeclareMathOperator{\card}{card}

				\DeclareMathOperator{\cyl}{cyl}
				\DeclareMathOperator{\proj}{proj}
				
				\DeclareMathOperator{\flow}{flow}
				\DeclareMathOperator{\sign}{sign}
								\DeclareMathOperator{\diver}{div}
								
    \newcommand{\sZ}{\mathbb{Z}}
    \newcommand{\sC}{\mathcal{C}}

    \newcommand{\ep}{\varepsilon}
    \newcommand{\dis}{\mathfrak{d}}
        \newcommand{\Cor}{\mathfrak{Cor}}

    \newcommand{\fC}{\mathfrak{C}}

    \newcommand{\di}{\mathrm{d}}

    \newcommand{\ind}{\mathds{1}}
 \theoremstyle{plain}   
 \newtheorem{thm}{Theorem}[section]
\newtheorem{lem}[thm]{Lemma}

\newtheorem{defn}[thm]{Definition}
\newtheorem{prop}[thm]{Proposition}
\newtheorem{rk}[thm]{Remark}
\newtheorem{hypo}{Hypothesis}
\numberwithin{equation}{section}
 
\title{Large deviation principle for the streams and the maximal flow in first passage percolation\thanks{Research was partially supported by the ANR project PPPP (ANR-16-CE40-0016) and the Labex MME-DII (ANR 11-LBX-0023-01).}}
\date{}
\author{Barbara Dembin\thanks{LPSM UMR 8001, Université Paris Diderot, Sorbonne Paris Cité, CNRS, F-75013 Paris, France. bdembin@lpsm.paris} , Marie Théret\thanks{Modal'X, UPL, Univ Paris Nanterre, F92000 Nanterre France and FP2M, CNRS FR 2036.
marie.theret@parisnanterre.fr} }

\begin{document}

 \selectlanguage{english}

\maketitle
{\bf Abstract:} We consider the standard first passage percolation model in the rescaled lattice $\sZ^d/n$ for $d\geq 2$ and a bounded domain $\Omega$ in $\sR^d$. We denote by $\Gamma^1$ and $\Gamma^2$ two disjoint subsets of $\partial \Omega$ representing respectively the source and the sink, \textit{i.e.}, where the water can enter in $\Omega$ and escape from $\Omega$. A maximal stream is a vector measure $\amu_n^{max}$ that describes how the maximal amount of fluid can enter through $\Gamma ^1$ and spreads in $\Omega$. Under some assumptions on $\Omega$ and $G$, we already know a law of large number for $\amu_n^{max}$. The sequence $(\amu_n^{max})_{n\geq 1}$ converges almost surely to the set of solutions of a continuous deterministic problem of maximal stream in an anisotropic network. We aim here to derive a large deviation principle for streams and deduce by contraction principle the existence of a rate function for the upper large deviations of the maximal flow in $\Omega$.

\section{Introduction}
\subsection{First definitions and main results}
\subsubsection{The environment, discrete admissible maximal streams}
We use here the same notations as in \cite{CT1}.
 Let $n\geq 1$ be an integer. We consider the graph $(\sZ^d_n,\E^d_n)$ having for vertices $\sZ^d _n=\sZ^d/n$ and for edges $\E_n^d$, the set of pairs of points of $\sZ^d_n$ at Euclidean distance $1/n$ from each other.  With each edge $e\in\E_n^d $ we associate a capacity $t(e)$, which is a random variable with value in $\sR^+$. The family $(t(e))_{e\in\E_n^d}$ is independent and identically distributed with a common law $G$. We interpret this capacity as a rate of flow, \textit{i.e.}, it corresponds to the maximal amount of water that can cross the edge per second. Throughout the paper, we work with a distribution $G$ on $\sR^+$ satisfying the following hypothesis.
\nomenclature{$(\sZ_n^d,\E_n^d)$}{The rescaled lattice}
\begin{hypo} \label{hypo:G} There exists $M>0$ such that $G([M,+\infty[)=0$.
\end{hypo}

Let $(\Omega,\Gamma^1,\Gamma^2)$ that satisfies the following hypothesis.
\begin{hypo}\label{hypo:omega}
The set $\Omega$ is an open bounded connected subset of $\sR^d$, that it is a Lipschitz domain. There exist $\cS_1,\dots,\cS_l$ oriented manifolds of class $\sC^1$ that intersect each other transversally such that the boundary $\Gamma$ of $\Omega$ is included in $ \cup_{i=1,\dots,l}\cS_i$.
 The sets $\Gamma^1$ and $\Gamma^2$ are two disjoint subsets of $\Gamma$ that are open in $\Gamma$ such that $\inf\{\|x-y\|,\, x\in\Gamma^1,\,y\in\Gamma^2\}>0$, and that their relative boundaries $\partial_\Gamma \Gamma^1$ and $\partial_\Gamma \Gamma ^2$ have null $\cH^{d-1}$ measure. 
\end{hypo}
The sets $\Gamma^1$ and $\Gamma ^2$ represent respectively the sources and the sinks. We aim to study the maximal streams from $\Gamma^1$ to $\Gamma^2$ through $\Omega$ for the capacities $(t(e))_{e\in\E_n^d}$. We shall define discretized versions for those sets. For $x=(x_1,\dots,x_d)\in\sR^d$, we define $$\|x\|_2=\sqrt{\sum_{i=1}^dx_i^2}\qquad\text{and}\qquad \|x\|_\infty =\max\big\{\,|x_i|,\,i=1,\dots,d\,\big\}.$$
We use the subscript $n$ to emphasize the dependence on the lattice $(\sZ^d_n,\E^d _n)$. Let $\Omega_n$, $\Gamma_n$, $\Gamma_n^1$ and $\Gamma_n^2$ be the respective discretized version of $\Omega$, $\Gamma$, $\Gamma^1$ and $\Gamma^2$:
\begin{align*}
&\Omega_n=\left\{\,x\in \sZ^d_n:\, d_\infty(x,\Omega)<\frac{1}{n}\right\},\\
&\Gamma_n=\big\{\,x\in \Omega_n:\, \exists y\notin \Omega_n,\,\langle x,y\rangle\in\E_n^d\,\big\}\,,\\
&\Gamma_n ^i=\left\{\,x\in \Gamma_n:\, d_\infty(x,\Gamma^i)<\frac{1}{n}, \,  d_\infty(x,\Gamma^{3-i})\geq\frac{1}{n}\right\},\qquad \text{for $i=1,2$},
\end{align*}
\nomenclature{$(\Omega,\Gamma^1,\Gamma^2)$}{The domain, the source and the sink}
\nomenclature{ $\Omega_n$, $\Gamma_n$, $\Gamma_n^1$ and $\Gamma_n^2$}{The discretized version of $\Omega$, $\Gamma$, $\Gamma^1$,$\Gamma^2$}
where $d_\infty$ is the $L^\infty$ distance and $\langle x,y\rangle$ represents the edge whose endpoints are $x$ and $y$. We denote by $\overrightarrow{\E}_n^d$ the set of oriented edges. We will denote by $\langle\langle x,y\rangle\rangle$ the oriented edge in $\overrightarrow{\E}_n^d$ whose first endpoint is $x$ and last endpoint is $y$. To each $\langle\langle x,y\rangle\rangle$ in $\overrightarrow{\E}_n^d$ we can associate the vector $\overrightarrow{xy}$ in $\sR^d$. Notice that $\|\overrightarrow{xy}\|_2=1/n$. Let $(\overrightarrow{e_1},\dots,\overrightarrow{e_d})$ be the canonical basis of $\sR^d$.
  We denote by $\cdot$ the standard scalar product in $\sR^d$. 

 \noindent 
{\bf Stream function.} A stream $f_n$ is a function $f_n : \E_n^d\rightarrow \sR^d$  such that the vector $f_n(e)$ is collinear with the geometric segment associated with $e$. For $e\in \E_n^d$, $\|f_n(e)\|_2$ represents the amount of water that flows through $e$ per second and $f_n(e)/\|f_n(e)\|_2$ represents the direction in which the water flows through $e$. 

\noindent 
{\bf Admissible streams through $\Omega$.} 
A stream $f_n:\E^d_n\rightarrow \sR^d$ from $\Gamma^1$ to $\Gamma^2$ through $\Omega$ is admissible if and only if
\begin{itemize}
\item[$\cdot$] \textit{The stream is inside $\Omega$} : for each edge $e=\langle x,y\rangle$ such that $(x,y)\notin \Omega_n^2\setminus (\Gamma_n^1\cup\Gamma_n^2) ^2$ we have $f_n(e)=0$
\item[$\cdot$] \textit{The stream respects the capacity constraints}: for each $e\in\E_n^d$ we have $\|f_n(e)\|_2\leq t(e)$
\item [$\cdot$] \textit{The stream satisfies the node law}: for each vertex $x\in \sZ_n^d\setminus(\Gamma_n^1\cup \Gamma_n^2)$ we have
$$\sum_{y\in\sZ_n^d:\,e =\langle x,y \rangle\in \E_n^d } f_n(e)\cdot \overrightarrow{xy} =0\,.$$
\end{itemize} 
The node law expresses that there is no loss or creation of fluid outside $\Gamma_1$ and $\Gamma_2$. The capacity constraint imposes that the amount of water that flows through an edge $e$ per second is limited by its capacity $t(e)$.
We denote by $\cS_n(\Gamma^1,\Gamma^2,\Omega)$ the set of admissible streams from $\Gamma^1$ to $\Gamma^2$ through $\Omega$.
\nomenclature{$\cS_n(\Gamma^1,\Gamma^2,\Omega)$}{The random set of admissible discrete stream }
As the capacities are random, the set of admissible streams $\cS_n(\Gamma^1,\Gamma^2,\Omega)$ is also random.  
 We denote by $\cS_n^M(\Gamma^1,\Gamma^2,\Omega)$ the set of streams $f_n:\E^d_n\rightarrow \sR^d$ such that
\begin{itemize}
\item[$\cdot$] for each edge $e=\langle x,y\rangle$ such that $(x,y)\notin \Omega_n^2\setminus (\Gamma_n^1\cup\Gamma_n^2) ^2$ we have $f_n(e)=0$
\item[$\cdot$] for each $e\in\E_n^d$ we have $\|f_n(e)\|_2\leq M$
\item [$\cdot$]the stream satisfies the node law for any vertex $x\in \sZ_n^d\setminus(\Gamma_n^1\cup \Gamma_n^2)$.
\end{itemize} 
Note that the set $\cS_n^M(\Gamma^1,\Gamma^2,\Omega)$ is a deterministic set.
\nomenclature{$\cS_n^M(\Gamma^1,\Gamma^2,\Omega)$}{The deterministic set of admissible discrete stream }
To each $f_n$, we can define its associated vector measure $\amu_n(f_n)$ by
$$\amu_n(f_n)=\frac{1}{n^d}\sum_{e\in\E_n^d}f_n(e)\delta_{c(e)}\,,$$
where $c(e)$ denotes the center of the edge $e$.

\noindent{\bf Maximal flow through $\Omega$.}
For each admissible stream $f_n$ in $\cS_n(\Gamma^1,\Gamma^2,\Omega)$, we define its flow by 
$$\flow_n(f_n)=\sum_{x\in \Gamma_n^1}\,\sum_{\substack{y\in\Omega_n:\\\,e =\langle x,y \rangle\in \E_n^d} }n\,f_n(e)\cdot \overrightarrow{xy}\,$$
where we recall that $\|\overrightarrow{xy}\|_2=1/n$.
This corresponds to the amount of water that enters in $\Omega_n$ through $\Gamma^1_n$ per second for the stream $f_n$. 
The maximal flow  between $\Gamma^1$ and $\Gamma^2$ through $\Omega$ for the capacities $(t(e))_{e\in\E_n^d}$, denoted by $\phi_n(\Gamma^1,\Gamma^2, \Omega)$, is the supremum of the flows of all admissible streams  from $\Gamma^1$ to $\Gamma^2$ through $\Omega$:
\begin{align}\label{eq:defphin}
\phi_n(\Gamma^1,\Gamma^2, \Omega)=\sup\left\{\flow_n(f_n)\,:\,f_n\in \cS_n(\Gamma^1,\Gamma^2,\Omega)\right\}\,.
\end{align}
We define $\amu_n^{max}$, the measure corresponding to a given stream $f_n ^{max}\in \cS_n(\Gamma^1,\Gamma^2,\Omega)$ that achieves the maximal flow $\phi_n(\Gamma^1,\Gamma^2, \Omega)$.

\noindent{\bf Admissible streams through a connected set $C$ without prescribed sinks and sources.} Let $C\subset \sR^d$.
We denote by $\cS_n(C)$ the set of admissible streams through $C$, that is streams $f_n:\E^d_n\rightarrow \sR^d$ such that:
\begin{itemize}
\item[$\cdot$] \textit{The stream respects the capacity constraint}: for each edge $\langle x,y\rangle\in\E_n^d$ such that $x\in C$ and there exists $i\in\{1,\dots,d\}$ such that $\overrightarrow {xy}\cdot \overrightarrow{e_i}>0$, we have $\|f_n(e)\|_2\leq t(e)$.
\item [$\cdot$] \textit{The stream respects the node law}: for each vertex $x\in \sZ_n^d\cap C$ such that for any $i\in\{1,\dots,d\}$, $(x-\overrightarrow{e}_i/n)\in C$, we have
$$\sum_{y\in\sZ_n^d:\,e =\langle x,y \rangle\in \E_n^d } f_n(e)\cdot \overrightarrow{xy} =0\,.$$
\end{itemize} 
In what follows, we will say that $x$ is the left endpoint of the edge $e=\langle x,y\rangle \in\E_n^d$, if there exists $i\in\{1,\dots,d\}$ such that $\overrightarrow{xy}=\overrightarrow{e_i}/n$. Moreover, we say that $e$ belongs to $C$ if its left endpoint $x$ belongs to $C$.
Note that the event $\{f_n\in \cS_n(C)\}$ only depends on edges such that their left endpoint is in $C$.

\subsubsection{Presentation of the limiting objects and main results}
We want to define the possible limiting objects for $\amu_n$. To lighten the presentation of the object, we will do as if our limiting object $\ssigma:\sR^d\rightarrow \sR^d$ were a nice $\sC^1$ vector field. Actually, the convergence holds in a distributional sense and $\ssigma$ is a distribution. More rigorous definitions will be given in section \ref{sect:defmaxflow}. Let us denote by $\sS^{d-1}$ the unit sphere in $\sR^d$. For $x\in \Omega$ and $\vv\in\sS^{d-1}$, the quantity $\ssigma(x)\cdot \vv$ corresponds to the quantity of flow sent by $\ssigma$ at the position $x$ in the direction $\vv$. 
 The limiting object $\ssigma$ inherits the properties of $\amu_n$:
 \begin{enumerate}[label=(\roman*)]
 \item \textit{the stream is inside} $\Omega$ : $\ssigma=0$ on $\sR^d\setminus \overline{\Omega}$;
 \item\label{condstream:3} \textit{conservation law}: $\diver \ssigma=0$ on $\Omega$ and $\ssigma\cdot \overrightarrow{n}_\Omega=0 $ on $\Gamma \setminus (\Gamma^1\cup\Gamma ^2)$ .
  \end{enumerate}
  Here $\overrightarrow{n}_\Omega(x)$ denotes the exterior unit vector normal to $\Omega$ at $x$.
 The condition \ref{condstream:3} is a consequence of the fact that $f_n$ satisfies the node law on $\Omega$ and that no water escapes from $\Omega$ except at the sinks and the sources. We will say that $\ssigma$ is an admissible stream if it satisfies these two conditions and we will denote by $\Sigma(\Gamma ^1,\Gamma^2,\Omega)$ the set of admissible continuous streams (a more rigorous definition will be given later). Just as in the discrete setting, we are interested in the maximal amount of flow that can enter through $\Gamma^1$. We need to give a definition of the flow for continuous streams. The flow is the amount of water that enters the source per second, it translates here as follows
\[\flow ^{cont}(\ssigma)=-\int_{\Gamma^1}\ssigma\cdot \overrightarrow{n}_\Omega d\cH^{d-1}\,.\]
Here, the vector $\overrightarrow{n}_\Omega$ is exterior, that is, exiting from $\Omega$, whereas the vector $-\overrightarrow{n}_\Omega$ is entering $\Omega$, this accounts for the minus sign. $\cH^{d-1}$ denotes the Hausdorff measure in dimension $d-1$.
The goal of this paper is roughly speaking to find a proper rate function $\widehat{I}$ on $\Sigma$ such that 
\[\forall \ssigma\in \Sigma(\Gamma ^1,\Gamma^2,\Omega) \qquad\Prb\big(\exists f_n\in\cS_n(\Gamma ^1,\Gamma ^2,\Omega): \amu_n(f_n)\approx \ssigma\cL^d\big)\approx \exp\big(-\widehat{I}(\ssigma)n^d\big)\,\]
where $\cL^d$ denotes the Lebesgue measure in dimension $d$.
To be able to compare $f_n$ and $\ssigma$, we introduce a distance $\dis$ between vector measures that we will explicit in section \ref{sect:tools}. The convergence for this distance implies the weak convergence.
Roughly speaking, the larger $\widehat{I}(\ssigma)$ is the more atypical the continuous stream is. We here express the rate function $\widehat{I}$ as the integral on $\Omega$ of an elementary rate function $I$:
\begin{align}\label{eq:defhatI}
\widehat{I}(\ssigma)=\int_{\Omega}I(\ssigma(x))d\cL^d(x)\,.
\end{align}
 This elementary rate function $I$ characterizes locally how atypical the stream is. When we consider a small cube $C$ in $\Omega$, we have if the cube is small enough that $\ssigma$ is almost constant, there exists $s>0$ and $\vv\in\sS^{d-1}$ such that $\ssigma\approx s\vv$ in $C$. The function $I$ characterizes how likely it is possible that there exists an admissible stream in $C$ close to $s\vv\cL^d$.  Let us now give a more rigorous definition.
Let us denote by $\fC$ the unit cube centered at $0$, that is
\[\fC=\left[-\frac{1}{2},\frac{1}{2}\right[^d\,.\]
We first study the probability of having a stream in $\cS_n(\fC)$ that looks like some constant continuous stream $s\vv\in\sR^d$.
\begin{thm}\label{thmbrique}Let $G$ that satisfies hypothesis \ref{hypo:G}. Let $\vv\in\sS^{d-1}$ and $s>0$. 
We have
\begin{align*}
-\lim_{\ep\rightarrow 0} &\limsup_{n\rightarrow \infty}\frac{1}{n^d}\log \Prb\left( \exists f_n\in\cS_n(\fC) : \,\dis\big(\amu_n(f_n)\ind_{\fC},s\vv\ind_{\fC}\cL^d\big)\leq \ep\right)\\&\hspace{2cm}=-\lim_{\ep\rightarrow 0} \liminf_{n\rightarrow \infty}\frac{1}{n^d}\log \Prb\left(\exists f_n\in\cS_n(\fC) : \,\dis\big(\amu_n(f_n)\ind_{\fC},s\vv\ind_{\fC}\cL^d\big)\leq \ep\right)\,.
\end{align*}
We will denote this limit by $I(s\vv)$.
\end{thm}

Let $\cM(\overline{\cV_\infty(\Omega,1)})$ be the set of measures on $\sR^d$ with support included in  $\overline{\cV_\infty(\Omega,1)}$ where $\cV_\infty(\Omega,1)=\{x\in\sR^d:\, d_\infty(\Omega,x)\leq 1\}$.
We endow $\cM(\overline{\cV_\infty(\Omega,1)})^d$ with the topology $\mathcal{O}$ associated with the distance $\dis$ and the Borelian $\sigma$-field $\cB$. Write $\Prb_n$ the following probability:
$$\forall A\in\cB\qquad \Prb_n(A)=\Prb(\exists f_n\in\cS_n(\Gamma^1,\Gamma^2,\Omega):\amu_n(f_n)\in A)\,.$$
We define the following rate function $\widetilde{I}$ on $\cM(\overline{\cV_\infty(\Omega,1)})^d$ as follows:
\[\forall \nu \in \cM(\overline{\cV_\infty(\Omega,1)})^d\qquad  \widetilde{I}(\nu) = \left\{
    \begin{array}{ll}
        +\infty  & \mbox{if } \nu\notin\{\ssigma\cL^d:\,\ssigma\in\Sigma(\Gamma^1,\Gamma^2,\Omega)\cap \Sigma^M(\Gamma^1,\Gamma^2,\Omega)\} \\
        \widehat{I}(\ssigma) & \mbox{if } \nu=\ssigma\cL^d,\,\ssigma\in\Sigma(\Gamma^1,\Gamma^2,\Omega)\,
    \end{array}
\right.
\]
where $\Sigma^M(\Gamma^1,\Gamma^2,\Omega)$ will be defined more rigorously later, it represents the continuous streams that corresponds to weak limit of a sequence of discrete streams in $\cS_n^M(\Gamma^1,\Gamma^2,\Omega)$.
We have the following large deviation principle for the stream :
\begin{thm}[Large deviation principle for admissible streams] \label{thm:pgd}Under some regularity hypothesis on $\Omega$, $\Gamma^1$ and $\Gamma^2$, for distributions $G$ compactly supported, the sequence $(\Prb_n)_{n\geq 1}$ satisfies a large deviation principle with speed $n^d$ governed by the good rate function $\widetilde{I}$ and with respect to the topology $\mathcal{O}$, \textit{i.e.}, for all $A\in\cB$
$$ -\inf\left\{\,\widetilde{I}(\nu ):\,\nu\in \mathring{A}\,\right\}\leq\liminf_{n\rightarrow\infty}\frac{1}{n^d}\log\Prb_n(A)\leq \limsup_{n\rightarrow\infty}\frac{1}{n^d}\log\Prb_n(A)\leq -\inf\left\{\,\widetilde{I}(\nu ):\,\nu\in \overline{A}\,\right\}\,.$$
\end{thm}
We can deduce from theorem \ref{thm:pgd}, by a contraction principle, a large deviation principle for the maximal flows.
Let $J$ be the following function on $\sR^+$:
$$\forall \lambda\geq 0\qquad J(\lambda)=\inf\left\{\widehat{I}(\ssigma) : \ssigma\in\Sigma(\Gamma^1,\Gamma^2, \Omega)\cap\Sigma^M(\Gamma^1,\Gamma^2,\Omega),\,\flow^{cont}(\ssigma)=\lambda\right\}\,$$
and $$\lambda_{max}=\sup\left\{ \flow^{cont}(\ssigma):\ssigma\in \Sigma(\Gamma^1,\Gamma^2,\Omega)\right\}\,.$$
To prove an upper large deviation principle for maximal flows, we will need the following lower large deviation principle for maximal flows that was proven in \cite{dembin2021large}.
\begin{thm}[Lower large deviation principle for maximal flows]\label{thm:lldmf} 
Let $G$ that satisfies hypothesis \ref{hypo:G}. Let $(\Omega,\Gamma^1,\Gamma^2)$ that satisfies hypothesis \ref{hypo:omega}. There exist $\phi_{\Omega}\geq 0$ and $\lambda_{min}\geq 0$ depending on $\Omega$, $\Gamma^1$, $\Gamma^2$ and $G$ such that the sequence $(\phi_n(\Gamma^1,\Gamma^ 2,\Omega)/n^{d-1},n\in\sN)$ satisfies a large deviation principle of speed $n^{d-1}$ with the good rate function $\widetilde J_l$. ~Moreover, the map $\widetilde J_l$ is infinite on $[0,\lambda_{min}[\cup]\phi_\Omega,+\infty[$, decreasing on $]\lambda_{min},\phi_{\Omega}[$, positive on $]\lambda_{min},\phi_{\Omega}[$. Besides, for every $\lambda< \lambda_{min}$, there exists $n_0\geq 1$ such that
$$\forall n\geq n_0 \qquad\Prb\left(\frac{\phi_n(\Gamma^1,\Gamma^ 2,\Omega)}{n^{d-1}}\leq \lambda\right)=0\,.$$
\end{thm}
Let $\lambda_{min}\geq 0$ depending on $G$ and $\Omega$  given by theorem \ref{thm:lldmf}.
We define the following rate function:
\begin{align}\label{eq:defJu}
\widetilde{J}_u(\lambda)=\left\{\begin{array}{ll}J(\lambda) &\mbox{if } \lambda\in[\lambda_{min},\lambda_{max}[\\
\lim_{\substack{\lambda\rightarrow\lambda_{max}\\\lambda<\lambda_{max}}}J(\lambda)&\mbox{if }\lambda=\lambda_{max}\\
+\infty &\mbox{if }\lambda\in[0,\lambda_{min}[\cup]\lambda_{max},+\infty[
\end{array}\right.\,.
\end{align}
\begin{thm}[Upper large deviation principle for maximal flows]\label{thm:uldmf} 
Let $G$ that satisfies hypothesis \ref{hypo:G}. Let $(\Omega,\Gamma^1,\Gamma^2)$ that satisfies hypothesis \ref{hypo:omega}. Let $\phi_\Omega$, $\lambda_{min}$ given by theorem \ref{thm:lldmf}. The sequence $(\phi_n(\Gamma^1,\Gamma^ 2,\Omega)/n^{d-1},n\in\sN)$ satisfies a large deviation principle of speed $n^d$ with the good rate function $\widetilde J_u.$ Moreover, the map $\widetilde J_u$ is convex on $\sR_+$, infinite on $[0,\lambda_{min}[\cup]\lambda_{max},+\infty[$, $\widetilde J_u $ is null on $[\lambda_{min},\phi_\Omega]$ and strictly positive on $]\phi_\Omega,+\infty[$. 
\end{thm}
\noindent Theorems \ref{thm:pgd} and \ref{thm:uldmf} are the main results of this article. To prove these theorems, we will need tools from the realm of large deviations and the following key theorem:
\begin{thm}\label{thm:ULDpatate} Let $G$ that satisfies hypothesis \ref{hypo:G}. Let $(\Omega,\Gamma^1,\Gamma^2)$ that satisfies hypothesis \ref{hypo:omega}. For any $\ssigma \in \Sigma(\Gamma^1,\Gamma^2,\Omega)$, we have 
\begin{align*}
-&\lim_{\ep\rightarrow 0} \limsup_{n\rightarrow \infty}\frac{1}{n^d}\log \Prb\left(\exists f_n\in\cS_n(\Gamma^1,\Gamma^2,\Omega) : \,\dis\big(\amu_n(f_n),\ssigma\cL^d\big)\leq \ep\right)\\&=-\lim_{\ep\rightarrow 0} \liminf_{n\rightarrow \infty}\frac{1}{n^d}\log \Prb\left(\exists f_n\in\cS_n(\Gamma^1,\Gamma^2,\Omega) : \,\dis\big(\amu_n(f_n),\ssigma\cL^d\big)\leq \ep\right)=\int_{\Omega}I(\ssigma(x))d\cL ^d(x)=\widehat{I}(\ssigma)\,.
\end{align*}
\end{thm}

\begin{rk}
Theorems \ref{thm:lldmf} and \ref{thm:uldmf} give the full picture of large deviations of $\phi_n(\Gamma^1,\Gamma^ 2,\Omega)$. The lower large deviations are of surface order since it is enough to decrease the capacities of the edges along a surface to obtain a lower large deviations event. The lower large deviations have been studied in the companion paper \cite{dembin2021large}. The upper large deviations are of volume order, to create an upper large deviations event, we need to increase the capacities of constant fraction of the edges. This is the reason why to study lower large deviations, it is natural to study cutsets that are $(d-1)$-dimensional objects, whereas to study the upper large deviations, we study streams that are $d$-dimensional objects. 
\end{rk}

\subsection{Background} We now present the mathematical background needed in what follows. We present two different flows in cylinders and give a rigorous definition of the limiting objects.

\subsubsection{Flows in cylinders and minimal cutsets}
Dealing with admissible streams is not so easy, but hopefully we can use an alternative definition of maximal flow which is more convenient. Here $n=1$, \textit{i.e.}, we consider the lattice $(\sZ^d,\E^d)$. Let $E\subset\E^d$ be a set of edges. We say that $E$ cuts $\Gamma^1$ from $\Gamma^2$ in $\Omega$ (or is a cutset, for short) if there is no path from $\Gamma^1_1$ to $\Gamma^2_1$ in $(\Omega_1,\E^d \setminus E)$. More precisely, let $\gamma$ be a path from  $\Gamma^1_1$ to $\Gamma^2_1$ in $\Omega_1$, we can write $\gamma$ as a finite sequence $(v_0,e_1,v_1,\dots,e_n,v_n)$ of vertices $(v_i)_{i=0,\dots,n}\in \Omega_1^{n+1}$ and edges $(e_i)_{i=1,\dots,n}\in (\E^d) ^n$ where $v_0\in  \Gamma^1_1$, $v_n\in\Gamma^2_1$ and for any $1\leq i \leq n$, $e_i=\langle v_{i-1}, v_i \rangle \in \E^d$. Then, $E$ cuts $\Gamma^1$ from $\Gamma^2$ in $\Omega$ if for any path $\gamma$ from  $\Gamma^1_1$ to $\Gamma^2_1$ in $\Omega_1$, we have $\gamma\cap E\neq \emptyset$. Note that $\gamma$ can be seen as a set of edges or a set of vertices and we define $|\gamma|=n$.
We associate with any set of edges $E$ its capacity $T(E)$ defined by $$T(E)=\sum _{e\in E} t(e)\,.$$ The max-flow min-cut theorem, see \cite{Bollobas}, a result of graph theory, states that 
$$\phi_1(\Gamma^1,\Gamma^2, \Omega)=\min\big\{\,T(E)\,:\, E \text{ cuts $\Gamma^1$ from $\Gamma^2$ in $\Omega$}\,\big\}\,.$$
We recall that $\phi_1$ was defined in \eqref{eq:defphin}.
The idea behind this theorem is quite intuitive. By the node law, the flow is always smaller than the capacity of any cutset. Conversely, consider a maximal flow through $\Omega$, some of the edges are jammed. We say that $e$ is jammed if the amount of water that flows through $e$ is equal to the capacity $t(e)$. These jammed edges form a cutset, otherwise we would be able to find a path $\gamma$ from $\mathfrak{G}_1$ to $\mathfrak{G}_2$ of non-jammed edges, and we could increase the amount of water that flows through $\gamma$ which contradicts the fact that the flow is maximal. Thus, the maximal flow is limited by the capacity of these jammed edges: the maximal flow is given by one of the  $T(E)$ where $E$ cuts $\Gamma^1$ from $\Gamma^2$ in $\Omega$. It follows that the maximal flow is equal to the minimal capacity of a cutset. 

We are interested in the maximal flow $\Phi$ that can cross a cylinder oriented according to $\vv\in\sS^{d-1}$ from its top to its bottom per second for admissible streams. A first issue is to understand if the maximal flow in the box properly renormalized converges when the size of the box grows to infinity. This boils down to understand the maximal amount of water that can flow in the direction $\vv$.
Let us first define rigorously the maximal flow from the top to the bottom of a cylinder. Let $A$ be a non-degenerate hyperrectangle, \textit{i.e.}, a rectangle of dimension $d-1$ in $\sR^d$. Let $\vv\in\sS^{d-1}$ such that $\vv$ is not contained in an hyperplane parallel to $A$. 
We denote by $\cyl(A,h,\vv)$ the cylinder of basis $A$ and of height $h>0$ in the direction $\vv$ defined by
$$\cyl(A,h,\vv)=\big\{\,x+t\vv:\,x\in A,\,t\in[0,h]\,\big\}\,.$$
If $\vv$ is one of the two unit vectors normal to $A$, we denote by $\cyl(A,h)$ 
$$\cyl(A,h)=\big\{\,x+t\vv:\,x\in A,\,t\in[-h,h]\,\big\}\,.$$
We have to define discretized versions of the bottom $B(A,h)$ and the top $T(A,h)$ of the cylinder $\cyl(A,h)$. We define them by 
$$B(A,h):= \left\{x\in\sZ^d\cap\cyl(A,h)\,:\,\begin{array}{c}
\exists y \notin \cyl(A,h),\, \langle x,y \rangle\in\E^d \\\text{ and $\langle x,y \rangle$ intersects } A-h\vv
\end{array} \right\}$$
and
$$T(A,h):= \left\{x\in\sZ^d\cap\cyl(A,h)\,:\,\begin{array}{c}
\exists y \notin \cyl(A,h),\, \langle x,y \rangle\in\E^d \\\text{ and $\langle x,y \rangle$ intersects } A+h\vv
\end{array} \right\}\,.$$

We denote by $\Phi(A,h)$ the maximal flow from the top to the bottom of the cylinder $\cyl(A,h)$ in the direction $\vv$, defined by 
$$\Phi(A,h)=\phi_1(T(A,h),B(A,h) ,\cyl(A,h,\vv))\,.$$
The maximal flow $\Phi(A,h)$ is not well suited to use ergodic subadditive theorems, because we cannot glue two cutsets from the top to the bottom of two adjacent cylinders together to build a cutset from the top to the bottom of the union of these two cylinders. Indeed, the intersection of these two cutsets with the adjacent face will very likely not coincide. 

To fix this issue, we need to introduce another maximal flow through the cylinder for which the subadditivity would be recover. Let $T'(A,h)$ (respectively $B'(A,h)$) be the a discretized version of the upper half part (resp. lower half part) of the boundary of $ \cyl(A,h)$; that is if we denote by $z$ the center of $A$:
\begin{align}\label{eq:defT(A,h)}
T'(A,h)=\big\{\,x\in\sZ^d\cap \cyl(A,h)\,:\,\overrightarrow{zx}\cdot\vv> 0 \text{ and }\exists y\notin \cyl(A,h),\,\langle x,y\rangle\in\E^d\,\big\}\,,
\end{align}
\begin{align}\label{eq:defB(A,h)}
B'(A,h)=\big\{\,x\in\sZ^d\cap \cyl(A,h)\,:\,\overrightarrow{zx}\cdot\vv<0 \text{ and }\exists y\notin \cyl(A,h),\,\langle x,y\rangle\in\E^d\,\big\}\,.
\end{align}
We denote by $\tau(A,h)$ the maximal flow from the upper half part to the lower half part of the boundary of the cylinder, \textit{i.e.}, 
$$\tau(A,h)=\phi_1(T'(A,h),B'(A,h),\cyl(A,h))\,.$$
By the max-flow min-cut theorem, the flow $\tau(A,h)$ is equal to the minimal capacity of a set of edges $E$ that cuts $T'(A,h)$ from $B'(A,h)$ inside the cylinder $\cyl(A,h)$. The intersection of $E$ with the boundary of the cylinder has to be close to the relative boundary $\partial A$ of the hyperrectangle $A$.

\subsubsection{Some mathematical tools and definitions}\label{sect:tools}
Let us first recall some mathematical definitions.
For a subset $X$ of $\sR^d$, we denote by $\overline{X}$ the closure of $X$, by $\mathring{X}$ the interior of $X$. Let $a\in\sR^d$, the set $a+X$ corresponds to the following subset of $\sR^d$
$$a+X=\{a+x:\,x\in X\}\,.$$
For $r>0$, the $r$-neighborhood $\cV_i(X,r)$ of $X$ for the distance $d_i$, that can be Euclidean distance if $i=2$ or the $L^\infty$-distance if $i=\infty$, is defined by
$$\cV_i(X,r)=\left\{\,y\in\sR^d:\,d_i(y,X)<r\,\right\}\,.$$
We denote by $B(x,r)$ the closed ball centered at $x\in\sR^d$ of radius $r>0$.
Let $\sC_b(\sR^d,\sR)$ be the set of continuous bounded functions from $\sR^d$ to $\sR$.
We denote by $\sC^k_c(A,B)$ for $A\subset \sR^p$ and $B\subset\sR^q$, the set of functions of class $\sC^k$ defined on $\sR^p$, that takes values in $B$ and whose domain is included in a compact subset of $A$. The set of functions of bounded variations in $\Omega$, denoted by $BV(\Omega)$, is the set of all functions $u\in L^1(\Omega\rightarrow\sR,\cL^d)$ such that
\[\sup\left\{\int_\Omega\diver \overrightarrow{h}d\cL^d:\, \overrightarrow{h}\in\sC_c^\infty(\Omega,\sR^d),\,\forall x\in\Omega \quad\overrightarrow{h}(x)\in B(0,1)\right\}<\infty\,.\]
Let $\nu$ be a signed-measure on $\sR^d$, we write $\nu=\nu^+-\nu^-$ for the Hahn-Jordan decomposition of the signed measures $\nu$. Then $\nu^+$ and $\nu^-$ are positive measures, respectively, the positive and negative part of $\nu$. We define the total variation $|\nu|$ of $\nu$ as $|\nu|=\nu^++\nu ^-$.

Let $x\in\sR^d$ and $\alpha>0$, we define the homothety $\pi_{x,\alpha}:\sR^d\rightarrow\sR^d$ as follows
\begin{align}\label{eq:defpixalpha}
\forall y \in\sR^d\qquad \pi_{x,\alpha}(y)=\alpha y+x\,.
\end{align}
We will need the following proposition that enables to relate the Lebesgue measure of a neighborhood of the boundary of a set $E$ with the $\cH^{d-1}$-measure of its boundary $\partial E$. 
\begin{prop}\label{prop:minkowski}
Let $E$ be a subset of $\sR^d$ such that $\partial E$ is piecewise of class $\sC^1$ and $\cH^{d-1}(\partial E)<\infty$. Then, we have
$$\lim_{r\rightarrow 0 }\frac{ \cL^d(\cV_2(\partial E,r))}{2 r}= \cH^{d-1}(\partial E)\,.$$
\end{prop}
\noindent  This proposition is a consequence of the existence of the $(d-1)$-dimensional Minkowski content. We refer to Definition 3.2.37 and Theorem 3.2.39 in \cite{FED}.

Let us now define the distance $\dis$.
Let $k\in\sN$. Let $\lambda\in[1,2]$. We denote by $\Delta^k_\lambda$ the set of dyadic cubes at scale $k$ with scaling parameter $\lambda$, that is,
$$\Delta^k_\lambda =\left\{\,2^{-k}\lambda\left(\left[-\frac{1}{2},\frac{1}{2}\right[^d+x\right):\,x\in\sZ^d\right\}\,.$$
Let $\Delta^k_\lambda (\Omega)$ denote the dyadic cubes at scale $k$ that intersect $\overline{\cV_\infty(\Omega,2)}$, that is
$$\Delta^k _\lambda(\Omega)=\left\{\, Q\in\Delta^k_\lambda : \,Q\cap \overline{\cV_\infty(\Omega,2)}\neq\emptyset\,\right\}\,.$$
Let $\nu,\mu\in \cM(\overline{\cV_\infty(\Omega,1)})^d$ be vectorial measures, we set
\begin{align}\label{eq:def:dis}
\dis(\nu,\mu)=\sup_{x\in[-1,1[^d}\sup_{\lambda\in[1,2]}\sum_{k=0}^\infty\frac{1}{2^{k}}\sum_{Q\in\Delta^k_\lambda }\left\|\mu(Q+x)-\nu(Q+x)\right\|_2.\end{align}
\begin{rk}Although working with topological neighborhood is the most general setting, we chose here to work with a distance to reduce the amount of technical details. The choice of a distance is arbitrary. However, this distance satisfies some properties that are not satisfied by other more standard distances. This distance was inspired by the distance that appears in \cite{Gold}. The key property that this distance satisfies is that if for $\nu,\mu\in \cM(\overline{\cV_\infty(\Omega,1)})^d$ the distance $\dis(\nu,\mu)$ is small, then the distance restricted to some $Q\subset \Omega$ is also small. This property will be proven later. 
\end{rk}

\subsubsection{Continuous streams}\label{sect:defmaxflow}
We give here the mathematical definitions to properly define the max-flow min-cut theorem as in the paper of Nozawa \cite{Nozawa}. 
A stream in $\Omega$ is a vector field $\ssigma\in L^\infty(\Omega\rightarrow \sR^d,\cL^d)$ that satisfies 
\[\diver\ssigma=0\qquad\text{on   $\Omega$,}\]
in the distributional sense, that is, $\diver\ssigma$ is a distribution defined on $\Omega$ by
\[\forall h \in\sC^\infty _c(\Omega,\sR)\qquad \int_{\sR^d}h\diver\ssigma d\cL^d=-\int_{\sR^d}\ssigma\cdot \overrightarrow{\nabla}hd\cL^d\,.\]
Thus, a stream $\ssigma$ satisfies
\[\forall h \in\sC^\infty _c(\Omega,\sR)\qquad \int_{\sR^d}\ssigma\cdot \overrightarrow{\nabla}hd\cL^d=0\,.\]
This condition is the continuous analogue of the node law, it expresses the fact that there is no loss or gain of fluid for the stream $\ssigma$ inside $\Omega$. 

For a stream $\ssigma$ from $\Gamma^1$ to $\Gamma^2$ in $\Omega$, the fluid can enter or exit only through the source $\Gamma ^1$ and the sink $\Gamma^2$, we have to mathematically express the fact that no water escapes through $\Gamma\setminus (\Gamma ^1\cup \Gamma ^2)$. Since $\ssigma$ is defined in the distributional sense, we need to give a sense to the value of $\ssigma$ on $\Gamma$ that is a set of null $\cL^d$-measure. To do so we need to define the trace on $\Gamma$ for any $u \in BV(\Omega)$. According to Nozawa in \cite{Nozawa}, there exists a linear mapping $\gamma$ from $BV(\Omega)$ to $L^1(\Gamma\rightarrow \sR,\cH^{d-1})$, such that, for any $u\in BV(\Omega)$,
\[\lim_{r\rightarrow 0, r>0}\frac{1}{\cL^d(\Omega\cap B(x,r))}\int_{\Omega\cap B(x,r)}|u(y)-\gamma(u)(x)|d\cL^d(y)=0 \quad\text{for $\cH^{d-1}$-a.e. $x\in\Gamma$.} \]
According to Nozawa in \cite{Nozawa}, Theorem 2.3, for every $\overrightarrow{\rho}=(\rho_1,\dots,\rho_d):\Omega\rightarrow\sR^d$ such that $\rho_i\in L^\infty(\Omega\rightarrow\sR^d,\cL^d)$ for all $i=1,\dots,d$ and $\diver\overrightarrow{\rho}\in L^d(\Omega\rightarrow \sR, \cL^d)$, there exists $g\in L^\infty(\Gamma\rightarrow\sR^d,\cH^ {d-1})$ defined by
$$\forall u\in W^{1,1}(\Omega)\qquad \int_{\Gamma}g\gamma(u)\,d\cH^{d-1}=\int_{\Omega}\overrightarrow{\rho}\cdot \overrightarrow{\nabla}ud\cL^d+\int_{\Omega}u\diver\overrightarrow{\rho}d\cL^d\,.$$
The function $g$ is denoted by $\overrightarrow{\rho}\cdot\overrightarrow{n}_\Omega$.
For any stream $\ssigma$, since $\diver\ssigma=0$ $\cL^d$-a.e. on $\Omega$, we have
\[\forall u \in W^{1,1}(\Omega)\qquad\int_{\Gamma}(\ssigma\cdot \overrightarrow{n}_{\Omega})\gamma(u)\,d\cH^{d-1}=\int_{\Omega}\ssigma\cdot \overrightarrow{\nabla}ud\cL^d\,.\]
We need to impose some boundary conditions for any stream $\ssigma$ from $\Gamma^1$ to $\Gamma^2$ in $\Omega$: the water can only enters through $\Gamma^1$ , \textit{i.e.},
\[\ssigma\cdot \overrightarrow{n}_{\Omega}\leq 0\qquad\text{$\cH^{d-1}$-a.e. on $\Gamma ^1$}\]
and no water can enter or exit through $\Gamma\setminus (\Gamma^1\cup\Gamma^2)$, \textit{i.e.}
\[\ssigma\cdot \overrightarrow{n}_{\Omega}= 0\qquad\text{$\cH^{d-1}$-a.e. on $\Gamma\setminus (\Gamma^1\cup\Gamma^2)$}\,.\]
Of course, we also need to add a constraint on the local capacity, otherwise the continuous maximal flow is infinite. This local constraint is here anisotropic which means that the maximal amount of water that can spreads in a direction depends on the direction but not on the location. This local constraint is given by a function $\nu:\sR^d\rightarrow \sR_+$, that is a continuous convex function that satisfies $\nu(\vv)=\nu(-\vv)$. In the setting of \cite{CT1}, the function $\nu$ corresponds the flow constant that will be properly defined in section \ref{sect:flowconstant}. The local capacity constraint is expressed by
\[\cL^d\text{-a.e. on $\Omega$, }\quad\qquad\forall \vv\in\sS^{d-1}\qquad \ssigma\cdot\vv\leq \nu(\vv)\,.\]
To each admissible stream $\ssigma$, we associate its flow 
\[\flow ^{cont}(\ssigma)=-\int_{\Gamma^1}\ssigma\cdot \overrightarrow{n}_\Omega d\cH^{d-1}\]
which corresponds to the amount of water that enters in $\Omega$ through $\Gamma^1$ for the stream $\ssigma$ per second.
Nozawa considered the following variational problem
\begin{align}\label{eq:def:phiomega}
\phi_\Omega=\sup\left\{\flow ^{cont}(\ssigma): \begin{array}{c}\ssigma\in L^\infty(\Omega\rightarrow \sR^d,\cL^d), \,\diver\ssigma =0 \,\cL^d\text{-a.e. on $\Omega$},\\  \ssigma\cdot\vv\leq \nu(\vv)\text { for all }\vv\in\sS^{d-1}\,\cL^d\text{-a.e. on $\Omega$, }\\\ssigma\cdot \overrightarrow{n}_{\Omega}\leq 0\quad\text{$\cH^{d-1}$-a.e. on $\Gamma ^1$}
\\
\ssigma\cdot \overrightarrow{n}_{\Omega}= 0\qquad\text{$\cH^{d-1}$-a.e. on $\Gamma\setminus (\Gamma^1\cup\Gamma^2)$}
\end{array} \right\}\,.
\end{align}
Note that we can extend $\ssigma$ to $\sR^d$ by defining $\ssigma=0$ $\cL^d$-a.e. on $\Omega^c$.
We denote by $\Sigma_\nu$ the set of admissible streams solution of the variational problem, \textit{i.e.},
\begin{align}\label{def:sigmaomega}
\Sigma_\nu=\left\{\ssigma\in L^\infty(\sR^d\rightarrow \sR^d,\cL^d): \begin{array}{c}\ssigma=0\text{ $\cL^d$-a.e. on $\Omega^c$}, \,\diver\ssigma =0 \,\cL^d\text{-a.e. on $\Omega$},\\  \ssigma\cdot\vv\leq \nu(\vv)\text { for all }\vv\in\sS^{d-1}\,\cL^d\text{-a.e. on $\Omega$, }\\\ssigma\cdot \overrightarrow{n}_{\Omega}\leq 0\quad\text{$\cH^{d-1}$-a.e. on $\Gamma ^1$}
\\
\ssigma\cdot \overrightarrow{n}_{\Omega}= 0\qquad\text{$\cH^{d-1}$-a.e. on $\Gamma\setminus (\Gamma^1\cup\Gamma^2)$}\\\flow ^{cont}(\ssigma)=\phi_\Omega
\end{array} \right\}
\,.
\end{align}
Depending on the domain, the source and the sink, there might be several solutions to the continuous max-flow problem.
There is also a formulation of this continuous problem in terms of minimal cutset, but we won't present it here as we are only interested in streams. We refer to \cite{Nozawa} for more details on this formulation.
When we study law of large numbers for maximal streams, the capacity constraint comes naturally from the law of large numbers for maximal flow. Namely, a discrete stream cannot send more water that $\nu(\vv)$ in the direction $\vv$ almost surely where $\nu(\vv)$ is the flow constant defined in section \ref{sect:flowconstant}. Otherwise there exists a cylinder in the direction $\vv$ where the maximal flow properly renormalized exceeds $\nu(\vv)$, this event is very unlikely. However, when we study large deviations, we are specifically interested in these unlikely events and so the capacity constraint given by $\nu$ is not relevant anymore. Of course, if $G$ is compactly supported on $[0,M]$, the limiting streams have a capacity constraint depending on $M$, $d$ and $\vv$.
We define the set of admissible continuous streams $\Sigma(\Gamma^1,\Gamma ^2,\Omega)$ without capacity constraint as
\begin{align}\label{eq:defadmsigma}
\Sigma(\Gamma^1,\Gamma ^2,\Omega)=\left\{\ssigma\in L^\infty(\sR^d\rightarrow \sR^d,\cL^d): \begin{array}{c}\ssigma=0\text{ $\cL^d$-a.e. on $\Omega^c$}, \,\diver\ssigma =0 \,\cL^d\text{-a.e. on $\Omega$},\\  
\ssigma\cdot \overrightarrow{n}_{\Omega}= 0\qquad\text{$\cH^{d-1}$-a.e. on $\Gamma\setminus (\Gamma^1\cup\Gamma^2)$}\\\forall i\in\{1,\dots,d\}\quad |\ssigma\cdot \overrightarrow{e_i}|\leq M\quad\text{$\cL^d$-a.e. on $\Omega$}
\end{array} \right\}
\,.
\end{align}

\begin{rk}Unlike the definition of $\Sigma_\nu$, in the definition of admissible streams $\Sigma(\Gamma^1,\Gamma ^2,\Omega)$ we do not constrain the water to enter through $\Gamma^1$. Indeed, we are interested in admissible streams that are not necessarily maximal.
\end{rk}

\subsection{State of the art}

\subsubsection{Flow constant}\label{sect:flowconstant}
In 1984, Grimmett and Kesten initiated the study of maximal flows in dimension $2$ in \cite{GrimmettKesten84}.
In 1987, Kesten studied maximal flows in dimension $3$ in \cite{Kesten:flows} for straight boxes, \textit{i.e.}, in the direction $\vv=\overrightarrow{v_0}:=(0,0,1)$ and basis $A=[0,k]\times[0,l]\times \{0\}$ with $k\geq l\geq 0$. He proved the following theorem.
\begin{thm}[Kesten \cite{Kesten:flows}]Let $d=3$. Let $G$ be a distribution that admits an exponential moment and such that $G(\{0\})$ is small enough. Let $k\geq l$ and $m=m(k,l)\geq 1$.
If $m(k,l)$ goes to infinity when $k$ and $l$ go to infinity in such a way there exists $\delta\in]0,1[$, such that
\[\lim_{k,l\rightarrow \infty}\frac{1}{k^\delta}\log m(k,l)= 0\,,\]
then
\[\lim_{k,l\rightarrow \infty}\frac{\Phi\big([0,k]\times[0,l]\times\{0\},m(k,l)\big)}{k\,l}=\nu\quad\text{a.s. and in $L^1$}\,\]
where $\nu$ is a constant depending only on $d$ and $G$.
\end{thm}

\noindent The proof is very technical and tries to give a rigorous meaning to the notion of surface. Moreover, it strongly relies on the  fact that the symmetry of the straight boxes preserves the lattice, there is no hope to extend this technique to tilted cylinders. In \cite{Zhang2017}, Zhang generalized the result of Kesten for $d\geq 3$ and $G(\{0\})<1-p_c(d)$.

To be able to define the flow constant in any direction, we would like to use a subadditive ergodic theorem. Since we cannot recover a subadditive property from the maximal flow $\Phi$, we consider the flow $\tau$ instead. The simplest case to study maximal flows is still for a straight cylinder, \textit{i.e.}, when $\vv=\overrightarrow{v_0}:=(0,0,\dots,1)$ and $A=A(\overrightarrow{k},\overrightarrow{l})=\prod_{i=1}^{d-1}[k_i,l_i]\times \{0\}$ with $k_i\leq 0<l_i\in\sZ$. In this case, the family of variables $(\tau(A(\overrightarrow{k},\overrightarrow{l}),h))_{\overrightarrow{k},\overrightarrow{l}}$ is subadditive since minimal cutsets in adjacent cylinders can be glued together  along the common side of these cylinders. By applying ergodic subadditive theorems in the multi-parameter case (see Krengel and Pyke \cite{KrengelPyke} and Smythe \cite{Smythe}), we obtain the following result.

\begin{prop}\label{prop1} Let $G$ be an integrable probability measure on $[0,+\infty[$, \textit{i.e.}, $\int_0^{+\infty}xdG(x)<\infty$. Let $A=\prod_{i=1}^{d-1}[k_i,l_i]\times \{0\}$ with $k_i\leq 0<l_i\in\sZ$. Let $h:\,\sN\rightarrow \sR^+$ such that $\lim_{n\rightarrow \infty}h(n)=+\infty$. Then there exists a constant $\nu(\overrightarrow{v_0})$, that does not depend on $A$ and $h$ but depends on $G$ and $d$, such that 
$$\lim_{n\rightarrow \infty}\frac{\tau(nA,h(n))}{\cH^{d-1}(nA)}=\nu(\overrightarrow{v_0})\text{  a.s. and in $L^1$}.$$
\end{prop}
\noindent The constant $\nu(\overrightarrow{v_0})$ is called the flow constant. In fact, the property that $\nu(\vv_0)$ does not depend on $h$ is not a consequence of ergodic subadditive theorems, but the property can be proved quite easily. Next, a natural question to ask is whether we can define a flow constant in any direction. When we consider tilted cylinders, we cannot recover perfect subadditivity because of the discretization of the boundary. Moreover, the use of ergodic subadditive theorems is not possible when the direction $\vv$ we consider is not rational.  These issues were overcome by Rossignol and Théret in \cite{Rossignol2010} where they proved the following law of large numbers.

\begin{thm}[Rossignol-Théret \cite{Rossignol2010}]
 Let $G$ be an integrable probability measure on $[0,+\infty[$ , \textit{i.e.}, $\int_0^{+\infty}xdG(x)<\infty$. For any $\vv\in\sS^{d-1}$, there exists a constant $\nu(\vv)\in[0,+\infty[$ such that for any non-degenerate hyperrectangle $A$ normal to $\vv$, for any function $h:\,\sN\rightarrow \sR^+$ such that $\lim_{n\rightarrow \infty}h(n)=+\infty$, we have
 $$\lim_{n\rightarrow \infty}\frac{\tau(nA,h(n))}{\cH^{d-1}(nA)}=\nu(\vv)\text{ in $L^1$}.$$
 If moreover the origin of the graph belongs to $A$, or if  $\int_0^{+\infty}x^{1+1/(d-1)}dG(x)<\infty$, then
  $$\lim_{n\rightarrow \infty}\frac{\tau(nA,h(n))}{\cH^{d-1}(nA)}=\nu(\vv)\text{  a.s.}.$$
  If the cylinder is flat, \textit{i.e.}, if $\lim_{n\rightarrow\infty} h(n)/n=0$, then the same convergence also holds for $\Phi(nA,h(n))$.
  Moreover, either $\nu(\vv)$ is null for all $\vv \in\sS^{d-1}$ or $\nu(\vv)>0$ for all $\vv\in\sS^{d-1}$.
\end{thm}

%

\subsubsection{Upper large deviations for maximal flows in cylinders}
We present here some result on upper large deviations for the maximal flows $\Phi(nA,h(n))$ in cylinders and $\tau(nA,h(n))$. The theorem 4 in \cite{Theret:uppertau} states upper large deviations results for the variable $\Phi(nA,h(n))$ above the value $\nu(\vv)$. 
\begin{thm}[Théret \cite{Theret:uppertau}]\label{thm:theretuppertau} Let us consider a distribution $G$ on $\sR_+$ that admits an exponential moment. Let $\vv$ be a unit vector and $A$ be an hyperrectangle orthogonal to $\vv$, let $h:\sN\rightarrow \sR_+$ be a height function such that $\lim_{n\rightarrow \infty}h(n)=+\infty$. We have for every $\lambda>\nu(\vv)$
\[\liminf_{n\rightarrow\infty} -\frac{1}{\cH^{d-1}(nA)h(n)}\log\Prb\left(\frac{\Phi(nA,h(n))}{\cH^{d-1}(nA)}\geq \lambda\right) >0\,.\]
\end{thm}
Let us give an intuition of the speed of deviation. If $\Phi(nA,h(n))$ is abnormally large, there are two possible scenarios. Either there are an order $n^{d-1}$ of paths from the top to the bottom of the cylinder that use edges of slightly abnormally large capacity, or there are a fewer number of paths from the top to the bottom of the cylinder with edges whose capacities are extremely big (with a capacity that goes to infinity with $n$). Both scenarios enable to transmit more water from the top to the bottom than the expected value. Actually, when $G$ has an exponential moment, the first scenario is the most likely one. Since the paths from the top to the bottom have a cardinality of order at least $h(n)$, this implies that a positive fraction of edges inside the cylinder have a slightly abnormally large capacity. This accounts for the speed of deviation of volume order.

\begin{rk}
We insist on the fact that $\nu(\vv)$ is not in general the limit of $\Phi(nA,h(n))/\cH^{d-1}(nA)$ when $n$ goes to infinity. We can prove that the limit is equal to $\nu(\vv)$ only for straight cylinders or flat cylinders. The existence of the limit of $\Phi(nA,h(n))/\cH^{d-1}(nA)$ when $n$ goes to infinity is known when $h(n)= Cn$. The limit may be expressed as the solution of a deterministic variational problem of the same kind than $\phi_\Omega$ defined in \eqref{eq:def:phiomega}. Proving that the limit is smaller or equal than $\nu(\vv)$ is trivial. Proving the strict inequality has been done only in dimension $2$ by Rossignol and Théret in \cite{lldRT}. They proved that the limit of $\Phi(nA,h(n))/\cH^{d-1}(nA)$ is strictly smaller for tilted cylinder. We expect that this result also holds for higher dimensions but the question is still open.
\end{rk}
The corresponding large deviation principle have been obtained in the case of straight cylinders by Théret in \cite{TheretUpper}.
\begin{thm}[Théret \cite{TheretUpper}] Let $h:\sN\rightarrow \sR_+$ be a height function such that \[\lim_{n\rightarrow \infty}\frac{h(n)}{\log n}=+\infty\,.\] Set $A=[0,1]^{d-1}\times\{0\}$.
Then for every $\lambda\geq 0$, the limit
\[\psi(\lambda)=\lim_{n\rightarrow\infty}-\frac{1}{n^{d-1}h(n)}\log\Prb\left(\Phi(nA,h(n))\geq \lambda n^{d-1}\right) \]
exists and is independent of $h$. Moreover, the function $\psi$ is convex on $\sR_+$, finite and continuous on the set $\{\lambda:\,G([\lambda,+\infty[)>0\}$. If $G$ has a first moment then $\psi$ vanishes on $[0,\nu((0,\dots,0,1))]$. If $G$ has an exponential moment then $\psi$ is strictly positive on $]\nu((0,\dots,0,1)),+\infty[$, and the sequence 
$$\left(\frac{\Phi(nA,h(n))}{n^{d-1}}\right)_{n\geq 1}$$ satisfies a large deviation principle with speed $n^{d-1}h(n)$ and governed by the good rate function $\psi$.
\end{thm}
This result crucially depends on the symmetry of the lattice with regards to reflexion along the vertical faces of the cylinders. The proof strategy may not be extended to tilted cylinders. The upper large deviations results for $\tau$ are a bit different because the speed of deviation depends on the tail of the distribution $G$. Indeed if the edges around $\partial A$ have very large capacities it will increase the flow $\tau$ in a non negligible way. Since the minimal cutsets corresponding to $\tau(A,h)$ are anchored around $\partial A$, their capacity depends a lot on these edges. Théret proved in Theorem 3 in \cite{Theret:uppertau} upper large deviations of the variable $\tau$.

\begin{thm}[Théret \cite{Theret:uppertau}]  Let $\vv$ be a unit vector and $A$ be an hyperrectangle orthogonal to $\vv$, let $h:\sN\rightarrow \sR_+$ be a height function such that $\lim_{n\rightarrow \infty}h(n)=+\infty$. The upper large deviations of $\tau(nA,h(n))/\cH^{d-1}(nA)$ depend on the tail of the distribution of the capacities. We have
\begin{enumerate}[$(\roman*)$]
\item If the law $G$ has bounded support, then for every $\lambda>\nu(\vv)$ we have
\[\liminf_{n\rightarrow\infty} -\frac{1}{\cH^{d-1}(nA)\min(h(n),n)}\log\Prb\left(\frac{\tau(nA,h(n))}{\cH^{d-1}(nA)}\geq \lambda\right) >0\,.\]
\item If the law $G$ is exponential of parameter $1$, then there exists $n_0$ such that for every $\lambda>\nu(\vv)$ there exists a positive constant $D$ depending on $d$ and $\lambda$ such that 
\[\forall n\geq n_0\qquad \Prb\left(\frac{\tau(nA,h(n))}{\cH^{d-1}(nA)}\geq \lambda\right)\geq\exp(-D\cH^{d-1}(nA))\,.\]
\item If the law $G$ satisfies
\[\forall  \theta>0\qquad \int_{\sR_+}e^{\theta x}dG(x)<\infty\,,\]
then for every $\lambda>\nu(\vv)$ we have
\[\lim_{n\rightarrow\infty} \frac{1}{\cH^{d-1}(nA)}\log\Prb\left(\frac{\tau(nA,h(n))}{\cH^{d-1}(nA)}\geq \lambda\right)=-\infty\,.\]
\end{enumerate}
\end{thm}

Let us give an intuition of why the speed is $n^{d}$ when $h(n)\geq n$. Since the cutsets are anchored in $\partial (nA)$, they cannot deviate too far away from $nA$: as a result, most of the edges outside $\cyl(nA,n)$ do not have an influence on $\tau(nA,h(n))$. There is no large deviation principle for maximal flows $\tau$ or $\Phi$ in tilted cylinder.
\subsubsection{Law of large numbers for the maximal stream in a domain}
We work here with the lattice $(\sZ_n^d,\E_n^d)$.
In \cite{CT1}, Cerf and Théret proved a law of large number for the maximal flow and the maximal stream. The maximal stream converges in some sense towards the solution of the continuous max-flow problem $\phi_\Omega$. We recall that the continuous max-flow problem was presented in section \ref{sect:defmaxflow} and here the local capacity constraint corresponds to the flow constant $\nu$.
\begin{thm}\label{thm:CerfTheret}[Cerf-Théret \cite{CT1}] Let $G$ that satisfies hypothesis \ref{hypo:G} and such that $G(\{0\})<1-p_c(d)$ (to ensure that $\nu$ is a norm). Let $(\Omega,\Gamma^1,\Gamma^2)$ that satisfies hypothesis \ref{hypo:omega}. We have that the sequence $(\amu_n^{max})_{n\geq 1}$ converges weakly a.s. towards the set $\Sigma_{\nu}$ (defined in \eqref{def:sigmaomega}), that is,
\[a.s.,\,\forall f\in\sC_b(\sR^d,\sR)\qquad \lim_{n\rightarrow\infty}\inf_{\ssigma\in\Sigma_{\nu}}\left\|\int_{\sR^d}fd\amu_n^{max}-\int_{\sR^d}f\ssigma d\cL^d\right\|_2=0\,.\]
Moreover, we have
\[\lim_{n\rightarrow\infty}\frac{\phi_n(\Gamma^1,\Gamma^2, \Omega)}{n^{d-1}}=\phi_\Omega\,.\]
\end{thm}
The strategy of their proof is to first prove that the measure $\amu_n^{max}$ converges towards $\ssigma\cL^d$ where $\ssigma$ is an admissible continuous stream and that
\[\lim_{n\rightarrow\infty}\frac{\phi_n(\Gamma^1,\Gamma^2, \Omega)}{n^{d-1}}=\flow^{cont}(\ssigma)\,.\]
 These properties come from the fact that the continuous stream $\ssigma$ inherits the properties of the discrete stream $f_n ^{max}$. In particular, the local capacity constraint comes from the fact that the maximal flow in a cylinder in a direction $\vv$ properly renormalized converges towards $\nu(\vv)$ when the dimension of the cylinders goes to infinity. This implies that almost surely the stream $f_n ^{max}$ cannot send more water than $\nu(\vv)$ in the direction $\vv$. The remaining part is to prove that $\ssigma\in\Sigma_{\nu}$, \textit{i.e.}, that $\flow^{cont}(\ssigma)=\phi_\Omega$. To prove it, they need to study discrete minimal cutsets associated with $f_n^{max}$ and their continuous counterpart.
The originality of this paper is the use of new techniques by working with maximal streams instead of minimal cutsets. This object is more natural than cutsets to study upper large deviations since the upper large deviations are of volume order, whereas cutsets are $(d-1)$-dimensional objects.

Actually, the convergence of $\phi_n(\Gamma^1,\Gamma^2, \Omega)/n^{d-1}$ towards $\phi_\Omega$ when $n$ goes to infnity was already known as a consequence of the companion papers of Cerf and Théret \cite{CT4}, \cite{CT3} and \cite{CT2} with an alternative definition for $\phi_\Omega$. Instead, of expressing $\phi_\Omega$ as the solution of a variational problem for maximal stream, they expressed it as the solution $\widetilde{\phi}_\Omega$ of a variational problem for minimal continuous cutsets.
In \cite{CT2}, Cerf and Théret proved using upper large deviations result in cylinders (theorem \ref{thm:theretuppertau}) that the large deviations of $\phi_n$ is of volume order.
\begin{thm}[Cerf-Théret \cite{CT2}]\label{thm:CT2}If $d(\Gamma^1,\Gamma^2)>0$ and if the law $G$ admits an exponential moment, then there exists a constant $\widetilde{\phi}_\Omega$ such that for all $\lambda>\widetilde{\phi}_\Omega$,
$$\limsup_{n\rightarrow\infty}\frac{1}{n^d}\log\Prb\left(\phi_n(\Gamma^1,\Gamma^2, \Omega)\geq \lambda n^{d-1}\right)<0\,.$$
\end{thm}
\noindent Their strategy does not able to prove the existence of the limit of $\log\Prb\left(\phi_n(\Gamma^1,\Gamma^2, \Omega)\geq \lambda n^{d-1}\right)/n^d$ when $n$ goes to infinity.

\subsubsection{Upper large deviation principle for the first passage percolation random pseudo-metric}
We consider here the lattice $\sZ^d$. There exists another interpretation of the model of first passage percolation which has been much more studied. In this interpretation we say that the random variable $t(e)$ represents a passage time, \textit{i.e.}, the time needed to cross the edge $e$. We can define a random pseudo-metric $T$ on the graph: for any pair of vertices $x$, $y\in\sZ^d$, the random variable $T(x,y)$ is the shortest time to go from $x$ to $y$, that is,
\[T(x,y)=\inf\left\{\sum_{e\in\gamma} t(e) :\,\gamma \text{ path from $x$ to $y$}\right\}\,.\]

A natural question is to understand how this random pseudo-metric behaves. In particular, what is the asymptotic behavior of the quantity $T(0,nx)$ when $n$ goes to infinity ?  Under some assumptions on the distribution $G$, one can prove that asymptotically when $n$ is large, the random variable $T(0,nx)$ behaves like $n\, \mu(x)$ where $\mu(x)$ is a deterministic constant depending only on the distribution $G$ and the point $x$, \textit{i.e.},
\[\lim_{n\rightarrow\infty}\frac{T(0,nx)}{n}=\mu(x)\qquad\text{almost surely and in $L^1$}\] when this limit exists. This constant $\mu$ is the so-called time constant. This implies the existence of a limiting metric $D$ such that 
\[\forall x,y\in\sR^d\qquad D(x,y)=\mu(y-x)\,.\] This metric approximates well $T(x,y)$ when $\|x-y\|_2$ is large. We refer to \cite{Kesten:StFlour} and \cite{MR3729447} for reviews on the subject.

For $d=2$, let $\overrightarrow{e_1}=(1,0)$. In \cite{basu2017upper}, Basu, Ganguly and Sly study the decay of the probability of the upper large deviations event $\{T(0,n\overrightarrow{e_1})\geq (\mu(\overrightarrow{e_1})+\ep)n\}$. They prove the following result:
\begin{thm}Let $d=2$. Let $b>0$, let $G$ be a probability distribution with support $[0,b]$ and a continuous density. Then for $\ep\in]0,b-\mu(\overrightarrow{e_1})[$ there exists $r\in]0,+\infty[$ depending on $\ep$ and $G$ such that
\[\lim_{n\rightarrow \infty} - \frac{\log \Prb( T(0,n\overrightarrow{e_1})\geq (\mu(\overrightarrow{e_1})+\ep)n)}{n^2}=r\,.\]
\end{thm}
\begin{rk}Their proof strategy also holds for $d\geq 2$ and for tilted directions.

\end{rk}

The result in \cite{basu2017upper} answers an old open question that was first formulated by Kesten in \cite{Kesten:StFlour}. The correct order of large deviations was already known (see \cite{Kesten:StFlour}). A large deviation principle was proved by Chow and Zhang in \cite{Chow-Zhang} for the time between two opposite faces of a box. However, their strategy cannot be generalized for proving the existence of a rate function for the time between two points.

We here briefly present the sketch of their proof. Let $N\geq n\geq 1$. The aim is to build the upper large deviations event at the higher scale $N$ using upper large deviations events at the smaller scale $n$.
 Let us define $B_N=[-N,N]^d$. Since the passage times are bounded, there exists a positive constant $c$ depending on $b$, such that geodesics between $0$ and $N\overrightarrow{e_1}$ remain almost surely in the box $B_{cN}$. The strategy of the proof is to create a configuration $\omega$ of the edges in the box of size $B_{cN}$ such that $\omega\in\{T(0,N\overrightarrow{e_1})\geq (\mu(\overrightarrow{e_1})+\ep)N\}$ using configurations of upper large deviations events at the smaller scale $n$. Namely, we consider $\omega_1,\dots,\omega_{(N/n)^2}$ $(N/n)^2$ independent realizations of the edges in the box $B_{cn}$ for the event $\{T(0,n\overrightarrow{e_1})\geq (\mu(\overrightarrow{e_1})+\ep)n\}$. 
The key idea is that, even on the upper large deviations event, there exists a limiting metric structure in the configurations $\omega_i$. Roughly speaking, at large scales the distance $T(x,x+n\vv)$ in a given direction $\vv$ from a given point $x$ grows linearly with speed $\nabla_x(\vv)$, \textit{i.e.}, we have $T(x,x+n\vv)\approx n \nabla_x(\vv)$.  
Up to paying a negligible price, they can pick configurations $(\omega_i, i=1,\dots,(N/n)^2)$ with the same limiting metric. 
Each configuration $\omega_i$ is cut into different regions such that for any $x,y$ in a given region $\nabla_x\approx \nabla_y$.  They reassemble all the configurations $\omega_i$ by gluing together the corresponding regions, in order to create a configuration $\omega$ in the box $B_{cN}$ that also has the same limiting metric. It follows that for any path $\gamma$ in the configuration $\omega$ we can build a path $\widetilde{\gamma}$ in the configuration $\omega_1$ such that
$$T(\gamma)\approx\frac{N}{n}T(\widetilde{\gamma})\,.$$
The path $\gamma$ is the dilated version of $\widetilde{\gamma}$. It follows that
$$T(\gamma)\geq\frac{N}{n}T(0,n\overrightarrow{e_1})\geq (\mu(\overrightarrow{e_1})+\ep)N\,.$$
The remaining of the proof uses techniques from large deviations theory to deduce the existence of a rate function.

Our paper finds its inspiration in the philosophy of \cite{basu2017upper}: we use large deviations events at a small scale to build large deviations events at a higher scale. We here manage to formalize the idea of a limiting environment. We obtain something stronger than upper large deviations for maximal flow: we manage to relate an abnormally large flow with local abnormalities on the domain $\Omega$. To obtain this stronger result, we need to deal with complex technical issues: in particular, we need to reconnect streams in adjacent cubes.
\subsection{Sketch of the proof }
Most of the proofs in this paper are about reconnecting streams defined in cubes. Let $A$ be an hyperrectangle of dimension $d-1$ of side length $\kappa>0$ normal to $\overrightarrow{e_1}=(1,0,\dots,0)$. We consider two streams $f_n\in\cS_n(\cyl(A,\kappa,\overrightarrow{e_1}))$ and $g_n\in\cS_n(\cyl(A+(\kappa+\delta)\overrightarrow{e_1},\kappa,\overrightarrow{e_1}))$ for some $\delta>0$. We would like to exploit the region $\cyl(A+\kappa \overrightarrow{e_1},\delta,\overrightarrow{e_1})$ between these two cubes -that we call the corridor- to connect the streams $f_n$ and $g_n$ (see figure \ref{figintro}). Namely, we would like to prove the existence of a stream $h_n$ such that $h_n\in\cS_n(\cyl(A,2\kappa+\delta,\overrightarrow{e_1}))$, $h_n=f_n$ in $\cyl(A,\kappa,\overrightarrow{e_1})$ and $h_n=g_n$ in $\cyl(A+(\kappa+\delta)\overrightarrow{e_1},\kappa,\overrightarrow{e_1})$. Moreover, we want that no water exists or enters from the lateral sides of the corridor for $h_n$. In particular the stream $h_n$ satisfies the node law in the corridor. 
Note that a necessary condition for the existence of $h_n$ is that the flow for $f_n$ through the face $A+\kappa \overrightarrow{e_1}$ is equal to the flow for $g_n$ through the face $A+(\kappa+\delta)\overrightarrow{e_1}$ (we say that their flow match). Indeed, if such $h_n$ exists, since it satisfies the node law inside the corridor and that no flow escapes from its lateral sides, the flows of $f_n$ and $g_n$ must match.
\begin{figure}[H]
\begin{center}
\def\svgwidth{0.6\textwidth}
   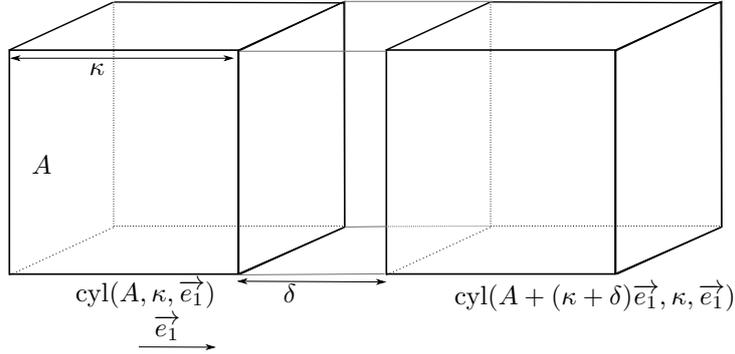
   \caption{\label{figintro}Connecting two streams in a cylinder}
   \end{center}
\end{figure}
The ideal situation to connect the streams is to take $\delta=0$ but this is too restrictive because it requires that the outputs of $f_n$ perfectly match the inputs of $g_n$. The outputs (respectively inputs) correspond to the values of $f_n$ (respectively $g_n$) for edges exiting (respectively entering)  the cylinder by the corridor. However, it seems reasonable that if the capacities of the edges in the corridor are large enough and $\delta$ is large enough, then as long as their flow match we can reconnect $f_n$ and $g_n$. This is the key property: \newline
\noindent {\bf Key property.} There exists a constant $c_d$ depending only on $d$ such that for $f_n$ and $g_n$ with matching flows, if their inputs and outputs are all smaller in absolute value than some constant $b>0$, then we can always connect the two streams as long as $\delta>\kappa c_d$ and all the edges inside the corridor have capacity larger than or equal to $b$.

%
%
%

The first step before proving theorem \ref{thm:ULDpatate} is to prove the existence of an elementary rate function: theorem \ref{thmbrique}.

\subsubsection{Sketch of the proof of theorem \ref{thmbrique}: existence of an elementary rate function}
We recall that 
$$\fC=\left[-\frac{1}{2},\frac{1}{2}\right[^d\,.$$
Consider a stream $f_n\in\cS_n(\fC)$ that is close to a continuous stream $s\vv$ (in the sense that $\amu_n(f_n)$ is close to $s\vv\ind_\fC\cL^d$ for the distance $\dis$). \newline
\noindent {\bf Step 1.} We first prove that at a mesoscopic level the flow of $f_n$ through each face of $\fC$ is almost uniform and close to the flow for the continuous stream $s\vv$. Namely, for each face $F$ of $\fC$, we can split $F$ into a collection $\cP(F)$ of small isometric $(d-1)$-dimensional hypercubes of side-length $\kappa$. For each $C\in\cP(F)$ the quantity $\psi(f_n,C)$ of water that flow through $C$ for $f_n$ is close to the flow for the continuous stream $s\vv$,\textit{i.e.},
 $$\psi(f_n,C)\approx n^{d-1}s\vv\cdot \overrightarrow{e_i}\cH^{d-1}(C)$$ where $\overrightarrow{e_i}$ is the vector of the canonical basis that is normal to $C$.\newline
\noindent {\bf Step 2.} We prove that up to paying a negligible price, we can increase the capacities of a negligible number of edges in $\fC$ in such a way we guarantee the existence of a stream $\widetilde{f}_n$  such that for each face $F$ of $\fC$, for each $C\in\cP(F)$,
\[\psi(\widetilde{f}_n,C)=n^{d-1}s\vv\cdot \overrightarrow{e_i}\cH^{d-1}(C)\,.\]
We call such a stream a well-behaved stream. To build such a stream we do small modifications to $f_n$ to ensure that the water spreads uniformly at the mesoscopic level. To do so we increase the capacities for a negligible portion of the edges in order to add a small amount of water that will correct the differences of flow with the continuous stream. Since these corrections are small, the modified stream $\widetilde{f}_n$ is still close to the continuous stream $s\vv$. The price we have to pay to modify the original configuration is negligible since only a negligible portion of the edges have been modified, \textit{i.e.}, these modifications won't appear in the limit.\newline
\noindent {\bf Step 3.} We prove theorem \ref{thmbrique}. We fix $N\geq n$. We consider $(N/n)^d$ different configurations of the event $\{\exists f_n\in\cS_n(\fC):f_n\approx s\vv \text {  and $f_n$ is well-behaved}\}$. Using these configurations, we connect the streams in order to create a stream in $\cS_N(\fC)$ that is close to $s\vv$. Since the streams we consider at the small scale are well-behaved, the water flow uniformly at a mesoscopic level and we are able to connect at the macroscopic level two adjacent streams by using the key property for connection at the mesoscopic level for each $C\in\cP(F)$ (see figure \ref{fig6}). The length of the corridor we need to connect two adjacent streams is $c_d\kappa$ that is negligible for small $\kappa$. The remaining of the proof uses standard techniques from the realm of large deviations.

\begin{figure}[H]
\begin{center}
\def\svgwidth{0.6\textwidth}
   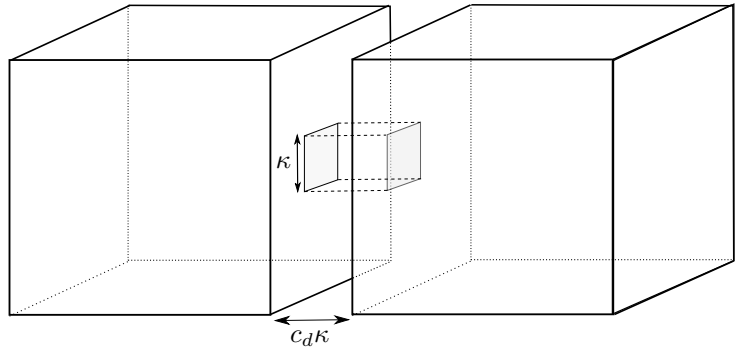
   \caption{\label{fig6}Connecting streams in cubes at mesoscopic level}
   \end{center}
\end{figure}
\subsubsection{Sketch of the proof of theorem \ref{thm:ULDpatate}} We aim to prove that for $\ssigma\in\Sigma(\Gamma^1,\Gamma ^2,\Omega)$, we have \[\Prb\big(\exists f_n\in\cS_n(\Gamma ^1,\Gamma ^2,\Omega): f_n\approx \ssigma\big)\approx \exp\big(-\widehat{I}(\ssigma)n^d\big)\]
where $\widehat{I}(\ssigma)=\int_{\Omega}I(\ssigma(x))d\cL^d(x)$. To prove this result we prove separately an upper and a lower bound on the probability we try to estimate. We can till $\Omega$ into a family of small cubes $\cE$ with disjoint interiors such that $\ssigma$ is almost constant in each cube $C\in\cE$. Using the independence of the capacities, we have
\[\Prb\big(\exists f_n\in\cS_n(\Gamma ^1,\Gamma ^2,\Omega): f_n\approx \ssigma\big)\leq \prod_{C\in\cE}\Prb(\exists f_n\in\cS_n(C):f_n\approx \ssigma\ind_C)\,.\]
Since $\ssigma$ is almost constant in $C\in\cE$, we can use theorem \ref{thmbrique} and prove that 
$$\prod_{C\in\cE}\Prb(\exists f_n\in\cS_n(C):f_n\approx \ssigma\ind_C)\approx \exp\left(-n^d\int_{\Omega}I(\ssigma(x))d\cL^d(x)\right)\,.$$
To prove the upper bound, we deconstruct a stream in $\Omega$, to prove the lower bound we do the reverse: we construct a stream in $\Omega$ close to $\ssigma$ from a collection of streams inside small cubes. For each cube $C$, we consider a discrete stream $g_n^C$ in $\cS_n(C)$ that is close to the constant approximation of $\ssigma$ in $C$. We use the ideas of corridors and well-behaved streams to reconnect these streams $(g_n^C,C\in\cE)$ altogether in order to create a stream $f_n\in\cS_n(\Omega)$ that is close to $\ssigma$. The main difficulty in the proof of the lower bound is to create from $f_n$ a stream $\overline{f}_n\in\cS_n(\Gamma^1,\Gamma^2,\Omega)$, \textit{i.e.}, to remove all the water that is entering or exiting through $\Gamma \setminus (\Gamma ^1\cup\Gamma^2)$ for $f_n$ and make sure that $\overline{f}_n$ is still close to $\ssigma$. This is the most technical part of the proof.
  \newline
\subsubsection{Organization of the paper.}

In section \ref{sect:properties}, we give some properties of the distance $\dis$, we also give necessary conditions on the stream $\ssigma$ in order to have $\widehat{I}(\ssigma)<\infty$. In section \ref{sect: technicallem}, we gather all the technical non probabilistic lemmas. In section \ref{sect:basicbrick}, we prove the existence of the elementary rate function $I$ by proving theorem \ref{thmbrique} and we also prove the convexity of $I$. In section \ref{sect:ULD}, we prove theorem \ref{thm:ULDpatate}, that is the key result to prove theorem \ref{thm:pgd}. Finally, in section \ref{sect:goodtaux}, we prove the theorem \ref{thm:pgd} and we deduce an upper large deviation principle for the maximal flow theorem \ref{thm:uldmf}.

\section{Properties of the distance and of the admissible continuous streams}\label{sect:properties}
In this section, we introduce the metric we use and derive some properties for the limiting continuous streams. 
\subsection[Properties of the metric]{Properties of the metric $\dis$}
We recall that $\dis$ was defined in \eqref{eq:def:dis}. We state here some key properties that this distance satisfies that will be useful in what follows. The proofs of the following lemmas will be given after their statements.
The convergence for the distance $\dis$ of a sequence of measures that are uniformly bounded in the total variation norm implies the weak convergence. We recall that for a signed-measure $\nu$ on $\sR^d$, we write $\nu=\nu^+-\nu^-$ for the Hahn-Jordan decomposition of $\nu$ and we write $|\nu|$ the total variation of $\nu$ defined as $|\nu|=\nu^++\nu ^-$.
\begin{lem}\label{lem:weakconv}Let $\nu=(\nu^1,\dots,\nu ^d)\in \cM(\overline{\cV_\infty(\Omega,1)})^d$. Let $(\nu_n=(\nu_n^1,\dots,\nu_n ^d))_{n\geq 1}$ be a sequence of measures in $\cM(\overline{\cV_\infty(\Omega,1)})^d$ such that $$\lim_{n\rightarrow\infty}\dis(\nu_n,\nu)=0\,.$$ 
Moreover, suppose that the measures $\nu$ and $(\nu_n)_{n\geq 1}$ are uniformly bounded in the total variation norm, that is there exists a positive constant $C_1$ such that
$$\forall n\geq 1\qquad \sum_{i=1}^d|\nu_n^i|(\overline{\cV_\infty(\Omega,1)})\leq C_1 \cL^d(\overline{\cV_\infty(\Omega,1)})\,,$$
and 
$$\sum_{i=1}^d|\nu^i|(\overline{\cV_\infty(\Omega,1)})\leq C_1 \cL^d(\overline{\cV_\infty(\Omega,1)})\,.$$
Then, the sequence of measure $(\nu_n)_{n\geq 1}$ weakly converges towards $\nu$, that is 
$$\forall f\in\sC_b(\sR^d,\sR)\qquad \lim_{n\rightarrow \infty }\int_{\sR^d}fd\nu_n=\int_{\sR^d}fd\nu\,.$$
\end{lem}
The converse result holds for a sequence of measures absolutely continuous with respect to Lebesgue measure.
\begin{lem}\label{lem:weakconvautresens}Let $M>0$. Let $\nu=h\cL^d$ with $h\in L ^\infty(\sR^d\rightarrow\sR^d,\cL^d)$ such that
$$\|h\|_2\leq M\quad\text{a.e. on $\sR^d$},\quad h=0\quad\text{a.e. on $\cV_2(\Omega,1)^c$}.$$
$\bullet$  Let $(h_n)_{n\geq 1}$ be a sequence of functions in  $L ^\infty(\sR^d\rightarrow\sR^d,\cL^d)$ such that
$$\forall n\geq 1\qquad \|h_n\|_2\leq M\quad\text{a.e. on $\sR^d$}\quad h_n=0\quad\text{a.e. on $\cV_2(\Omega,1)^c$}$$
and the sequence of measure $(h_n\cL^d)_{n\geq 1}$ weakly converges towards $h\cL^d$, that is 
$$\forall g\in\sC_b(\sR^d,\sR)\qquad \lim_{n\rightarrow \infty }\int_{\sR^d}gh_nd\cL ^ d=\int_{\sR^d}ghd\cL^d\,$$
then, 
$$\lim_{n\rightarrow\infty}\dis(h_n\cL^d,h\cL^d)=0\,.$$

\noindent $\bullet$ Let $(f_n)_{n\geq 1}$ be a sequence of streams inside $\Omega$ such that 
$$\forall e\in\E_n^d\qquad\|f_n(e)\|_2\leq M$$ and
the sequence of measures $(\amu_n(f_n))_{n\geq 1}$ weakly converges towards $h\cL^d$.
Then,
$$\lim_{n\rightarrow\infty}\dis(\amu_n(f_n),h\cL^d)=0\,.$$

\end{lem}

We can control the distance between two measures that are absolutely continuous with respect to the Lebesgue measure by the $L^1$-distance.
\begin{lem}\label{lem:propdis2}Let $f,g\in L^1(\sR^d\rightarrow \sR^d,\cL^d)$. We have
\[\dis(f\cL^d,g\cL^d)\leq 2 \int_{\sR^d}\|f(x)-g(x)\|_2d\cL^d(x)= 2\|f-g\|_{L^1}\,.\]
\end{lem}
We say that $(E_i)_{i\geq 1}$ is a paving of $\sR^d$ if the sets $E_i$ are of pairwise disjoint interior, for $i\geq 1$, the set $E_i$ is a translate of $E_1$ and $\sR^d= \cup_{i\geq 1}E_i$.
For a subset $E$ of $\sR^d$, we denote by $\diam E$ its diameter, \textit{i.e.},
$$\diam E = \sup\left\{\|x-y\|_2:\,x,y\in E\right\}\,.$$
The following lemma will be very useful in what follows, it enables to control the number of elements of a paving that intersect the boundary of a given cube.
\begin{lem}\label{lem:interbord}There exists a positive constant $\ep_\fC$ depending only on the dimension such that for any $\delta \in ]0,1[$ and $z\in\sR^d$, for any paving $(E_i)_{i\geq 1}$ of $\sR^d$ such  that $\diam E_1\leq \ep_{\fC}\delta$, we have
\[\left|\left\{i \geq 1:E_i\cap (\partial(\delta\fC+z))\right\}\right|\leq 2\frac{\cH^{d-1}(\partial(\delta\fC+z))}{\cL^d(E_1)}\diam E_1\,.\]

\end{lem}

The result of the following lemma is a key property of the distance $\dis$ that does not necessarily hold for standard distances: if the distance $\dis(\nu,\mu)$ is small then for a cube $Q\subset \Omega$, the distance $\dis(\nu\ind_Q,\mu\ind_Q)$ is also small.

\begin{lem}\label{lem:propdis3}Let $M>0$. Let $G$ be a distribution such that $G([M,+\infty[)=0$. Let $\nu\in \cM(\overline{\cV_\infty(\Omega,1)})^d$ such that 
\[\forall x\in[-1,1[^d\quad\forall \lambda\in[1,2]\quad\forall k\geq 0\quad  \forall Q\in(\Delta_\lambda^k+x)\qquad \|\nu(Q)\|_2\leq M\cL^d(Q)\,.\]
There exist positive constants $\beta_1$, $\beta_2$ depending only on $M$, $\Omega$ and $d$, and $\ep_\fC$ depending on $d$ such that for any $\delta\in[0,1]$ and $z\in\sR^d$, we have  for any $\rho\leq \delta\ep_\fC$, for $n$ large enough depending on $\rho$, for any $f_n\in\cS_n(\Omega)$
$$ \dis(\amu_n(f_n)\ind_{\delta\fC+z},\nu\ind_{\delta\fC+z})\leq \beta_1 \frac{\dis(\amu_n(f_n),\nu)}{\rho }+ \beta_2\rho\delta ^{d-1}\,.$$
\end{lem}
The following lemma implies that to upper-bound the distance between two measures $\mu$, $\nu$, given a partition of $\Omega$, it is sufficient to upper-bound separately the distance $\dis(\mu\ind_ A,\nu\ind_A)$ on each set $A$ of the partition.
\begin{lem}\label{lem:propdis4} Let $\mu,\nu\in \cM(\overline{\cV_\infty(\Omega,1)})^d$. Let $(A_i,1\leq i\leq r)$ be a family of pairwise disjoint subsets of $\sR^d$ such that $$\overline{\cV_\infty(\Omega,1)}\subset\bigcup_{i=1}^rA_i\,.$$
Then, we have
$$\dis(\mu,\nu)\leq \sum_{i=1}^r\dis(\mu\ind_{A_i},\nu\ind_{A_i})\,.$$
\end{lem}

We now prove the lemmas above.
\begin{proof}[Proof of lemma \ref{lem:weakconv}]
Let $f\in\sC_b(\sR^d,\sR)$. Let $\ep>0$.  For $k\geq 1$, we set $$f_k=\sum_{Q\in\Delta^k_1 (\Omega)}f(c(Q))\ind_Q\,,$$
where $c(Q)$ denotes the center of $Q$.
Since the function $f$ is uniformly continuous on the compact set $\overline{\cV_\infty(\Omega,2)}$, we fix $k$ large enough (depending on $f$ and $\ep$) such that
$$\forall Q\in\Delta^k_1(\Omega)\quad \forall x\in Q\qquad |f_k(x)-f(x)|\leq \ep\,.$$
Besides, we have
\begin{align}
&\left\|\int_{\sR^d}fd\nu_n-\int_{\sR^d}fd\nu\right\|_2\nonumber\\
&\hspace{2cm}\leq \left\|\int_{\sR^d}fd\nu_n-\int_{\sR^d}f_kd\nu_n\right\|_2+\left\|\int_{\sR^d}f_kd\nu_n-\int_{\sR^d}f_kd\nu\right\|_2+\left\|\int_{\sR^d}fd\nu-\int_{\sR^d}f_kd\nu\right\|_2\nonumber \\
&\hspace{2cm}\leq \ep\sum_{i=1}^d |\nu_n^i|(\overline{\cV_\infty(\Omega,1)})+\|f\|_\infty\sum_{Q\in\Delta^k_1(\Omega)}\,\|\nu_n(Q)-\nu(Q)\|_2+\ep \sum_{i=1}^d |\nu_n^i|(\overline{\cV_\infty(\Omega,1)})\nonumber\\
&\hspace{2cm}\leq 2\ep C_1\cL^d(\overline{\cV_\infty(\Omega,1)})+2^{k} \|f\|_\infty\,\dis(\nu_n,\nu)\,.
\end{align}
Hence for $n$ large enough,
$$\left\|\int_{\sR^d}fd\nu_n-\int_{\sR^d}fd\nu\right\|_2\leq 3\ep C_1\cL^d(\overline{\cV_\infty(\Omega,1)})\,.$$
This yields the result.
\end{proof}
\begin{proof}[Proof of lemma \ref{lem:weakconvautresens}]  Let $h$ in  $L ^\infty(\sR^d\rightarrow\sR^d,\cL^d)$ and $(h_n)_{n\geq 1}$ be a sequence of functions in  $L ^\infty(\sR^d\rightarrow\sR^d,\cL^d)$ as in the statement of lemma \ref{lem:weakconvautresens}.
Let $\ep>0$. Let $z\in[-1,1[^d$, $\lambda\in[1,2]$.
For $k_0\geq 1$ large enough depending on $\ep$ and $\Omega$, we have
\begin{align*}
\sum_{k=k_0}^\infty\frac{1}{2^{k}}\sum_{Q\in(z+\Delta^k_\lambda) }\left\|(h_n\cL^d)(Q+x)-(h\cL^d)(Q+x)\right\|_2&\leq \sum_{k=k_0}^\infty\frac{1}{2^{k}}\sum_{Q\in(z+\Delta^k_\lambda) }2M\cL^d(Q)\ind_{Q\cap \cV_2(\Omega,1)\neq\emptyset}\\
&\leq \sum_{k=k_0}^\infty\frac{1}{2^{k}}2M \cL^d(\cV_2(\Omega,3))\\
&\leq 4M\cL^d(\cV_2(\Omega,3)) 2^{-k_0}\leq \ep\,.
\end{align*}
We aim at obtaining a uniform control in $z$ and $\lambda$ of
$$\sum_{k=0}^{k_0}\frac{1}{2^{k}}\sum_{Q\in(z+\Delta^k_\lambda) }\left\|(h_n\cL^d)(Q+x)-(h\cL^d)(Q+x)\right\|_2\,.$$
Let $\delta >0$ such that $d \delta\leq 2^{-{k_0}}$. Let $B=\delta(y+\fC)$ with $y\in\sZ^d$. 
Since $(h\cL^d)(\partial B)=0$, we have by Portmanteau theorem 
\begin{align}\label{eq:weakconvcyl}
\lim_{n\rightarrow \infty }\|(h_n\cL^d)(B)-(h\cL^ d)(B)\|_2=0\,.
\end{align}
Besides, using lemma \ref{lem:interbord}, we have
\begin{align*}
\sum_{k=0}^{k_0}&\frac{1}{2^{k}}\sum_{Q\in(z+\Delta^k_\lambda) }\left\|(h_n\cL^d)(Q+x)-(h\cL^d)(Q+x)\right\|_2\\&\leq \sum_{k=0}^{k_0}\frac{1}{2^{k}}\sum_{\substack{Q\in(z+\Delta^k_\lambda):\\Q\cap\cV_2(\Omega,1) }}\sum_{\substack{y\in\sZ^d:\\B= \delta(y+\fC)\subset Q}}\|(h_n\cL^d)(B)-(h\cL^d)(B)\|_2+\sum_{\substack{y\in\sZ^d:\\B= \delta(y+\fC)\text{ s.t. }B\cap\partial Q\neq \emptyset}}2M\delta ^d \\
&\leq 2\sum_{\substack{y\in\sZ^d:\\B= \delta(y+\fC)\cap \cV_2(\Omega,1)\neq\emptyset}}\|(h_n\cL^d)(B)-(h\cL^d)(B)\|_2+ \sum_{k=0}^{k_0}\frac{1}{2^{k}}\sum_{\substack{Q\in(z+\Delta^k_\lambda):\\Q\cap\cV_2(\Omega,1) \neq\emptyset}} 8dM\cH^{d-1}(\partial Q)\delta\\
&\leq 2\sum_{\substack{y\in\sZ^d:\\B= \delta(y+\fC)\cap \cV_2(\Omega,1)\neq\emptyset}}\|(h_n\cL^d)(B)-(h\cL^d)(B)\|_2 +\sum_{k=0}^{k_0}\frac{1}{2^{k}}\frac{\cL^d(\cV_2(\Omega,3))}{(\lambda 2^{-k}) ^d}2d(\lambda 2^{-k}) ^{d-1}8dM\delta\\
&\leq 2\sum_{\substack{y\in\sZ^d:\\B= \delta(y+\fC)\cap \cV_2(\Omega,1)\neq\emptyset}}\|(h_n\cL^d)(B)-(h\cL^d)(B)\|_2+(k_0+1)\cL^d(\cV_2(\Omega,3))16d^2M\delta\,.
\end{align*}
It follows that
$$\dis(h_n\cL^d,h\cL^d)\leq \ep+2\sum_{\substack{y\in\sZ^d:\\B= \delta(y+\fC)\cap \cV_2(\Omega,1)\neq\emptyset}}\|(h_n\cL^d)(B)-(h\cL^d)(B)\|_2+(k_0+1)\cL^d(\cV_2(\Omega,3))16d^2M\delta\,.$$
By taking the limsup in $n$, we obtain 
$$\limsup_{n\rightarrow \infty}\dis(h_n\cL^d,h\cL^d)\leq \ep+k_0\cL^d(\cV_2(\Omega,3))16d^2M\delta\,.$$
By first letting $\delta$ goes to $0$ and then 
by letting $\ep$ goes to $0$, we obtain
$$\lim_{n\rightarrow \infty}\dis(h_n\cL^d,h\cL^d)=0\,.$$
This yields the result.
The same arguments may be adapted in the case of a sequence $(\amu_n(f_n))_{n\geq1}$ using the fact that for any $B\in(x+\Delta_\lambda ^k)$ for $n$ large enough
$$\left\|\amu_n(f_n)(B)\right\|_2\leq 3d\cL^d(B)M\,.$$
\end{proof}
\begin{proof}[Proof of lemma \ref{lem:propdis2}]
Write $\mu=f\cL^d$ and $\nu=g\cL^d$.
Let $x\in[-1,1[^d$, $\lambda\in[1,2]$.
We have

\begin{align*}
\sum_{k=0}^\infty\frac{1}{2^{k}}\sum_{Q\in\Delta^k_\lambda }\left\|\mu(Q+x)-\nu(Q+x)\right\|_2&=\sum_{k=0}^\infty\frac{1}{2^{k}}\sum_{Q\in\Delta^k_\lambda }\left\|\int_{Q+x}(f(y)-g(y))d\cL^d(y)\right\|_2\\
&\leq \sum_{k=0}^\infty\frac{1}{2^{k}}\sum_{Q\in\Delta^k_\lambda }\int_{Q+x} \left\|f(y)-g(y)\right\|_2d\cL^d(y)\\
&=\sum_{k=0}^\infty\frac{1}{2^{k}}\int_{\sR^d} \left\|f(y)-g(y)\right\|_2d\cL^d(y)\\
&\leq 2\|f-g\|_{L^1}\,.
\end{align*}
It follows that 
$$\dis(f\cL^d,g\cL^d)\leq 2\|f-g\|_{L^1}\,.$$
This yields the result.
\end{proof}
\begin{proof}[Proof of lemma \ref{lem:interbord}]
 Write $B=\delta\fC+z$. By proposition \ref{prop:minkowski}, there exists a positive constant $\ep_{\fC}$ depending on $\fC$ such that
$$\forall \ep_0\in[0, \ep_\fC]\qquad\frac{\cL^d(\cV_2(\partial \fC,\ep_0))}{2\ep_0}\leq 2\cH^{d-1}(\partial \fC)\,.$$
It follows that 
$$\forall \ep_0\in[0, \ep_\fC\delta ]\qquad\frac{\cL^d(\cV_2(\partial B,\ep_0))}{2\ep_0}\leq 2\cH^{d-1}(\partial B)\,.$$
Let  $(E_i)_{i\geq 1}$ be a paving of $\sR^d$ such that  $\diam E_1\leq \ep_{\fC}\delta$, we have
\[\left|\left\{i \geq 1:E_i\cap (\partial(\delta\fC+z))\right\}\right|\leq \frac{\cL^d(\cV_2(\partial B,\diam E_1))}{\cL^d(E_1)}\leq 4\frac{\cH^{d-1}(\partial(\delta\fC+z))}{\cL^d(E_1)}\diam E_1\,.\]
This yields the result.
\end{proof}
\begin{proof} [Proof of lemma \ref{lem:propdis3}]
Let $\nu$ that satisfies the conditions in the statement of the lemma \ref{lem:propdis3}. Let $\delta\in[0,1]$ and $z\in\sR^d$. Write $B=\delta\fC+z$.

Let $w\in[-1,1[^d$ and $\lambda\in[1,2]$. Let $\ep_\fC$ be given by lemma \ref{lem:propdis3}.  Let $\rho\leq \ep_\fC \delta$.  Let $f_n\in\cS_n(\Omega)$. Write $\amu_n=\amu_n(f_n)$.
Let $j$ be the smallest integer such that $d\lambda2^{-j}\leq\rho$. Hence, $w+\Delta_j^\lambda$ is a paving of $\sR^d$ such that for any $Q\in \Delta_j^\lambda$, we have $\diam Q\leq d\lambda 2^{-j}\leq \ep_\fC \delta$.
Using lemma \ref{lem:interbord}, we have
\begin{align*}
\sum_{k=0}^\infty&\frac{1}{2^{k}}\sum_{Q\in(\Delta_k^\lambda+w)}\|\amu_n(B\cap Q)-\nu(B\cap Q)\|_2\\
&\leq \sum_{k=0}^j\frac{1}{2^{k}}\sum_{Q\in(\Delta_k^\lambda+w)}\|\amu_n(B\cap Q)-\nu(B\cap Q)\|_2+\sum_{k=j+1}^\infty\frac{1}{2^{k}}\sum_{\substack{Q\in(\Delta_k^\lambda+w):\\ Q\cap B\neq \emptyset}}
\|\amu_n(B\cap Q)-\nu(B\cap Q)\|_2\\
&\leq \sum_{k=0}^j\frac{1}{2^{k}}\sum_{Q\in(\Delta_j^\lambda+w)}\|\amu_n(B\cap Q)-\nu(B\cap Q)\|_2+\sum_{k=j+1}^\infty\frac{1}{2^{k}}\sum_{\substack{Q\in(\Delta_k^\lambda+w):\\ Q\cap B\neq \emptyset}}
\left(\|\amu_n(B\cap Q)\|_2 +\|\nu(B\cap Q)\|_2\right)\\
&\leq \sum_{k=0}^j\frac{1}{2^{k}}\left(\sum_{Q\in(\Delta_j^\lambda+w): Q\subset B}\|\amu_n( Q)-\nu( Q)\|_2+\sum_{Q\in(\Delta_j^\lambda+w): Q\cap \partial B\neq\emptyset}\left(\|\amu_n(B\cap Q)\|_2 +\|\nu(B\cap Q)\|_2\right)\right)\\
&\qquad+\sum_{k=j+1}^\infty\frac{1}{2^{k}}\left(\frac{1}{n^d}\sum_{e\in\E_n^d: c(e)\in\cV_2(B,\rho)}\|f_n(e)\|_2+ M\cL^d( 2\delta \fC+z)\right)\\
&\leq 2\sum_{Q\in(\Delta_j^\lambda+w): Q\subset B}\|\amu_n( Q)-\nu( Q)\|_2+\frac{2}{n^d}\sum_{\substack{e\in\E_n^d:\\ c(e)\in\cV_2(\partial B, d\lambda 2^{-j})}}\|f_n(e)\|_2\\
&\qquad + 2M\cL^d (\lambda 2^{-j}\fC)|\{Q\in(\Delta_j^\lambda+w): Q\cap \partial B\neq\emptyset\}|+\frac{M}{2^{j}}(2d +1)2^d\delta^d\\
&\leq 2^{j+1}\dis(\amu_n,\nu)+8Md\cH^{d-1}(\partial B)\rho+8M \cH^{d-1}(\partial B)\rho+\frac{M}{2^{j}}(2d +1)2^d\delta^d\\
&\leq \frac{4d\lambda}{\rho}\dis(\amu_n,\nu)+M\left( 16d^2+16d +(2d+1)2^d\right)\rho\delta^{d-1}
\end{align*}
where we use in the last inequality that by definition of $j$, we have $d\lambda2^{-j+1}>\rho$. It follows that 
$$\dis(\amu_n(f_n)\ind_B,\nu\ind_B)\leq \beta_1 \frac{\dis(\amu_n(f_n),\nu)}{\rho }+ \beta_2\rho\delta ^{d-1}\,.$$
where $\beta_1$ and $\beta_2$ are positive constant depending only on $M$ and $d$.
\end{proof}
\begin{proof} [Proof of lemma \ref{lem:propdis4}]
Let $\mu,\nu\in\cM(\overline{\cV_\infty(\Omega,1)})^d$ and $(A_i,1\leq i\leq r)$ be a family of pairwise disjoint subsets of $\sR^d$ such that
 $$\overline{\cV_\infty(\Omega,1)}\subset\bigcup_{i=1}^rA_i\,.$$
We have for $w\in[-1,1[^d$ and $\lambda\in[1,2]$
\begin{align*}
\sum_{k=0}^\infty\frac{1}{2^{k}}\sum_{Q\in(\Delta_k^\lambda+w)}\|\mu( Q)-\nu( Q)\|_2
&\leq \sum_{k=0}^\infty\frac{1}{2^{k}}\sum_{Q\in(\Delta_k^\lambda+w)}\sum_{i=1}^r\|\mu( Q\cap A_i)-\nu( Q\cap A_i)\|_2\\
&=\sum_{i=1}^r \sum_{k=0}^\infty\frac{1}{2^{k}}\sum_{Q\in(\Delta_k^\lambda+w)}\|\mu( Q\cap A_i)-\nu( Q\cap A_i)\|_2 \\
&\leq\sum_{i=1}^r\dis(\mu\ind_{A_i},\nu\ind_{A_i})\,.
\end{align*}
It follows that 
$$\dis(\mu,\nu)\leq \sum_{i=1}^r\dis(\mu\ind_{A_i},\nu\ind_{A_i})\,.$$
This yields the result.
\end{proof}

\subsection{Properties of the admissible streams}
The aim of this section is to prove properties that the continuous streams $\ssigma$ must be in $\Sigma(\Gamma^1,\Gamma ^2,\Omega)$ to get $\widehat{I}(\ssigma)<\infty$.
Denote by $\Sigma^M(\Gamma^1,\Gamma ^2,\Omega)$ the following set 
\begin{align}\label{eq:defSigmaM}
\Sigma^M(\Gamma^1,\Gamma ^2,\Omega)=\left\{\ssigma\in L^\infty(\sR^d\rightarrow\sR^d,\cL^d):\begin{array}{c}\exists \psi:\sN\rightarrow\sN \quad\forall n\geq 1 \quad\exists f_{\psi(n)}\in \cS_{\psi(n)}^M(\Gamma^1,\Gamma^2,\Omega)\\\lim_{n\rightarrow\infty}\dis(\amu_{\psi(n)}(f_{\psi(n)}),\ssigma\cL^d)=0\end{array} \right\}
\end{align}
 with $f_{\psi(n)}\in\cS_{\psi(n)}^M (\Gamma^1,\Gamma^2,\Omega)$.
 
\begin{prop}\label{propadmissiblestream}Let $G$ that satisfies hypothesis \ref{hypo:G}. Let $(\Omega,\Gamma^1,\Gamma^2)$ that satisfies hypothesis \ref{hypo:omega}. Let $\nu\in\cM(\sR^d)^d\setminus (\{\ssigma\cL ^d : \,\ssigma\in\Sigma(\Gamma^1,\Gamma ^2,\Omega)\cap \Sigma^M(\Gamma^1,\Gamma ^2,\Omega)\})$ (we recall that $\Sigma(\Gamma^1,\Gamma ^2,\Omega)$ was defined in \eqref{eq:defadmsigma}), we have
$$\lim_{\ep\rightarrow 0}\liminf_{n\rightarrow \infty}\frac{1}{n^d} \log\Prb(\exists f_n\in\cS_n(\Gamma^1,\Gamma^2,\Omega) : \,\dis\big(\amu_n(f_n),\nu\big)\leq \ep)=-\infty\,.$$
\end{prop}
To prove this proposition we need the following deterministic lemma. It states that the limit of a sequence of discrete stream inherits the properties of the discrete streams. The main ingredient of the proof of this lemma were already present in \cite{CT1}. We postpone its proof after the proof of proposition \ref{propadmissiblestream}.
\begin{lem} \label{lem:convdiscstream} Let $(\Omega,\Gamma^1,\Gamma^2)$ that satisfies hypothesis \ref{hypo:omega}. Let $M>0$. Let $\psi:\sN\rightarrow\sN$ be an increasing function. Let $ f_{\psi(n)}\in \cS_{\psi(n)}^M(\Gamma^1,\Gamma^2,\Omega)$, for $n \geq 1$ such that $\mu_{\psi(n)}(f_{\psi(n)})$ weakly converges towards a measure $\nu \in \cM(\sR^d)^d$. Then, we have
$$\nu\in\left\{\ssigma\cL^d:\ssigma\in\Sigma^M(\Gamma^1,\Gamma^2,\Omega)\cap \Sigma(\Gamma^1,\Gamma^2,\Omega)\right\}\,.$$
\end{lem}

\begin{proof}[Proof of Proposition \ref{propadmissiblestream}]
Let $\nu\in\cM(\sR^d)^d$. We start by extracting a deterministic sequence of good realizations of $\amu_n(f_n)$ that converges weakly towards $\nu$.
Let us assume there exists $\kappa>0$ such that 
$$\forall \ep>0\qquad\liminf_{n\rightarrow\infty}\frac{1}{n^d}\log \Prb(\exists f_n\in\cS_n(\Gamma^1,\Gamma^2,\Omega) : \,\dis\big(\amu_n(f_n),\nu\big)\leq \ep)\geq -\kappa\,.$$
Hence, we can build iteratively an increasing sequence $(a_n)_{n\geq 1}$ of integers such that
$$\forall n\geq 1\qquad \frac{1}{a_n^d}\log \Prb\left(\exists f_{a_n}\in\cS_{a_n}(\Gamma^1,\Gamma^2,\Omega) : \,\dis\big(\amu_{a_n}(f_{a_n}),\nu\big)\leq \frac{1}{n}\right)\geq -2\kappa\,.$$
It follows that the event\[\left\{\exists f_{a_n}\in\cS_{a_n}(\Gamma^1,\Gamma^2,\Omega) : \,\dis\big(\amu_{a_n}(f_{a_n}),\nu\big)\leq \frac{1}{n}\right\}\] is not empty. We choose according to some deterministic rule a realization $\omega_n$ of the capacities of the set $\Omega\cap\E_{a_n}^d$ that belongs to this event. According to some deterministic rule, on the fixed realization $\omega_n$, we choose a stream $f_{a_n}\in\cS_{a_n}(\Gamma^1,\Gamma^2,\Omega)$ that satisfies 
\begin{align}\label{eq:distmeasure}
\dis\big(\amu_{a_n}(f_{a_n}),\nu\big)\leq \frac{1}{n}\,.
\end{align}
By lemma \ref{lem:weakconv}, it follows that $\amu_{a_n}(f_{a_n})$ weakly converges towards $\nu$.
By lemma \ref{lem:convdiscstream}, we have that 
$$\nu\in \left\{\ssigma\cL ^d : \,\ssigma\in\Sigma(\Gamma^1,\Gamma ^2,\Omega)\cap \Sigma^M(\Gamma^1,\Gamma ^2,\Omega)\right\})\,.$$
This concludes the proof.
\end{proof}

\begin{proof}[Proof of lemma \ref{lem:convdiscstream}]
To lighten the notation we will write $\amu_n$ instead of $\amu_{\psi(n)}(f_{\psi(n)})$. Note that for all $n\geq 1$, the support of $\amu_n$ is included in the compact set $\overline{\cV_\infty(\Omega,1)}$. 
For $i\in\{1,\dots,d\}$, let $\amu_n^i=\amu_n^{i,+}-\amu_n^{i,-}$ be the Hahn-Jordan decomposition of the signed measure $\amu_n^i$.
Notice that we have
$$|\amu_n^i|\left(\overline{\cV_\infty(\Omega,1)}\right)\leq \frac{1}{n^d}\sum_{e\in\E_n^d}\|f_n(e)\|_2\leq   \frac{M}{n^d} 2d |\Omega_n|\leq 2d\cL^d(\overline{\cV_\infty(\Omega,1)})M\,.$$
Hence, the sequence $(\amu_n)_{n\geq 1}$ is uniformly tight and uniformly bounded in the total variation norm. By Prohorov theorem (see for example, Theorem 8.6.2 in volume II of \cite{Bogachev}), it follows that up to extraction, we can assume that 
$$\amu_n^{i,+}\rightharpoonup \amu^{i,+},\quad \amu_n^{i,-}\rightharpoonup \amu^{i,-}\,,$$
and by using inequality \eqref{eq:distmeasure} and lemma \ref{lem:weakconv}, we deduce that
$$\forall i\in\{1,\dots,d\}\qquad \nu^i =\amu^{i,+}-\amu^{i,-}\,.$$

\noindent{\bf Step 1 : We prove that $\nu$ is absolutely continuous with respect to $\cL ^d$.}
This proof is an adaptation of the proof of proposition 4.2. in \cite{CT1}. 
Let $A$ be a Borel subset of $\sR^d$. Since the Lebesgue measure $\cL^d$ is outer regular, for $\ep>0$ there exists an open set $O$ such that $A\subset O$ and $\cL^d(O\setminus A)\leq \ep$. By the Vitali covering theorem for Radon measures (see Theorem 2.8. in \cite{Mattila}), there exists a countable family $(B(x_j,r_j),j\in J)$ of disjoint closed balls included in $O$ such that 
$$\amu^{i,+}\left(O\setminus \bigcup _{j\in J}B(x_j,r_j)\right)=0\,.$$
We have for $\delta>0$, using Portmanteau theorem
$$\amu^{i,+}(O)\leq \amu^{i,+}\big(\cup _{j\in J}\mathring{B}(x_j,r_j+\delta)\big)\leq \liminf_{n\rightarrow\infty} \amu_n^{i,+}\big(\cup_{j\in J} \mathring{B}(x_j,r_j+\delta)\big)\,.$$
Moreover, we have on the realization $\omega_n$,
 $$\amu_n^{i,+}\big(\cup_{j\in J} \mathring{B}(x_j,r_j+\delta)\big)\leq M\sum_{j\in J}\cL^d (B(x_j,r_j+\delta+2n^{-1})\,.$$
Hence, by taking the liminf in $n$ in the previous inequality and then by letting $\delta$ goes to $0$, we obtain
$$\amu^{i,+}(O)\leq M\sum_{j\in J}\cL^d (B(x_j,r_j))\leq M\cL^d(O)\leq M(\ep+\cL^d(A))\,$$
and $$\amu ^{i,+}(A)\leq  \amu^{i,+}(O)\leq  M(\ep+\cL^d(A))\,.$$
Finally, we let $\ep$ goes to $0$, we deduce that 
$\amu ^{i,+}(A)\leq  M\cL^d(A)$. Similarly, $\amu ^{i,-}(A)\leq  M\cL^d(A)$. We deduce that $\nu$ is absolutely continuous with respect to the Lebesgue measure; that is, there exists $\ssigma\in L^1(\sR^d\rightarrow \sR^d,\cL ^d)$ such that $\nu=\ssigma \cL ^d$. Hence, we have $\ssigma\in\Sigma^M (\Gamma^1,\Gamma^2,\Omega)$. We use the notation $\ssigma=(\sigma^1,\dots,\sigma^d)$. We have proved that for all $i\in\{1,\dots,d\}$,
$$\forall A \in\cB(\sR^d)\qquad\int_A|\sigma^i|d\cL^d\leq \amu^{i,+}(A)+\amu^{i,-}(A)\leq 2M\cL ^d(A)\,,$$
and $$\forall A \in\cB(\sR^d)\qquad-M\cL^d(A)\leq \int_A \ssigma\cdot \overrightarrow{e_i} d\cL^d= \amu^{i,+}(A)-\amu^{i,-}(A)\leq M\cL^d(A)$$
which implies that $|\ssigma\cdot\overrightarrow{e_i}|\leq M$ $\cL^d$-almost everywhere. It follows that $\ssigma\in L^\infty(\sR^d\rightarrow \sR^d,\cL ^d)$.
Moreover, we have 
\begin{align}\label{eq:sigmaL1mun}
\|\ssigma\|_{L^1}=\int_\Omega\|\ssigma\|_2d\cL^d&\leq\sum_{i=1}^d \int_\Omega |\sigma^i|d\cL^d\leq \sum_{i=1}^d \amu^{i,+}(\Omega)+\amu^{i,-}(\Omega)\nonumber\\
&\leq \sum_{i=1}^d \liminf_{n\rightarrow\infty}\amu_n^{i,+}(\Omega)+\liminf_{n\rightarrow\infty}\amu_n^{i,-}(\Omega)\nonumber\\
&\leq \liminf_{n\rightarrow\infty}\sum_{i=1}^d \amu_n^{i,+}(\Omega)+\amu_n^{i,-}(\Omega)\nonumber\\&=\liminf_{n\rightarrow\infty}\frac{1}{n^d}\sum_{e\in\E_n^d}\|f_n(e)\|_2
\end{align}
where we use Portmanteau theorem and the fact that $\Omega$ is an open set.

\noindent{\bf Step 2 : We prove that $\nu(\sR^d\setminus \Omega)=0$.}
The set $\sR^d\setminus \overline{\Omega}$ is an open set. Using Portmanteau theorem, we have
\[\forall i\in\{1,\dots,d\}\quad\forall \diamond\in\{+,-\}\qquad \amu^{i,\diamond}(\sR^d\setminus \overline{\Omega})\leq\liminf_{n\rightarrow \infty}\amu_n^{i,\diamond}(\sR^d\setminus \overline{\Omega})\,.\]
Besides, using proposition \ref{prop:minkowski}, and by construction of $\amu_n$ we have for $n$ large enough
$$\amu_n^{i,\diamond}(\sR^d\setminus \overline{\Omega})\leq 2dM \frac{|\cV_\infty(\partial \Omega,1/n)\cap \sZ_n^d|}{n^d}\leq  2dM\cL^d(\cV_2(\partial \Omega,d/n))\leq 8d^2\frac{M}{n} \cH^{d-1}(\partial \Omega)\,.$$
It follows that 
\[\forall i\in\{1,\dots,d\}\quad\forall \diamond\in\{+,-\}\qquad \amu^{i,\diamond}(\sR^d\setminus \overline{\Omega})=0\,\]
and $\nu(\sR^d\setminus \overline{\Omega})=0$. Finally, since $\nu$ is absolutely continuous with respect to the Lebesgue measure we have $\nu(\partial \Omega)=0$ and the result follows.

\noindent{\bf Step 3 : We prove that $ \diver\ssigma=0$ $\cL^d$-almost everywhere.} 
Thanks to the previous step, we can write $\nu=\ssigma\cL^d$. 
Let $h\in \sC_c ^\infty (\Omega,\sR)$. 
For all $x\in \Gamma_n^1\cup \Gamma_n ^2$, let $\di f_n(x)$ be the amount of water that appears at $x$ according to the stream $f_n$:
\begin{align}\label{eq:discdiv}
\di f_n(x)= n\sum_{y\in\sZ_n^d:\,e =\langle x,y \rangle\in \E_n^d } f_n(e)\cdot \overrightarrow{yx}\,.
\end{align}
We have that $f_n$ satisfies the node law at $x$ if and only if $\di f_n(x)=0$.
We recall that $\|\overrightarrow{yx}\|_2=1/n$, this accounts for the $n$ factor in the expression above. The function $\di f_n$ corresponds to a discrete divergence. Since $f_n$ satisfies the node law, $\di f_n$ is null on $\Omega_n\setminus(\Gamma_n^1\cup\Gamma_n^2)$. We state here the equality obtained in \cite{CT1} just after equality (4.6):
$$\int_{\sR^d}\overrightarrow{\nabla}h\cdot d\amu_n=-\frac{1}{n ^{d-1}}\sum_{x\in \Gamma_n^1\cup \Gamma_n ^2}h(x)\,\di f_n(x)+\frac{\alpha_n(h,f_n,\Omega)}{n ^d}\,$$
with
$$\lim_{n\rightarrow \infty} \frac{\alpha_n(h,f_n,\Omega)}{n ^d}=0\,.$$
This equality is not difficult to obtain, it uses the fact that the stream $f_n$ has a null discrete divergence to control the divergence of the limiting object $\ssigma$. The proof of this result may be found in Proposition 4.5 in \cite{CT1}.
Notice that since $h\in \sC_c ^\infty (\Omega,\sR)$, $h$ is null on $\Gamma_n^1\cup \Gamma_n ^2$, for all $n$, and
$$\lim_{n\rightarrow \infty}\int_{\sR^d}\overrightarrow{\nabla}h\cdot d\amu_n= 0\,.$$
Since $\overrightarrow{\nabla} h\in\sC_c ^\infty (\Omega,\sR)$, by Portmanteau theorem, we have
$$\int_{\sR ^d}h\diver \ssigma d\cL ^d=\int_{\sR^d}\ssigma\cdot\overrightarrow{\nabla}h\,d\cL^d=\lim_{n\rightarrow\infty}\int_{\sR^d}\overrightarrow{\nabla}h\cdot d\amu_n=0\,.$$
This yields the result.

\noindent{\bf Step 4 : We prove that $\ssigma\cdot\overrightarrow{n}_\Omega=0$ $\cH^{d-1}$-almost everywhere on $\Gamma\setminus( \Gamma^1\cup\Gamma^2)$.}  Thanks to inequality $(4.8)$ in the proof of the Corollary 2 in \cite{CT1}, $\ssigma\cdot\overrightarrow{n}_\Omega$ is an element of $L ^\infty (\Gamma,\cH ^{d-1})$ characterized by
\begin{align}\label{eq:caracterisationsigman}
\forall u\in\sC_c ^\infty (\sR ^d,\sR )\qquad \int_\Gamma (\ssigma\cdot\overrightarrow{n}_\Omega)u\,d\cH ^{d-1}=\int_{\sR ^d} \ssigma\cdot \overrightarrow{\nabla}ud\cL ^d\,.
\end{align}
Let $u\in\sC^\infty_ c((\Gamma^1\cup\Gamma^2)^c,\sR)$ . 
As in the previous step, we have
$$\lim_{n\rightarrow \infty}\int_{\sR^d}\overrightarrow{\nabla}u\cdot d\amu_n=\lim_{n\rightarrow\infty}-\frac{1}{n ^{d-1}}\sum_{x\in \Gamma_n^1\cup \Gamma_n ^2}u(x)\,\di f_n(x)+\frac{\alpha_n(h,f_n,\Omega)}{n ^d}\,.$$
Since $u$ is null on $\Gamma_n^1\cup \Gamma_n ^2$ for $n$ large enough, we  have $$\lim_{n\rightarrow \infty}\int_{\sR^d}\overrightarrow{\nabla}u\cdot d\amu_n=0\,.$$
Finally, using Portmanteau theorem we have 
$$\int_\Gamma (\ssigma\cdot\overrightarrow{n}_\Omega)u\,d\cH ^{d-1}=\int_{\sR ^d} \ssigma\cdot \overrightarrow{\nabla}u\,d\cL ^d=\lim_{n\rightarrow \infty}\int_{\sR^d}\overrightarrow{\nabla}u\cdot d\amu_n=0\,.$$
This ends the proof.
\end{proof}

\section{Technical lemmas \label{sect: technicallem}}
\subsection{Mixing\label{sect:mixing}}
This section is only geometrical and does not contain any randomness. The aim of this section is to prove that we can reconnect two different streams if the incoming flow coincides with the outcoming flow. Namely, if we consider  two families of inputs and outputs such that the sum of the inputs is equal to the sum of the outputs, we can connect the inputs with the outputs. To connect the streams, we are going to give an algorithm that enables to build a stream that connects the inputs to the outputs. To lighten the notations, all the lemmas of this section are stated and proved in $\sZ^d$ instead of $\sZ_n^d$.
 \begin{lem}[Mixing]\label{lem:mixing} Let $M> 0$, $n\geq 1$. For any two sequences of real numbers $(f_{in}(y),y\in \{1,\dots,n\}^{d-1})$ and $(f_{out}(y),y\in \{1,\dots,n\}^{d-1})$ satisfying 
 $$\forall y\in \{1,\dots,n\}^{d-1}\qquad |f_{in}(y)|\leq M\, ,\quad  |f_{out}(y)|\leq M\,$$ and
 $$\sum_{y\in \{1,\dots,n\}^{d-1}}f_{in}(y)=\sum_{y\in \{1,\dots,n\}^{d-1}}f_{out}(y)\,,$$
for any $m\geq 2(d-1)n$, there exists a stream $f:\E^d\rightarrow \sR^d$ such that:
\begin{itemize}
\item[$\cdot$] for each edge $e\notin [0,m[\times [1,n]^{d-1}$ (we recall that $e$ belong to a set if its left endpoint belong to this set), we have $f(e)=0$,
\item[$\cdot$] for each $e\in\E^d$ we have $\|f(e)\|_2\leq M$,
\item[$\cdot$] for each $y\in \{1,\dots,n\}^{d-1} $, we have $f(\langle (0,y),(1,y)\rangle)=f_{in}(y)\overrightarrow{e_1}$ and $f(\langle(m-1,y), (m,y)\rangle)=f_{out}(y)\overrightarrow{e_1}$,
\item [$\cdot$] for each vertex $v\in \sZ^d\setminus((\{0\}\times  \{1,\dots,n\}^{d-1})\cup( \{m\}\times  \{1,\dots,n\}^{d-1}))$ the node law is respected.
\end{itemize} 
Moreover, if the outputs are uniform, \textit{i.e}, 
$$\forall y \in\{1,\dots,n\}^{d-1}\qquad f_{out}(y)=\frac{1}{n^{d-1}}\sum_{z\in \{1,\dots,n\}^{d-1}}f_{in}(z)$$
then the same result holds for any $m\geq (d-1)n$.
\end{lem}
Before proving this lemma, we need to prove that we can reconnect streams in the particular case of the dimension 2 with uniform outputs. We build the stream by an algorithm, this algorithm will be used in other proofs of this section.
\begin{lem}[Mixing in dimension 2]\label{lem:mixing in dim2}Let $M> 0$, $n\geq 1$. For any sequence of real number $(f_{in}(j),j=1,\dots,n)$ satisfying
 $$\forall j\in \{1,\dots,n\}\qquad| f_{in}(j)|\leq M\,,$$
there exists a stream $f:\E^2\rightarrow \sR^2$ 
\begin{enumerate}[(i)]
\item for each edge $e\notin [0,n[\times[1,n]$ we have  $$f(e)=0\,$$,\label{item:1}
\item for each $e\in\E^2$ we have $\|f(e)\|_2\leq M$,\label{item:2}
\item for each $j\in\{1,\dots,n\}$, we have $f(e_j)=f_{in}(j)\overrightarrow{e_1}$ and $$f( e_j+n\overrightarrow{e_1})=\frac{1}{n}\sum_{i=1}^{n}f_{in}(i)\overrightarrow{e_1},$$ where $e_j=\langle (0,j),(1,j)\rangle$ and $\langle x,y\rangle+k\overrightarrow{e_1}=\langle x+k\overrightarrow{e_1},y+k\overrightarrow{e_1}\rangle$,\label{item:3}
\item for each vertex $v\in \sZ^2\setminus((\{0\}\times  \{1,\dots,n\})\cup(\{n\}\times  \{1,\dots,n\}))$, the node law is respected.\label{item:4}
\end{enumerate}

\end{lem}
\begin{proof}Up to multiplying by $-1$ all the inputs, we can always assume that 
$$\sum_{i=1}^{n}f_{in}(i)\geq 0\,.$$
We set $$\beta=\frac{1}{n}\sum_{i=1}^{n}f_{in}(i)\,.$$
We start by sending the minimum between $\beta$ and $f_{in}(i)$ through straight lines:
$$f^{(0)}=\sum_{i=1}^{n}\min(f_{in}(i),\beta)\sum_{k=0}^{n-1} \overrightarrow{e_1}\ind _{e_i+k\overrightarrow{e_1}}\,.$$
We are going to perform an algorithm starting with the stream $f^{(0)}$ to build a stream $f$ satisfying all the conditions of the lemma. At any step of the algorithm, $f$ will satisfy condition \ref{item:4} but also the following conditions:
\begin{enumerate}[label=(\alph*)]
\item \label{eqcondmix1}The vertical edges in the column $n-i$ are only used by 
the source $i\in\{1,\dots,n\}$: 
\begin{align*}
\forall i\in\{1,\dots,n\}&\quad \forall j\in\{1,\dots,n-1\}\qquad \|f(\langle(n-i,j),(n-i,j+1)\rangle)\|_2\leq \|f(e_i)\|_2\ind_{f_{in}(i)>\beta}\,.
\end{align*}
\item \label{eqcondmix2}If $f_{in}(i)\leq\beta$ then the flow through the line $i$ in the direction $\overrightarrow {e_1}$ is non-decreasing:
\begin{align*}
\forall i\in\{1,\dots,n\} &\text{ s.t. }f_{in}(i)<\beta\quad \forall j<k \in\{1,\dots,n-1\}\qquad \nonumber\\
& f_{in}(i)= f(e_i)\cdot \overrightarrow{e_1}\leq  f(e_i+j\overrightarrow{e_1})\cdot \overrightarrow{e_1}\leq f(e_i+k\overrightarrow{e_1})\cdot \overrightarrow{e_1}\leq \beta \,.
\end{align*}
\item \label{eqcondmix3}If $f_{in}(i)>\beta$ then the flow through the line $i$ in the direction $\overrightarrow {e_1}$ is non-increasing:
\begin{align*}
\forall& i\in\{1,\dots,n\} \text{ s.t. }f_{in}(i)>\beta\quad \forall j<k \in\{1,\dots,n-1\}\qquad \nonumber\\
&f_{in}(i)\geq  f(e_i)\cdot \overrightarrow{e_1}\geq  f(e_i+j\overrightarrow{e_1})\cdot \overrightarrow{e_1}\geq f(e_i+k\overrightarrow{e_1})\cdot \overrightarrow{e_1}\geq f(e_i+n\overrightarrow{e_1})\cdot\overrightarrow{e_1}=\beta \,.
\end{align*}
\end{enumerate}
We set $f=f^{(0)}$.
It is clear that the stream $f$ satisfies the node law \ref{item:4} and conditions \ref{eqcondmix1}, \ref{eqcondmix2} and \ref{eqcondmix3}.
 Let us assume there exists $i$ such that
$$\|f(e_i)\|_2<|f_{in}(i)|\,.$$
We consider the smallest integer $i$ such that the previous inequality is satisfied. Since $f$ satisfies condition \ref{eqcondmix2}, necessarily $f_{in}(i)>\beta$ (if not we have $\|f(e_i)\|_2=|f_{in}(i)|)$).By condition \ref{eqcondmix3}, we have $f(e_i)\cdot\overrightarrow{e_1}\geq 0$ and so 
 $\|f(e_i)\|_2=f(e_i)\cdot\overrightarrow{e_1}<f_{in}(i)$.
Since $f$ satisfies condition \ref{item:4}, it yields by the node law
$$\sum_{k=1}^n f(e_k+n\overrightarrow{e_1})\cdot\overrightarrow{e_1}=\sum_{k=1}^n f(e_k)\cdot\overrightarrow{e_1}<\sum_{k=1}^n f_{in}(k)=n\beta\,.$$
Then, there exists $j$ such that
$$f(e_j+n\overrightarrow{e_1})\cdot\overrightarrow{e_1}<\beta\,.$$
We pick the smallest integer $j$ such that the previous inequality holds. By condition \ref{eqcondmix3}, we have $f_{in}(j)<\beta$.
We set
$$\gamma_{i,j} =\sum_{k=0}^{n-1-i}\overrightarrow{e_1}\ind_{e_i+k\overrightarrow{e_1}}+\sum_{k=-(j-i)_-}^{(j-i)_+-1}\sign(j-i)\overrightarrow{e_2}\ind_{\langle (n-i,i+k),(n-i,i+k+1)\rangle}+\sum_{k=n-i}^{n-1}\overrightarrow{e_1}\ind_{e_j+k\overrightarrow{e_1}}\,.$$
Note that $\gamma_{i,j}$ is a stream, \textit{i.e.}, a function from $\E^2$ to $\sR^2$. We can associate with $\gamma_{i,j}$ an oriented path corresponding to the path the water takes to go from source $i$ to sink $j$ for the stream $\gamma_{i,j}$ (see figure \ref{fig:cheminogamij}).
\begin{figure}[H]
\begin{center}
\def\svgwidth{0.6\textwidth}
   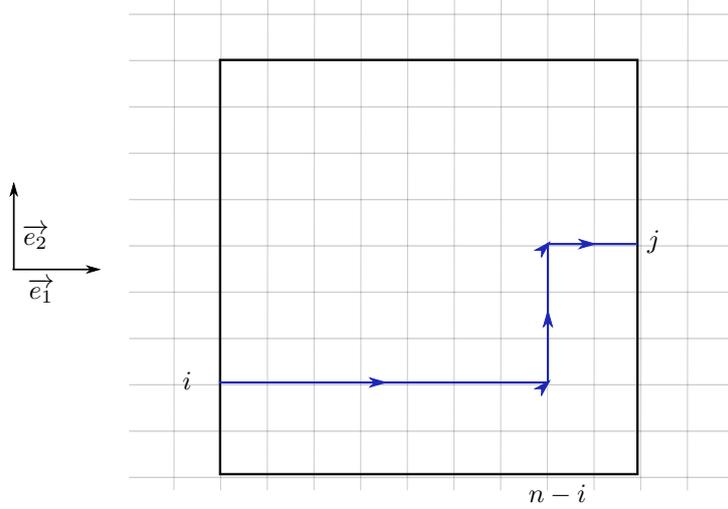
   \caption{\label{fig:cheminogamij}The path associated with $\gamma_{i,j}$ for $j>i$}
   \end{center}
\end{figure}
\noindent  We set $$f'=f+ \min\left(f_{in}(i)-\|f(e_i)\|_2,\beta- f(e_j+n\overrightarrow{e_1})\cdot\overrightarrow{e_1}\right)\gamma_{i,j}\,.$$
 It is clear that $f'$ satisfies condition \ref{item:4}. We have to check that the conditions \ref{eqcondmix1}, \ref{eqcondmix2} and \ref{eqcondmix3} are still satisfied for $f'$. Note that we have
 $$m_{i,j}:=\min\left(f_{in}(i)-\|f(e_i)\|_2,\beta- f(e_j+n\overrightarrow{e_1})\cdot\overrightarrow{e_1}\right)>0\,.$$
 We start by checking if the condition \ref{eqcondmix1} is satisfied:
 for any $k\in\{1,\dots,n-1\}$, we have
 \begin{align*}
 \|f'(\langle(n-i,k),(n-i,k+1)\rangle)\|_2&\leq \|f(\langle(n-i,k),(n-i,k+1)\rangle)\|_2+ m_{i,j} \\
 &\leq \|f(e_i)\|_2 +m_{i,j}\\
 & =f(e_i)\cdot\overrightarrow{e_1}+m_{i,j}\\
 &=f'(e_i)\cdot\overrightarrow{e_1}=\|f'(e_i)\|_2\,.
 \end{align*}
 We have
 $$f'(e_i)\cdot\overrightarrow{e_1}=f(e_i)\cdot\overrightarrow{e_1}+m_{i,j}=\|f(e_i)\|_2+m_{i,j}\leq f_{in}(i)\,.$$
 Moreover, it is clear that the flow through the line $i$ in the direction $\overrightarrow{e_1}$ is non-increasing for $f'$ since it was the case for $f$. Furthermore, we have $f'(e_i+n\overrightarrow{e_1})\cdot\overrightarrow{e_1}=f(e_i+n\overrightarrow{e_1})\cdot\overrightarrow{e_1}=\beta$. Hence, the stream $f'$ satisfies condition \ref{eqcondmix3}.
 Since the flow through the line $j$ is non-decreasing for $f$, it is easy to check that it is also true for $f'$.
 Moreover, we have by definition of $m_{i,j}$
 $$f'(e_j+n\overrightarrow{e_1})\cdot\overrightarrow{e_1}=f(e_j+n\overrightarrow{e_1})\cdot\overrightarrow{e_1}+m_{i,j}\leq \beta$$
 and
 $$f'(e_j)\cdot\overrightarrow{e_1}=f(e_j)\cdot\overrightarrow{e_1}= f_{in}(j)\,.$$
 It follows that $f'$ also satisfies condition \ref{eqcondmix2}.
Since after each step of the algorithm the number of such couples $(i,j)$ is decreasing, the algorithm will eventually end.
Let us assume there exists no such $i$. Therefore, we have
\[\sum_{i=1}^nf(e_i)\cdot\overrightarrow{e_1}=\sum_{i=1}^n f_{in}(i)\,.\]
By the node law again, we have
\[\sum_{i=1}^n f(e_i+n\overrightarrow{e_1})\cdot\overrightarrow{e_1}=\sum_{i=1}^n f_{in}(i)=n\beta\]
and by conditions \ref{eqcondmix2} and \ref{eqcondmix3}, it follows that
\[\forall j \in\{1,\dots,n\}\qquad f(e_j+n\overrightarrow{e_1})=\beta\overrightarrow{e_1}\,.\]
This yields property \ref{item:3}. At the end of the algorithm, the stream $f$ satisfies conditions \ref{item:1}, \ref{item:2}, \ref{item:3} and \ref{item:4}. 
\end{proof}

\begin{proof}[Proof of lemma \ref{lem:mixing}] We first prove the result for uniform outputs, that is, for any family $(f_{in}(y),y\in\{1,\dots,n\}^{d-1})$ satisfying for any $y\in\{1,\dots,n\}^{d-1}$ $|f_{in}(y)|\leq M$,
 if we set
$$\forall y\in\{1,\dots,n\}^{d-1}\qquad f_ {out}(y)=\frac{1}{n^{d-1}}\sum_{w\in\{1,\dots,n\}^{d-1}}f_{in}(w)\,,$$
there exists a stream in $[0,(d-1)n[\times[1,n]^{d-1}$ satisfying all the requirements in the statement of lemma \ref{lem:mixing}.
We prove this result by induction on the dimension. The result holds for the dimension $2$ thanks to lemma \ref{lem:mixing in dim2}. Let us now consider $d\geq 3$ and assume the result holds for the dimension $d-1$. Consider a family of inputs bounded by $M$:
$(f_{in}(y),y\in\{1,\dots,n\}^{d-1})$. Using the induction hypothesis for the dimension $d-1$, for each $i\in\{1,\dots,n\}$, we can build a stream $f_{d-1}^{(i)}$ in $[0,(d-2)n[\times \{i\}\times [1,n]^{d-2}$ given the inputs $(f_{in} (i,z),z\in\{1,\dots,n\}^{d-2})$ and the uniform outputs equal to
$$\forall z\in\{1,\dots,n\}^{d-2}\qquad g(i,z)=\frac{1}{n^{d-2}}\sum_{z\in\{1,\dots,n\}^{d-2}}f_{in} (i,z)\,.$$
It is clear that the streams $f_{d-1}^{(i)}$ are defined on disjoint sets of edges.
Finally, for each $x\in\{1,\dots,n\}^{d-2}$, using lemma \ref{lem:mixing in dim2}, we denote by $f_2^{(x)}$ the stream in $ [0,n[\times [1,n]\times \{x\}$ with inputs $(g(i,x),i=1,\dots,n)$ and uniform output equal to
$$\frac{1}{n}\sum_{i=1}^n g(i,x)=\frac{1}{n^{d-1}}\sum_{i=1}^n\sum_{z\in\{1,\dots,n\}^{d-2}}f_{in} (i,z)=\frac{1}{n^{d-1}}\sum_{w\in\{1,\dots,n\}^{d-1}}f_{in} (w)\,.$$
The stream $f_2^{(x)}$ are also defined on disjoint set of edges.
Finally, the stream $$g=\sum_{i=1}^n f_{d-1}^{(i)}+\sum_{x\in\{1,\dots,n\}^{d-2}}f_2^{(x)}(\cdot-(d-2)n\overrightarrow{e_1})$$ is defined on $[0,(d-1)n[\times[1,n]^{d-1}$, $\|g(e)\|_2\leq M$ for any $e$ and $g$ mixes uniformly the inputs since
$$\forall y\in\{1,\dots,n\}^{d-1}\qquad g(\langle((d-1)n-1,y),((d-1)n,y)\rangle)=\frac{1}{n^{d-1}}\sum_{w\in\{1,\dots,n\}^{d-1}}f_{in} (w)\,.$$ It follows that the result holds for the dimension $d$. This concludes the induction.

Let $d\geq 2$, let us now consider two families $(f_{in}(y))_y$ and $(f_{out}(y))_y$ of arbitrary inputs and outputs that satisfy the conditions in the statement of the lemma. Let $f^{i}$ be the stream in $[0,(d-1)n[\times[1,n]^{d-1}$ with inputs $(f_{in}(y))_y$ and uniform outputs. Let $f^{o}$ be the stream in $[0,(d-1)n[\times[1,n]^{d-1}$ with inputs $(f_{out}(y))_y$ and uniform outputs. Denote by $S$ the reflexion with regards to the hyperplane $\{x\in\sR^d,\,x_1=0\}$, \textit{i.e.}, $$\forall (x_1,\dots,x_d)\in\sR^d\qquad S(x_1,x_2,\dots,x_d)=(-x_1,x_2,\dots,x_d)\,.$$
We denote by $Sf^{o}$ the symmetric of the stream $f^o$ by $S$:
$$\forall e\in\E^d\qquad Sf^o(e)=S(f^o(S(e)))\,$$
where for $e=\langle x,y\rangle$ the edge $S(e)$ corresponds to $\langle S(x),S(y)\rangle$.
Note that for any edge $e$ parallel to $\overrightarrow{e_1}$, we have $f^{o}(e)=a\overrightarrow{e_1}$ with $a\in\sR$ and 
$$Sf^o (S(e))=S(f^o(S(S(e))))=S(f^o(e))=-a\overrightarrow{e_1}=-f^o(e)\,.$$
We have for $y\in\{1,\dots,n\}^{d-1}$, 
\begin{align*}
f^i(\langle ((d-1)n-1,y),((d-1)n,y)\rangle)&=\frac{1}{n^{d-1}}\sum_{y\in \{1,\dots,n\}^{d-1}}f_{in}(y)=\frac{1}{n^{d-1}}\sum_{y\in \{1,\dots,n\}^{d-1}}f_{out}(y)\\
&=f^o(\langle ((d-1)n-1,y),((d-1)n,y)\rangle)\\
&=-Sf^o (\langle (-(d-1)n,y),(-(d-1)n+1,y)\rangle)\,.
\end{align*}
It follows that the stream $g=f^i-Sf^{o}(\cdot -(2(d-1)n-1)\overrightarrow{e_1})$ connects the inputs with the outputs in $[0,2(d-1)n[\times[1,n]^{d-1}$ and satisfies all the properties stated in the lemma.
If $m>2(d-1)n$, we extend the stream outside $[0,2(d-1)n[\times[1,n]^{d-1}$ through straight lines:
$$g+\sum_{y\in\{1,\dots,n\}}\sum_{k=2(d-1)n+1}^m f_{out}(y)\overrightarrow{e_1}\ind_{\langle(k-1,y),(k,y)\rangle}\,.$$
\end{proof}
We will need a special result of mixing in the case where the non-null inputs and outputs are regularly spaced in the lattice. Namely, there exists an integer $K\geq 1$ such that any input or output whose index does not belong to $K\sZ^{d-1}\cap[1,n]^{d-1}$ is null. In that case, we want to prove that we do not use a lot of edges to reconnect the inputs with the outputs.
For any integer $K\geq 1$, we denote by $\E_K ^d$ the following set of edges:
 \begin{align*}
\E_K ^d=&\left\{e=\langle x,y\rangle \in\E^d: y-x=\overrightarrow{e_1},\,\forall j\neq 1\quad x_j\in K \sZ\right\}\\
&\quad\cup\left\{e \in\E^d: \,\exists z\in\sZ \,\,\exists x,y\in K\sZ^{d-1}\text{ s.t. }\|x-y\|_1=K\text{ and } e\subset[(z,x),(z,y)] \right\}\,.
\end{align*}
\begin{lem} \label{lem:mixbin}Let $d\geq 2$, $M> 0$, $n\geq 1$. There exists a positive integer $c_d$ such that for any integer $K$ satisfying $K\geq c_d$, for any two $(f_{in}(y),y\in \{1,\dots,n\}^{d-1}\cap K\sZ^{d-1})$ and $(f_{out}(y),y\in \{1,\dots,n\}^{d-1}\cap K\sZ^{d-1})$ sequences of real numbers satisfying
 $$\forall y\in \{1,\dots,n\}^{d-1}\cap K\sZ^{d-1}\qquad |f_{in}(y)|\leq M\, ,\quad  |f_{out}(y)|\leq M\,$$ and
 $$\sum_{y\in \{1,\dots,n\}^{d-1}\cap K\sZ^{d-1}}f_{in}(y)=\sum_{y\in \{1,\dots,n\}^{d-1}\cap K\sZ^{d-1}}f_{out}(y)\,,$$
there exists a stream $f:\E^d\rightarrow \sR^d$
such that
\begin{itemize} 
\item[$\cdot$] for each $e\notin\E_K ^d\cap [0,n[\times[1,n]^{d-1}$ we have $f(e)=0$,
\item[$\cdot$] for each $e\in\E^d$ we have $\|f(e)\|_2\leq M$,
\item[$\cdot$] for each $y\in \{1,\dots,n\}^{d-1}\cap K\sZ^{d-1} $, we have $f(\langle(0,y),(1,y)\rangle)=f_{in}(y)\overrightarrow{e_1}$ and $f(\langle(n-1,y),(n,y)\rangle)=f_{out}(y)\overrightarrow{e_1}$,
\item [$\cdot$] for each vertex $v\in \sZ^d\setminus((\{0\}\times  (\{1,\dots,n\}^{d-1}\cap K\sZ^{d-1}))\cup(
 \{n\}\times ( \{1,\dots,n\}^{d-1}\cap K\sZ^{d-1})))$ the node law is respected.
\end{itemize} 
Moreover, we have
\[\left|\{e\in\E^d:f(e)\neq 0\}\right|\leq \left|\E_K ^d\cap[0,n[\times[1,n]^{d-1}\right|\leq \frac{3d}{K ^{d-2}}n^d\,.\]
\end{lem}
\begin{proof}[Proof of lemma \ref{lem:mixbin}]Let $c_d$ be an integer we will choose later. Let $K\geq c_d$. Let us consider the following bijection $\pi$ between the lattice $\sZ\times K\sZ^{d-1}$  and $\sZ^d$ defined as follows 
\[\forall x\in\sZ\quad\forall y\in\sZ^{d-1}\qquad \pi((x,K y))=(x,y)\,.\] 
Therefore, the problem boils down to finding a stream that joins the inputs $(f_{in}(K y),y\in\{1,\dots,n_0\}^{d-1})$ with the outputs $(f_{out}(K y),y\in\{1,\dots,n_0\}^{d-1})$ in $[0,n[\times [1,n_0]^{d-1}$ where $n_0=\lfloor n/K\rfloor$. Note that $n\geq K n_0\geq c_d n_0$.  By setting $c_d=2(d-1)$, this ensures that we can apply lemma \ref{lem:mixing}. We obtain a stream $f_n$ in $[0,n[\times [1,n_0]^{d-1}\cap\E^d$ that satisfies all the properties stated in the lemma \ref{lem:mixing}. It remains to build upon $f_n$ a stream $\widetilde{f}_n$ in the original lattice. To do so, we set
\[\forall e=\langle x,y\rangle\in\sZ^d\quad\forall \widetilde{e}=\langle w_1,w_2\rangle \text{  s.t.  }[w_1,w_2]\subset [\pi^{-1}(x),\pi ^{-1}(y)]\qquad \widetilde{f}_n(\widetilde{e})=f_n(e)\,.\]
It is easy to check that the stream $\widetilde{f}_n$ is supported on $\E^d_K\cap[0,n[\times[1,n]^{d-1}$ and that it satisfies all the properties stated in the lemma \ref{lem:mixbin}. 
It remains to upper-bound the quantity $|\E_K ^d\cap[0,n[\times[1,n]^{d-1}|$. We have
\begin{align*}
|\E_K ^d\cap[0,n[\times[1,n]^{d-1}|&\leq n\, |K \sZ^{d-1}\cap[1,n]^{d-1}|  + n\,|K \sZ^{d-1}\cap[1,n]^{d-1}|2dK \\
&\leq n\left(\frac{n}{K}\right)^{d-1}+2dK n\left(\frac{n}{K}\right)^{d-1}\leq \frac{3d}{K ^{d-2}}n^d\,.
\end{align*}
This yields the result.
\end{proof}

In what follows, we will need the following lemma. This lemma gives a precise description on the way the edges are used. The hypothesis of this lemma may seem strange but should be more clear in its context of application (see proofs of lemma \ref{lem:preconv} and proposition \ref{prop:step2uld}).
\begin{lem}\label{lem:mixprecis}
Let $d\geq 2$, $M> 0$, $\ep>0$ and $n\geq 1$. For any sequence of real number $(f_{in}(y),y\in \{1,\dots,n\}^{d-1})$ satisfying
 $$\forall y\in \{1,\dots,n\}^{d-1}\qquad -M\leq  f_{in}(y)\leq \ep$$ and
\begin{align*} 
 \forall i\in&\{0,\dots,d-2\}\quad\forall y\in\{1,\dots,n\}^i\\ &\sum_{x\in \{1,\dots,n\}^{d-1-i}}f_{in}(y,x)\geq 0 \quad\text{ or }\quad\forall x,z\in \{1,\dots,n\}^{d-1-i}\quad |f_{in}(y,x)-f_{in}(y,z)|\leq \ep\,,
 \end{align*}
there exists a stream $f:\E^d\rightarrow \sR^d$
such that
\begin{itemize} 
\item[$\cdot$] for each $e\notin [0,(d-1)n[\times[1,n]^d$ we have $f(e)=0$;
\item[$\cdot$] for each $e\in\E^d$, if $e$ is parallel to $\overrightarrow{e_1}$, then we have $-M\leq f(e)\cdot \overrightarrow{e_1}\leq \ep$, otherwise $\|f(e)\|_2\leq \ep$;
\item[$\cdot$] for each $y\in \{1,\dots,n\}^{d-1} $, we have $f(\langle(0,y),(1,y)\rangle)=f_{in}(y)\overrightarrow{e_1}$ and $$f(\langle((d-1)n-1,y),((d-1)n,y)\rangle)=\frac{1}{n^{d-1}}\sum_{z\in\{1,\dots,n\}^{d-1}}f_{in}(z)\overrightarrow{e_1}\,;$$
\item [$\cdot$] for each vertex $v\in \sZ^d\setminus((\{0\}\times  (\{1,\dots,n\}^{d-1}))\cup(
 \{(d-1)n\}\times ( \{1,\dots,n\}^{d-1})))$ the node law is respected.
\end{itemize} 
\end{lem}
\begin{proof}We prove this result by induction on the dimension. For $d=2$. Let $(f_{in}(j),1\leq j\leq n)$ be a family that satisfies the conditions stated in the lemma: 
$$\forall\, j\in\{1,\dots,n\}\qquad -M\leq f_{in}(j)\leq \ep\,$$
and \[\beta=\sum_{k=1}^nf_{in}(k)\geq 0\qquad\text{ or} \qquad \forall k,j\in\{1,\dots,n\}\quad|f_{in}(k)-f_{in}(j)|\leq \ep\,.\]
If $\beta\geq 0$, we apply directly the algorithm in the proof of lemma \ref{lem:mixing in dim2} to obtain a stream $f$. If $\beta<0$, then $\forall k,j\in\{1,\dots,n\}\quad|f_{in}(k)-f_{in}(j)|\leq \ep$ and  we set $\alpha=\min\{f_{in}(j):\,1\leq j\leq n\}$. It follows that for any $j\in\{1,\dots,n\}$, we have $f_{in}(j)-\alpha\in[0,\ep]$, we apply the lemma \ref{lem:mixing in dim2} to the sequence of real numbers $(f_{in}(j)-\alpha, j=1,\dots,n)$ to obtain a stream $g$ in $[0,n[\times[1,n]$, finally we set 
$$f=g+\sum_{i=1}^{n}\alpha\sum_{k=0}^{n-1} \overrightarrow{e_1}\ind _{e_i+k\overrightarrow{e_1}}\,.$$
In both cases, the stream $f$ we obtain satisfies all the required properties: if $\beta\geq 0$, due to condition \ref{eqcondmix1}, only the inputs $i\in\{1,\dots,n\}$ such that $0\leq \beta\leq f_{in}(i)\leq \ep$ can use the edges parallel to $\overrightarrow{e_2}$. Thanks to condition \eqref{eqcondmix2} and \eqref{eqcondmix3}, for each $e\in\E^d$ parallel to $\overrightarrow{e_1}$, we have $-M\leq f(e)\cdot \overrightarrow{e_1}\leq \ep$.

Let us assume the result holds for $d-1$ where $d\geq 3$. Let $(f_{in}(y),y\in\{1,\dots,n\}^ {d-1})$ be a family that satisfies the condition stated in the lemma \ref{lem:mixprecis}. 
 For $i\in\{1,\dots,n\}$, it is easy to check that the family $(f_{in}(i,x),x\in\{1,\dots,n\}^{d-2})$ also satisfies the conditions of the lemma \ref{lem:mixprecis}. By induction hypothesis, there exists a stream $f^{(i)}_{(d-1)}$ in $[0,(d-2)n[\times\{i\}\times[1,n]^{d-2}$ that satisfies all the conditions of the lemma \ref{lem:mixprecis}. We build the family $g$ as follows
$$ \forall x\in\{1,\dots,n\}^{d-2} \qquad g(i,x)= f^{(i)}_{(d-1)}(\langle ((d-2)n-1,i,x),((d-2)n,i,x)\rangle)\cdot\overrightarrow{e_1}=\frac{1}{n^{d-2}}\sum_{z\in\{1,\dots,n\}^{d-2}}f_{in}(i,z)\,.$$
It is clear that for any $y\in\{1,\dots,n\}^{d-1}$, $g(y)\in [-M,\ep]$. Besides, we have
$$ \forall x\in\{1,\dots,n\}^{d-2} \qquad\sum_{k=1}^ng(k,x)=\sum_{y\in\{1,\dots,n\}^{d-1}}f_{in}(y)\,.$$
By the properties of the family $(f_{in}(y),y\in\{1,\dots,n\}^{d-1})$ we have for any $x\in\{1,\dots,n\}^{d-2}$\[\sum_{k=1}^ng(k,x)\geq 0\quad\text{ or }\quad \forall y,z\in\{1,\dots,n\}^{d-1}\quad |f_{in}(y)-f_{in}(z)|\leq \ep\,.\] If $\sum_{k=1}^ng(k,x)<0$, it follows that for any $k,j\in\{1,\dots,n\}$, we have $|g(k,x)-g(j,x)|\leq\ep$. 
 In both cases, we can apply the result for the dimension $2$: we denote by $f_2^{(x)}$ the stream in $[0,n[\times [1,n]\times\{x\}$ with inputs $(g(i,x),i=1,\dots,n)$.
We can check as in the proof of lemma \ref{lem:mixing} that the stream 
$$\sum_{i=1}^n f^{(i)}_{(d-1)}+\sum_{x\in\{1,\dots,n\} ^{d-2}}f_2^{(x)}(\cdot - (d-2)n\overrightarrow{e_1})$$
satisfies all the required conditions.

\end{proof}
\subsection{Decomposition of a stream}
In all this section, we consider $(\Omega,\Gamma^1,\Gamma^2)$ that satisfy hypothesis \ref{hypo:omega}. Let $n\geq 1$.
We say that $\overrightarrow{\gamma}=(\overrightarrow{g_1},\dots,\overrightarrow{g_r})$ is an oriented self-avoiding path if there exists $r+1$ distinct points $x_1,\dots,x_{r+1}\in\E_n^d$ such that for any $i\in\{1,\dots,r\}$, $\overrightarrow{g_i}=\langle\langle x_i,x_{i+1}\rangle\rangle \in\overrightarrow{\E}_n^d$.

\begin{lem}[Decomposition of a stream]\label{lem:res}
Let $f_n$ be a stream inside $\Omega$ that satisfies the node law everywhere except points in $\Gamma_n^1\cup\Gamma_n^2$. There exists a finite set of self-avoiding oriented path $\oGam$ (that may be empty) such that for any $\overrightarrow{\gamma}\in \oGam$, the starting point and the ending point belong to $\Gamma_n ^1\cup\Gamma_n^2$, all the other vertices in $\ogam$ belong to $\Omega_n\setminus (\Gamma_n ^1\cup\Gamma_n^2)$. To each oriented path $\ogam\in\oGam$ we can associate a positive real number $p(\overrightarrow{\gamma})$ such that
$$f_n=\sum_{\overrightarrow{\gamma}\in\oGam}p(\ogam)\sum_{\langle\langle x,y\rangle\rangle\in\overrightarrow{\gamma}}n\,\overrightarrow{xy}\ind_{\langle x,y\rangle}\,.$$
Moreover, we have
\[\forall \ogam \in\oGam\quad\forall\, \overrightarrow{e}=\langle\langle x,y\rangle\rangle \in\ogam\qquad f_n(e)\cdot\overrightarrow{xy}>0\,.\]
\end{lem}

\begin{proof}
We are going to perform an algorithm to build iteratively the couple $(\overrightarrow{\Gamma},(p(\ogam))_{\ogam\in\overrightarrow{\Gamma}}) $. We set $$\widehat{f}_n=\sum_{\overrightarrow{\gamma}\in\oGam}p(\ogam)\sum_{\langle\langle x,y\rangle\rangle\in\overrightarrow{\gamma}}n\,\overrightarrow{xy}\ind_{\langle x,y\rangle}$$
and $f_n^{res}=f_n-\widehat{f}_n$.
At any step of the algorithm, we have
\begin{align}\label{cond:res}\forall \ogam \in\oGam\quad\forall\, \overrightarrow{e}=\langle\langle x,y\rangle\rangle \in\ogam\qquad f_n(e)\cdot\overrightarrow{xy}\geq\widehat{f}_n(e)\cdot\overrightarrow{xy}>0\,.
\end{align}
Moreover, for any $\ogam\in\oGam$, the path $\ogam$ has both of its endpoints in $\Gamma_n^1\cup\Gamma_n^2$ and all the other vertices in $\ogam$ belong to $\Omega_n\setminus (\Gamma_n^1\cup\Gamma_n^2)$. Consequently, at any step of the algorithm, the stream $f_n ^{res}$ satisfies the node law for any point in $\sZ_n^d\setminus(\Gamma_n^1\cup\Gamma_n^2)$.

We start with $\oGam=\emptyset$.
Let $x\in\Gamma_n^1\cup\Gamma_n ^2$ and $y\in\Omega\cap\sZ_n^d$ such that $e=\langle x,y\rangle\in\E_n^d$ and $f_n^{res}(e)\neq 0$.
We distinguish two cases either $f_n^{res}(e)\cdot \overrightarrow{xy}>0$ or $f_n^{res}(e)\cdot \overrightarrow{xy}<0$. Let us assume  $f_n^{res}(e)\cdot \overrightarrow{xy}>0$. 
Since $f_n^{res}$ satisfies the node law and since there exists only a finite number of self avoiding path using edges with endpoints $\Omega_n$, there exists $z\in (\Gamma_n^1\cup\Gamma_n^2)\setminus\{x\}$ and an oriented self-avoiding path $\ogam_0$ starting from $x$ and ending at $z$ with vertices in $\Omega_n$ such that the first edge of $\ogam_0$ is $\langle\langle x,y\rangle\rangle$ and for any $\overrightarrow{e_0}=\langle\langle w_0,w_1\rangle\rangle \in\ogam_0$, we have $f_n^{res}(e_0)\cdot \overrightarrow{w_0w_1}>0$.

If $f_n^{res}(e)\cdot \overrightarrow{xy}<0$. Then there exists $z\in\Gamma_n^1\cup\Gamma_n^2\setminus\{x\}$ and an oriented self-avoiding path $\ogam_0$ starting from $z$ and ending at $x$ with vertices in $\Omega_n$ such that the last edge of $\ogam_0$ is $\langle\langle y,x\rangle\rangle$ and for any $\overrightarrow{e_0}=\langle\langle w_0,w_1\rangle\rangle \in\ogam_0$, we have $f_n^{res}(e_0)\cdot \overrightarrow{w_0w_1}>0$.

Note that, up to removing a section of $\ogam_0$, we can always assume that all the vertices of $\ogam_0$ 
except its two endpoints are in $\Omega_n\setminus (\Gamma_n^1\cup\Gamma_n^2)$. If it is not the case, we denote by $w$ the first vertex in $\Gamma_n^1\cup\Gamma_n^2$ along the path $\ogam_0$ starting from $x$ and we replace $\ogam_0$ by the section of $\ogam_0$ between the vertices $w$ and $x$.

Besides, we have
$$m(\overrightarrow{\gamma}_0)=\inf\left\{f_n ^{res}(e)\cdot (n\,\overrightarrow{w_0w_1}):\,\overrightarrow{e}= \langle \langle w_0,w_1\rangle\rangle \in\overrightarrow{\gamma}_0\right\}>0\,.$$
Let  $\overrightarrow{e}=\langle\langle w_0,w_1\rangle\rangle \in\ogam_0$. By construction, we have
$f_n^{res}(e)\cdot\overrightarrow{w_0w_1}>0$ and 
$f_n(e)\cdot\overrightarrow{w_0w_1}>\widehat{f}_n(e)\cdot\overrightarrow{w_0w_1}$. Hence, $\langle\langle w_1,w_0\rangle\rangle$ cannot belong to one of the $\ogam$ in $\oGam$ since it would contradict \eqref{cond:res}. Necessarily, we have $\widehat{f}_n(e)\cdot\overrightarrow{w_0w_1}\geq 0$.
It yields that
\begin{align*}
0\leq \left(\widehat{f}_n+m(\ogam_0)\sum_{\langle\langle x_0,y_0\rangle\rangle\in\overrightarrow{\gamma}_0}n\,\overrightarrow{x_0y_0}\ind_{\langle x_0,y_0\rangle}\right)(e)\cdot\overrightarrow{w_0w_1}&=\widehat{f}_n(e)\cdot\overrightarrow{w_0w_1}+\frac{1}{n}m(\ogam_0)\\&\leq (\widehat{f}_n+f_n^{res})(e)\cdot\overrightarrow{w_0w_1}\leq f_n(e)\cdot\overrightarrow{w_0w_1}\,.
\end{align*}
We add $(\ogam_0,m(\ogam_0))$ to $(\oGam,(p (\ogam)_{\ogam\in\oGam})$ and the condition \eqref{cond:res} still holds.
We can iterate this process finitely many times with every possible self-avoiding oriented paths ending or starting with the edge $e$ (according to the sign of $f_n^{res}(e)\cdot \overrightarrow{xy}$). At any iteration, $|f_n^{res}(e)\cdot \overrightarrow{xy}|$  decrease. Eventually, the stream function we obtain satisfies $f_n^{res}(e)\cdot \overrightarrow{xy}=0$. 

The algorithm ends when for any $x\in\Gamma_n^1\cup\Gamma_n ^2$ and $y\in\Omega\cap\sZ_n^d$ such that $e=\langle x,y\rangle\in\E_n^d$, we have $f_n^{res}(e)= 0$. Consequently, at the end of the algorithm we have $f_n^{res}=0$ and
$$f_n=\widehat{f}_n=\sum_{\overrightarrow{\gamma}\in\oGam}p(\ogam)\sum_{\langle\langle x,y\rangle\rangle\in\overrightarrow{\gamma}}n\,\overrightarrow{xy}\ind_{\langle x,y\rangle}\,.$$
This concludes the proof.
\end{proof}

\section{Construction  and convexity of the elementary rate function\label{sect:basicbrick}}
In this section, we build the elementary rate function $I$ that is the basic brick to build the rate function $\widehat{I}$. We start by proving preliminary lemmas that we need in order to prove theorem \ref{thmbrique} but also theorem \ref{thm:ULDpatate} 
\subsection{Preliminary lemmas}
Before proving theorem \ref{thmbrique}, we are going to prove that we can slightly modify a stream $f_n\in\cS_n(\fC)$ without paying too much probability such that the stream is well-behaved in the sense that at a mesoscopic level for each face of the cube $\fC$ the stream spreads uniformly.

We recall that $(\overrightarrow{e_1},\dots,\overrightarrow{e_d})$ denotes the oriented canonical basis of $\sR^d$. For $i\in\{1,\dots,d\}$, we denote by $\fC_i^+$ and $\fC_i^-$ the two faces of $\fC=[-1/2,1/2[^d$ associated with the vector $\overrightarrow{e_i}$ that is
\[\fC_i^-=\left[-\frac{1}{2},\frac{1}{2}\right[^{i-1}\times\left\{-\frac{1}{2}\right\}\times \left[-\frac{1}{2},\frac{1}{2}\right[^{d-i}\qquad \text{and}\qquad   \fC_i^+=\left[-\frac{1}{2},\frac{1}{2}\right[^{i-1}\times\left\{\frac{1}{2}\right\}\times \left[-\frac{1}{2},\frac{1}{2}\right[^{d-i}\,.\] 
Let  $i\in\{1,\dots,d\}$ and let $\cA$ be an hyperrectangle normal to $\overrightarrow{e_i}$. We denote by $\E_n^{i,+}[\cA]$ and $\E_n^{i,-}[\cA]$ the following set of edges (see figure \ref{figra}):
\begin{align}\label{eq:defeni+}
\E_n^{i,+}[\cA]=\left\{e=\left\langle x,x+\frac{\overrightarrow{e_i}}{n}\right\rangle\in\E_n^d: \left]x,x+\frac{\overrightarrow{e_i}}{n}\right]\cap\cA\neq\emptyset\right\}
\end{align}
and
\begin{align}\label{eq:defeni-}
\E_n^{i,-}[\cA]=\left\{e=\left\langle x,x+\frac{\overrightarrow{e_i}}{n}\right\rangle\in\E_n^d: \left]x-\frac{\overrightarrow{e_i}}{n},x\right]\cap \cA\neq\emptyset\right\}\,.
\end{align}
The choice of the definitions of $\E_n^{i,-}$ and $\E_n^{i,+}$ is to ensure that for $\cA\subset\fC_i^- $ and $\cB\subset\fC_i^+$, we have $\E_n^{i,-}[\cA]\subset\E_n^d\cap\fC$ and $\E_n^{i,+}[\cB]\subset\E_n^d\cap\fC$.
\begin{figure}[H]
\begin{center}
\def\svgwidth{0.6\textwidth}
   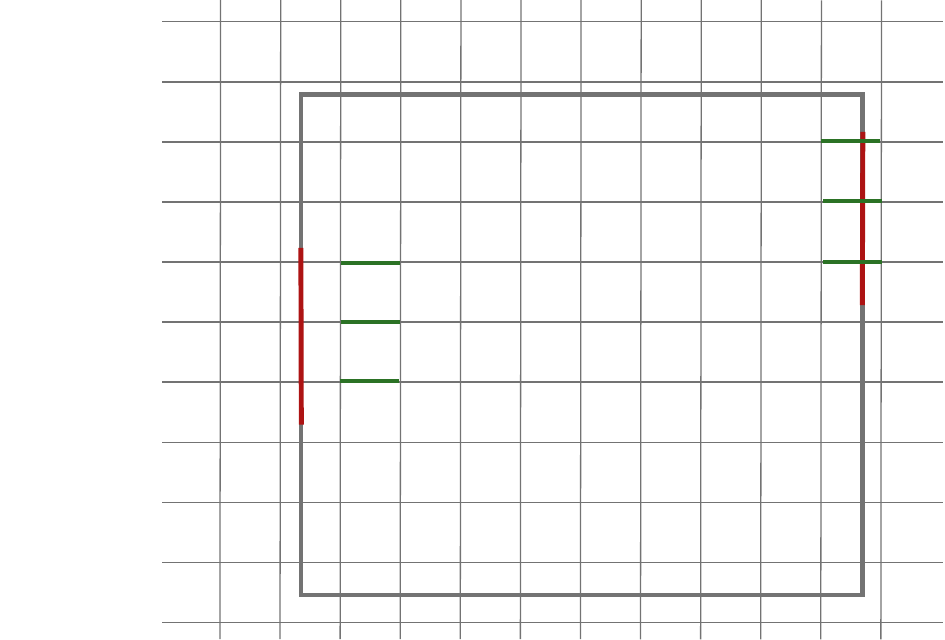
   \caption{\label{figra}The sets $\E_n^{1,-}[\cA]$ and $\E_n^{1,+}[\cB]$}
   \end{center}
\end{figure}

 Let $m\geq 1$.
 We partition all the faces of $\fC$ in hypersquares of side-length $1/m$.
We denote $\cP_i^+(m)$ and $\cP_i^-(m)$ the following sets (see figure \ref{fig6b})
\begin{align}\label{eq:defPi-}
\cP_i^-(m)=\left\{\left[-\frac{1}{2},-\frac{1}{2}+\frac{1}{m}\right[^{i-1}\times\left\{-\frac{1}{2}\right\}\times \left[-\frac{1}{2},-\frac{1}{2}+\frac{1}{m}\right[^{d-i}+\sum_{\substack{k=1,\dots,d\\k\neq i}}\frac{a_k}{m}\overrightarrow{e}_k:\,\begin{array}{c} a_k\in\{0,\dots,m-1\},\\ k\in\{1,\dots,d\}\setminus\{i\}\end{array}\right\}
\end{align}
and 
\begin{align}\label{eq:defPi+}
\cP_i^+(m)=\left\{\left[-\frac{1}{2},-\frac{1}{2}+\frac{1}{m}\right[^{i-1}\times\left\{\frac{1}{2}\right\}\times \left[-\frac{1}{2},-\frac{1}{2}+\frac{1}{m}\right[^{d-i}+\sum_{\substack{k=1,\dots,d\\k\neq i}}\frac{a_k}{m}\overrightarrow{e}_k:\,\begin{array}{c} a_k\in\{0,\dots,m-1\},\\ k\in\{1,\dots,d\}\setminus\{i\}\end{array}\right\}\,.
\end{align}
Note that for $A\in\cP_i^-(m)$, $\overrightarrow{e_i}$ is normal to $A$. The cube splits into $m ^{d-1}$ tubes according to the direction $\overrightarrow{e_i}$:
\[\fC\cup\fC_i^+=\bigcup_{A\in\cP_i^-(m)}\cyl(A,1,\overrightarrow{e_i})\,.\]
Note that for $A\in\cP_i^-(m)\cup\cP_i^+(m)$, we have $\cH^{d-1}(A)=1/m^{d-1}$.
\begin{figure}[h]
\begin{center}
\def\svgwidth{0.6\textwidth}
   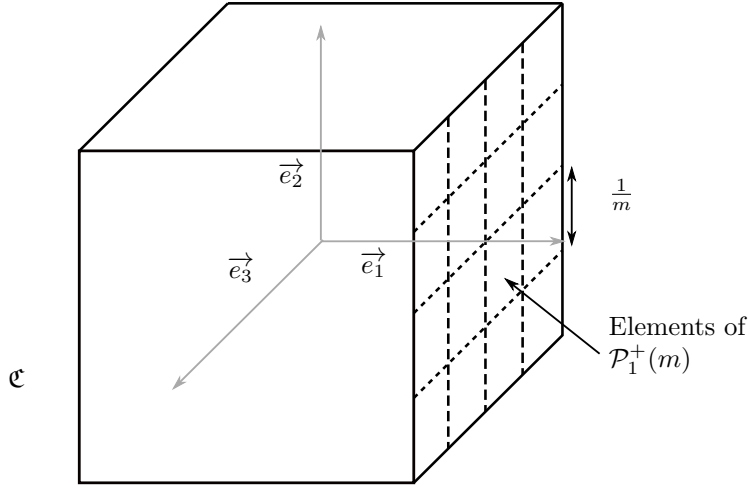
   \caption{\label{fig6b}Splitting the face $\fC_1 ^+$ into $m^{d-1}$ hyperrectangles}
   \end{center}
\end{figure}
Let $f_n\in\cS_n(\fC)$. For $\diamond\in\{+,-\}$ , $i\in\{1,\dots,d\}$ and $A\in\cP_i^\diamond(m)$, let us denote by $\psi_i^\diamond(f_n,A)$ the intensity of the stream $f_n$ through the face $A\in\cP_i^\diamond(m)$ in the direction $\overrightarrow{e_i}$, that is 
 \[\psi_i^\diamond(f_n,A)=\sum_{e\in\E_n^{i,\diamond}[A]}f_n\left(e\right)\cdot\overrightarrow{e_i}\,.\] 
%
Note that the intensity of the stream through a face is not uniform at the mescopic level. In the following lemma, we prove that there exists a repartition at a mesoscopic level for the intensity of the stream through the faces of $\fC$ that is more likely.
\begin{lem}\label{lem:lemexistencefamille}
Let $s>0$ and $\vv=(v_1,\dots,v_d)\in\sS^{d-1}$. Let $\ep>0$. There exists $\kappa_d$ and $\alpha$ depending only on $d$, for $m=\lfloor \ep ^{-\alpha}\rfloor$, for any $n\geq 1$ there exist two families of real numbers in $\sqrt{\ep}\sZ$, namely $(\lambda_{A}^+,A\in\cup_{i=1}^d\cP_i^+(m))$, $(\lambda_{A}^-,A\in\cup_{i=1}^d\cP_i^-(m))$, that satisfy
\begin{align}\label{cond:famille}
\forall \diamond\in\{+,-\}\quad\forall i\in\{1,\dots,d\}\quad\forall A\in\cP_i^\diamond(m)\qquad \left|\lambda_A^\diamond-sv_i\cH^{d-1}(A)n^{d-1}\right|\leq \kappa_d\frac{\ep^\alpha}{m^{d-1}} n ^{d-1}
\end{align}
and
\begin{align}\label{cond:famille2}
\sum_{A\in\cup_{i=1}^d\cP_i^+(m)}\lambda_A^+=\sum_{A\in\cup_{i=1}^d\cP_i^-(m)}\lambda_A^-
\end{align}
such that
\begin{align*}
&\liminf_{\ep\rightarrow 0}\limsup_{n\rightarrow\infty}\frac{1}{n^d}\log\Prb\left( \exists f_n\in\cS_n(\fC) : \begin{array}{c}\forall \diamond\in\{+,-\}\,\forall i\in\{1,\dots,d\}\,\forall A\in\cP_i^\diamond(m) \quad \psi_i^\diamond(f_n,A)=\lambda_A^\diamond\\ \text{and }\,\dis\big(\amu_n(f_n),s\vv\ind_{\fC}\cL^d\big)\leq \ep^{\alpha}\end{array}\right)\\
&\hspace{5cm}= \lim_{\ep\rightarrow 0}\limsup_{n\rightarrow\infty}\frac{1}{n^d}\log\Prb\left( \exists f_n\in\cS_n(\fC) : \,\dis\big(\amu_n(f_n),s\vv\ind_{\fC}\cL^d\big)\leq \ep\right)
\end{align*}
and \begin{align*}
&\liminf_{\ep\rightarrow 0}\liminf_{n\rightarrow\infty}\frac{1}{n^d}\log\Prb\left( \exists f_n\in\cS_n(\fC) : \begin{array}{c}\forall \diamond\in\{+,-\}\,\forall i\in\{1,\dots,d\}\,\forall A\in\cP_i^\diamond(m) \quad \psi_i^\diamond(f_n,A)=\lambda_A^\diamond\\ \text{and }\,\dis\big(\amu_n(f_n),s\vv\ind_{\fC}\cL^d\big)\leq \ep^{\alpha}\end{array}\right)\\
&\hspace{5cm}= \lim_{\ep\rightarrow 0}\liminf_{n\rightarrow\infty}\frac{1}{n^d}\log\Prb\left( \exists f_n\in\cS_n(\fC) : \,\dis\big(\amu_n(f_n),s\vv\ind_{\fC}\cL^d\big)\leq \ep\right)\,.
\end{align*}
\end{lem}
Note that in the statement of this lemma, we cannot chose the families $(\lambda_A^+)$ and $(\lambda_A ^-)$, we only know that these families exist. Moreover, notethat these families depend on $n$ and o$\ep$. Actually, if we consider families that satisfy condition \eqref{cond:famille} and \eqref{cond:famille2}, we can prove the same result for these families by slightly modifying the environment to create a new stream. The following lemma, which is an improvement of lemma \ref{lem:lemexistencefamille}, will be useful in what follows. We postpone its proof to the end of the proof of lemma \ref{lem:lemexistencefamille}.
\begin{lem}\label{lem:toutefamille}
Let $s>0$ and $\vv\in\sS^{d-1}$. Let $\kappa_d$ and $\alpha$ the positive constants from lemma \ref{lem:lemexistencefamille}. Let $\ep>0$. For $m=\lfloor \ep ^{-\alpha}\rfloor$, for any $n\geq 1$ for any families of real numbers $(\beta_{A}^+,A\in\cup_{i=1}^d\cP_i^+(m))$, $(\beta_{A}^-,A\in\cup_{i=1}^d\cP_i^-(m))$ (not necessarily in $\sqrt{\ep}\sZ$) that satisfy conditions \eqref{cond:famille} and \eqref{cond:famille2}, we have 
\begin{align*}
&\liminf_{\ep\rightarrow 0}\limsup_{n\rightarrow\infty}\frac{1}{n^d}\log\Prb\left( \exists f_n\in\cS_n(\fC) : \begin{array}{c}\forall \diamond\in\{+,-\}\,\forall i\in\{1,\dots,d\}\,\forall A\in\cP_i^\diamond(m) \\ \psi_i ^\diamond(f_n,A)=(1-\ep ^{\alpha/4})\beta_A^\diamond\\ \,\text{and } \dis\big(\amu_n(f_n),s\vv\ind_{\fC}\cL^d\big)\leq \ep^{\alpha_0}\end{array}\right)\\
&\hspace{5cm}= \lim_{\ep\rightarrow 0}\limsup_{n\rightarrow\infty}\frac{1}{n^d}\log\Prb\left( \exists f_n\in\cS_n(\fC) : \,\dis\big(\amu_n(f_n),s\vv\ind_{\fC}\cL^d\big)\leq \ep\right)
\end{align*}
and
\begin{align*}
&\liminf_{\ep\rightarrow 0}\liminf_{n\rightarrow\infty}\frac{1}{n^d}\log\Prb\left( \exists f_n\in\cS_n(\fC) : \begin{array}{c}\forall \diamond\in\{+,-\}\,\forall i\in\{1,\dots,d\}\,\forall A\in\cP_i^\diamond(m) \\\psi_i ^\diamond(f_n,A)=(1-\ep ^{\alpha/4})\beta_A^\diamond\\ \,\text{and }\dis\big(\amu_n(f_n),s\vv\ind_{\fC}\cL^d\big)\leq \ep^{\alpha_0}\end{array}\right)\\
&\hspace{5cm}=\lim_{\ep\rightarrow 0}\liminf_{n\rightarrow\infty}\frac{1}{n^d}\log\Prb\left( \exists f_n\in\cS_n(\fC) : \,\dis\big(\amu_n(f_n),s\vv\ind_{\fC}\cL^d\big)\leq \ep\right)
\end{align*}
where $\alpha_0$ is a constant depending on $\alpha$ and $d$.
\end{lem}

 \begin{proof}[Proof of lemma \ref{lem:lemexistencefamille}]
 The proof is divided into three steps. In the first step, we prove that if $f_n\in\cS_n(\fC)$ satisfies 
 $\dis\big(\amu_n(f_n),s\vv\ind_{\fC}\cL^d\big)\leq \ep$, then
 the flow for $f_n$ through any tube $\cyl(A,1,\overrightarrow{e_i})$ for $A\in\cP_i^-(m)$ is close to the value of the flow through this tube for $s\vv$. In a second step, we modify the stream in such a way that the corresponding flows through the tubes are in $\sqrt{\ep}\sZ$. This ensures that the possible values the flow can take at a mesoscopic level belong to a finite deterministic set. Finally, we do a pigeonhole principle to prove that there exists a deterministic set of possible values for the flows through the tubes that can be observed with a large enough probability.

Let $n\geq 1$. Let us consider $\omega\in\left\{\exists f_n\in\cS_n(\fC) : \,\dis\big(\amu_n(f_n),s\vv\ind_{\fC}\cL^d\big)\leq \ep\right\}.$ On the configuration $\omega$, we choose a stream $f_n(\omega)$ such that $ \dis\big(\amu_n(f_n),s\vv\ind_{\fC}\cL^d\big)\leq \ep$. If there are several possible choices, we select one according to a deterministic rule. For short, we write $\amu_n$ for $\amu_n(f_n)$.

\noindent{\bf Step 1: Control the incoming and outcoming flow in the tubes.} In this step we use some tools from \cite{CT1}.
The aim is now to show that the strength of the stream $f_n$ that flows through $A\in\cP_i^+(m)$ is close to $$\int_A s\vv\cdot \overrightarrow{e_i}\, d\cH^{d-1}(x)\,n ^{d-1}=sv_i\cH^{d-1}(A)\,n^{d-1}\,.$$
In \cite{CT1} (more precisely in the display equality after inequality (4.18)), the authors define the following flow through the bottom half of the cylinder $\cyl(A,h)$ for $h>0$:
\[\Psi(\amu_n,\cyl(A,h),\overrightarrow{e_i})=\sum_{\substack{e=\langle x,y\rangle\in\E_n^d: e\subset \cyl(A,h),\\x\in B'(A,h),\,y\notin B'(A,h)}}f_n(e)\cdot (n\overrightarrow{xy})\,\]
where we recall that $T'(A,h)$ and $B'(A,h)$ were defined in equalities \eqref{eq:defT(A,h)} and \eqref{eq:defB(A,h)}.
It is easy to check that the set of edges $\E_n ^{i,+}[A]$ is a cutset from $B'(A,h)$ to $T'(A,h)$ in $\cyl(A,h)$ which is minimal for the inclusion. By the node law, it follows that
$$\Psi(\amu_n,\cyl(A,h),\overrightarrow{e_i})=\psi_i^+(f_n,A)\,.$$
We refer to equation (4.20) in \cite{CT1} for more details about this fact.

Note that both expressions only depend on edges that have their left endpoint in $\cyl(A,h,-\overrightarrow{e_i})\subset \fC$. The value of the streams outside this set has no importance in the estimation made below.
Therefore, we can use the estimates proven in proposition 4.5 in \cite{CT1} and the expression of $h_0$ given just after the display inequality $(4.24)$. There exists $C(d)>0$ depending on $d$ such that 
$$\forall \eta>0\quad\forall h\leq  \frac{\eta}{C(d)M m}\qquad \left|\frac{1}{h}\int_{\cyl(A,h,-\overrightarrow{e_i})}\overrightarrow{e_i}\cdot d\amu_n(f_n)- \frac{\psi_i^+(f_n,A)}{n^{d-1}}\right|\leq \eta\,. $$
Hence, we have
\begin{align*}
\left|\frac{\psi_i^+(f_n,A)}{n^{d-1}}-s\vv\cdot\overrightarrow{e_i}\cH^{d-1}(A)\right|&\leq  \left| \frac{\psi_i^+(f_n,A)}{n^{d-1}}-\frac{1}{h}\int_{\cyl(A,h,-\overrightarrow{e_i})}\overrightarrow{e_i}\cdot d\amu_n(f_n)\right|\\
&\quad+\left|\frac{1}{h}\int_{\cyl(A,h,-\overrightarrow{e_i})}\overrightarrow{e_i}\cdot d\amu_n(f_n)-\frac{1}{h}\int_{\cyl(A,h,-\overrightarrow{e_i})}s\vv\cdot\overrightarrow{e_i}d\cL ^d(x)\right|\\
&\leq \eta +\frac{1}{h}\|\amu_n(\cyl(A,h,-\overrightarrow{e_i}))-s\vv\cL^d(\cyl(A,h,-\overrightarrow{e_i}))\|_2\,.
\end{align*}
We choose $h$ in such a way that $1/mh\in\sZ$: we set
\[h=\frac{1}{m}\left\lceil \frac{MC(d)}{\eta}\right\rceil^{-1}\,\]
where $\lceil x\rceil $ denotes the ceil of the real number $x$.
We have
$$h\leq \frac{1}{m}\frac{\eta}{MC(d)}\,.$$
Write 
\[h=\frac{\lambda}{2^j}\]
with $j\geq 1$ and $\lambda\in[1,2]$. 
Since $1/mh\in\sZ$, there exists $x\in[0,1]^d$ such that
\[\cyl(A,h,-\overrightarrow{e_i})\setminus A=\bigcup_{\substack{Q\in \Delta_\lambda^j \\(Q+x)\subset \cyl(A,h,-\overrightarrow{e_i})}}(Q+x)\,.\]
Hence, we have for $n$ large enough depending on $\ep$ and $M$
\begin{align*}
&\|\amu_n(\cyl(A,h,-\overrightarrow{e_i}))-s\vv\cL^d(\cyl(A,h,-\overrightarrow{e_i}))\|_2\\
&\hspace{2cm}\leq\|\amu_n(\cyl(A,h,-\overrightarrow{e_i})\setminus A)-s\vv\cL^d(\cyl(A,h,-\overrightarrow{e_i})\setminus A)\|_2+ \|\amu_n(A)\|_2\\
&\hspace{2cm}\leq \sum_{Q\in\Delta_\lambda^j}\|\amu_n(Q+x)-s\vv\cL^d(Q)\|_2+\frac{M}{n m^{d-1}}\\&\hspace{2cm}\leq 2^{j}\dis(\amu_n,s\vv\ind_\fC\cL^d)+\frac{M}{n m^{d-1}}\leq 2\frac{\lambda}{ h}\ep\,.
\end{align*}
It yields that
\[\left|\frac{\psi_i^+(f_n,A)}{n^{d-1}}-s\vv\cdot\overrightarrow{e_i}\cH^{d-1}(A)\right|\leq \eta  +\frac{2\lambda\ep}{h^{2}}\leq \frac{1}{m^{d-1}}\left(\eta m ^{d-1} +2\lambda\ep\left\lceil \frac{MC(d)}{\eta}\right\rceil^{2} m ^{d+1}\right) \,.\]
Set $\eta=\frac{1}{m ^d}$, we have
\begin{align*}
\left|\frac{\psi_i^+(f_n,A)}{n^{d-1}}-s\vv\cdot\overrightarrow{e_i}\cH^{d-1}(A)\right|&\leq \frac{1}{m^{d-1}}\left(\frac{1}{m} +8M^2C(d)^2\ep m^{3d+1}\right) \,.
\end{align*}
Setting $m=\left\lfloor\ep ^{-\alpha}\right\rfloor $
where 
\begin{align}\label{eq:defalpha}
\alpha=\frac{1}{2(3d+1)}\,
\end{align}
where $\lfloor x\rfloor$ denotes the integer part of the real number $x$. 
Consequently, there exists $\kappa_d$ depending on $d$ and $M$ such that
\begin{align}\label{eq:contpsi}
\left|\frac{\psi_i^+(f_n,A)}{n^{d-1}}-s\vv\cdot\overrightarrow{e_i}\cH^{d-1}(A)\right|\leq \kappa_d\frac{\ep^\alpha}{m^{d-1}} \,.
\end{align}
By the same arguments, we can prove the same result for $A\in\cP_i^-(m)$.

\noindent{\bf Step 2: Modify the stream in such a way the flow is in $\sqrt{\ep}\sZ$ in each tube.}
The aim is now to correct the stream so that $\psi_i ^\diamond(A,f_n)\in\sqrt{\ep}\sZ$ for $i=1,\dots,d$, $\diamond\in\{+,-\}$, $A\in\cP_i^\diamond(m)$. 
Using the arguments of the proof of lemma \ref{lem:res}, there exists $\overrightarrow{\Gamma}$ a set of self-avoiding oriented paths in $\fC$ such that for any path $\ogam\in\overrightarrow{\Gamma}$ only its first and last edges belong to $\cup_{i=1,\dots,d}\E_n^{i,-}[\fC_i^-]\cup\E_n^{i,+}[\fC_i^+]$, and we can associate a positive real number $p(\ogam)$ such that
\[f_n=\sum_{\ogam\in\overrightarrow{\Gamma}}p(\ogam)\sum_{\overrightarrow{e}=\langle\langle x,y\rangle\rangle\in\ogam}n\,\overrightarrow{xy}\ind_{e}\]
and
\[\forall \ogam\in\overrightarrow{\Gamma}\quad\forall \overrightarrow{e}=\langle\langle x,y\rangle\rangle\in \ogam \qquad f_n(e)\cdot \overrightarrow{xy}\geq 0 \,.\]
For $i,j\in\{1,\dots,d\}$, $\diamond,\circ\in\{+,-\}$, $A_1\in\cP_i^\diamond(m)$ and $A_2\in\cP_j^\circ(m)$, we set
\[g_n[A_1,A_2]=\sum_{\substack{\ogam\in\overrightarrow{\Gamma}:\\ \gamma^f\in\E_n^{i,\diamond}[ A_1],\gamma^l\in\E_n^{i,\circ}[ A_2]} }p(\ogam)\sum_{\overrightarrow{e}=\langle\langle x,y\rangle\rangle\in\ogam}n\,\overrightarrow{xy}\ind_{e}\,.\]
where $\gamma^f$ (respectively $\gamma ^l$) corresponds to the first (resp. last) edge of $\ogam$.
Hence, we have
$$f_n= \sum_{A_1\in\cup_{i=1}^d\cP_i^+(m)\cup\cP_i^-(m)}\,\,\sum_{A_2\in\cup_{k=1}^d\cP_k^+(m)\cup\cP_k^-(m)}g_n[A_1,A_2]\,.$$
On the configuration $\omega$, $g_n[A_1,A_2]\in\cS_n(\fC)$.
Moreover, we have
\[\forall A_1,A_2\in\cup_{i=1}^d\cP_i^+(m)\cup\cP_i^-(m)\quad\forall e\in\E_n^d\qquad f_n(e)\cdot g_n[A_1,A_2](e)\geq 0\,.\]
 Since this decomposition is not necessarily unique, we choose one according to a deterministic rule.
 Let $t$ be a real number, we define
\[\proj(t,\ep)=\sign(t) \sqrt{\ep}\left\lfloor\frac{|t|}{\sqrt{\ep}}\right\rfloor\,\]
where $\sign(t)$ corresponds to the sign of $t$.
We define $\widehat{f}_n$ in the following way
\[\widehat{f}_n= \sum_{\substack{1\leq i,j\leq d\\\diamond,\circ\in\{+,-\}}}\sum_{\substack{A_1\in\cP_i^\diamond(m)\\A_2\in\cP_j^\circ(m)}}\frac{\proj\big(\psi_i(g_n[A_1,A_2],A_1),\ep\big)}{\psi_i(g_n[A_1,A_2],A_1)}g_n[A_1,A_2]\ind_{\psi_i^\diamond(g_n[A_1,A_2],A_1)\neq 0}\,.\]
It is easy to check that on the configuration $\omega$, we have $\widehat{f}_n\in\cS_n(\fC)$ because we have
\[0\leq \frac{\proj\big(\psi_i(g_n[A_1,A_2],A_1)),\ep\big)}{\psi_i(g_n[A_1,A_2],A_1)}\leq 1\,.\]
Moreover, for $i\in\{1,\dots,d\}$, $\diamond\in\{+,-\}$ and $A_1\in\cP_i^\diamond(m)$, we have
\begin{align*}
\psi_i^\diamond( \widehat{f}_n ,A_1)&=\sum_{A_2\in\cup_{k=1}^d\cP_k^+(m)\cup\cP_k^-(m)}\frac{\proj\big(\psi_i^\diamond(g_n[A_1,A_2],A_1)),\ep\big)}{\psi_i^\diamond(g_n[A_1,A_2],A_1)}\psi_i^\diamond(g_n[A_1,A_2],A_1)\ind_{\psi_i^\diamond(g_n[A_1,A_2],A_1)\neq 0}\\
&=\sum_{A_2\in\cup_{k=1}^d\cP_k^+(m)\cup\cP_k^-(m)}\proj\big(\psi_i^\diamond(g_n[A_1,A_2],A_1)),\ep\big)
\in\sqrt{\ep}\sZ
\end{align*}
and since 
\[\forall t\in\sR\qquad |t-\proj(t,\ep)|\leq \sqrt{\ep}\,\]
we have
\[\left|\psi_i^\diamond( \widehat{f}_n ,A_1)-\psi_i^\diamond( f_n ,A_1)\right|\leq\card\left(\cup_{k=1}^d\cP_k^+(m)\cup\cP_k^-(m)\right)\sqrt{\ep}\leq  2dm ^{d-1}\sqrt{\ep}\,.\]
It follows that for $n$ large enough (depending on $\ep$) using inequality \eqref{eq:contpsi}, we have
\begin{align}\label{eqpsi}
\left|\psi_i^\diamond( \widehat{f}_n ,A_1)-sv_i\cH^{d-1}(A_1)n ^{d-1}\right|&\leq \left|\psi_i^\diamond( \widehat{f}_n ,A_1)-\psi_i^\diamond( f_n ,A_1)\right|+\left|\psi_i^\diamond( f_n ,A_1)-sv_i\cH^{d-1}(A_1)n^{d-1}\right|\nonumber\\&\leq 2\kappa_d\frac{\ep^\alpha}{m ^{d-1}}n^{d-1}\,.
\end{align}
Let us compute  the distance $\dis(\amu_n(\widehat{f}_n),\amu_n(f_n))$. Let $x\in [-1,1[^d$, $\lambda\in[1,2]$, let $k\geq 1$, let $Q\in\Delta_\lambda ^k$ such that $Q\cap\fC\neq \emptyset$, we have
\begin{align*}
&\left\|\amu_n(f_n)(Q+x)-\amu_n(\widehat{f}_n)(Q+x)\right\|_2 \\
&\hspace{0.5cm}\leq  \sum_{\substack{1\leq i,j\leq d\\\diamond,\circ\in\{+,-\}}}\sum_{\substack{A_1\in\cP_i^\diamond(m)\\A_2\in\cP_j^\circ(m)}}\left|1-\frac{\proj(\psi_i^\diamond(g_n[A_1,A_2],A_1),\ep)}{\psi_i^\diamond(g_n[A_1,A_2],A_1)}\right|\|\amu_n(g_n[A_1,A_2])(Q+x)\|_2\ind_{\psi_i^\diamond(g_n[A_1,A_2],A_1)\neq 0}\\
&\hspace{0.5cm}\leq \sum_{\substack{1\leq i,j\leq d\\\diamond,\circ\in\{+,-\}}}\sum_{\substack{A_1\in\cP_i^\diamond(m)\\A_2\in\cP_j^\circ(m)}}\frac{\sqrt{\ep} }{\psi_i^\diamond(g_n[A_1,A_2],A_1)}\|\amu_n(g_n[A_1,A_2])(Q+x)\|_2\ind_{\psi_i^\diamond(g_n[A_1,A_2],A_1)\neq 0}\,.
\end{align*}
Besides, we have by construction:
\[\psi_i^\diamond(g_n[A_1,A_2],A_1)=\sum_{\substack{\ogam\in\overrightarrow{\Gamma}:\\ \gamma^f\in\E_n^{i,\diamond}[ A_1],\gamma^l\in\E_n^{j,\circ}[ A_2]} }p(\ogam)\geq 0\,\]
and since the path $\ogam$ is self-avoiding, we have for $n$ large enough
\begin{align*}
\sum_{Q\in\Delta^k_\lambda }\|\amu_n(g_n[A_1,A_2])(Q+x)\|_2&\leq \sum_{Q\in\Delta^k_\lambda }\sum_{\substack{\ogam\in\overrightarrow{\Gamma}:\\ \gamma^f\in\E_n^{i,\diamond}[ A_1],\gamma^l\in\E_n^{j,\circ}[ A_2]} }p(\ogam)\frac{|\ogam \cap (Q+x)|}{n^d}\\
&\leq \sum_{\substack{\ogam\in\overrightarrow{\Gamma}:\\ \gamma^f\in\E_n^{i,\diamond}[ A_1],\gamma^l\in\E_n^{j,\circ}[ A_2]} }p(\ogam)
\frac{|\ogam \cap \fC|}{n^d}= d \psi_i^\diamond(g_n[A_1,A_2],A_1)\,
\end{align*}
where we use that 
\[|\{e\in\E_n^d:e\in\fC\}|=|\{e=\langle x,y\rangle\in\E_n^d:\, x\in\fC, \,\exists i\in\{1,\dots,d\}\quad n\,\overrightarrow{xy}=\overrightarrow{e_i}\}|=d|\fC\cap\sZ_n^d|=dn^d\,.\]
It follows that
\begin{align*}
\sum_{Q\in\Delta^k_\lambda }\left\|\amu_n(f_n)(Q+x)-\amu_n(\widehat{f}_n)(Q+x)\right\|_2 \leq \sqrt{\ep}\sum_{\substack{1\leq i,j\leq d\\\diamond,\circ\in\{+,-\}}}\sum_{\substack{A_1\in\cP_i^\diamond(m)\\A_2\in\cP_j^\circ(m)}}d\leq d\sqrt{\ep}(2dm^{d-1})^2
\end{align*}
and
\begin{align*}
\dis(\amu_n(\widehat{f}_n),\amu_n(f_n))&=\sup_{x\in[0,1]^d}\sup_{\lambda\in[1,2]}\sum_{k=0}^\infty\frac{1}{2^{k}}\sum_{Q\in\Delta^k_\lambda }\left\|\amu_n(f_n)(Q+x)-\amu_n(\widehat{f}_n)(Q+x)\right\|_2\\
&\leq 8d^3\sqrt{\ep} m ^{2(d-1)}\leq 8d^3\sqrt{\ep}\, \ep ^{-\frac{d-1}{3d+1}}\,.
\end{align*}
Hence, we have for $n$ large enough
\begin{align*}
\dis(\amu_n(\widehat{f}_n),s\vv\ind_\fC\cL^d)&\leq \dis(\amu_n(f_n),s\vv\ind_\fC\cL^d)+\dis(\amu_n(\widehat{f}_n),\amu_n(f_n))\\
&\leq \ep + 8d^3\ep ^{\frac{d+3}{2(3d+1)}}\,.
\end{align*}
Finally, for $\ep$ small enough depending on $d$, for $n$ large enough depending on $d$ and $\ep$, we have
$$\dis(\amu_n(\widehat{f}_n),s\vv\ind_\fC\cL^d)\leq\ep ^{\alpha}\,.$$
\noindent{\bf Step 3: Do a pigeonhole principle for possible values of $\psi_i^\diamond(\widehat{f}_n,A)$.}
We would like to project $\psi_i^\diamond(\widehat{f}_n,A)$ on the possible values it can take for $i=1,\dots,d$ and $A\in\cP_i^+(m)\cup\cP_i^-(m)$. Note that the two families $(\psi_i^\diamond(\widehat{f}_n,A),\,i=1,\dots,d, \, A\in\cP_i^+(m))$ and $(\psi_i^\diamond(\widehat{f}_n,A),\,i=1,\dots,d, \, A\in\cP_i^-(m))$ satisfy the conditions \eqref{cond:famille} and \eqref{cond:famille2}.
Since $\psi_i^\diamond(\widehat{f}_n,A)$ satisfies inequality \eqref{eqpsi}, for $n$ large enough, there are at most $4\kappa_d\ep ^\alpha n^{d-1}/(m^{d-1}\sqrt{\ep})$ possible values for $\psi_i^\diamond(\widehat{f}_n,A)$. It follows that by a pigeonhole principle, there exist two deterministic families $(\lambda_{A}^+,A\in\cup_{i=1}^d\cP_i^+(m))$ and $(\lambda_{A}^-,A\in\cup_{i=1}^d\cP_i^-(m))$ of real numbers in $\sqrt{\ep}\sZ$  that satisfies the condition \eqref{cond:famille} depending on $n$ and $\ep$ and such that
\begin{align}\label{eq3.1:2}
\Prb&\left( \exists \widehat{f}_n\in\cS_n(\fC) : \begin{array}{c}\forall \diamond\in\{+,-\}\,\forall i\in\{1,\dots,d\}\,\forall A\in\cP_i^\diamond(m) \quad \psi_i^\diamond(\widehat{f}_n,A)=\lambda_A^\diamond\\ \text{and }\,\dis\big(\amu_n(\widehat{f}_n),s\vv\ind_{\fC}\cL^d\big)\leq \ep^{\alpha}\end{array}\right)\nonumber\\
&\qquad\geq\left(\frac{m ^{d-1}\sqrt{\ep}}{4\kappa_d\ep^\alpha n^{d-1}}\right)^{2dm^{d-1}}\Prb\left( \exists f_n\in\cS_n(\fC) : \dis\big(\amu_n(f_n),s\vv\ind_{\fC}\cL^d\big)\leq \ep \right)\,.
\end{align}
The families $(\lambda_A^+)_A$ and $(\lambda_A^-)_A$ satisfy the condition \eqref{cond:famille2} since $\widehat{f}_n$ satisfies the node law.
Hence, we have by taking the limsup in $n$ and then the liminf in $\ep$ in inequality \eqref{eq3.1:2}
\begin{align}\label{eq3.1:2n}
\liminf_{\ep\rightarrow 0}\limsup_{n\rightarrow\infty}\frac{1}{n^d}\log\Prb&\left( \exists f_n\in\cS_n(\fC) : \begin{array}{c}\forall \diamond\in\{+,-\}\,\forall i\in\{1,\dots,d\}\,\forall A\in\cP_i^\diamond(m) \quad \psi_i^\diamond(f_n,A)=\lambda_A^\diamond\\\text{and } \,\dis\big(\amu_n(f_n),s\vv\ind_{\fC}\cL^d\big)\leq \ep^{\alpha}\end{array}\right)\nonumber\\
&\qquad\geq\lim_{\ep\rightarrow 0}\limsup_{n\rightarrow\infty}\frac{1}{n^d}\log\Prb\left( \exists f_n\in\cS_n(\fC) : \dis\big(\amu_n(f_n),s\vv\ind_{\fC}\cL^d\big)\leq \ep \right)\,.
\end{align}
Moreover, we have for all $n\geq 1$
\begin{align*}
\Prb&\left( \exists f_n\in\cS_n(\fC) : \begin{array}{c}\forall \diamond\in\{+,-\}\,\forall i\in\{1,\dots,d\}\,\forall A\in\cP_i^\diamond(m) \quad \psi_i^\diamond(f_n,A)=\lambda_A^\diamond\\ \text{and }\,\dis\big(\amu_n(f_n),s\vv\ind_{\fC}\cL^d\big)\leq \ep^{\alpha}\end{array}\right)\nonumber\\
&\qquad\leq\Prb\left( \exists f_n\in\cS_n(\fC) : \dis\big(\amu_n(f_n),s\vv\ind_{\fC}\cL^d\big)\leq \ep^{\alpha} \right)\,.
\end{align*}
It follows that
\begin{align}\label{eq3.1:2nn}
\liminf_{\ep\rightarrow 0}\limsup_{n\rightarrow\infty}\frac{1}{n^d}\log\Prb&\left( \exists f_n\in\cS_n(\fC) : \begin{array}{c}\forall \diamond\in\{+,-\}\,\forall i\in\{1,\dots,d\}\,\forall A\in\cP_i^\diamond(m) \quad \psi_i^\diamond(f_n,A)=\lambda_A^\diamond\\\text{and } \,\dis\big(\amu_n(f_n),s\vv\ind_{\fC}\cL^d\big)\leq \ep^{\alpha}\end{array}\right)\nonumber\\
&\qquad\leq\lim_{\ep\rightarrow 0}\limsup_{n\rightarrow\infty}\frac{1}{n^d}\log\Prb\left( \exists f_n\in\cS_n(\fC) : \dis\big(\amu_n(f_n),s\vv\ind_{\fC}\cL^d\big)\leq \ep^{\alpha} \right)\,.
\end{align}
Combining inequalities \eqref{eq3.1:2n} and \eqref{eq3.1:2nn} we obtain the equality. We can do the same computations by taking the liminf in $n$ instead of the limsup. The result follows.
\end{proof}

\begin{proof}[Proof of lemma \ref{lem:toutefamille}] 
Let $\ep>0$. Let $m=\lfloor \ep ^{-\alpha}\rfloor$ and $(\lambda_{A}^+,A\in\cup_{i=1}^d\cP_i^+(m))$ and $(\lambda_{A}^-,A\in\cup_{i=1}^d\cP_i^-(m))$ be the two families of real numbers in $\sqrt{\ep}\sZ$ defined in lemma \ref{lem:lemexistencefamille}. We consider the event 
\[\cE=\left\{ \exists f_n\in\cS_n(\fC) : \begin{array}{c}\forall \diamond\in\{+,-\}\,\forall i\in\{1,\dots,d\}\,\forall A\in\cP_i^\diamond(m) \quad \psi_i^\diamond(f_n,A)=\lambda_A^\diamond\\\text{and } \,\dis\big(\amu_n(f_n),s\vv\ind_{\fC}\cL^d\big)\leq \ep^{\alpha}\end{array}\right\}\,.\]
From now on, $f_n$ stands for a stream that satisfies the requirements stated in the previous event. If there are several possible choices, we select one according to a deterministic rule.
We consider two families of real numbers $(\beta_{A}^+,A\in\cup_{i=1}^d\cP_i^+(m))$ and $(\beta_{A}^-,A\in\cup_{i=1}^d\cP_i^-(m))$ that satisfy conditions \eqref{cond:famille} and \eqref{cond:famille2} (these families are not necessarily in $\sqrt{\ep}\sZ$).
The aim is now to correct the stream by modifying slightly the environment so that we obtain a stream $\widetilde{f}_n\in\cS_n(\fC)$ that satisfies
\[\forall \diamond\in\{+,-\}\quad\forall i\in\{1,\dots,d\}\quad\forall A\in\cP_i^\diamond(m) \qquad \psi_i^\diamond(\widetilde{f}_n,A)=(1-\ep ^{\alpha/4})\beta_A^\diamond\,.\]
Since both families satisfy condition \eqref{cond:famille}, we have
\begin{align}\label{eq:ecartbetalambda}
\forall \diamond\in\{+,-\}\quad\forall i\in\{1,\dots,d\}\quad\forall A\in\cP_i^\diamond(m) \qquad \left|\lambda_A^\diamond-\beta_{A}^\diamond\right|\leq 2\kappa_d \frac{\ep^{\alpha}}{m ^{d-1}}n ^{d-1}\,.
\end{align}
We set 
\begin{align}\label{eq:defkappa}
K=\left \lfloor \left(\frac{1}{2\kappa_d\ep ^{\alpha/2}}\right) ^{1/(d-1)}\right\rfloor\,.
\end{align}
For $\diamond\in\{+,-\}$,  $i\in\{1,\dots,d\}$ and $  A\in\cP_i^\diamond(m)$, we denote by $V_{A}$ the following set
\[V_{A}=\left\{x\in\fC:\{x\diamond\lambda\overrightarrow{e_i}:\lambda\in[0,1/n[\}\cap A\neq \emptyset , \,\forall j\neq i\quad x_j\in K \sZ/n\right\}\,.\]
We set $$V_{i}^\diamond=\bigcup_{A\in\cP_i ^\diamond }V_{A}\,.$$
Let us define the function $w_i$ on $V_i^+\cup V_i^-$ as follows:
\[\forall A\in\cP_i^\diamond(m)\quad \forall x\in V_{A}\qquad w_i(x)=\frac{\beta^\diamond_A-\lambda_A^{\diamond}}{|V_{A}|}\,.\]
Hence, we have using \eqref{eq:ecartbetalambda} and \eqref{eq:defkappa}
\begin{align}\label{eq:defwi}
 \forall x\in V_{A}\qquad|w_i(x)|\leq 2 \kappa_d \frac{\ep^\alpha}{m ^{d-1}}n ^{d-1}\frac{m ^{d-1}K ^{d-1}}{n^{d-1}}\leq \ep ^{\alpha/2}\,.
 \end{align}
For $i=1,\dots,d$, we denote by $\mu_i^\diamond$ the following quantity:
\[\mu_i ^\diamond=\sum_{A\in\cP_i^\diamond(m)}\beta_A ^\diamond -\lambda_A^{\diamond}\,.\]
The quantity $\mu_i^\diamond$ corresponds to the difference between the flow through $\fC_i^\diamond$ for $f_n$  and the flow we would like to obtain.
Since both families satisfy condition \eqref{cond:famille2}, then we have
\[\sum_{i=1} ^d\mu_i ^-=\sum_{i=1}^d\mu_i ^+\,.\]
We split the $2d$ faces of $\fC$ into three categories: the faces $\cF_{in}$ where there is an excess of flow, the faces $\cF_{out}$ where there is a default of flow and the faces $\cF_0$ where the difference of flow is null, \textit{i.e.},
\[\cF_{in}=\{\fC_i^-:\,\mu_i^->0,\,i=1,\dots,d\}\cup \{\fC_i^+:\,\mu_i^+<0,\,i=1,\dots,d\};\]
\[\cF_{out}=\{\fC_i^-:\,\mu_i^-<0,\,i=1,\dots,d\}\cup \{\fC_i^+:\,\mu_i^+>0,\,i=1,\dots,d\}\] and\[ \cF_{0}=\{\fC_i^\diamond:\,\mu_i^\diamond=0,\,i=1,\dots,d\}\,.\]
By the node law, we have
\[\sum_{\fC_i ^\diamond\in\cF_{in}}|\mu_i ^\diamond|=\sum_{\fC_j ^\diamond\in\cF_{out}}|\mu_j ^\diamond|\]
For any $i\in \{1,\dots,d\}$, we define the function $\mathfrak{p}_i:\sR^d\rightarrow \sR ^{d-1}$ as 
\begin{align}\label{eq:defpi}
\forall x=(x_1,\dots,x_d)\in\sR^d\qquad \mathfrak{p}_i(x)=(x_1,\dots,x_{i-1},x_{i+1},\dots,x_d)\,.
\end{align}
Let us assume $\fC_i^-\in \cF_0$. We denote by $f_n^{res}[\fC_i ^-]$, the stream given by lemma \ref{lem:mixbin} associated with the families \[\left(f_{in}(\mathfrak{p}_i(x)) =w_i(x),\,x\in \bigcup_{A\in\cP_i^-(m)}V_A\right)\] and the null family (that corresponds here to uniform outputs) \[\left(f_{out}(\mathfrak{p}_i(x)) =0,\,x\in V_i ^+\right)\]  in the direction $\overrightarrow{e_i}$. The number of edges on which the stream is not null is at most $3dn^d/K ^{d-2}$ and for each edge $e\in\E^d$, $\|f_n^{res}[\fC_i ^-](e)\|_2\leq \ep ^{\alpha/2}$ by \eqref{eq:defwi}.
We can do the same thing for $\fC_i ^+\in \cF_0$ in the direction $-\overrightarrow{e_i}$ with the families of null inputs
 \[\left(f_{in}(\mathfrak{p}_i(x)) =0,\,x\in V_i ^+\right)\]  and for outputs \[\left(f_{out}(\mathfrak{p}_i(x)) =w_i(x),\,x\in V_i ^-\right)\,.\]  We obtain a stream $f_n^{res}[\fC_i ^+]$.

Let us now consider $\fC_i^- \in \cF_{in}$ and $\fC_j ^+\in \cF_{out}$. Let $\alpha>0$ such that $\alpha\leq|\mu_i^-|,|\mu_j^+|$. Let us first assume that $i=j$. We denote by $ f_n ^{res}[\fC_i^-,\fC_i^+,\alpha]$ the stream given by lemma \ref{lem:mixbin} in the direction $\overrightarrow{e_i}$ with the families of inputs
\[\left(f_{in}(\mathfrak{p}_i(x)) =\frac{\alpha}{|\mu_i ^-|}w_i(x),\,x\in V_i ^-\right)\] 
and the family of outputs
\[\left(f_{out}(\mathfrak{p}_i(x)) =\frac{\alpha}{|\mu_i ^+|}w_i(x),\,x\in V_i ^+\right)\,.\] 
It remains to deal with the case $i\neq j$.
Let us call $\tau_{i,j}$ the bijection that inverts the $i^{th}$ coordinate with the $j^{th}$ one. Notice that $\tau_{i,j}(V_i^-)=V_j^-$. We can build 
a family of pairwise disjoint oriented paths $(\ogam_x, x\in V_i^-)$ of length at most $2n$, such that for $x\in V_i^-$ the path $\ogam_x$ joins $x$ to $\tau_{i,j}(x)$ in $\fC$ (see figure \ref{fig:cheminogam}):
$$\ogam_x=\sum_{k=0}^{x_j-x_i-1}\overrightarrow{e_i}\ind_{\langle x+k\overrightarrow{e_i}/n,x+(k+1)\overrightarrow{e_i}/n\rangle}-\sum_{k=0}^{x_j-x_i-1}\overrightarrow{e_j}\ind_{\langle x+(x_j-x_i)\overrightarrow{e_i}/n+k\overrightarrow{e_j}/n,x+(x_j-x_i)\overrightarrow{e_i}/n+(k+1)\overrightarrow{e_j}/n\rangle}\,.$$
$\ogam_x$ represents in fact the stream of intensity $1$ through this oriented path.
\begin{figure}[H]
\begin{center}
\def\svgwidth{0.6\textwidth}
   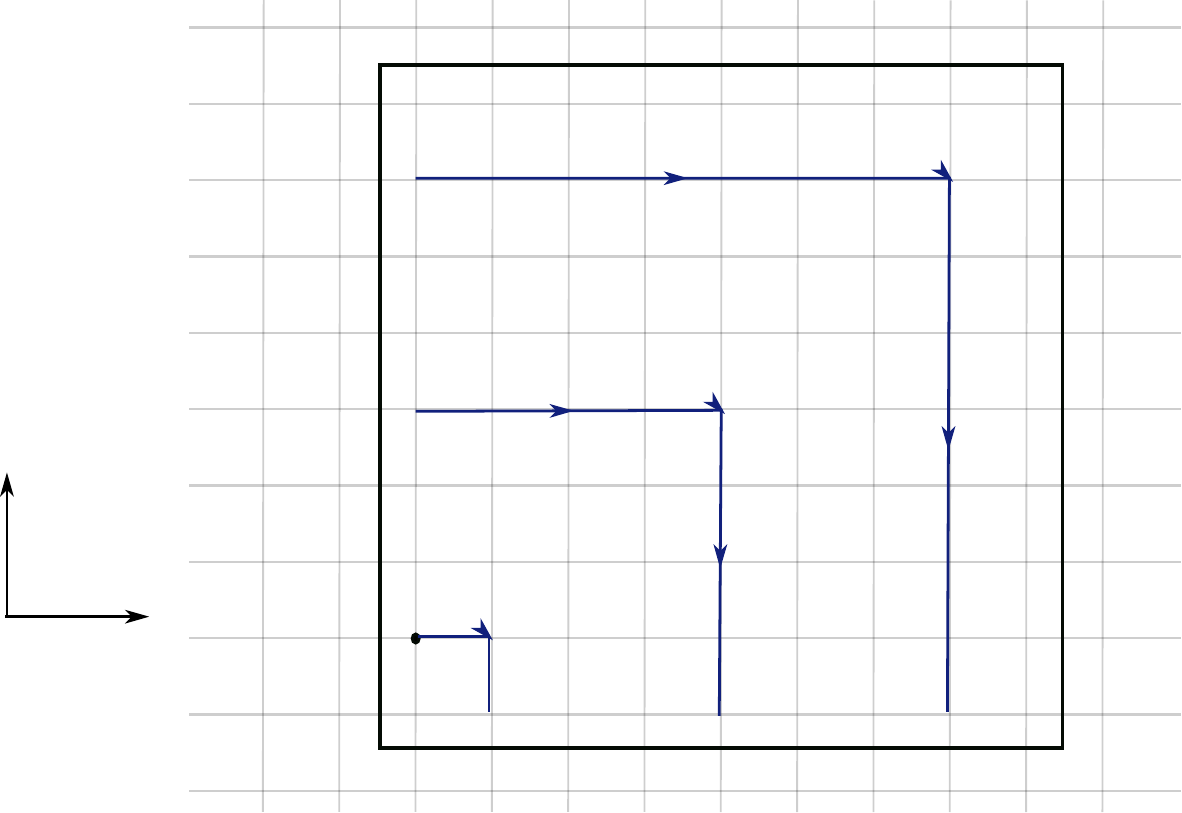
   \caption[figaaa]{\label{fig:cheminogam}The path $\overrightarrow{\gamma}_x$ for $x\in V_1^-$}
   \end{center}
\end{figure}
\noindent By lemma \ref{lem:mixbin}, we can build a stream $g_n^{i,j}$ in the direction $\overrightarrow{e_j}$
with the family of inputs
\[\left(f_{in}(\mathfrak{p}_j(x)) =\frac{\alpha}{|\mu_i ^-|}w_i(\tau_{i,j}^{-1}(x)),\,x\in V_j ^-\right)\] 
and the family of outputs
\[\left(f_{out}(\mathfrak{p}_j(x)) =\frac{\alpha}{|\mu_j ^+|}w_j(x),\,x\in V_j ^+\right)\,.\] 
Finally, we set 
\[f_n ^{res}[\fC_i^-,\fC_j^+,\alpha]=g_n^{i,j} +\sum_{x\in V_i^-} \frac{\alpha}{|\mu_i ^-|}w_i(x)\ogam_x\,.\]
We have
$$\sum_{x\in V_i ^-}|\ogam_x|\leq 2n|V_i^-|\leq 2n \frac{n^{d-1}}{K ^{d-1}}\leq 2\frac{n^d}{K ^{d-1}}\,.$$
Hence, using the previous inequality and lemma \ref{lem:mixbin}, we have that the stream $f_n ^{res}[\fC_i^-,\fC_j^+,\alpha]$ is supported by at most $2n^d/K ^{d-1}+3dn^d/K ^{d-2}\leq 4dn^d/K ^{d-2}$ for $\ep$ small enough depending on $d$. Moreover, for each edge $e\in\fC$, we have $\|f_n ^{res}[\fC_i^-,\fC_j^+,\alpha](e)\|_2\leq 2\ep ^{\alpha/2}$. By symmetry, the same construction holds for any $\fC_i ^\diamond \in \cF_{in}$ and $\fC_j ^\diamond\in\cF_ {out}$.

The aim is now to build a residual stream $f_n^{res}$ such that 
\[\forall i \in\{1,\dots,d\}\quad\forall \diamond\in\{-,+\}\quad\forall A\in\cP_i ^{\diamond}(m)\qquad \psi_i^\diamond(f_n^{res},A)=\beta^\diamond_A- \lambda^\diamond_A\,.\]
We do the following algorithm to build this stream.
\begin{algorithm}[H]
\caption{Build the stream $f_n ^{res}$}
\begin{algorithmic}
\STATE $f_n ^{res} \leftarrow 0$
\FOR{$i=1,\dots,d$, $\circ=-,+$}
\IF{$\fC_i ^{\circ}\in\cF_0$}
\STATE $f_n ^{res} \leftarrow f_n ^{res}+f_n ^{res}[\fC_i ^\circ]$
\ENDIF
\ENDFOR
\FOR{$i=1,\dots,d$, $\circ=-,+$}
\IF{$\fC_i ^{\circ}\in\cF_{in}$}
\WHILE {$|\psi_i^\circ(f_n ^{res},\fC_i^\circ)|< |\mu_i^\circ|$}
\STATE By the node law, there exists $\fC_j^\diamond\in\cF_{out}$ such that \[|\psi_j(f_n ^{res},\fC_j^\diamond)|< |\mu_j^\diamond|\,.\] We set $\alpha= \min(|\mu_i^\circ|-|\psi_i^\circ(f_n ^{res},\fC_i^\circ)|,|\mu_j^\diamond|-|\psi_j(f_n ^{res},\fC_j^\diamond)|)$. 
\STATE $f_n ^{res} \leftarrow f_n ^{res}+f_n ^{res}[\fC_i^\circ,\fC_j^\diamond,\alpha]$
\ENDWHILE
\ENDIF
\ENDFOR
\RETURN$f_n ^{res}$.
\end{algorithmic}
\end{algorithm}


\noindent The number of steps of this algorithm is at most $(2d)^2$. Finally, the stream $f_n ^{res}$ has its support included in a set $\Gamma$ such that 
\[|\Gamma|\leq \frac{\kappa'_d}{K ^{d-2}} n^d\] where $\kappa'_d$ depends only on the dimension. Moreover, each edge $e\in \Gamma$ is used a most twice at each step, hence we have  
\begin{align}\label{eq:contfnres}
\|f_n^{res}(e)\|_2\leq 2(2d)^2\ep ^{\alpha/2}\leq 8 d^2 \ep ^{\alpha/2}\,.
\end{align}
We set \[\widetilde{f}_n=(1-\ep ^{\alpha/4})(f_n+f_n ^{res})\,.\] This ensures that for $\ep$ small enough depending on $s$ and $d$, on the event $\{\forall e \in\Gamma \quad t(e)\geq s/(2d)\}$, we have $\widetilde{f}_n\in\cS_n(\fC)$.
Indeed, for $e\in\Gamma$, we have 
\begin{align*}
\|\widetilde{f}_n(e)\|_2&\leq (1-\ep^{\alpha/4})(\|f_n(e)\|_2+8d^2\ep ^{\alpha/2})\leq (1-\ep^{\alpha/4})\left(1+16d^3\frac{\ep ^{\alpha/2}}{s}\right)t(e)\leq t(e)
\end{align*}
for $\ep$ small enough depending on $d$ ans $s$.
Doing so we obtain
\[\forall \diamond\in\{+,-\}\,\forall i\in\{1,\dots,d\}\,\forall A\in\cP_i^\diamond(m) \quad \psi_i^\diamond(\widetilde{f}_n,A)=(1-\ep ^{\alpha/4})(\psi_i^\diamond(f_n,A)+\beta_A^\diamond-\lambda_A ^\diamond)=(1-\ep ^{\alpha/4})\beta_A^\diamond\,\]
where we recall that $\psi_i^\diamond(f_n,A)=\lambda_ A ^\diamond$.
Moreover, using inequality \eqref{eq:contfnres} and the expression of $K$ in terms of $\ep$ given by \eqref{eq:defkappa}, we obtain for $n$ large enough depending on $d$
\begin{align*}
\dis(\amu_n(\widetilde{f}_n),s\vv\ind_\fC\cL^d)&\leq \dis(\amu_n(f_n),s\vv\ind_\fC\cL^d)+\frac{2\ep^{\alpha/4}}{n^d}\sum_{e\in\E_n^d\cap\fC}\|f_n(e)\|_2+\frac{2}{n^d}\sum_{e\in\E_n^d\cap\fC}\|f_n^{res}(e)\|_2\\
&\leq \ep^{\alpha} +6d\ep^{\alpha/4}M+ 16d^2\ep ^{\alpha/2}\frac{|\Gamma|}{n^d}\\
&\leq\ep^{\alpha} +6d\ep^{\alpha/4}M+ 16d^2K'_ d\frac{\ep ^{\alpha/2}}{K ^{d-2}}\leq K \ep ^{\frac{\alpha}{2}\left(1-\frac{d-2}{d-1}\right)}=K\ep^{\frac{\alpha}{2(d-1)}}\leq\ep^{\frac{\alpha}{4(d-1)}}\,.
\end{align*}
where $K$ depends on $M$ and $d$ and the last inequalities holds for $\ep$ small enough depending on $d$.
We set 
\begin{align}\label{eq:defalpha0}
\alpha_0=\frac{\alpha}{4(d-1)}\,.
\end{align}
On the following event
$$\cE\cap\cE'=\left\{ \exists f_n\in\cS_n(\fC) : \begin{array}{c}\,\forall \diamond\in\{+,-\}\,\forall i\in\{1,\dots,d\}\,\forall A\in\cP_i^\diamond(m) \\ \psi_i^\diamond(f_n,A)=\lambda_A^\diamond\\ \text{and }\,\dis\big(\amu_n(f_n),s\vv\ind_{\fC}\cL^d\big)\leq \ep^{\alpha},\,\end{array}\right\}\cap\left\{\forall e\in\Gamma\quad t(e)\geq \frac{s}{2d}\right\}\,,$$
the stream $\widetilde{f}_n$ is admissible since it satisfies the capacity constraint.
Using the fact that the two events $\cE$ and $\cE'$ are increasing (requiring large capacities will always help to obtain a given stream), we have by FKG inequality
\begin{align}\label{eq3.1:3}
\Prb&\left(\exists \widetilde{f}_n\in\cS_n(\fC) : \begin{array}{c}\,\forall \diamond\in\{+,-\}\,\forall i\in\{1,\dots,d\}\,\forall A\in\cP_i^\diamond(m) \quad \psi_i^\diamond(\widetilde{f}_n,A)=(1-\ep ^{\alpha/4})\beta_A^\diamond\\ \text{and }\,\dis\big(\amu_n(\widetilde{f}_n),s\vv\ind_{\fC}\cL^d\big)\leq \ep ^{\alpha_0}\end{array}\right)\nonumber\\
&\geq\Prb\left(\left\{\exists f_n\in\cS_n(\fC) : \begin{array}{c}\,\forall \diamond\in\{+,-\}\,\forall i\in\{1,\dots,d\}\,\forall A\in\cP_i^\diamond(m) \\\psi_i^\diamond(f_n,A)=\lambda_A^\diamond\\\text{and } \,\dis\big(\amu_n(f_n),s\vv\ind_{\fC}\cL^d\big)\leq \ep^{\alpha}\end{array}\right\}\cap\left\{\forall e\in\Gamma\quad t(e)\geq \frac{s}{2d}\right\}\right)\nonumber\\
&\geq \Prb\left(\exists f_n\in\cS_n(\fC) : \begin{array}{c}\,\forall \diamond\in\{+,-\}\,\forall i\in\{1,\dots,d\}\,\forall A\in\cP_i^\diamond(m) \quad \psi_i^\diamond(f_n,A)=\lambda_A^\diamond\\ \text{and }\,\dis\big(\amu_n(f_n),s\vv\ind_{\fC}\cL^d\big)\leq \ep^{\alpha}\end{array}\right)\nonumber\\
 &\hspace{2cm}\times\Prb\left(\forall e\in\Gamma\quad t(e)\geq \frac{s}{2d}\right)\,.
\end{align}
Using the independence of the capacities, we get
\begin{align}\label{eq3.1:4}
\Prb\left(\forall e\in\Gamma\quad t(e)\geq \frac{s}{2d}\right) \geq G\left(\left[\frac{s}{2d},+\infty\right[\right)^{|\Gamma|}\geq G\left(\left[\frac{s}{2d},+\infty\right[\right)^{\kappa'_dn^d/K ^{d-2}}\,.
\end{align}
Combining inequalities \eqref{eq3.1:3} and \eqref{eq3.1:4}, we obtain for $\ep$ small enough (depending on $d$ and $s$)
\begin{align*}
\limsup_{n\rightarrow\infty}&\frac{1}{n^d}\log \Prb\left(\exists f_n\in\cS_n(\fC) : \begin{array}{c}\,\forall \diamond\in\{+,-\}\,\forall i\in\{1,\dots,d\}\,\forall A\in\cP_i^\diamond(m) \quad \psi_i^\diamond(f_n,A)=(1-\ep ^{\alpha/4})\beta_A^\diamond\\ \text{and }\,\dis\big(\amu_n(f_n),s\vv\ind_{\fC}\cL^d\big)\leq  \ep ^{\alpha_0}\end{array}\right)\\
&\geq \limsup_{n\rightarrow\infty}\frac{1}{n^d}\log \Prb\left( \exists f_n\in\cS_n(\fC) : \begin{array}{c}\,\forall \diamond\in\{+,-\}\,\forall i\in\{1,\dots,d\}\,\forall A\in\cP_i^\diamond(m) \quad \psi_i^\diamond(f_n,A)=\lambda_A^\diamond\\ \text{and }\,\dis\big(\amu_n(f_n),s\vv\ind_{\fC}\cL^d\big)\leq \ep^{\alpha}\end{array}\right)\\
&\qquad+\frac{\kappa'_d}{K^{d-2}} \log G\left(\left[\frac{s}{2d},+\infty\right[\right)\,
\end{align*}
where we recall that $K$ goes to infinity when $\ep$ goes to $0$.
Finally, by lemma \ref{lem:lemexistencefamille}, we obtain by letting $\ep$ goes to $0$ and choosing a fixed $s\leq 2d M$:
\begin{align*}
\liminf_{\ep\rightarrow 0}\limsup_{n\rightarrow\infty}&\frac{1}{n^d}\log \Prb\left(\exists f_n\in\cS_n(\fC) : \begin{array}{c}\,\forall \diamond\in\{+,-\}\,\forall i\in\{1,\dots,d\}\,\forall A\in\cP_i^\diamond(m) \\\psi_i^\diamond(f_n,A)=(1-\ep ^{\alpha/4})\beta_A^\diamond\\ \text{and }\,\dis\big(\amu_n(f_n),s\vv\ind_{\fC}\cL^d\big)\leq \ep ^{\alpha_0}\end{array}\right)\\
&\geq \lim_{\ep\rightarrow 0}\limsup_{n\rightarrow\infty}\frac{1}{n^d}\log \Prb\left( \exists f_n\in\cS_n(\fC) : \,\dis\big(\amu_n(f_n),s\vv\ind_{\fC}\cL^d\big)\leq \ep\right)\,.
\end{align*}
Moreover, we have 
\begin{align*}
\liminf_{\ep\rightarrow 0}\limsup_{n\rightarrow\infty}&\frac{1}{n^d}\log \Prb\left(\exists f_n\in\cS_n(\fC) : \begin{array}{c}\,\forall \diamond\in\{+,-\}\,\forall i\in\{1,\dots,d\}\,\forall A\in\cP_i^\diamond(m) \\\psi_i^\diamond(f_n,A)=(1-\ep ^{\alpha/4})\beta_A^\diamond\\ \text{and }\,\dis\big(\amu_n(f_n),s\vv\ind_{\fC}\cL^d\big)\leq \ep ^{\alpha_0}\end{array}\right)\\
&\leq \lim_{\ep\rightarrow 0}\limsup_{n\rightarrow\infty}\frac{1}{n^d}\log \Prb\left( \exists f_n\in\cS_n(\fC) : \,\dis\big(\amu_n(f_n),s\vv\ind_{\fC}\cL^d\big)\leq \ep^{\alpha_0}\right)\,.
\end{align*}
This yields the result. The same result holds for the liminf.
\end{proof}
\begin{defn}Let $s>0$ and $\vv\in\sS^{d-1}$. Let $\pi$ be an homothety of $\sR^d$ (see \eqref{eq:defpixalpha}). We will say that $f_n\in\cS_n(\pi(\fC))$ is $(\ep$, $s\vv$,$\pi$)-well-behaved if
\[\forall \diamond\in\{+,-\}\,\forall i\in\{1,\dots,d\}\,\forall A\in\cP_i^\diamond(m) \quad \psi_i^\diamond(f_n,\pi(A))=(1-\ep ^{\alpha/4})(s\vv\cdot \overrightarrow{e_i})\, \cH^{d-1}(\pi(A))n ^{d-1}\,.\]
If $\pi=\mathrm{Id}$, we will write $(\ep,s\vv)$-well-behaved instead of $(\ep,s\vv,\mathrm{Id})$-well-behaved.
\end{defn}

\subsection{Definition and existence of the elementary rate function}
\begin{proof}[Proof of Theorem \ref{thmbrique}] Here $M>0$ denotes the supremum of the support of $G$.
Let $\vv\in\sS^{d-1}$. Let $s>0$. Let $\ep>0$. Let $m=\lfloor \ep ^{-\alpha}\rfloor$ where $ \alpha$ was defined in \eqref{eq:defalpha}. Let $N,n\in\sN$ such that $n\leq N$ and $n\geq m+1$. 
Write $$\fC'=\frac{n}{N}\fC\,.$$
\noindent{\bf Step 1: Paving $\fC$ with smaller cubes.}
We want to almost cover $\fC$ with translates of $\fC'$ by letting enough space between them to allow to reconnect the streams inside the different translates of $\fC'$ together. Let us set \[p=\left\lfloor n\left(1+\frac{2d}{m}\right)\right\rfloor\,.\]
\begin{figure}[h]
\begin{center}
\def\svgwidth{0.6\textwidth}
   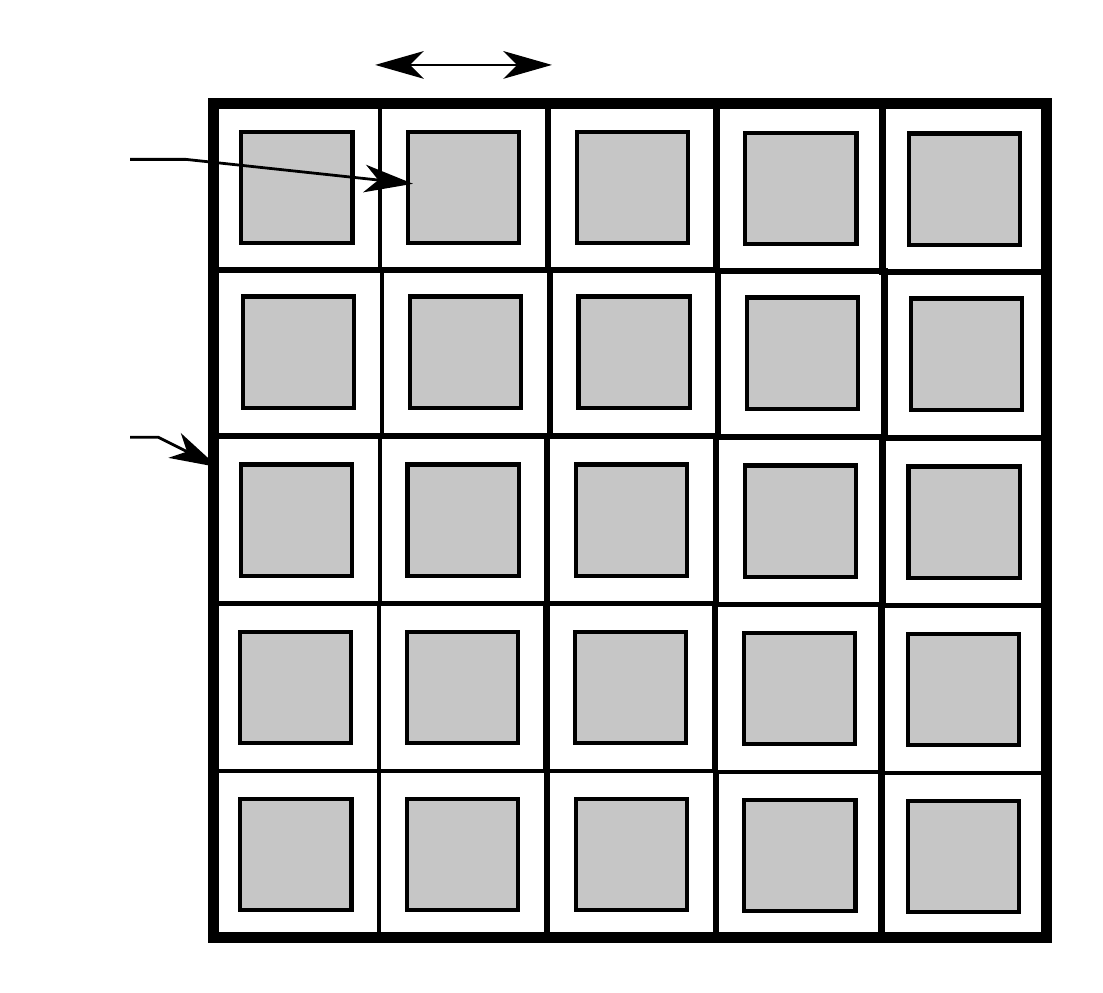
   \caption{\label{fig1}Paving $\fC$ with translates of $\fC'$}
   \end{center}
\end{figure}

\noindent Let $(x_i)_{i\in I}$ be the points in $p\sZ^d_{N}\cap (1-p/N)\fC$. Any two distinct points $x_i$ and $x_j$, $i,j\in I$, are at distance at least $p/N$ from each other. We set $$\pi_i=\pi_{x_i,\frac{n}{N}}\text{, for  }i\in I\,$$
where we recall that $\pi_{x_i,n/N}$ was defined in \eqref{eq:defpixalpha}. We have $\pi_i(\fC)=x_i+\fC'$.
The family $(\pi_i(\fC))_{i\in I}$ is a disjoint collection of translates of $\fC'$ such that 
\[\forall i \in I \qquad  \pi_i(\fC)\subset\fC\,.\]
We define the set $\Cor=\fC\setminus \cup_{i\in I}\pi_i(\fC)$. The set $\Cor$ represents the "corridor", this space will allow the streams in different $\pi_i(\fC)$ to be connected altogether (see figure \ref{fig1}).
It is easy to check that $\pi_i(\sZ_n^d)=\sZ_N^d$ and so that $\pi_i$ induces a bijection from $\E_n^d$ to $\E_N^d$.
We write 
$$\cE_i=\left\{\,\exists f_N\in\cS_N(\pi_i(\fC)) \text{ $(\ep,s\vv,\pi_i)$-well-behaved}: \,\dis\big(\amu_N(f_N)\ind_{\pi_i(\fC)},s\vv\ind_{\pi_i(\fC)}\cL^d\big)\leq 4 \frac{n^d}{N^d}\ep ^{\alpha_0}\,\right\}\,$$
where $\alpha_0$ was defined in \eqref{eq:defalpha0}.
On the event $\cE_i$, we will denote by $f_{N}^{(i)}$ a well-behaved stream satisfying the property described in $\cE_i$ (chosen according to a deterministic rule if there is more than one such stream). We denote by $\Cor_N$ the edges in $\E_N^d$ whose left endpoints are in $\Cor$:
$$\Cor_N=\left\{\langle x,y\rangle\in\E_N^d: x\in\Cor \text{ and }\exists i\in\{1,\dots,d\}\quad\overrightarrow{xy}=\frac{\overrightarrow{e_i}}{N}\right\}\,.$$
Let us denote by $\cF$ the event 
\[\cF=\left\{\, \forall e\in \Cor_N\qquad t(e)\geq M-\mathrm{H}(\ep)\right\}\]
where $\mathrm{H}:\sR_+\rightarrow\sR_+$ is a function we will chose later in such a way $\lim_{\ep\rightarrow0}\mathrm{H}(\ep)=0$.
We aim to prove that on the event $\cF\cap\bigcap_{i\in I }\cE_i$, we can build a stream $f_N\in \cS_N(\fC)$ such that $f_N$ coincides with $f_N^{(i)}$ on the cubes $\pi_i(\fC)$, $i\in I$ and 
$$\forall e\in \E_N^d\cap \Cor\qquad \|f_N(e)\|_2\leq  M-\mathrm{H}(\ep)\,.$$
\noindent{\bf Step 2: Reconnecting streams in the different cubes.}
We now explain how to reconnect the streams in the different cubes. Let $i,j \in I$ such that $\|x_i-x_j\|_1=p/N$. There exists $l\in\{1,\dots,d\}$ such that \[x_j=\frac{p}{N}\overrightarrow{e_l}+x_i\,.\]
Note that for all $A\in\cP_l^+ (m)$, we have on the event $\cE_i\cap\cE_j$
\[\psi_l^+\left(f_N^{(i)},\pi_i(A)\right)=(1-\ep ^{\alpha/4})(s\vv\cdot \overrightarrow{e_l})\, \cH^{d-1}(\pi_i(A))n ^{d-1}=\psi_l^-\left(f_{N}^{(j)},\pi_j\left(A-\overrightarrow{e_l}\right)\right)\,.\]
\begin{figure}[h]
\begin{center}
\def\svgwidth{0.5\textwidth}
   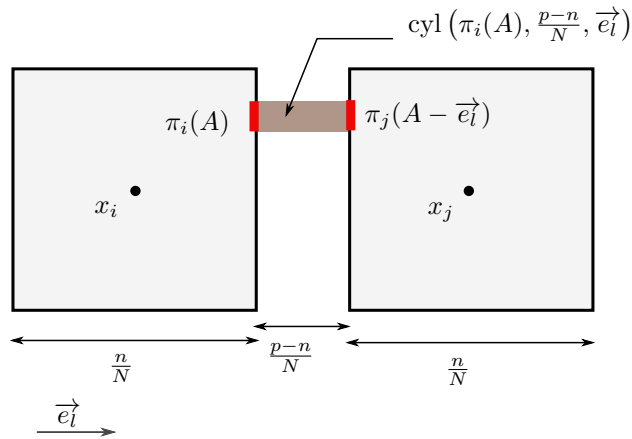
   \caption{\label{fig6c}Connecting streams in two adjacent cubes at mesoscopic level}
   \end{center}
\end{figure}

\noindent We can therefore apply lemma \ref{lem:mixing} to connect these two streams using only edges in the cylinder $\cyl(\pi_i(A),(p-n)/N,\overrightarrow{e_l})$ (see figure \ref{fig6c}) since $(p-n)/N\geq 2(d-1)/(Nm)$. We denote by $f_N ^{mix, i,j,A}$ the corresponding stream. Note that each edge in the corridor is used at most once by the streams $f_N ^{mix, i,j,A}$. We set
\[f_N^{mix,i,j}=\sum_{A\in\cP_l^+ (m)}f_N^{mix,i,j,A}\,.\]
For $i\in I$, $\diamond\in\{+,-\}$ such that there exists $l\in \{1,\dots,d\}$ with $x_i\diamond p/N\overrightarrow{e_l}\notin \{x_j,j\in I\}$, we have to connect the stream $f_N^{(i)}$ to the boundary of $\fC$ by exiting the water in the straight direction $\overrightarrow{e_l}$. More formally, we set
\[g_N^{i,l,+}=\sum_{e\in\E_N^{l,+}[ \pi_i(\fC_l^+)]}f_N^{(i)}(e)\sum_{k=1}^{p-n} \ind_{e+k\frac{\overrightarrow{e_l}}{N}}\ind_{\left(e+k\frac{\overrightarrow{e_l}}{N}\right)\in\fC}\,\]
and
\[g_N^{i,l,-}=\sum_{e\in\E_N^{l,-}[ \pi_i(\fC_l^-)]}f_N^{(i)}(e)\sum_{k=1}^{p-n} \ind_{e-k\frac{\overrightarrow{e_l}}{N}}\ind_{\left(e-k\frac{\overrightarrow{e_l}}{N}\right)\in\fC}\,\]
where for an edge $e=\langle x,y\rangle\in\E_N^d$ for $z\in\sZ_N^d$, we denote by $e+z$ the edge $\langle x+z,y+z\rangle\in\E_N^d$.
Finally, we set 
\[f_N=\sum_{i\in I}f_N^{(i)}+\sum_{(i,j)\in I^2:\|x_i-x_j\|_1=\frac{p}{N}}f_N^{mix,i,j}+\sum_{i\in I}\sum_{\substack{l=1,\dots,d, \diamond\in\{+,-\}: \\x_i\diamond \frac{p}{N}\overrightarrow{e_l}\notin\{x_j,j \in I \}}}g_N^{i,l,\diamond}\,.\]
By construction, $f_N$ coincides with $f_N^{(i)}$ on all $\fC'+x_i$ for $i\in I$. But, the value of $\|f_N(e)\|_2$ may exceed $M-\mathrm{H}(\ep)$ for edges in the corridor. To fix this issue we consider the stream $\widehat{f}_N=(1-\mathrm{H}(\ep)/M)f_N$. On the event $\cF\cap\bigcap_{i\in I}\cE_i$, the stream $\widehat{f}_N$ is in $\cS_N(\fC)$.

\noindent{\bf Conclusion.}
Using lemma \ref{lem:propdis4}, we obtain
\begin{align}\label{eq:contdis1}
\dis( \amu_N(f_N),s\vv\ind_{\fC}\cL^d)&\leq  \sum_{i\in I}\dis( \amu_N(f_N^{(i)}),s\vv\ind_{\pi_i(\fC)}\cL^d)+\dis( \amu_N(f_N)\ind_{\Cor},s\vv\ind_{\Cor}\cL^d)\nonumber\\
&\leq 4|I|\frac{n^d}{N^d}\ep^{\alpha_0}+ 2s\cL^d(\Cor)+2M\frac{|\Cor_N|}{N^d}\,.
\end{align}
Moreover, we have
\begin{align}\label{eq:contdis2}
|I|\leq \frac{N^d}{\lfloor n (1+ 2d/m)\rfloor ^d}\leq \frac{N^d}{n^d}
\end{align}
and since $(1-2p/N)\fC\subset \cup_{i\in I}(x_i+(p/N)\fC)$, it follows that
\[|I|\geq \left(\frac{1-2p/N}{p/N}\right)^d=\left(\frac{N}{\lfloor n(1+ 2d/m)\rfloor}-2\right)^d\geq \left(\frac{N}{ n(1+ 2d/m)}-2\right)^d\]
and
\begin{align}\label{eq:contdis3}
\cL ^d(\Cor)&\leq 1-\frac{n^d}{N^d}|I|\leq 1-\left(\frac{1}{1+ 2d/m}-2\frac{n}{N}\right)^d\,.
\end{align}
We have using proposition \ref{prop:minkowski}, for $N$ large enough depending on $d$ and $\ep$ and $n$
\begin{align}\label{eq:contdis4}
|\Cor_N|\leq 2d\frac{\cL^d(\cV_2(\Cor,d/N))}{1/N^d}&\leq 2d\left(\cL^d(\Cor)+\cL^d(\cV_2(\partial \Cor,d/N))\right)N^d\nonumber\\&\leq 2d\left(\cL^d(\Cor)+\frac{4d}{N}\cH^{d-1}(\partial\Cor)\right)N^d 
\,.
\end{align}
We also have using inequality \eqref{eq:contdis2}
\begin{align}\label{eq:contdis5}
\cH^{d-1}(\partial \Cor)\leq \cH^{d-1}(\partial \fC)+\sum_{i\in I}\cH ^{d-1}\left(\partial(\pi_i(\fC)\right)\leq 2d + 2d|I| \left(\frac{n}{N }\right)^{d-1}\leq 2d+ 2d\frac{N}{n}\,.
\end{align}
Combining inequalities \eqref{eq:contdis1}, \eqref{eq:contdis2}, \eqref{eq:contdis3}, \eqref{eq:contdis4} and \eqref{eq:contdis5}, we obtain for $N, n$ large enough depending on $\ep$, on the event $\cF\cap\bigcap_{i\in I}\cE_i$, that
\begin{align}
\dis(\amu_N(f_N),s\vv\ind_{\fC}\cL^d)\leq g(\ep)
\end{align}
where $g:\sR_+\rightarrow\sR_+$ is a function such that $\lim_{\ep\rightarrow 0}g(\ep)=0$.
It follows that for $N$ large enough
\begin{align}
\dis(\amu_N(\widehat{f}_N),s\vv\ind_{\fC}\cL^d)&\leq \dis(\amu_N(\widehat{f}_N),\amu_N(f_N))+\dis(\amu_N(f_N),s\vv\ind_{\fC}\cL^d)\nonumber\\
&\leq \frac{2\mathrm{H}(\ep)}{MN^d}\sum_{e\in\E_N^d\cap\fC}\|f_N(e)\|_2+g(\ep)\leq 2d\frac{\mathrm{H}(\ep)}{M}+g(\ep)\,
\end{align}
where we recall that $|\E_N^d\cap\fC|=dN^d$.
Hence, using the independence and  inequality \eqref{eq:contdis2}, we obtain
\begin{align}\label{eq:contdis6}
\frac{1}{N^d}&\log\Prb\left(\exists f_N\in\cS_N(\fC):\,\dis(\amu_N(f_N),s\vv\ind_{\fC}\cL^d)\leq g(\ep)+2d\frac{\mathrm{H}(\ep)}{M}\right)\nonumber\\
&\geq \frac{1}{N^d}\log\Prb\left(\cF\cap\bigcap_{i\in I}\cE_i\right)=\frac{1}{N^d}\log\Prb(\cF)+\frac{1}{N^d}|I|\log\Prb(\cE_1)\nonumber\\
&\geq \frac{|\Cor_N|}{N^d}\log G([M-\mathrm{H}(\ep),M])+\frac{1}{n^d} \log\Prb(\cE_1)\,.
\end{align}
We define $\mathrm{H}(\ep)$ as follows:
\begin{align}\label{eq:defh}
\mathrm{H}(\ep)=\inf\left\{a>0: G([M-a,M])\geq 1-\left(\frac{1}{1+2d/m}\right)^d\right\}\,.
\end{align}
We recall that $m=\lfloor \ep^{-\alpha}\rfloor$. It is clear that $\mathrm{H}$ is non-decreasing. We denote by $l=\lim_{\ep\rightarrow 0}\mathrm{H}(\ep)$. Let us assume that $l>0$. By defintion of $\mathrm{H}$ it follows that
$$\forall \ep>0\qquad G([M-l/2,M])<1-\left(\frac{1}{1+2d/m}\right)^d\,$$
and so $G([M-l/2,M])=0$. This contradicts the fact that $M$ is the supremum of the support of $G$. Hence, $l=0$.
 Thanks to inequality \eqref{eq:contdis3} and by definition of $\mathrm{H}$, we have
\begin{align*} 
\left(1-\left(\frac{1}{1+2d/m}\right)^d\right) \log\left(1-\left(\frac{1}{1+2d/m}\right)^d\right)
 & \leq\liminf_{N\rightarrow\infty} \cL^d(\Cor)\log G([M-\mathrm{H}(\ep),M])\leq 0\,.
 \end{align*}
 Since  $$\lim_{\ep\rightarrow 0}\left(1-\left(\frac{1}{1+2d/m}\right)^d\right) \log\left(1-\left(\frac{1}{1+2d/m}\right)^d\right)=0\,,$$
 it follows that
\begin{align}\label{eq:cortend0} 
 \lim_{\ep\rightarrow 0} \liminf_{N\rightarrow\infty}\cL^d(\Cor)\log G([M-\mathrm{H}(\ep),M])=0\,.
 \end{align}
We admit the following result, we postpone its proof (see lemma \ref{lem:scaling1} below).
\begin{align}\label{eq:scaling}
\Prb(\cE_i)=\Prb&\left(\exists f_N\in\cS_N(\pi_i(\fC)) \text{ $(\ep,s\vv,\pi_i)$-well-behaved}: \,\dis\big(\amu_N(f_N),s\vv\ind_{\pi_i(\fC)}\cL^d)\leq 4\frac{n^d}{N^d}\ep^{\alpha_0}\right)\nonumber\\
&\hspace{3cm}\geq \Prb(\exists f_n\in\cS_n(\fC) \quad(\ep,s\vv)\text{-well-behaved}: \dis(\amu_n(f_n),s\vv\ind_{\fC}\cL^d)\leq \ep^{\alpha_0})\,.
\end{align}
For $f_n\in\cS_n(\fC)$ $(\ep,s\vv)$-well-behaved, we have
\[\forall \diamond\in\{+,-\}\,\forall i\in\{1,\dots,d\}\,\forall A\in\cP_i^\diamond(m) \quad \psi_i^\diamond(f_n,A)=(1-\ep ^{\alpha/4})(s\vv\cdot \overrightarrow{e_i})\, \cH^{d-1}(A)n ^{d-1}\,.\]
The families $((s\vv\cdot \overrightarrow{e_i})\, \cH^{d-1}(A)n ^{d-1}, A\in\cup_{i=1}^d\cP_i^-(m))$ and $((s\vv\cdot \overrightarrow{e_i})\, \cH^{d-1}(A)n ^{d-1}, A\in\cup_{i=1}^d\cP_i^+(m))$ clearly satisfy conditions \eqref{cond:famille} and \eqref{cond:famille2}. Hence, by applying lemma \ref{lem:toutefamille}, we obtain 
\begin{align}\label{eq:eglimwb}
&\liminf_{\ep\rightarrow 0}\limsup_{n\rightarrow\infty}\frac{1}{n^d} \log\Prb(\exists f_n\in\cS_n(\fC) \quad(\ep,s\vv)\text{-well-behaved}: \dis(\amu_n(f_n),s\vv\ind_{\fC}\cL^d)\leq \ep^{\alpha_0} )\nonumber\\
&=\lim_{\ep\rightarrow 0}\limsup_{n\rightarrow\infty}\frac{1}{n^d} \log\Prb(\exists f_n\in\cS_n(\fC): \dis(\amu_n(f_n),s\vv\ind_{\fC}\cL^d)\leq \ep )\,.
\end{align}
Using inequalities \eqref{eq:contdis4}, \eqref{eq:contdis6} and \eqref{eq:scaling}, by taking first the liminf in $N$ and then the limsup in $n$ we obtain
\begin{align*}
&\liminf_{N\rightarrow\infty}\frac{1}{N^d}\log\Prb\left(\exists f_N\in\cS_N(\fC):\,\dis(\amu_N(f_N),s\vv\ind_{\fC}\cL^d)\leq g(\ep)+2d\frac{\mathrm{H}(\ep)}{M}\right)\\
&\geq \liminf_{N\rightarrow\infty}2d\cL^d(\Cor)\log G([M-\mathrm{H}(\ep),M])\\
&\hspace{1cm}+\limsup_{n\rightarrow\infty}\frac{1}{n^d} \log\Prb(\exists f_n\in\cS_n(\fC) \quad(\ep,s\vv)\text{-well-behaved}: \dis(\amu_n(f_n),s\vv\ind_{\fC}\cL^d)\leq \ep^{\alpha_0})\,.
\end{align*}
Finally, taking the limit when $\ep$ goes to $0$ in the previous inequality (the probability are non-decreasing in $\ep$) and using equalities \eqref{eq:cortend0}, \eqref{eq:scaling} and \eqref{eq:eglimwb}, we obtain
\begin{align*}
\lim_{\ep\rightarrow 0}\liminf_{N\rightarrow\infty}\frac{1}{N^d}&\log\Prb(\exists f_N\in\cS_N(\fC):\,\dis(\amu_N(f_N),s\vv\ind_{\fC}\cL^d)\leq\ep)\\
&\geq\lim_{\ep\rightarrow 0}\limsup_{n\rightarrow\infty}\frac{1}{n^d} \log\Prb(\exists f_n\in\cS_n(\fC): \dis(\amu_n(f_n),s\vv\ind_{\fC}\cL^d)\leq \ep )\,.
\end{align*}
This yields the result.

\end{proof}
The following lemma proves inequality \eqref{eq:scaling} in a slightly more general setting.
\begin{lem}[Scaling and Translation]\label{lem:scaling1} Let  $\vv\in\sS^{d-1}$ and $s>0$. Let $\ep>0$. Let $m\in\sN$. Let $N\geq n\geq 1$. Let  $(\rho_{A}^+,A\in\cup_{i=1}^d\cP_i^+(m))$ and $(\rho_{A}^-,A\in\cup_{i=1}^d\cP_i^-(m))$ be two families of real numbers (potentially depending on $\ep$, $n$ and $N$). Let $x\in \sZ_N^d$. Set $\delta=n/N$.  Then $\pi_{x,\delta}(\sZ_n ^d)=\sZ_N^d$: $\pi_{x,\delta}$ induces a bijection from $\E_n^d$ to $\E_N ^d$ (we refer to \eqref{eq:defpixalpha} for the definition of $\pi_{x,\delta}$).
Then, we have
\begin{align*}
\Prb&\left( \exists f_n\in\cS_n(\fC) : \begin{array}{c}\forall \diamond\in\{+,-\}\,\forall i\in\{1,\dots,d\}\,\forall A\in\cP_i^\diamond(m) \quad \psi_i^\diamond(f_n,A)=\rho_A^\diamond\\ \text{and }\,\dis\big(\amu_n(f_n),s\vv\ind_{\fC}\cL^d\big)\leq \ep\end{array}\right)\\
&\leq \Prb\left( \exists f_N\in\cS_N(\pi_{x,\delta} (\fC)) : \begin{array}{c}\forall \diamond\in\{+,-\}\,\forall i\in\{1,\dots,d\}\,\forall A\in\cP_i^\diamond(m) \quad \psi_i^\diamond(f_N,\pi_{x,\delta} (A))=\rho_A^\diamond\\ \text{and }\,\dis\big(\amu_N(f_N),s\vv\ind_{\pi_{x,\delta}(\fC)}\cL^d\big)\leq 4\delta ^{d}\ep\end{array}\right)\,.
\end{align*}
\end{lem}
\begin{proof} First notice that for $y\in\sZ_n^d$, we have $\pi_{x,\delta}(y)=ny/N+x\in\sZ_N^d$. Then, $\pi_{x,\delta}$ induces a bijection from $\E_n^d$ to $\E_N ^d$.
Let us consider $\omega\in(\sR_+) ^{\E_{n}^d}$ a configuration for which there exists $f_{n}\in\cS_{n}( \fC)$ such that \[\dis\big(\amu_{n}(f_{n}),s\vv\ind_{\fC}\cL^d\big)\leq\ep\]
and
\[\forall \diamond\in\{+,-\}\,\forall i\in\{1,\dots,d\}\,\forall A\in\cP_i^\diamond(m) \quad \psi_i^\diamond(f_n,A)=\rho_A^\diamond\,.\]
Let $f_{n}=f_{n}(\omega)$ be such a stream in the configuration $\omega$ and define $\amu_n=\amu_n(f_n)$. We aim to prove that on the configuration $\omega\circ\pi_{x,\delta} ^{-1}$ the stream $f_{n}\circ\pi_{x,\delta} ^{-1}$ belongs to $\cS_{N}(\pi_{x,\delta}(\fC))$, satisfies 
\[\forall \diamond\in\{+,-\}\quad \forall i\in\{1,\dots,d\}\quad\forall A\in\cP_i^\diamond(m) \qquad \psi_i^\diamond(f_n\circ\pi_{x,\delta}^{-1},\pi_{x,\delta} (A))=\rho_A^\diamond\]
and
 \[\dis\big(\amu_{N}(f_{n}\circ\pi_{x,\delta} ^{-1}),s\vv\ind_{\pi_{x,\delta}(\fC)}\cL^d\big)\leq 4 \delta ^{d} \ep\,.\]
We set
\[
\amu_{N}=\amu_N(f_n\circ\pi_{x,\delta}^{-1})=\frac{1}{N^d}\sum_{e\in\E_{N}^d}f_n\circ\pi_{x,\delta}^{-1}(e)\delta_{c(e)}\,.
\]
It is clear that  $f_n\circ\pi_{x,\delta}^{-1}\in\cS_{N}(\pi_{x,\delta}(\fC))$ for the configuration $\omega\circ\pi_{x,\delta} ^{-1}$. Moreover, we have for $A\in\cP_i ^\diamond(m)$
\begin{align*}
\psi_i^\diamond(f_n\circ\pi_{x,\delta}^{-1},\pi_{x,\delta} (A)))&=\sum_{e\in\E_N^{i,\diamond}[\pi_{x,\delta}(A)]}f_n(\pi_{x,\delta} ^{-1}(e))\cdot\overrightarrow{e_i}=\sum_{e\in\E_n^{i,\diamond}[A]}f_n(e)\cdot\overrightarrow{e_i}=\psi_i^\diamond(f_n,A)=\rho_A^\diamond\,.
\end{align*}
It remains to compute the distance $\dis\big(\amu_{N},s\vv\ind_{\pi_{x,\delta}(\fC)}\cL^d\big)$.
Let $\lambda\in[1,2]$, $y\in[0,1]^d$. 
Let $j\geq 1$ such that
\[2^{j}<\lambda \frac{N}{n} \leq 2^{j+1}\,.\]
Let $\lambda'\in[1,2]$ such that
\[\lambda \frac{N}{n} =\lambda'2^{j}\,.\]
Let $z\in[-1,1[^d$ such that $\frac{N}{n} (y-x)\in(z+\lambda'\sZ^d)$.
Let $k\geq j$.
Let $Q\in\Delta^k_\lambda$, set $B=Q+y$, we have

\[
\pi_{x,\delta}^{-1}(B)=\frac{N}{n}(B-x)=\frac{N}{n}Q+\frac{N}{n} (y-x)\,.
\]
Since $\frac{N}{n} (y-x)\in(z+\lambda'\sZ^d)\subset (z+\lambda'2^{j-k}\sZ^d)$, it yields that \[\pi_{x,\delta}^{-1}(B)\in\left(z+\Delta^{k-j}_{\lambda'}\right)\,.\]
We have by change of variable
\begin{align*}
\cL^d(B\cap\pi_{x,\delta}(\fC))&=\int_{B\cap\pi_{x,\delta}(\fC)}d\cL^d(y)=\frac{n^d}{N^d}\int _{\pi_{x,\delta}^{-1}(B)\cap\fC}d\cL^d(y)=\frac{n^d}{N^d}\cL^d(\pi_{x,\delta}^{-1}(B)\cap\fC)\,.
\end{align*}
It follows that for $B\in (y+\Delta^k _\lambda)$, we have
$$\|\amu_{N}(B)-s\vv\cL^d(B\cap\pi_{x,\delta}(\fC))\|_2=\frac{n^d}{N^d}\|\amu_{n}(\pi_{x,\delta}^{-1}(B))-s\vv\cL^d(\pi_{x,\delta}^{-1}(B)\cap\fC)\|_2$$
where $\pi_{x,\delta}^{-1}(B)\in(z+\Delta^{k-j}_{\lambda'})$. Hence, we have
\begin{align*}
\sum_{B\in(y+\Delta^k_\lambda)} \|\amu_{N}(B)-s\vv\cL^d(B\cap\pi_{x,\delta}(\fC))\|_2= \sum_{\widehat{B}\in(z+\Delta^{k-j}_{\lambda'})}\frac{n^d}{N^d}\|\amu_{n}(\widehat{B})-s\vv\cL^d(\widehat{B}\cap\fC)\|_2\,.
\end{align*}
Besides, for $k<j$, we have by triangular inequality
\begin{align*}
\sum_{Q\in(y+\Delta^k_\lambda)} \|\amu_{N}(Q)-s\vv\cL^d(Q\cap\pi_{x,\delta}(\fC))\|_2&\leq \sum_{Q\in(y+\Delta^j_\lambda)} \|\amu_{N}(Q)-s\vv\cL^d(Q\cap\pi_{x,\delta}(\fC))\|_2\\
&=\sum_{Q'\in(z+\Delta^{0}_\lambda)}\frac{n^d}{N^d}\|\amu_{N}(Q')-s\vv\cL^d(Q'\cap\fC)\|_2\,.
\end{align*}
Combining the two previous inequalities, it follows that
\begin{align*}
&\sum_{k=0}^\infty\frac{1}{2^{k}}\sum_{Q\in(y+\Delta^k_\lambda)} \|\amu_{N}(Q)-s\vv\cL^d(Q\cap\pi_{x,\delta}(\fC))\|_2\\
&\leq \frac{n^d}{N^d}\sum_{k=0}^{j-1}\frac{1}{2^{k}}\sum_{Q'\in(z+\Delta^{0}_\lambda)}\|\amu_{N}(Q')-s\vv\cL^d(Q'\cap\fC)\|_2+\frac{n^d}{N^d}\sum_{k=j}^\infty\frac{1}{2^{k}}\sum_{Q\in(z+\Delta^{k-j}_{\lambda'})}  \|\amu_{n}(Q)-s\vv\cL ^d(Q\cap\fC)\|_2 \\
&\leq 2\frac{n^d}{N^d}\sum_{Q\in(z+\Delta^{0}_\lambda)}\|\amu_{N}(Q)-s\vv\cL^d(Q\cap\fC)\|_2+ 2^{-j}\frac{n^d}{N^d}\sum_{k=0}^\infty\frac{1}{2^{k}}\sum_{Q\in(z+\Delta^{k}_{\lambda'})}  \|\amu_{n}(Q)-s\vv\cL ^d(Q\cap\fC)\|_2\\
&\leq \frac{n^{d}}{N^{d}}\left( 2+ \frac{\lambda'}{\lambda}\frac{n}{N}\right)\dis(\amu_{n},s\vv\ind_{\fC}\cL^d)\leq 4\frac{n^d}{N^d}\ep\,.\\
\end{align*}
Hence, it yields that on the configuration $\omega\circ\pi_{x,\delta} ^{-1}$,
\[\dis(\amu_N(g_N),s\vv\ind_{\delta\fC+x}\cL^d)\leq 4\delta^{d}\ep\]
and the result follows.
\end{proof}

\subsection{Lower semi-continuity}
In this section, we prove that the map $I$ is lower semi-continuous on $\sR^d$.
\begin{prop}\label{prop:lscI}The map $I$ is lower semi-continuous on $\sR^d$.
\end{prop}
\begin{proof}[Proof of proposition \ref{prop:lscI}]
Let $\vv\in\sR^d$ and let $(\vv_p)_{p\geq1}$ be a sequence such that $\lim_{p\rightarrow\infty}\vv_p=\vv$. Let us first assume that $I(\vv)<\infty$.
Let $\delta>0$.
Let $\ep_0=\ep_0(\delta)>0$ such that 
$$\forall\ep\leq \ep_0\qquad-\limsup_{n\rightarrow\infty}\frac{1}{n^d}\log\Prb(\exists f_n\in\cS_n(\fC):\dis(\amu_n(f_n),\vv\ind_\fC\cL^d)\leq \ep)\geq I(\vv)-\delta\,.$$
Let $p_0\geq 1$ be such that for any $p\geq p_0$, $\|\vv-\vv_p\|_2 \leq \ep_0/4$.
Using lemma \ref{lem:propdis2}, it yields that
\[\dis(\amu_n(f_n),\vv\ind_\fC\cL^d)\leq \dis(\amu_n(f_n),\vv_p\ind_\fC\cL^d)+ 2 \|\vv-\vv_p\|_2\leq\dis(\amu_n(f_n),\vv_p\ind_\fC\cL^d)+\frac{\ep_0}{2} \,,\]
and
\begin{align*}
\forall \ep\leq \ep_0\quad\forall p\geq p_0\qquad -\limsup_{n\rightarrow\infty}\frac{1}{n^d}&\log\Prb(\exists f_n\in\cS_n(\fC):\dis(\amu_n(f_n),\vv\ind_\fC\cL^d)\leq \ep_0)\\&\leq  -\limsup_{n\rightarrow\infty}\frac{1}{n^d}\log\Prb(\exists f_n\in\cS_n(\fC):\dis(\amu_n(f_n),\vv_p\ind_\fC\cL^d)\leq \ep/4)\,.
\end{align*}
It follows that
$$\forall \ep\leq \ep_0\quad\forall p\geq p_0\qquad  -\limsup_{n\rightarrow\infty}\frac{1}{n^d}\log\Prb(\exists f_n\in\cS_n(\fC):\dis(\amu_n(f_n),\vv_p\ind_\fC\cL^d)\leq \ep/4) \geq I(\vv)-\delta\,.$$
By letting first $\ep$ goes to $0$ and then taking the liminf in $p$, we obtain
$$\liminf_{p\rightarrow\infty} I(\vv_p)\geq I(\vv)-\delta\,.$$
Since the previous inequality holds for any $\delta>0$, it follows that 
$$\liminf_{p\rightarrow\infty} I(\vv_p)\geq I(\vv)\,.$$
Let us now assume that $I(\vv)=+\infty$. By the same reasoning, we can prove that for any $M>0$,
$$\liminf_{p\rightarrow\infty} I(\vv_p)\geq M\,.$$
It follows that $\liminf_{p\rightarrow\infty} I(\vv_p)=I(\vv)=+\infty$. This yields the proof.
\end{proof}
\subsection{Convexity}
In this subsection, we aim to prove that the map $I$ is convex, this property will allow us to obtain regularity properties on $I$.
\begin{thm}\label{thm:convexity}
The map $I:\sR^d\rightarrow \sR_+\cup\{+\infty\}$ is convex, that is
\[\forall \lambda\in[0,1]\quad \forall\, \vv_1,\vv_2\in\sR^d\qquad I(\lambda\vv_1+(1-\lambda)\vv_2)\leq \lambda I(\vv_1)+(1-\lambda)I(\vv_2)\,.\]
\end{thm}
\noindent Let us define $\mathcal{D}_I$ as the set of points where $I$ is finite, that is,
$$\mathcal{D}_I=\left\{x\in\sR^d: I(x)<+\infty\,\right\}\,.$$
It is easy to check thanks to theorem \ref{thm:convexity}, that the set $\mathcal{D}_I$ is convex.
From theorem \ref{thm:convexity}, we can deduce the following proposition that is a corollary of Theorem 6.7.(i) in \cite{EVGA}.
\begin{prop}\label{prop:continuityI}The map $I$ is continuous on $\mathring{\mathcal{D}}_I$. 
\end{prop}
Let $C$ be the cube of side-length $1/n$ centered at $0$, that is
\[C=\left[-\frac{1}{2n},\frac{1}{2n}\right]^d\,.\]
For any edge $e\in\E_n^d$, write $\cP(e)$ the dual of the edge $e$, \textit{i.e.}, the hypersquare of dimension $d-1$ of side-length $1/n$, orthogonal to $e$ and centered at the center of $e$.
Let $\vv\in\sS^{d-1}$, $s\in[0,M]$, $h>0$,  and  $\cA$ be an hyperrectangle of $\sR^d$ such that $\vv$ is not contained in an hyperplane parallel to $\cA$.
We need to define a new set of admissible streams $\widehat{\cS}_n(\cyl(\cA,h,\vv),s\vv)$ that is defined only in the interior of $\cyl(\cA,h,\vv)$ and have prescribed values near the boundary of $\cyl(\cA,h,\vv)$. Let $g:\sR_+\rightarrow [0,1]$ be a function such that
$$\lim_{\ep\rightarrow 0}g(\ep)=1\,.$$
 A stream $f_n$ is in $\widehat{\cS}_n(\cyl(\cA,h\vv),s\vv,g(\ep))$ if 
\begin{itemize}
\item The stream respects the capacity constraint:$\forall e\in \E_n ^d\qquad \|f_n(e)\|\leq t(e)$.
\item The stream is null outside the cylinder: $\forall e \in\E_n^d\qquad \cP(e)\not\subset\cyl(A,h,\vv)\implies f_n(e)=0$.
\item The values of the stream for edges closed to the boundary are prescribed by the continuous stream $s\vv$:  $\forall e=\langle x,y\rangle$ such that $((x+C)\cup(y+C))\not\subset \cyl(\cA,h,\vv)$ and $\cP(e)\subset  \cyl(\cA,h,\vv)$, we have $f_n(e)=g(\ep)n^2(s\vv\cdot \overrightarrow{xy})\overrightarrow{xy}$.
\item The node law is respected for any $x\in\sZ_n^d$ such that $(x+C)\subset\cyl(\cA,h,s\vv)\,.$
\end{itemize}

To prove theorem \ref{thm:convexity}, we need first to prove the following lemma. This lemma controls the probability of having a constant stream $s\vv$ in a cylinder oriented in the direction $\vv$.
\begin{lem}\label{lem:preconv}Let $\vv\in\sS^{d-1}$, $s\in[0,dM]$, $h>0$ and  $\cA$ be an hyperrectangle of $\sR^d$ such that $\vv$ is not contained in an hyperplane parallel to $\cA$. There exists two positive functions $g_0:\sR_+\rightarrow \sR_+$ and $g_1:\sR_+\rightarrow [0,1]$ that satisfy
$$\lim_{\ep\rightarrow 0}g_0(\ep)=0\qquad\text{   and   }\qquad\lim_{\ep\rightarrow 0}g_1(\ep)=1\,$$
such that we have
\begin{align*}
-\liminf_{\ep\rightarrow 0} \limsup_{n\rightarrow \infty}\frac{1}{n^d}\log \Prb\left(\cE_n(\cyl(\cA,h,\vv),s\vv,g_0(\ep),g_1(\ep))\right)\leq \cL^d(\cyl(\cA,h,\vv))I(s\vv)
\end{align*}
where \[\cE_n(\cyl(\cA,h,\vv),s\vv,g_0(\ep),g_1(\ep))=\left\{ \begin{array}{c}\exists f_n\in\widehat{\cS}_n(\cyl(\cA,h,\vv),s\vv,g_1(\ep)) :\\ \,\dis\big(\amu_n(f_n),s\vv\ind_{\cyl(\cA,h,\vv)}\cL^d\big)\leq g_0(\ep)\cL^d(\cyl(\cA,h,\vv))\end{array}\right\}\,.\]
\end{lem}

\begin{proof}[Proof of lemma \ref{lem:preconv}]
To prove lemma \ref{lem:preconv}, we proceed similarly as in the proof of theorem \ref{thmbrique}. We pave the cylinder with small cubes, we consider streams in these small cubes and we try to reconnect these streams using the corridor. The main difference with the proof of theorem \ref{thmbrique} is that we require that edges close to the boundary of the cylinder have a prescribed value. This prescribed value corresponds to a discretized version of the continuous stream $s\vv$.

Here $M>0$ denotes the supremum of the support of $G$. Without loss of generality, we can assume that for any $i\in\{1,\dots,d\}$, $\vv\cdot\overrightarrow{e_i}=v_i\geq 0$.

\noindent{\bf Step 1: Paving $\cyl(\cA,h,\vv)$ with cubes.}
Let $m=\lfloor \ep ^{-\alpha}\rfloor$ where $\alpha$ was defined in \eqref{eq:defalpha}. We set
$$\kappa = \frac{m}{n}\left\lfloor \frac{n\ep}{m}\right\rfloor\,.$$
Hence we have $n\kappa\in\sZ$, $n\kappa(1+2d/m)\in\sZ$ and $\lim_{n\rightarrow\infty}\kappa=\ep$.
Write $$E=\kappa\left(1+\frac{2d}{m}\right)\fC\,.$$
We want to cover $\cyl(\cA,h,\vv)$ by translates of $E$.
Let $\fT(\cyl(\cA,h,\vv))$ be the following set of translates of $E$ contained in $\cyl(\cA,h,\vv)$:
\[\fT(\cyl(\cA,h,\vv))=\left\{x\in\kappa(1+2d/m)\sZ^d:\, (E+x)\subset\left(\cyl(\cA,h,\vv)\setminus \cV_\infty(\partial \cyl(\cA,h,\vv), d\kappa)\right)\right\}\,.\]
Write $\Cor$ the following set
\[\Cor=\cyl(\cA,h,\vv)\setminus \bigcup_{x\in\fT(\cyl(\cA,h,\vv))}\pi_{x,\kappa}(\fC)\,\]
and the set of edges $\Cor_n$ whose left endpoint is in $\Cor$:
$$\Cor_n=\left\{\langle x,y\rangle\in\E_n^d: x\in\Cor \text{ and }\exists i\in\{1,\dots,d\}\quad\overrightarrow{xy}=\frac{\overrightarrow{e_i}}{n}\right\}\,.$$
For $x\in\fT(\cyl(\cA,h,\vv)))$ we write 
$$\cE_x=\left\{\,\exists f_n\in\cS_n(\pi_{x,\kappa}(\fC)) \text{ $(\ep,s\vv,\pi_{x,\kappa})$-well-behaved}: \,\dis\big(\amu_n(f_n),s\vv\ind_{\pi_{x,\kappa}(\fC)}\cL^d\big)\leq 4\ep^{\alpha_0}\kappa ^d\,\right\}\,.$$
On the event $\cE_x$, we will denote by $f_n^x$ a $(\ep,s\vv,\pi_{x,\kappa})$-well-behaved stream satisfying 
$$\dis\big(\amu_n(f_n),s\vv\ind_{\pi_{x,\kappa}(\fC)}\cL^d\big)\leq 4\ep^{\alpha_0}\kappa ^d$$ (chosen according to a deterministic rule if there is more than one). 
Let us denote by $\cF$ the event 
$$\cF=\left\{\, \forall e\in\Cor_n\qquad t(e)\geq M-\mathrm{H}(\ep)\right\}\,$$
where $\mathrm{H}(\ep)\geq 0$ will be defined later in a similar way than in \eqref{eq:defh}. The function $\mathrm{H}$ satisfies $\lim_{\ep\rightarrow 0}\mathrm{H}(\ep)=0$.
We aim to prove, that on the event $\cF\cap\bigcap_{x\in\fT(\cyl(\cA,h,\vv)))}\cE_x$, we can build a stream $$f_n\in \widehat{\cS}_n\left(\cyl(\cA,h,\vv),s\vv,(1-\mathrm{H}(\ep))(1-\ep^{\alpha/4})\right)$$ such that $f_n$ coincides with $f_n^x$ on $\pi_{x,\kappa}(\fC)$, for $x\in\fT(\cyl(\cA,h,\vv))$.

\noindent{\bf Step 2: Construction of the stream inside $\cup_{x\in\fT(\cyl(\cA,h,\vv))}(x+E)$.}
By lemma \ref{lem:mixing}, for any $x\in\fT(\cyl(\cA,h,\vv))$, for any $i\in\{1,\dots,d\}$, for any $\diamond\in\{+,-\}$, for any $A_0\in\cP_i^\diamond(m)$,
there exists a stream $\overline{f}_n^{x,A_0}$ in $\cyl(\pi_{x,\kappa}( A_0),\kappa d/m,\diamond\overrightarrow{e_i})$ such that 
\[\forall e\in \E_n^{i,\diamond}[\pi_{x,\kappa} (A_0)] \qquad \overline{f}_n^{x,A_0}\left(e\diamond\frac{\overrightarrow{e_i}}{n}\right)=f_n^x(e)\]
and
\[\forall e\in \E_n^{i,\diamond}\left[\pi_{x,\kappa} \left(A_0\diamond \frac{ d}{m}\overrightarrow{e_i}\right)\right]\qquad \overline{f}_n^{x,A_0}(e)=\frac{\psi_i^\diamond(f_n^x,\pi_{x,\kappa}( A_0))}{|\E_n^{i,\diamond}[\pi_{x,\kappa} (A_0)] |}\,.\]
This stream mix the inputs in such a way the outputs are uniform.
We set $f_n^{prel}$ the stream inside $\cup_{x\in\fT(\cyl(\cA,h,\vv))}(x+E)$ as 
\[f_n^{prel}=\sum_{x\in\fT(\cyl(\cA,h,\vv))}\left(f_n^{x}+\sum_{A_0\in\cup_i\cP_i^+(m)\cup\cP_i^-(m)}\overline{f}_n^{x,A_0}\right)\,.\]
For any $x_1,x_2\in\fT(\cyl(\cA,h,\vv))$ such that $x_2=x_1+\kappa(1+2d/m)\overrightarrow{e_i}$, we have on the event $\cE_{x_1}\cap\cE_{x_2}$
\[\forall A_0\in\cP_i^+ (m)\qquad\psi_i^+\left(f_n^{x_1},\pi_{x_1,\kappa}( A_0)\right)=(1-\ep ^{\alpha/4})sv_i\, \cH^{d-1}(\pi_{x,\kappa}(A_0))n ^{d-1}=\psi_i^-\left(f_n^{x_2}, \pi_{x_2,\kappa}( A_0-\overrightarrow{e_i})\right)\,.\]
Moreover, since $2\kappa d/m\in\sZ_n$, it follows that:
$$|\E_n^{i,\diamond}[\pi_{x,\kappa} (A_0)] |=\left|\E_n^{i,\diamond}\left[\pi_{x,\kappa} (A_0)+2\frac{\kappa d}{m}\overrightarrow{e_i}\right]\right |\,.$$
Combining the two latter equalities, we obtain that
 \[\forall e\in \E_n^{i,\diamond}\left[\pi_{x_1,\kappa} \left(A_0\diamond \frac{ d}{m}\overrightarrow{e_i}\right)\right]\qquad \overline{f}_n^{x_1,A_0}(e)=\overline{f}_n^{x_2,A_0-\overrightarrow{e_i}}\left(e+\frac{\overrightarrow{e_i}}{n}\right)\,.\]
The latter equality ensures that the streams $ \overline{f}_n^{x_1,A_0}$ and $\overline{f}_n^{x_2,A_0-\overrightarrow{e_i}}$ glue well together and so that the node law is satisfied  everywhere inside $\cup_{x\in\fT(\cyl(\cA,h,\vv))}(x+E)$.
It remains to extend $f_n^{prel}$ to $\cyl(\cA,h,\vv)\setminus\cup_{x\in\fT(\cyl(\cA,h,\vv))}(x+E)$ in such a way that we respect the boundary conditions. To do so, we define a discretized version $f_n^{disc}$ of the continuous stream $s\vv$. Since the stream $f_n ^{disc}$ will not perfectly match with the boundary conditions of $f_n^{prel}$, we will also need to build a stream that corrects the differences on the boundary. 
 
\noindent{\bf Step 3: Construction of a discrete stream from the continuous one.}
Let $x\in \fT(\cyl(\cA,h,\vv)$. Let $i\in\{1,\dots,d\}$ and $\diamond \in\{+,-\}$, let us compute
$|\E_n^{i,\diamond}[\pi_{x,\kappa(1+2d/m)}(\fC_i^\diamond)]|$.
By symmetry of the lattice and since $x\in\kappa(1+2d/m)\sZ^d\subset\sZ_n^d$, it is equal to $|\E_n^{1,+}[\kappa(1+2d/m)\fC_1^+]|=\kappa^{d-1}(1+2d/m)^{d-1}n^{d-1}$ since $n\kappa(1+2d/m)\in\sZ$.
We consider the following stream $f^{disc}_n$ that is the discretized version of $s\vv\ind_{\cyl(\cA,h,\vv)}$ defined by
\begin{align*}
\forall e=\langle x,y\rangle \in\E_n^d &\text{  such that  } \cP(e)\subset \cyl(\cA,h,\vv)\qquad\\
& f_n ^{disc}(e)= (1+2d/m)^{-(d-1)}(1-\ep ^{\alpha/4})n^2 (s\vv\cdot \overrightarrow{xy})\overrightarrow{xy}\,.
\end{align*}
In particular, if $\overrightarrow{xy}=\overrightarrow{e_i}/n$, for $i\in\{1,\dots,d\}$, we have
$$f_n ^{disc}(e)= (1+2d/m)^{-(d-1)}(1-\ep ^{\alpha/4}) sv_i\overrightarrow{e_i}\,.$$
Hence, we have for $x\in\fT(\cyl(\cA,h,\vv)$,  $i\in\{1,\dots,d\}$, $\diamond\in\{+,- \}$
\begin{align}\label{eq:egalitépreldisc}
\psi_i ^\diamond(f_n^{disc},\pi_{x,\kappa(1+2d/m)}(\fC_i ^\diamond))&=(1+2d/m)^{-(d-1)}(1-\ep ^{\alpha/4}) sv_i |\E_n^{i,\diamond}[\pi_{x,\kappa(1+2d/m)}(\fC_i^\diamond)]|\nonumber\\
&=(1-\ep ^{\alpha/4}) sv_i \kappa^{d-1}n^{d-1}\nonumber\\
&=\sum_{A_0\in\cP_i^\diamond(m)}\psi_i^\diamond(f_n^x,\pi_{x,\kappa}(A_0))\nonumber\\&=\psi_i^\diamond(f_n^x,\pi_{x,\kappa}(\fC_i^\diamond))=\psi_i ^\diamond(f_n^{prel},\pi_{x,\kappa(1+2d/m)}(\fC_i ^\diamond))\,.
\end{align}
Let $w\in\sZ_n ^d$  such that $w+C\subset\cyl(\cA,h,\vv)$, we have
\[\sum_{\substack{y\in\sZ_n^d:\\e=\langle w,y\rangle\in\E_n^d}}f_n^{disc}(e)\cdot (n\overrightarrow{wy})=\\(1+2d/m)^{-(d-1)}(1-\ep ^{\alpha/4})n \sum_{\substack{y\in\sZ_n^d:\\e=\langle w,y\rangle\in\E_n^d}}(s\vv\cdot \overrightarrow{wy})=0\,\]
and the node law is satisfied at $w$ for the stream $f_n ^{disc}$.

\noindent{\bf Step 4: Gluing the streams and correcting the differences.}
Let us now consider $x\in\fT(\cyl(\cA,h,\vv))$ such that there exists $i\in\{1,\dots,d\}$ and $\diamond\in\{-,+\}$ such that $x\diamond\kappa(1+2d/m)\overrightarrow{e_i}\notin\fT(\cyl(\cA,h,\vv))$. Let us denote by $\partial ^{int}\fT(\cyl(\cA,h,\vv))$ such $x$, \textit{i.e.},
\[\partial ^{int}\fT(\cyl(\cA,h,\vv))=\left\{x\in \fT(\cyl(\cA,h,\vv)):\begin{array}{c}\exists i\in\{1,\dots,d\},\diamond\in\{+,-\}\\ x\diamond\kappa(1+2d/m)\overrightarrow{e_i}\notin\fT(\cyl(\cA,h,\vv))\end{array}\right\}\]
and for such an $x$, let us denote by $E_\kappa(x)$ the set of faces of $\pi_{x,\kappa(1+2d/m)}(\fC)$ that are "external", \textit{i.e.},
\[E_\kappa(x)=\left\{\pi_{x,\kappa(1+2d/m)}(\fC_i^\diamond):\,x\diamond\kappa(1+2d/m)\overrightarrow{e_i}\notin\fT(\cyl(\cA,h,\vv)) ,\,i\in\{1,\dots,d\},\diamond\in\{-,+\}\right\}\,.\]  
For those faces, we need to correct the stream $f_n^{prel}$ to be able to glue it with the discretized version $f^{disc}_n$. Let us consider $F_0=\kappa(1+2d/m)(\fC_i^\diamond+x)\in E_\kappa(x)$. By equality \eqref{eq:egalitépreldisc}, we have 
\begin{align}\label{eq:egaliteflux}
\psi_i^\diamond(f_n ^{disc},F_0)=\psi_i^\diamond(f_n^{prel},F_0)\,.
\end{align}
 We have
$$\forall e\in\E_n^{i,\diamond}[F_0] \qquad f_n ^{disc}(e)=( 1+2d/m)^{-(d-1)}(1-\ep ^{\alpha/4}) sv_i \overrightarrow{e_i}\,,$$
$$\forall e\in\E_n^{i,\diamond}[\kappa(\fC_i^\diamond+x\diamond d/m\overrightarrow{e_i})] \qquad f_n ^{prel}(e)= (1-\ep ^{\alpha/4}) sv_i \overrightarrow{e_i}\,.$$
It follows that for $e\in \E_n^{i,\diamond}[\kappa(\fC_i^\diamond+x\diamond d/m\overrightarrow{e_i})]$, we have
\begin{align}\label{eq:inputpos}
(-f_n ^{disc}(e)+f_n^{prel}(e))\cdot \overrightarrow{e_i}=(1-\ep^{\alpha/4})\left(1-\left(1+\frac{2d}{m}\right)^{-(d-1)}\right)sv_i\leq 4\frac{d^2}{m}M\leq 8d^2 M\ep ^{\alpha}
\end{align}
for small enough $\ep$ depending on $d$.
For $e\in \E_n^{i,\diamond}[F_0]\setminus \E_n^{i,\diamond}[\kappa(\fC_i^\diamond+x\diamond d/m\overrightarrow{e_i})]$, we have $f_n ^{prel}(e)=0$ thus
\begin{align}\label{eq:inputneg}
(-f_n ^{disc}(e)+f_n^{prel}(e))\cdot \overrightarrow{e_i}=-(1-\ep^{\alpha/4})\left(1+2\frac{d}{m}\right)^{-(d-1)}sv_i\geq -sv_i\,.
\end{align}
We can indexed the edges of $\E_n^{i,\diamond}[F_0]$ following the order given by the canonical basis such that to each edge $e$ we can associate its index $\zeta(e)\in\{1,\dots, \kappa(1+2d/m)n\}^{d-1}$. More precisely, we set 
$$\forall e\in\E_n^{i,\diamond}[F_0]\qquad\zeta(e)=n\mathfrak{p}_i(c(e))+\left(\left\lfloor \frac{\kappa(1+2d/m)n}{2}\right\rfloor+1\right)\sum_{j\in\{1,\dots,d\}\setminus\{i\}}\overrightarrow{e_j}\,$$
where we recall that the definition of $\mathfrak{p}_i$ was given in \eqref{eq:defpi}.
It is easy to check that $\zeta(e)\in\{1,\dots, \kappa(1+2d/m)n\}^{d-1}$ (we recall that $ \kappa(1+2d/m)n\in\sZ$).
Set for any $e\in\E_n^{i,\diamond}[F_0]$, $f_{in }(\zeta(e))=(-f_n ^{disc}(e)+f_n^{prel}(e))\cdot \overrightarrow{e_i}$.
If $e$ is such that $\zeta(e)\in\{\kappa d n/ m+1, \kappa(1+d/m)n\}^{d-1}$, then $$f_{in}(\zeta(e))=(1-\ep^{\alpha/4})\left(1-\left(1+\frac{2d}{m}\right)^{-(d-1)}\right)sv_i\,.$$
Otherwise, we have
$$f_{in}(\zeta(e))=-(1-\ep^{\alpha/4})\left(1+\frac{2d}{m}\right)^{-(d-1)}sv_i\,.$$
To apply lemma \ref{lem:mixprecis}, we have to check that the sequence $(f_{in}(y),y\in\{1,\dots, \kappa(1+2d/m)n\}^{d-1})$ satisfies the conditions stated in the lemma.
First note that thanks to equality \eqref{eq:egaliteflux}, we have
$$\sum_{y\in\{1,\dots, \kappa(1+2d/m)n\}^{d-1}}f_{in}(y)=0\,$$
and by inequalities \eqref{eq:inputneg} and \eqref{eq:inputpos}, $$\forall y\in \{1,\dots, \kappa(1+2d/m)n\}^{d-1}\qquad-M\leq f_{in}(y)\leq 8d^2M\ep^\alpha\,.$$
Let $l\in\{1,\dots,d-2\}$ and $x\in\{1,\dots, \kappa(1+2d/m)n\}^{l}$, if $x\notin \{\kappa d n/ m+1,\dots, \kappa(1+d/m)n\}^{l}$ then for any $y\in \{1,\dots, \kappa(1+2d/m)n\}^{d-1-l}$, we have
$$f_{in}(x,y)=-(1-\ep^{\alpha/4})\left(1+\frac{2d}{m}\right)^{-(d-1)}sv_i\,.$$
If $x\in \{\kappa d n/ m+1,\dots, \kappa(1+d/m)n\}^{l}$, then we have
\begin{align*}
&\sum_{y\in \{1,\dots, \kappa(1+2d/m)n\}^{d-1-l}}f_{in}(x,y)\\
&\quad = \left|\left\{ \frac{\kappa dn}{m},\dots,\kappa\left(1+\frac{d}{m}\right)n\right\}^{d-1-l}\right|\left(1-\left(1+\frac{2d}{m}\right)^{-(d-1)}\right)(1-\ep^{\alpha/4})sv_i\\
&\qquad - \left|\left\{1,\dots,\kappa\left(1+\frac{2d}{m}\right)n\right\}^{d-1-l}\setminus \left\{ \frac{\kappa dn}{m},\dots,\kappa\left(1+\frac{d}{m}\right)n\right\}^{d-1-l}\right|\left(1+\frac{2d}{m}\right)^{-(d-1)}(1-\ep^{\alpha/4})sv_i\\
&\quad=\left(1-\left(1+\frac{2d}{m}\right)^{-(d-1)}-\left(\left(1+\frac{2 d }{m}\right) ^{d-1-l}-1\right)\left(1+\frac{2d}{m}\right)^{-(d-1)} \right)(1-\ep^{\alpha/4})sv_i(\kappa n) ^{d-1-l}\\
&\quad =\left(1 -\left(1+\frac{2d}{m}\right)^{-l}\right)(1-\ep^{\alpha/4})sv_i(\kappa n) ^{d-1-l}\geq 0\,.
\end{align*}
It follows that the conditions to apply lemma \ref{lem:mixprecis} are satisfied.
By lemma \ref{lem:mixprecis}, there exists a stream $g_n^{x,F_0}$ in $\cyl(F_0,(d-1)\kappa(1+2d/m),\diamond\overrightarrow{e_i})\subset \cyl(\cA,h,\vv)$ such that 
\[\forall e\in\E_n^{i,\diamond}[F_0]\qquad g_n^{x,F_0}\left(e\diamond\frac{\overrightarrow{e_i}}{n}\right)=-f_n ^{disc}(e)+f_n^{prel}(e)\,\]
and
\[\forall e\in\E_n^{i,\diamond}[F_0\diamond (d-1)\kappa(1+2d/m)\overrightarrow{e_i}]\qquad  g_n^{x,F_0}(e)=0\,.\]
The stream $g_n^{x,F_0}$ satisfies the node law everywhere except for points in $\{w\in\sZ_n^d: \,\exists y\in\sZ_n^d \text{ s.t. }\overrightarrow{yw}=\diamond\overrightarrow{e_i}/n\text{ and } \langle y,w\rangle \in \E_n ^{i,\diamond}[F_0]\}$.
Moreover, using inequalities \eqref{eq:inputpos} and \eqref{eq:inputneg}, we have for any edge $e\in\E_n^d$ parallel to $\overrightarrow{e_i}$
\[g_n^{x,F_0}(e)\cdot\overrightarrow{e_i}\in[-sv_i,8d^2M\ep ^{\alpha}]\,\]
and for edge $e$ parallel to $\overrightarrow{e_j}$ with $j\neq i$:
\[\|g_n^{x,F_0}(e)\|_2\leq 8d^2M\ep^\alpha\,.\]
Finally, we build $f_n$ as follows: for any $e\in\cyl(\cA,h,\vv)\cap\E_n^d$
\[f_n(e)=\left\{
    \begin{array}{ll}
       f_n^{prel} (e) & \mbox{if $e\in\E_n^d\cap\cup_{x\in\fT(\cyl(\cA,h,\vv)}\pi_{x,\kappa(1+2d/m)}(\fC)$} \\
       f_n ^{disc}(e)+\sum_{x\in\partial ^{int}\fT(\cyl(\cA,h,\vv))}\sum_{F_0\in E_\kappa(x)}g_n^{x,F_0}(e)& \mbox{otherwise.}
    \end{array}
\right.\]
The node law is satisfied everywhere inside $\cyl(\cA,h,\vv)$. Note that by construction of $\fT(\cyl(\cA,h,\vv))$, each $e\in \cyl(F_0,(d-1)\kappa,\diamond\overrightarrow{e_i})$ belongs at most to $d$ such cylinder (one for each direction): for each $j\in\{1,\dots,d\}$ there exists at most one $\circ\in\{+,-\}$ and $y\in\partial ^{int}(\fT\cyl(\cA,h, \vv))$ such that $F_1=\pi_{y,\kappa(1+2d/m)}(\fC_i ^\circ)\in E_\kappa(y)$ and $e\in \cyl(F_1,(d-1)\kappa,\circ\overrightarrow{e_i})$.
It follows that for any $e\in\cyl(\cA,h,\vv)$, we have
$$\|f_n(e)\|_2\leq M+8d^3M\ep ^\alpha\,.$$
On the event $\cF$, the stream $\widehat{f}_n=(1-\ep ^{\alpha/4})(1 -\mathrm{H}(\ep)/M)f_n$ respects the capacity constraint for $\ep$ small enough depending on $d$.
Indeed, we have for $\ep$ small enough depending on $d$,
$$\|\widehat{f}_n(e)\|_2\leq (1-\ep^{\alpha/4})(1+8d^3\ep ^{\alpha})(M-\mathrm{H}(\ep))\leq  M-\mathrm{H}(\ep)\,.$$
  On the event $\cF\cap\cap_{x\in\fT(\cyl(\cA,h,\vv)))}\cE_x$, we have that $\widehat{f}_n\in \widehat{\cS}_n(\cyl(\cA,h,\vv),s\vv,g_1(\ep))$ for $n$ large enough where
$$g_1(\ep)=\left(1+\frac{2d}{m}\right)^{-(d-1)}\left(1-\frac{\mathrm{H}(\ep)}{M}\right)(1-\ep ^{\alpha/4})^2\,.$$

\noindent {\bf Conclusion.}
Using lemma \ref{lem:propdis4}, on the event $\cap_{x\in\fT(\cyl(\cA,h,\vv))}\cE_x$, we have
\begin{align}\label{eq:disconv}
&\dis(\amu_n(\widehat{f}_n),s\vv\ind_{\cyl(\cA,h,\vv)}\cL^d)\nonumber
\\&\quad\leq \dis(\amu_n(\widehat{f}_n),\amu_n(f_n))
+\dis(\amu_n(f_n),s\vv\ind_{\cyl(\cA,h,\vv)}\cL^d)\nonumber\\
&\quad\leq  2\left(1-(1-\ep ^{\alpha/4})(1 -\mathrm{H}(\ep)/M)\right)\frac{1}{n^d}\sum_{\substack{e\in\E_n^d:c(e)\in\cyl(\cA,h,\vv)}}\|f_n(e)\|_ 2+ \dis(\amu_n(f_n)\ind_\Cor,s\vv\ind_{\Cor}\cL^d)\nonumber\\
&\hspace{1cm}+ \sum_{x\in\fT(\cyl(\cA,h,\vv))}\dis(\amu_n(f_n^x),s\vv\ind_{\pi_{x,\kappa}(\fC)}\cL^d)\nonumber\\
&\quad\leq\left(\ep^{\alpha/4} +\frac{\mathrm{H}(\ep)}{M}(1+\ep ^{\alpha/4})\right)4dM \cL^d(\cV_2(\cyl(\cA,h,\vv),d/n))+2M\frac{|\Cor_n|}{n^d}+2dM\cL^d(\Cor)\nonumber\\
&\hspace{1cm}+4\ep^{\alpha_0} \cL^d(\cyl(\cA,h,\vv))\,.
\end{align}
We have for $n$ large enough depending on $\cA$ and $h$ that 
\begin{align}\label{eq:preldis1}
\cL^d(\cV_2(\cyl(\cA,h,\vv),d/n))\leq 2\cL^d(\cyl(\cA,h,\vv))\,.
\end{align}
Let us estimate the size of $\cH^{d-1}(\partial\Cor)$:
\begin{align}\label{eq:preldis2}
\cH^{d-1}(\partial\Cor) &\leq \cH^{d-1}(\partial \cyl(\cA,h,\vv))+|\fT(\cyl(\cA,h,\vv))|\cH^{d-1}(\partial (\kappa\fC))\nonumber\\
&\leq\cH^{d-1}(\partial \cyl(\cA,h,\vv))+ \frac{\cL^d(\cyl(\cA,h,\vv))}{\kappa ^d}2d\kappa ^{d-1}\nonumber\\
&\leq \cH^{d-1}(\partial \cyl(\cA,h,\vv))+ 2d\frac{\cL^d(\cyl(\cA,h,\vv))}{\kappa }
\end{align}
By doing similar computations than in \eqref{eq:contdis4}, we have
\begin{align}\label{eq:preldis3}
|\Cor_n|\leq 2d\left(\cL^d(\Cor)+\frac{2d}{n}\cH^{d-1}(\partial\Cor)\right)n^d 
\,.
\end{align}
Let us upper-bound the volume of the corridor for $n$ large enough depending on $\ep$:
\begin{align*}
\cL^d(\Cor)&\leq |\fT(\cyl(\cA,h,\vv))|\cL^d\left(\kappa\left(1+\frac{2d}{m}\right)\fC\setminus\kappa \fC\right)+\cL^d\left(\cV_2(\partial \cyl(\cA,h,\vv), d^2\kappa)\right)\nonumber\\
&\leq \frac{\cL^d(\cyl(\cA,h,\vv))}{\kappa^d(1+2d/m)^d}\kappa^d\left(\left(1+\frac{2d}{m}\right)^{d}-1\right)+4d^2\kappa\cH^{d-1}(\partial\cyl(\cA,h,\vv))\nonumber\\
&\leq \frac{2^{d+1}}{m}d\cL^d(\cyl(\cA,h,\vv))+4d^2\kappa\cH^{d-1}(\partial\cyl(\cA,h,\vv))\nonumber\\
&\leq 2^{d+2}d\ep ^{\alpha_0}\cL^d(\cyl(\cA,h,\vv)) +8d^2\ep \cH^{d-1}(\partial\cyl(\cA,h,\vv))\,
\end{align*}
where in the second inequality we use proposition \ref{prop:minkowski} for $\kappa$ small and in the last inequality we use the fact that $\kappa$ goes to $\ep$ when $n$ goes to infinity.
We detail here an inequality, we used in the previous inequality and that we will use again in what follows. For $x\in[0,1]$, we have
\begin{align}\label{eq:binomenewton}
(1+x)^d-1=\sum_{k=1}^d\binom{d}{k}x^k\leq x\sum_{k=1}^d\binom{d}{k}\leq 2^dx\,.
\end{align}
Finally, for $\ep$ small enough depending on $\cA$ and $h$, we have
\begin{align}\label{eq:contvolcor}
\cL^d(\Cor)\leq 2^{d+3}d\ep ^{\alpha_0}\cL^d(\cyl(\cA,h,\vv))\,.
\end{align}
For $\ep$ small enough depending on $\cA$ and $h$, for $n$ large enough depending on $\ep$, using inequalities \eqref{eq:disconv}, \eqref{eq:preldis1}, \eqref{eq:preldis2}, \eqref{eq:preldis3} and \eqref{eq:contvolcor}, we have
$$\dis(\amu_n(\widehat{f}_n),s\vv\ind_{\cyl(\cA,h,\vv)}\cL^d)\leq g_0(\ep)\cL^d(\cyl(\cA,h,\vv)) $$
where 
$$g_0(\ep)=8dM\left(\ep^{\alpha/4} +\frac{\mathrm{H}(\ep)}{M}(1+\ep ^{\alpha/4})\right)+2^{d+6}\ep ^{\alpha_0}d^2M +4\ep^{\alpha_0} \,.$$
Finally, we have
$$\Prb\left(\cF\cap\bigcap_{x\in\fT(\cyl(\cA,h,\vv))}\cE_x\right)\leq \Prb(\cE_n(\cyl(\cA,h,\vv),s\vv,g_0(\ep),g_1(\ep)))\,.$$
Using the independence, we have 
\begin{align}\label{eq:indepprelconv}
\Prb(\cF)&\prod_{x\in\fT(\cyl(\cA,h,\vv))}\Prb(\cE_x)\leq\Prb(\cE_n(\cyl(\cA,h,\vv),s\vv,g_0(\ep),g_1(\ep)))\,.
\end{align}
We have  $n\kappa(1+2d/m)\in\sN$. Therefore, for any $x\in\fT(\cyl(\cA,h,\vv))$ we have $x\in\sZ_n^d$. Hence, if we set $n_0= n\kappa$, the application $\pi_{x,n_0/n}$ is a bijection from $\E_{n_0}^d$ to $\E_n^d$.
We can apply lemma \ref{lem:scaling1}:
\begin{align}\label{eq:eqn2}
\Prb(\cE_x)\geq \Prb\left(\exists f_{n_0}\in\cS_{n_0}(\fC) \,(\ep,s\vv)\text{-well-behaved}:\quad \dis(\amu_{n_0}(f_{n_0}),s\vv\ind_ {\fC}\cL^d)\leq \ep^{\alpha_0} \right)
\end{align}
and we have using lemma \ref{lem:toutefamille} and theorem \ref{thmbrique}
\begin{align}\label{eq:eqn33}
\lim_{\ep\rightarrow 0}\limsup_{n\rightarrow \infty} \frac{1}{n^d\kappa^d}\log\Prb(\cE_x)\geq - I(s\vv)\,.
\end{align}
Besides, using inequality \eqref{eq:eqn33}, we have 
\begin{align}\label{eq:n3b}
\lim_{\ep\rightarrow 0}\limsup_{n\rightarrow \infty} \frac{1}{n^d}\sum_{x\in\fT(\cyl(\cA,h,\vv))}\log\Prb(\cE_x)&=\lim_{\ep\rightarrow 0}\limsup_{n\rightarrow \infty} \frac{1}{n^d}|\fT(\cyl(\cA,h,\vv))|\log\Prb(\cE_{x_0})\nonumber\\
&\geq \lim_{\ep\rightarrow 0}\limsup_{n\rightarrow \infty} \frac{1}{n^d}\frac{\cL^d(\cyl(\cA,h,\vv))}{\kappa^d(1+2d/m)^d}\log\Prb(\cE_{x_0})\nonumber\\
&\geq -\cL^d(\cyl(\cA,h,\vv))I(s\vv)\,
\end{align}
where $x_0\in\fT(\cyl(\cA,h,\vv))$.
Besides, we have
$$\limsup_{n\rightarrow \infty}\frac{1}{n^d}\Prb(\cF)=\limsup_{n\rightarrow \infty} \frac{|\Cor_n|}{n^d}\log G([M-\mathrm{H}(\ep),M])\,.$$
Hence, we can define the function $\mathrm{H}$ as in equality \eqref{eq:defh}, using the control of the volume of the corridor \eqref{eq:contvolcor}:
\begin{align}\label{eq:defh2}
\mathrm{H}(\ep)=\inf\left\{a>0: G([M-a,M])\geq 2^{d+3}d\ep ^{\alpha_0}\cL^d(\cyl(\cA,h,\vv)\right\}\,.
\end{align}
and since $M$ is the supremum of the support of $G$ with the same arguments as in the proof of theorem \ref{thmbrique}, we can prove that $\mathrm{H}(\ep)$ goes to $0$ when $\ep$ goes to $0$ and
\begin{align*} 
 \lim_{\ep\rightarrow 0} \limsup_{n\rightarrow\infty}\cL^d(\Cor)\log G([M-\mathrm{H}(\ep),M])=0\,.
 \end{align*} 
Hence, we get
\begin{align}\label{eq:eqn4}
\lim_{\ep\rightarrow 0}\limsup_{n\rightarrow \infty} \frac{1}{n^d}\log\Prb(\cF)=0\,.
\end{align}
Finally, combining inequalities \eqref{eq:indepprelconv}, \eqref{eq:n3b} and \eqref{eq:eqn4}, by taking the liminf when $\ep$ goes to $0$ we obtain:
\begin{align*}
\liminf_{\ep\rightarrow 0}&\limsup_{n\rightarrow \infty}\frac{1}{n^d}\Prb(\cE_n(\cyl(\cA,h,\vv),s\vv,g_0(\ep),g_1(\ep)))\geq -\cL^d(\cyl(\cA,h,\vv))I(s\vv)\,.
\end{align*}
The result follows.
\end{proof}
\begin{proof}[Proof of theorem \ref{thm:convexity}]
We have to treat separately the case where $\vv_1=\pm\vv_2$. \newline
\noindent {\bf $\bullet$ First Case $\vv_1=\pm\vv_2$.}
Let $\lambda\in[0,1]$ and $l>0$ be a small real number, in particular, we have $l<1$. Let $\vv\in\sS^{d-1}$ and $s_1,s_2\in\sR$. Without any loss of generality, we can assume that $\vv\cdot \overrightarrow{e_d}\neq 0$. Let $C^\lambda_1$ and $C^\lambda_2$ be the following sets (see figure \ref{fig2})
\[C^\lambda_1=\cyl\left([0,\lambda l ^2]\times [0,l^2]^{d-2}\times \{0\},l^2,\vv\right) \qquad \text{and}\qquad C^\lambda_2=\cyl\left([\lambda l ^2,l^2]\times [0,l^2]^{d-2}\times \{0\},l^2,\vv\right)\,.\] 
We pave $\fC$ with translates of $C^\lambda=C_1^\lambda\cup C_2^\lambda$.
\begin{figure}[h]
\begin{center}
\def\svgwidth{0.5\textwidth}
   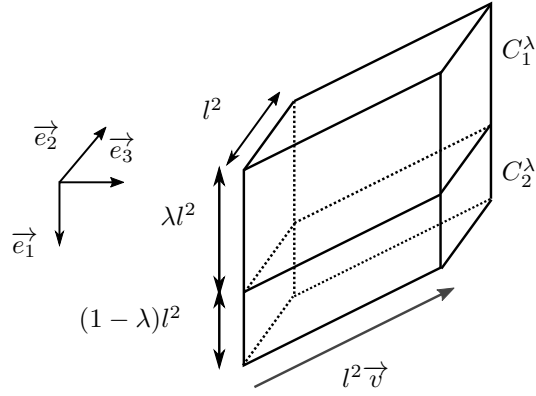
   \caption{\label{fig2}Representation of $C^\lambda$  ($d=3$)}
   \end{center}
\end{figure}
Note that we have
\[\cL^d(C_1^\lambda)=\lambda \cL^d(C^\lambda)\qquad\text{and}\qquad \cL^d(C_2^\lambda)=(1-\lambda) \cL^d(C^\lambda)\,.\]
We consider the following set $\fT$ of translated of $C^\lambda$:
\[\fT=\left\{C^\lambda+\sum_{i=1}^{d-1}k_il^2\overrightarrow{e_i}+k_dl ^2\vv:\,(k_1,\dots,k_d)\in\sZ^d\right\}\,.\]
Let $\fT(\fC)$ be the following set
\[\fT(\fC)=\{F\in\fT:\,F\cap\fC\neq \emptyset\}\,.\]
 Let $g_0$ and $g_1$ be the functions  defined in lemma \ref{lem:preconv}.
For $F=C^\lambda+x\in\fT(\fC)$, we denote by $\cG_F$ the following event (see the definition of $\cE_n$ in lemma \ref{lem:preconv})
\[\cG_F=\cE_n(C^\lambda_1+x,s_1\vv,g_0(\ep),g_1(\ep))\cap\cE_n(C^\lambda_2+x,s_2\vv,g_0(\ep),g_1(\ep))\,.\]
In other words, on the event $\cG_F$ we suppose the existence of a stream similar to $s_1\vv$ in $C_1^\lambda+x$ (respectively to $s_2\vv$ in $C_2 ^\lambda+x$). We denote by $f_n^{1,F}$ and $f_n^{2,F}$ the streams corresponding to these events (chosen according to a deterministic rule).
We denote by $\mathfrak{B}$ the following set of edges:
\[\mathfrak{B}= \left\{e\in\E_n^d\cap\fC:\cP(e)\cap \bigcup_{F=x+C^\lambda\in\fT(\fC)}(x+\partial C^\lambda_ 1)\cup (x+\partial C^\lambda_ 2)\neq \emptyset\right\}\,.\]
We denote by $\cF$ the following event 
$$\cF=\left\{\,\forall e\in \mathfrak{B}\qquad t(e)\geq g_1(\ep) \max(|s_1|,|s_2|)\|\vv\|_\infty\right\}\,.$$
On the event $ \cF\cap \cap_{F\in\fT(\fC)}\cG_F$,
we claim that there exists a stream $f_n\in\cS_n(\fC)$ obtained by concatenating all the streams $(f_n^{1,F},f_n^{2,F})$ such that
\[\dis(\amu_n(f_n),(\lambda s_1+(1-\lambda)s_2)\vv\ind_\fC)\leq K_1l+g_0(\ep)\,\]
where $K_1$ is a positive constant depending only on $M$ and $d$.
Set \[\ssigma=\sum_{F=( C^\lambda+z)\in \fT(\fC)}(s_1\ind_{C^\lambda_1+z}+s_2\ind_{C^\lambda_2+z})\vv\,.\]
Note that sets in $\fT(\fC)$ have pairwise disjoint interiors and $\fC\subset\cup_{F\in\fT(\fC)}F$.
We build $f_n$ as follows: for any $e=\langle y,w \rangle\in\E_n^d\cap\fC$
\[f_n(e) = \left\{
    \begin{array}{ll}
        f_n^{1,C^\lambda+x}(e) & \mbox{if } \cP(e)\subset C^\lambda_1+x, \,C^\lambda+x\in\fT(\fC) \\
        f_n^{2,C^\lambda+x}(e) & \mbox{if } \cP(e)\subset C^\lambda_2+x, \,C^\lambda+x\in\fT(\fC)\\
        g_1(\ep)n^{d-1}\left(\int_{\cP(e)}\ssigma(u)\cdot (n\,\overrightarrow{yw})\,d\cH^{d-1}(u)\right)n\,\overrightarrow{yw} &\mbox{if } e\in\mathfrak{B}.
    \end{array}
\right.
\]
It remains to check that $f_n$ satisfies the node law everywhere in $\fC$. Remember that $C=[-1/(2n),1/(2n)]^d$. By construction of $f_n$, if  $w\in\sZ_n^d$ is such that $(w+C)\subset (x+C_i^\lambda) $ for $i=1,2$ and $x$ such that $(x+C^\lambda)\in\fT(\fC)$ then since $f_n^{i,C^\lambda+x}$ satisfies the node law at $w$ it is also true for the stream $f_n$.
 Let $w\in\sZ_n^d$ such that $$(w+C)\cap \bigcup_{x: x+C^\lambda\in\fT(\fC)}(x+ (\partial C_1^\lambda\cup\partial C_2^\lambda))\neq \emptyset\,.$$
For $z\in\sZ_n^d$ such that $e=\langle w,z\rangle\in\E_n^d$, either $\cP(e)\subset (x+C^\lambda_1) $ for some $x$ such that $x+C^\lambda\in\fT(\fC)$  either $\cP(e)\subset (x+C^\lambda_2)$ or $e\in\mathfrak{B}$. In any case, we have for $e=\langle w,z\rangle$
\[ f_n(e)=g_1(\ep)n^{d-1}\left(\int_{\cP(e)}\ssigma(u)\cdot (n\,\overrightarrow{wz})\, d\cH^{d-1}(u)\right) n\,\overrightarrow{wz}\,.\]
We recall that $L(C_1^\lambda)$ denotes the lateral sides of the cylinder $C_1^\lambda$, $T(C_1^\lambda)$ its top and $B(C_1^\lambda)$ its bottom.
We apply Gauss-Green theorem for $\ssigma$ in $(w+C)\cap(x+C^\lambda_ 1)$:
\begin{align}\label{eq:gg}
&0=\int_{(w+C)\cap(x+C^\lambda_ 1)}\diver\ssigma d\cL ^d\nonumber\\
&=\sum_{\substack{y\in\sZ_n^d:\\\langle w,y\rangle\in\E_n^d}}\left(\int_{\cP(\langle w,y\rangle)}n s_1\vv\cdot\overrightarrow{wy}\ind_{x+C^\lambda_ 1}(z)\,d\cH^{d-1}(z)\right)+\int_{(C+w)\cap (x+\partial C^\lambda_1)} s_1\vv\cdot \nn_{x+ C^\lambda_1}(u)d\cH^{d-1}(u)\nonumber\\
&=\sum_{\substack{y\in\sZ_n^d:\\\langle w,y\rangle\in\E_n^d}}\left(\int_{\cP(\langle w,y\rangle)}n s_1\vv\cdot \overrightarrow{wy}\ind_{x+C^\lambda_ 1}(z)\,d\cH^{d-1}(z)\right)+\int_{(C+w)\cap (x+T(C^\lambda_1)\cup B(C^\lambda_1))}s_1\vv\cdot \nn_{x+ C^\lambda_1}(u)d\cH^{d-1}(u)
\end{align}
where we use in the last equality that if $u\in (x+ L(C_1^\lambda))$ then $\vv\cdot \overrightarrow{n}_{x+C_1^\lambda}(u)=0$.
Note that for $u\in (x+T(C_1^\lambda))=(x+l^2\vv+B(C_1^\lambda))$ we have
$\overrightarrow{n}_{x+C_1 ^\lambda}(u)=-\overrightarrow{n}_{x+C_1^\lambda +l^2 \vv}(u)=\overrightarrow{e_d}$.
It follows that 
$$\sum_{x:x+C^\lambda\in\fT(\fC)}\int_{(C+w)\cap (x+T(C^\lambda_1)\cup B(C^\lambda_1))}s_1\vv\cdot \overrightarrow{n}_{x+ C^\lambda_1}(u)d\cH^{d-1}(u)=0$$ 
and thanks to equality \eqref{eq:gg}, we have
\[\sum_{x:x+C^\lambda\in\fT(\fC)}\sum_{\substack{y\in\sZ_n^d:\\\langle w,y\rangle\in\E_n^d}}\int_{\cP(\langle w,y\rangle)}ns_1\vv\cdot \overrightarrow{wy}\ind_{x+C^\lambda_ 1}(z)\,d\cH^{d-1}(z)=0\,.\]
By similar arguments, we can prove that
\[\sum_{x:x+C^\lambda\in\fT(\fC)}\sum_{\substack{y\in\sZ_n^d:\\\langle w,y\rangle\in\E_n^d}}\int_{\cP(\langle w,y\rangle)}ns_2\vv\cdot \overrightarrow{wy}\ind_{x+C^\lambda_ 2}(z)\,d\cH^{d-1}(z)=0\,.\]
Hence, we have
\[\sum_{\substack{y\in\sZ_n^d:\\\langle w,y\rangle\in\E_n^d}}\int_{\cP(\langle w,y\rangle)}n\ssigma(z)\cdot \overrightarrow{wy}\,d\cH^{d-1}(z)=0\]
and $f_n$ satisfies the node law at $w$. We recall that $g_1(\ep)\leq 1$, so $f_n$ satisfies the capacity constraint in $\fC$ on the event $\cF\cap\cap_{F\in\fT(\fC)}\cG_F$. Finally, we have $f_n\in\cS_n(\fC)$. 

Write $\nu=(\lambda s_1+(1-\lambda)s_2)\vv\ind_\fC\cL^d$.
We want to upper bound the distance $\dis(\amu_n(f_n),\nu)$. To do so, we introduce another measure $\widetilde{\nu}$ and we upper bound separately the quantities $\dis(\amu_n(f_n),\widetilde{\nu})$ and $\dis(\widetilde{\nu},\nu)$.
We set
\[\widetilde{\nu}=\sum_{z:\,( C^\lambda+z)\in \fT(\fC)}(s_1\ind_{C^\lambda_1+z}+s_2\ind_{C^\lambda_2+z})\vv\ind_\fC\cL ^d=\ssigma\ind_\fC\cL^d\,.\]
We denote for short $\amu_n(f_n)$ by $\amu_n$.
Notice that for $z$ such that $( C^\lambda+z)\in \fT(\fC)$ and $( C^\lambda+z)\subset\fC$, we have
\[\|\nu(C^\lambda+z)-\widetilde{\nu}(C^\lambda+z)\|_2=0\,.\]
Let $k\geq 1$ such that $ l\leq \sqrt{\ep_\fC /d} 2^{-k}$ where $\ep_{\fC}$ was defined in lemma \ref{lem:interbord}. It follows that $dl^2\leq\ep_{\fC} 2^{-2k}\leq \ep_{\fC} 2^{-k}$. Such a $k$ exists when $l$ is small enough. Let $x\in[-1,1[^d$ and $\beta\in[1,2]$. Let $Q\in\Delta ^k _\beta$. 

\noindent $\bullet$ Let us first assume that $(Q+x)\subset \fC $ then
\begin{align*}
\|\nu(Q+x)-\widetilde{\nu}(Q+x)\|_2&\leq \sum_{z:(C^\lambda+z)\subset(Q+x)}\|\nu(C^\lambda+z)-\widetilde{\nu}(C^\lambda+z)\|_2\nonumber\\
&\qquad+\sum_{z:(C^\lambda+z)\cap(\partial Q+x)\neq\emptyset}\|\nu((Q+x)\cap (C^\lambda+z))-\widetilde{\nu}((Q+x)\cap (C^\lambda+z))\|_2\nonumber\\
&\leq \sum_{z:(C^\lambda+z)\cap(\partial Q+x)\neq\emptyset}2(|s_1|+|s_2|)\cL^d(C^\lambda)\,.
\end{align*}
Moreover, using lemma \ref{lem:interbord}, we have that $\fT$ is a paving and $\diam( C^\lambda)\leq dl^2\leq \ep_\fC\beta 2^{-k}$, it follows that
\[\left|\left\{z:(C^\lambda+z)\cap(\partial Q+x)\neq\emptyset\right\}\right|\leq 4 \frac{\cH^{d-1}(\partial Q)}{\cL^d(C^\lambda)}\diam (C^\lambda)\,.\]
Finally, we get
\begin{align}\label{eq:Qdansc}
\|\nu(Q+x)-\widetilde{\nu}(Q+x)\|_2\leq 8(|s_1|+|s_2|)dl^{2}\cH^{d-1}(\partial Q)\,.
\end{align}

\noindent $\bullet$ Let us assume now that $(Q+x)\cap\partial\fC \neq \emptyset $. Thus, we have
\begin{align*}
\|\nu(Q+x)-\widetilde{\nu}(Q+x)\|_2&\leq \sum_{z:(C^\lambda+z)\subset(Q+x)\cap \fC}\|\nu(C^\lambda+z)-\widetilde{\nu}(C^\lambda+z)\|_2\\
&\qquad+\sum_{z:(C^\lambda+z)\cap\partial ((Q+x)\cap\fC)\neq\emptyset}\|\nu((Q+x)\cap (C^\lambda+z))-\widetilde{\nu}((Q+x)\cap (C^\lambda+z))\|_2\\
&\leq \sum_{z:(C^\lambda+z)\cap\partial ((Q+x)\cap\fC)\neq\emptyset}2(|s_1|+|s_2|)\cL^d((Q+x)\cap \fC\cap(C^\lambda+z))\,
\end{align*}
where we used in the last inequality that $\nu$ and $\widetilde{\nu}$ are null outside $\fC$.
It follows that
\begin{align*}
\sum_{Q\in\Delta_\beta^k:(Q+x)\cap \partial \fC\neq\emptyset}&\|\nu(Q+x)-\widetilde{\nu}(Q+x)\|_2\\
&\leq \sum_{\substack{Q\in\Delta_\beta^k:\\(Q+x)\cap \partial \fC\neq\emptyset}}\sum_{z:(C^\lambda+z)\cap\partial ((Q+x)\cap\fC)\neq\emptyset}2(|s_1|+|s_2|)\cL^d((Q+x)\cap \fC\cap(C^\lambda+z))\\
&\leq \sum_{\substack{Q\in\Delta_\beta^k:\\(Q+x)\cap \partial \fC\neq\emptyset}}\sum_{z:(C^\lambda+z)\cap\partial \fC\neq\emptyset}2(|s_1|+|s_2|)\cL^d((Q+x)\cap \fC\cap(C^\lambda+z))\\
&\quad+\sum_{\substack{Q\in\Delta_\beta^k:\\(Q+x)\cap \partial \fC\neq\emptyset}}\sum_{z:(C^\lambda+z)\cap\partial (Q+x))\neq\emptyset}2(|s_1|+|s_2|)\cL^d((Q+x)\cap \fC\cap(C^\lambda+z))\\
&\leq  \left|\left\{z:(C^\lambda+z)\cap\partial \fC\neq\emptyset\right\}\right|2(|s_1|+|s_2|)\cL^d(C^\lambda)\\
&\quad +\sum_{\substack{Q\in\Delta_\beta^k:\\(Q+x)\cap \partial \fC\neq\emptyset}}\left|\left\{z:(C^\lambda+z)\cap\partial(Q+x)\neq\emptyset\right\}\right|2(|s_1|+|s_2|)\cL^d(C^\lambda)\,.
\end{align*}
Using again lemma \ref{lem:interbord}, we obtain
\begin{align}\label{eq:QpasdansC}
\sum_{Q\in\Delta_\beta^k:(Q+x)\cap \partial \fC\neq\emptyset}&\|\nu(Q+x)-\widetilde{\nu}(Q+x)\|_2\nonumber\\
&\leq  8(|s_1|+|s_2|)(\diam C^\lambda) \left(\cH^{d-1}(\partial\fC)+ |\{Q\in\Delta_\beta^k:(Q+x)\cap \partial \fC\neq\emptyset\}|\cH^{d-1}(\partial Q) \right)\,.
\end{align}
Combining inequalities \eqref{eq:Qdansc} and \eqref{eq:QpasdansC}, it follows that
\begin{align*}
\sum_{Q\in\Delta_\beta^k}\|\nu(Q+x)-\widetilde{\nu}(Q+x)\|_2&\leq \sum_{\substack{Q\in\Delta_\beta^k: \\(Q+x)\subset\fC}}\|\nu(Q+x)-\widetilde{\nu}(Q+x)\|_2+\sum_{\substack{Q\in\Delta_\beta^k:\\(Q+x)\cap \partial \fC\neq\emptyset}}\|\nu(Q+x)-\widetilde{\nu}(Q+x)\|_2\\
&\leq 8(|s_1|+|s_2|) d l^2\left( |\{Q\in\Delta_\beta^k: (Q+x)\cap \fC\neq\emptyset\}| \cH^{d-1}(\partial Q)+\cH^{d-1}(\partial\fC)\right)\\
&\leq 8(|s_1|+|s_2|)dl^2\left(\frac{\cL^d(3\fC)}{(\beta 2^{-k})^d}2d(\beta 2^{-k})^{d-1} +2d \right)\\
&\leq 8(|s_1|+|s_2|)dl^2\left(2\,3^dd2^k+2d \right)\\
&\leq C_d(|s_1|+|s_2|)l
\end{align*}
where in the last inequality we use the fact that $ 2^k\leq \sqrt{\ep_\fC/d}/l$ and where $C_d$ is a positive constant depending only on $d$. 
It follows that
\begin{align*}
&\sum_{k=1}^\infty\frac{1}{2^{k}}\sum_{Q\in\Delta_\beta^k}\|\nu(Q+x)-\widetilde{\nu}(Q+x)\|_2\\
&\quad\leq \sum_{k: 2^{-k}\geq l\sqrt{d/\ep_\fC}}\frac{1}{2^{k}}\sum_{Q\in\Delta_\beta^k}\|\nu(Q+x)-\widetilde{\nu}(Q+x)\|_2+\sum_{k: 2^{-k}<l\sqrt{d/\ep_\fC}}\frac{1}{2^{k}}\sum_{Q\in\Delta_\beta^k}\|\nu(Q+x)-\widetilde{\nu}(Q+x)\|_2\\
&\quad\leq 2C_d(|s_1|+|s_2|)l+2(|s_1|+|s_2|)\sum_{k: 2^{-k}<l\sqrt{d/\ep_\fC}}\frac{1}{2^{k}}\sum_{\substack{Q\in\Delta_\beta^k:\\(Q+x)\cap \fC\neq\emptyset}}\cL^d(Q)\\
&\quad\leq 2C_d(|s_1|+|s_2|)l+2(|s_1|+|s_2|)\sum_{k: 2^{-k}<l\sqrt{d/\ep_\fC}}\frac{1}{2^{k}}\cL^d(3\fC)\\
&\quad\leq 23^d\left(C_d+2\sqrt{\frac{d}{\ep_\fC}}\right) (|s_1|+|s_2|)l\,.
\end{align*}
Consequently, we have 
\[\dis(\nu,\widetilde{\nu})\leq 2\left(C_d+23^d\sqrt{\frac{d}{\ep_\fC}}\right) (|s_1|+|s_2|)l\,.\]
Let us now compute the distance $\dis(\amu_n,\widetilde{\nu})$.
Using lemma \ref{lem:propdis4} and lemma \ref{lem:interbord}, on the event $\cF\cap\cap_{F\in\fT(\fC)}\cG_F$ we have
\begin{align*}
\dis(\amu_n,\widetilde{\nu})&\leq \sum_{F=(w+C^\lambda)\in\fT: F\subset\fC}\dis(\amu_n\ind_{w+C^\lambda_1},s_1\vv\ind_{w+C^\lambda_1}\cL^d)+\dis(\amu_n\ind_{w+C^\lambda_2},s_2\vv\ind_{w+C^\lambda_2}\cL^d)\\
&\quad+\sum_{F=(w+C^\lambda)\in\fT: F\cap\partial\fC\neq \emptyset}4 (|s_1|+|s_2|)\cL^d(C^\lambda)\\
&\leq  \sum_{F=(w+C^\lambda)\in\fT: F\subset\fC}g_0(\ep)\cL^d(C ^\lambda)+16d (|s_1|+|s_2|)\cH^{d-1}(\partial\fC)l^2\leq g_0(\ep) + 16d^2l^2(|s_1|+|s_2|)\,.
\end{align*}
Note that if $I(s_1\vv),I(s_2\vv)$ are finite then necessarily for any $\ep>0$, then we have $$\Prb(t(e)\geq g_1(\ep)\max(|s_1|,|s_2|)\|v\|_\infty)>0\,.$$ Indeed, let us assume that $\Prb(t(e)\geq g_1(\ep_0)|s_1|\|v\|_\infty)=0$ for some $\ep_0>0$ and $I(s_1\vv)<\infty$, thanks to theorem \ref{thmbrique}, we have
$$\forall \ep>0\qquad\liminf_{n\rightarrow\infty}\frac{1}{n^d}\log \Prb(\exists f_n\in\cS_n(\fC) : \,\dis\big(\amu_n(f_n),s_1\vv\ind_\fC \cL ^d)\leq \ep)\geq -I(s_1\vv)\,.$$
By doing the same reasoning of proposition \ref{propadmissiblestream}, we can extract a subsequence and choose a sequence of configurations $(\omega_n)_{n\geq 1}$ such that $\amu_n(f_n)(\omega_n)$ weakly converges towards $s_1\vv\ind_\fC\cL^d$. Let $i\in\{1,\dots,d\}$ such that $|v_i|=\|v\|_\infty$. Without loss of generality, we can assume that $v_i\geq 0$ and $s_1\geq 0$. Then for any $e\in\E_n^d$ we have $f_n(e)\cdot \overrightarrow{e_i}\leq  g_1(\ep_0)s_1v_i$ and $$\int_\fC\amu_n(f_n)(\omega_n)\cdot \overrightarrow{e_i}d\cL^d\leq g_1(\ep_0)s_1v_i<\int_{\fC}s_1\vv\cdot \overrightarrow{e_i}d\cL^d\,.$$
This is a contradiction, it follows that if $I(s_1\vv)$ is finite then  $\Prb(t(e)\geq g_1(\ep_0)|s_1|\|v\|_\infty)>0$.

On the event $\cap_{F\in\fT(\fC)}\cG_F\cap\cF$, there exists $f_n\in\cS_n(\fC)$ such that 
\[\dis(\amu_n(f_n),\nu)\leq \dis(\amu_n(f_n),\widetilde{\nu})+\dis(\widetilde{\nu},\nu)\leq g_0(\ep) +K_1l\]
where $K_1$ is a constant depending only on $d$, $s_1$ and $s_2$. It follows that for $\ep$ small enough depending on $l$
\[\Prb\left(\cap_{F\in\fT(\fC)}\cG_F\cap\cF\right)\leq \Prb\left(\exists f_n\in\cS_n(\fC):\,\dis(\amu_n(f_n),(\lambda s_1+(1-\lambda)s_2)\vv\ind_\fC\cL^d)\leq 2K_1 l\right)\,.\]
Notice that $|\fB|/n^d$ goes to $0$ when $n$ goes to infinity, it follows that
\begin{align}\label{eq:contcf}
\limsup_{n\rightarrow \infty}\frac{1}{n^d}\log\Prb(\cF)=\limsup_{n\rightarrow \infty}\frac{1}{n^d}|\mathfrak{B}|\log\Prb\Big(t(e)\geq g_1(\ep) \max(|s_1|,|s_2|)\|v\|_\infty\Big)=0\,.
\end{align}
We have for $F\in\fT(\fC)$, using the independence:
$$\Prb(\cG_F)=\Prb(\cE_n(C^\lambda_1+x,s_1\vv,g_0(\ep),g_1(\ep)))\Prb(\cE_n(C^\lambda_2+x,s_2\vv,g_0(\ep),g_1(\ep)))\,.$$
By lemma \ref{lem:preconv},  we obtain
\begin{align}\label{eq:contgF}
-\liminf_{\ep\rightarrow 0} \limsup_{n\rightarrow \infty}\frac{1}{n^d}\log \Prb(\cG_F)\leq (\lambda I(s_1\vv)+(1-\lambda)I(s_2\vv))\cL^d(C^\lambda)\,.
\end{align}
Since the events $(\cG_F,\, F\in\fT(\fC))$ and $\cF$ are independent and using \eqref{eq:contcf}, we have
\begin{align*}
-&\limsup_{n\rightarrow \infty}\frac{1}{n^d}\log\Prb\left(\exists f_n\in\cS_n(\fC):\,\dis(\amu_n(f_n),(\lambda s_1+(1-\lambda)s_2)\vv)\leq 2K_1 l\right)\\
&\hspace{1cm} \leq-|\fT(\fC)|\limsup_{n\rightarrow \infty}\frac{1}{n^d}\log \Prb(\cG_{F_0})-\limsup_{n\rightarrow \infty}\frac{1}{n^d}\log\Prb(\cF)=-|\fT(\fC)|\limsup_{n\rightarrow \infty}\frac{1}{n^d}\log \Prb(\cG_{F_0})
\end{align*}
where $F_0\in\fT(\fC)$. Since the previous result holds for any $\ep$ small enough depending on $l$, inequality \eqref{eq:contgF} yields
\begin{align*}
-&\limsup_{n\rightarrow \infty}\frac{1}{n^d}\log\Prb\left(\exists f_n\in\cS_n(\fC):\,\dis(\amu_n(f_n),(\lambda s_1+(1-\lambda)s_2)\vv)\leq 2K_1 l\right)\\
&\hspace{4cm}\leq |\fT(\fC)|(\lambda I(s_1\vv)+(1-\lambda)I(s_2\vv)) \cL^d(C^\lambda)\,.
\end{align*}
Besides, we have
\[|\fT(\fC)|\leq \frac{\cL^d((1+2l^2)\fC)}{\cL ^d(C^\lambda)}\,.\]
Hence, we obtain
\begin{align*}
- \limsup_{n\rightarrow \infty}\frac{1}{n^d}\log\Prb&\left(\exists f_n\in\cS_n(\fC):\,\dis(\amu_n(f_n),(\lambda s_1+(1-\lambda)s_2)\vv)\leq 2K_1 l\right)\\
&\qquad\leq  (1+2l^2)^d(\lambda I(s_1\vv)+(1-\lambda)I(s_2\vv))\,.
\end{align*}
By letting $l$ go to $0$ (the left hand side is non-decreasing in $l$), we obtain
\begin{align*}
-\lim_{l\rightarrow 0} \limsup_{n\rightarrow \infty}\frac{1}{n^d}\Prb\left(\exists f_n\in\cS_n(\fC):\,\dis(\amu_n(f_n),(\lambda s_1+(1-\lambda)s_2)\vv)\leq 2K_1 l\right)\leq \lambda I(s_1\vv)+(1-\lambda)I(s_2\vv)\,.
\end{align*}
The result follows.

\noindent{\bf $\bullet$ Second Case:}
Let $\lambda\in[0,1]$ and $l\in[0,1]$.
Let $\vv_1, \vv_2\in \sS^{d-1}$ such that $\vv_1\neq\pm \vv_2$ and $s_1,s_2>0$. We claim that there exists $\overrightarrow{n}\in\sS^{d-1}$ such that $s_1\vv_1\cdot\an=s_2\vv_2\cdot \overrightarrow{n}\neq 0$. Indeed, we can complete $\vv=s_1\vv_1-s_2\vv_2$ into a normal basis $(\vv,\ff_2,\dots,\ff_d)$ of $\sR^d$ where $\overrightarrow{f}_2,\dots,\overrightarrow{f}_d$ are in $\sS^{d-1}$. There exists $i\in\{2,\dots,d\}$ such that $\vv_1\cdot\ff_i\neq 0$. If not, we have 
\[\forall i\in\{2,\dots,d\}\qquad \vv_1\cdot\ff_i=0\] and there exists $\lambda\in\sR$ such that we have $\vv_1=\lambda(s_1\vv_1-s_2\vv_2)$. This is a contradiction with $\vv_1\neq\pm \vv_2$. Hence, the vector $\overrightarrow{n}$ corresponds to the $\overrightarrow{f_i}$ such that  $\vv_1\cdot\ff_i\neq 0$ (if there are several choices we pick $\ff_i$ with the smallest $i$). Since $(\vv,\ff_2,\dots,\ff_d)$ is a normal basis, we have 
$$\vv\cdot\an =s_1\vv_1\cdot\an-s_2\vv_2\cdot \overrightarrow{n}=0\,.$$
\begin{figure}[h]
\begin{center}
\def\svgwidth{0.4\textwidth}
   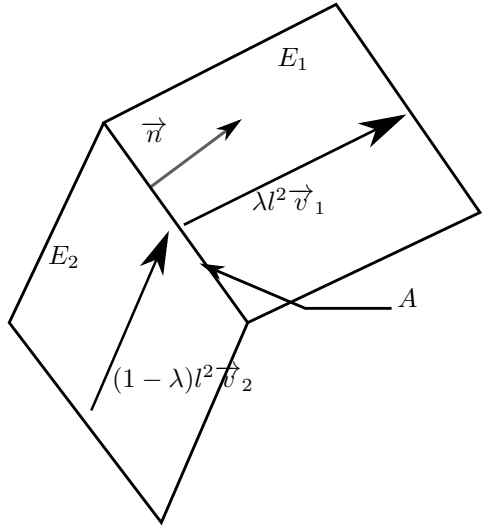
   \caption{\label{fig3}Representation of the set $E$ $(d=2)$}
   \end{center}
\end{figure}

\noindent Let $(\overrightarrow{g}_1,\dots,\overrightarrow{g}_{d-1},\overrightarrow{n})$ be an orthonormal basis. Let $A$ be the hyperrectangle normal to $\an$ whose expression in the basis $(\overrightarrow{g}_1,\dots,\overrightarrow{g}_{d-1},\overrightarrow{n})$ is $[0,l^2]^{d-1}\times\{0\}$.
Let $E_1,E_2$ and $E$ be the following sets (see figure \ref{fig3})
\[E_1=\cyl(A,\lambda l^2 ,\vv_1),\quad E_2=\cyl(A,(1-\lambda)l^2,-\vv_2)\quad\text{and}\quad E=E_1\cup E_2\,.\]
We have $\cL^d(E_1)=\lambda l^{2d}$ and $\cL^d(E_2)=(1-\lambda)l ^{2d}$.
We consider the following set $\fT$ of translated of $E$:
\[\fT=\left\{E+\sum_{i=1}^{d-1}k_il^2\overrightarrow{g}_i+k_dl ^2(\lambda \vv_1+(1-\lambda)\vv_2):\,(k_1,\dots,k_d)\in\sZ^d\right\}\,.\]
Let $\fT(\fC)$ be the following set
\[\fT(\fC)=\{F\in\fT:\,F\cap\fC\neq\emptyset\}\,.\]
Let $F\in\fT(\fC)$ write $F=E+x$. We denote by $\cF_F$ the following event
\[\cF_F=\cE_n(\cyl(A,\lambda l^2 ,\vv_1)+x,s_1\vv_1,g_0(\ep),g_1(\ep))\cap\cE_n(\cyl(A,(1-\lambda)l^2,-\vv_2)+x,s_2\vv_2,g_0(\ep),g_1(\ep))\,\]
where $g_0$ and $g_1$ are the functions defined in lemma \ref{lem:preconv}.
On $\cF_F$, we denote by $f_n^{1,F}$ and $f_n^{2,F}$ the streams corresponding to these events (chosen according to a deterministic rule if there are several possible choices).
We denote by $\mathfrak{B}$ the following set of edges:
\[\fB= \left\{e\in\E_n^d\cap\fC:\cP(e)\cap \bigcup_{F=x+E\in\fT(\fC)}(x+\partial\cyl(A,\lambda l^2 ,\vv_1))\cup (x+\partial \cyl(A,(1-\lambda)l^2,-\vv_2))\neq \emptyset\right\}\,.\]
We denote by $\cF$ the following event 
\[\cF=\left\{\forall e\in\fB\quad t(e)\geq \max(s_1\|v_1\|_\infty,s_2\|v_2\|_\infty)\right\}\,.\]
Set \[\ssigma=\sum_{z:\,( E+z)\in \fT(\fC)}(s_1\vv_1\ind_{\cyl(A,\lambda l^2 ,\vv_1))+z}+s_2\vv_2\ind_{\cyl(A,(1-\lambda)l^2,-\vv_2)+z})\ind_{\fC}\,.\]
On the event $\cap_{F\in\fT(\fC)}\cG_F\cap\cF$,
we define the following stream $f_n\in\cS_n(\fC)$ obtained by concatenating all the streams $f_n^{i,E+x}$,  $\forall  e=\langle y,w\rangle \in\E_n^d\cap\fC$:
\[f_n(e) = \left\{
    \begin{array}{ll}
        f_n^{1,E+x}(e) & \mbox{if } \cP(e)\subset \cyl(A,\lambda l^2 ,\vv_1)+x, \,E+x\in\fT(\fC) \\
        f_n^{2,E+x}(e) & \mbox{if } \cP(e)\subset \cyl(A,(1-\lambda)l^2,-\vv_2)+x, \,E+x\in\fT(\fC) \\
        g_1(\ep)n^{d-1}\left(\int_{\cP(e)}\ssigma(x)\cdot (n\,\overrightarrow{yw})\, d\cH^{d-1}(x)\right)n\,\overrightarrow{yw}  & \mbox{if } e\in\fB.
    \end{array}
\right.
\]
Let $y\in A$. Without loss of generality, we can assume that $\vv_1\cdot \overrightarrow{n}>0$. Hence, we have $\overrightarrow{n}_{\cyl(A,\lambda l^2,\vv_1)}(y)=-\overrightarrow{n}$. We have \[s_1\vv_1 \cdot \overrightarrow{n}_{\cyl(A,\lambda l^2,\vv_1)}(y)=-s_1\vv_1\cdot \overrightarrow{n}=s_2\vv_2\cdot \overrightarrow{n}= -s_2\vv_2 \cdot \overrightarrow{n}_{\cyl(A,(1-\lambda) l^2,-\vv_2)}(y)\,.\]
Hence, by same arguments than in the previous case, we have that the node law is satisfied at any point in $\fC\cap\sZ_n^d$.

By the same computations than in the previous case, we can prove that there exists a positive constant $K_2$ depending only on $d$, $s_1$ and $s_2$ such that on the event $\cF\cap_{F\in\fT(\fC)}\cG_F$, for $l$ small enough, we have
\[\dis\Big(\amu_n(f_n),(\lambda s_1\vv_1+(1-\lambda)s_2\vv_2)\ind_{\fC}\cL^d\Big)\leq g_0(\ep)+K_2l\,\]
and by similar arguments
\begin{align*}
&I(\lambda s_1\vv_1+(1-\lambda)s_2\vv_2)\\
&\qquad=-\lim_{l\rightarrow{0}}\limsup_{n\rightarrow \infty}\frac{1}{n^d}\log\Prb\left(\exists f_n\in\cS_n(\fC):\,\dis\Big(\amu_n(f_n),(\lambda s_1\vv_1+(1-\lambda)s_2\vv_2)\ind_{\fC}\Big)\leq 2K_2l\right)\\&\qquad\leq \lambda I(s_1\vv_1)+(1-\lambda)I(s_2\vv_2)\,.
\end{align*}
This yields the result.
%
%
\end{proof}
\subsection{Control of the elementary rate function}

In this section, we aim to obtain a control on the elementary rate function in terms of $G$.
\begin{prop}\label{prop:controleI}For any $\vv\in\sR^d$ such that $G([\|\vv\|_\infty,M])>0$, we have
$$I(\vv)\leq -d\log G([\|\vv\|_\infty,M])\,.$$
\end{prop}
\begin{proof} Let $\vv\in\sR^d$ such that $G([\|\vv\|_\infty,M])>0$. Let $\ep>0$.
We consider the following stream $f^{disc}_n$ that is the discretized version of $\vv\ind_{\fC}$ defined by
\begin{align*}
\forall e=\langle x,y\rangle \in\fC \qquad
& f_n ^{disc}(e)= n^2(\vv\cdot \overrightarrow{xy})\overrightarrow{xy}\,.
\end{align*}
In particular, if $\overrightarrow{xy}=\overrightarrow{e_i}/n$, for $i\in\{1,\dots,d\}$, we have
$$f_n ^{disc}(e)= v_i\overrightarrow{e_i}\,.$$
We aim to compute the distance $\dis(\amu_n(f_n^{disc}),\vv\ind_\fC\cL^d)$. We write $\amu_n$ for $\amu_n(f_n^{disc})$.
Let $z\in[-1,1[^d$, $k\geq 1$ and $\lambda\in[1,2]$. Let $Q\in\Delta ^k_\lambda$. 
We have 
\begin{align*}
n^d\amu_n((Q+z)\cap C)&=\sum_{x\in\sZ_n^d\cap(Q+z)\cap\fC}\vv+\sum_{x\in\sZ_n^d\setminus ((Q+z)\cap\fC)}\sum_{i=1}^d \vv_i\overrightarrow{e_i} \ind_{(x+\overrightarrow{e_i}/(2n))\in (Q+z)\cap\fC}\\
&- \sum_{x\in\sZ_n^d\cap (Q+z)}\sum_{i=1}^d \vv_i\overrightarrow{e_i} \ind_{(x+\overrightarrow{e_i}/(2n))\notin (Q+z)\cap\fC}
\end{align*}
Using proposition \ref{prop:minkowski}, it follows that for $n$ large enough depending on $Q$
\begin{align*}
&\|\amu_n((Q+z)\cap\fC)-\vv\cL^d((Q+z)\cap\fC)\|_2\\
&\qquad\leq  \left(\frac{|\sZ_n^d\cap(Q+z)\cap\fC|}{n^d}-\cL^d((Q+z)\cap\fC)\right)\|\vv\|_2+2d\frac{|\sZ_n^d \cap\cV_\infty(\partial((Q+z)\cap\fC),1/n)|}{n^d}\|\vv\|_\infty\\
&\qquad \leq \cL^d(\cV_2(\partial((Q+z)\cap\fC),d/n))(\|\vv\|_2+2d\|\vv\|_\infty)\\
&\qquad \leq \frac{4d}{n}\cH^{d-1}(\partial((Q+z)\cap\fC)(\|\vv\|_2+2d\|\vv\|_\infty)
\end{align*}
and
\begin{align*}
\sum_{Q\in(z+\Delta_k^\lambda)}\|\amu_n(Q\cap\fC)-\vv\cL^d(Q\cap\fC)\|_2&\leq \frac{4d}{n}\left(\cH^{d-1}(\partial\fC)+\frac{\cL^d(2\fC)}{\cL^d( Q)}\cH^{d-1}(\partial Q)\right)(\|\vv\|_2+2d\|\vv\|_\infty)\\
&\leq \frac{4d}{n}(2d+2^d2d 2^k)(\|\vv\|_2+2d\|\vv\|_\infty)
\end{align*}
Let $k_0\geq 1$ be the smallest integer such that
$$10dM 2^{-k_0}\leq \ep/2\,.$$
With this choice of $k_0$, we have for $n$ large enough
\begin{align*}
\sum_{k=k_0}^\infty \frac{1}{2^k}\sum_{Q\in(z+\Delta_k^\lambda)}\|\amu_n(Q\cap\fC)-\vv\cL^d(Q\cap\fC)\|_2\leq \sum_{k=k_0}^\infty \frac{1}{2^k}5dM=10dM 2^{-k_0}\leq \ep/2
\end{align*}
We can choose $n$ large enough depending on $\ep$ and $k_0$ such that
$$\sum_{k=0}^{k_0-1}\frac{1}{2^k}\frac{4d}{n}(2d+2^d2d 2^k)(\|\vv\|_2+2d\|\vv\|_\infty)\leq \frac {\ep}{2}\,.$$
It follows that for $n$ large enough depending on $\ep$, we have
$$\dis(\amu_n(f_n^{disc}),\vv\ind_\fC\cL^d)\leq \ep\,.$$
Note that on the event $\{\forall e\in\fC\quad t(e)\geq\|\vv\|_\infty\}$, the stream $f_n^{disc}$ belongs to $\cS_n(\fC)$.
We recall that
\[|\{e\in\E_n^d:e\in\fC\}|=|\{e=\langle x,y\rangle\in\E_n^d:\, x\in\fC, \,\exists i\in\{1,\dots,d\}\quad n\,\overrightarrow{xy}=\overrightarrow{e_i}\}|=d|\fC\cap\sZ_n^d|=dn^d\,.\]
Using the independence of the family $(t(e))_{e\in\E_n^d}$, it follows that for $n$ large enough depending on $\ep$
\begin{align}\label{eq:controleIparG}
-\frac{1}{n^d}\log \Prb(\exists f_n\in\cS_n(\fC):\dis(\amu_n(f_n),\vv\ind_\fC\cL^d)\leq \ep)&\leq -\frac{1}{n^d}\log \Prb(\forall e\in\fC\quad t(e)\geq\|\vv\|_\infty)\nonumber\\&=-d\log G([\|\vv\|_\infty,M])
\end{align}
Finally, by taking first the limsup in $n$ and then letting $\ep$ goes to $0$ in the previous inequality, we obtain that
$$I(\vv)\leq -d\log G([\|\vv\|_\infty,M])\,.$$
This yields the proof.
\end{proof}
\section{Upper large deviations for the stream in a domain\label{sect:ULD}}
The aim of this section is to prove that the function $\widehat{I}$ (defined in \eqref{eq:defhatI}), build from the elementary rate function $I$, is the rate function corresponding to the probability that a stream $f_n\in\cS_n(\Gamma^1,\Gamma ^2,\Omega)$ is close to some continuous stream $\ssigma\in\Sigma(\Gamma^1,\Gamma ^2,\Omega)$. This is the purpose of theorem \ref{thm:ULDpatate}.

We will need to approximate $\ssigma$ by a regular vector field.
\subsection{Approximation by a regular stream }

Define the function $\eta\in \sC^\infty_c(\sR^d,\sR)$ by
\[
\eta(x)=\left\{\begin{array}{ll}c\exp\left(\frac{1}{\|x\|_2-1}\right)&\mbox{if $\|x\|_2<1$}\\
0&\mbox{if $\|x\|_2\geq 1$}\end{array}\right.
\]
where the constant $c>0$ is adjusted such that $\int_{\sR^d}\eta(x)dx=1$.
For any $n\geq 1$, we denote by $K_n$ the following function
\begin{align}\label{eq:defKn}
\forall x\in\sR^d\qquad K_n(x)=n^d\eta(nx)\,.
\end{align}
The sequence $(K_n)_{n\geq 1}$ is a sequence of mollifiers.
Note that since for any $x\in\Omega$ we have $I(\ssigma(x))\geq 0$, it follows that $\widehat{I}(\ssigma)=\|I(\ssigma))\|_{L^1}$ and $x\mapsto I(\ssigma(x))\in L^1(\sR^d\rightarrow \sR, \cL^d)$.
\begin{prop}\label{prop:continuityIhat}Let $\ssigma\in\Sigma(\Gamma^1,\Gamma ^2,\Omega)$ such that $\widehat{I}(\ssigma)<\infty$. 
 Let $(K_n)_{n\geq1}$ be the sequence as defined in \eqref{eq:defKn}. We have
\[
\lim_{n\rightarrow\infty} \widehat{I}(\ssigma * K_n)=\widehat{I}(\ssigma)
\]
where $*$ denotes the convolution operator.
\end{prop}

\begin{proof}
Let $n\geq 1$.
Write $\ssigma_n=\ssigma * K_n$.
By classical properties (see for instance theorem 4.1. in \cite{EVGA}),
we have
\begin{align}\label{eq:limsigmap}
\forall p\geq 1\qquad\lim_{n\rightarrow\infty} \ssigma_n=\ssigma\qquad \text{in $L^p$}
\end{align}
and
\[\lim_{n\rightarrow\infty} \ssigma_n(x)=\ssigma(x)\qquad \text{for $\cL^d$-almost every $x$}\,.\]
Using Fatou lemma and the fact that $I$ is lower semi-continuous on $\sR^d$ (proposition \ref{prop:lscI}), we have
\begin{align}\label{eq:liminf}
\liminf_{n\rightarrow\infty}\widehat{I}(\ssigma_n)=\liminf_{n\rightarrow\infty}\int_{\Omega}I(\ssigma_n(x))d\cL^d(x)\geq \int_{\Omega}\liminf_{n\rightarrow\infty}I(\ssigma_n(x))d\cL^d(x) \geq \int_{\Omega}I(\ssigma(x))d\cL^d(x)= \widehat{I}(\ssigma)\,.
\end{align}
Besides, using the fact that $I$ is convex (theorem \ref{thm:convexity}) and Jensen's inequality, we have
\begin{align*}
\widehat{I}(\ssigma_n)&=\int_{\Omega}I\left(\int_{\sR^d}\ssigma(x-y)K_n(y)d\cL^d(y)\right)d\cL^d(x)\leq\int_{\Omega}\int_{\sR^d}I(\ssigma(x-y))K_n(y)d\cL^d(y)d\cL^d(x)\,.
\end{align*}
Hence, it yields
\begin{align}\label{eqcontlem22}
\widehat{I}(\ssigma_n)-\widehat{I}(\ssigma)&\leq\int_{\Omega}(I(\ssigma)\star K_n)(x)-I(\ssigma(x))d\cL^d(x)\nonumber\\
&\leq\int_{\Omega}|(I(\ssigma)*K_n)(x)-I(\ssigma(x))|d\cL^d(x)=\|I(\ssigma)* K_n-I(\ssigma)\|_{L^1}
\,.
\end{align}
Since $x\rightarrow I(\ssigma(x))\in L^1(\sR^d\rightarrow\sR,\cL^d)$, it follows that 
$$\lim_{n\rightarrow \infty}\|I(\ssigma)* K_n-I(\ssigma)\|_{L^1}=0$$ and by inequality \eqref{eqcontlem22}
\begin{align}\label{eq:limsup}
\limsup_{n\rightarrow \infty}\widehat{I}(\ssigma_n)\leq\widehat{I}(\ssigma)\,.
\end{align}
The result follows by combining inequalities \eqref{eq:liminf} and \eqref{eq:limsup}.
\end{proof}

\subsection{Proof of theorem \ref{thm:ULDpatate}}

We prove theorem \ref{thm:ULDpatate} in two steps. These two steps correspond to the two following propositions. 
\begin{prop}\label{prop:step1uld}Let $\ssigma\in\Sigma(\Gamma^1,\Gamma ^2,\Omega)$. We have
\[\lim_{\ep\rightarrow 0} \limsup_{n\rightarrow \infty}\frac{1}{n^d}\log \Prb\left(\exists f_n\in\cS_n(\Gamma^1,\Gamma^2,\Omega) : \,\dis\big(\amu_n(f_n),\ssigma\cL^d\big)\leq \ep\right)\leq-\int_{\Omega}I(\ssigma(x))d\cL ^d(x)\,.\]
\end{prop}
\begin{prop}\label{prop:step2uld}
Let $\ssigma\in\Sigma(\Gamma^1,\Gamma ^2,\Omega)\cap \Sigma^M(\Gamma^1,\Gamma ^2,\Omega)$. We have
\[\lim_{\ep\rightarrow 0} \liminf_{n\rightarrow \infty}\frac{1}{n^d}\log \Prb\left(\exists f_n\in\cS_n(\Gamma^1,\Gamma^2,\Omega) : \,\dis\big(\amu_n(f_n),\ssigma\cL^d\big)\leq \ep\right)\geq-\int_{\Omega}I(\ssigma(x))d\cL ^d(x)\,.\]
\end{prop}
The result of theorem \ref{thm:ULDpatate} follows immediately from propositions \ref{prop:step1uld} and \ref{prop:step2uld}. To prove proposition \ref{prop:step1uld}, on the event that there exists a stream $f_n\in\cS_n(\Gamma ^1,\Gamma^2,\Omega)$ such that $\dis(\amu_n(f_n),\ssigma\cL^d)\leq \ep$, we pick such a stream $f_n$. We can divide $\Omega$ into a collection of small cubes $(B_i)_{i\in J}$. Thanks to the choice of the distance, the restriction of $f_n$ to these cubes is close to the restriction of $\ssigma$ to these cubes, \textit{i.e.}, the quantity $\dis(\amu_n(f_n)\ind_{B_i},\ssigma\ind_{B_i}\cL^d)$ is small. By independence, we will be able to upper-bound the probability that there exists a stream $f_n\in\cS_n(\Gamma ^1,\Gamma^2,\Omega)$ such that $\dis(\amu_n(f_n),\ssigma\cL^d)\leq \ep$ by a product of probabilities of more elementary events related to the elementary rate function $I$ we have defined in theorem \ref{thmbrique}.
To prove proposition \ref{prop:step2uld}, we do the reverse. Namely, starting with a collection of elementary events,  we try to reconstruct the event that there exists a stream $f_n\in\cS_n(\Gamma ^1,\Gamma^2,\Omega)$ such that $\dis(\amu_n(f_n),\ssigma\cL^d)\leq \ep$. The proof of this proposition is much more difficult and technical than the proof of proposition \ref{prop:step1uld}.
\subsubsection{Proof of proposition \ref{prop:step1uld}}
To prove proposition \ref{prop:step1uld}, we will need the following lemma that enables to compare the probability of an event in a dilation of $\fC$ with the rescaled version of this event in $\fC$.
\begin{lem}[Scaling and Translation 2]\label{lem:scaling2} Let $\ssigma\in\Sigma(\Gamma^1,\Gamma ^2,\Omega)$. Let $n\geq 1$ and $x\in \sZ_n^d$. Let $n_0\leq n$. The application $\pi_{x,n_0/n}$ defines a bijection from $\E_{n_0}^d$ to $\E_{n}^d$ (we recall that $\pi_{x,n_0/n}$ was defined in \eqref{eq:defpixalpha}) and
\begin{align*}
\Prb&\left( \exists f_n\in\cS_n(\pi_{x,n_0/n}(\fC)) : \,\dis\big(\amu_n(f_n),\ssigma\ind_{\pi_{x,n_0/n}(\fC)}\cL^d\big)\leq \frac{1}{2}\frac{n_0^{d+1}}{n^{d+1}}\ep\right)\\
&\hspace{2cm}\leq \Prb\left(\exists f_{n_0}\in\cS_{n_0}(\fC) : \,\dis\big(\amu_{n_0}(f_{n_0}),\ssigma\circ\pi_{x,n_0/n} \ind_{\fC}\cL^d\big)\leq \ep\right)\,.
\end{align*}
\end{lem}
We postpone the proof of lemma \ref{lem:scaling2} after the proof of proposition \ref{prop:step1uld}.
\begin{proof}[Proof of proposition \ref{prop:step1uld}]Let $\ssigma\in\Sigma(\Gamma^1,\Gamma ^2,\Omega)$. Write $\nu=\ssigma\cL^d$. Let $\ep\in[0,1]$.

 On the event $\{\exists f_n\in\cS_n(\Gamma^1,\Gamma^2,\Omega) : \,\dis\big(\amu_n(f_n),\nu\big)\leq \ep\}$, we pick $f_n\in\cS_n(\Gamma^1,\Gamma^2,\Omega)$ such that $\dis\big(\amu_n(f_n),\nu\big)\leq \ep$. If there are several choices, we pick one according to a deterministic rule.

\noindent{\bf Step 1. Dividing $\Omega$ into a collection of small cubes.}
Let $i=i(\ep)$ be the integer such that 
$$2^i \leq \ep ^{-\frac{1}{d+5}}<2^{i+1}\,.$$ 
We set $$\lambda_i=\frac{\lceil n2^{-i}\rceil}{n2^{-i}}\,.$$
We have $\lambda_i n 2^{-i}\in\sN$ and for $n$ large enough $\lambda_i\in [1,2]$.
Let $z\in\sZ^d$ and $B=2^{-i}\lambda_i(\fC+ z)$. Using lemma \ref{lem:propdis3} with $\delta=\lambda_i2^{-i}$ and $\rho=\ep_\fC\delta ^3$, we have
\[\dis(\amu_n \ind_B,\ssigma\ind_B\cL^d)\leq \beta_1 \frac{\ep}{\rho}+\beta_2\rho\delta ^{d-1}=\beta_1\frac{\ep}{\ep_\fC\delta^3}+\beta_ 2\ep_\fC\delta^{d+2}\leq \frac{\beta_ 1}{\ep_\fC}\delta^{d+2}+\beta_ 2\ep_\fC\delta^{d+2}\,.\]
There exists a constant $\beta_0$ depending on $d$, $\beta_1$ and $\beta_2$ such that
$$\dis(\amu_n \ind_B,\ssigma\ind_B\cL^d)\leq \beta_0\delta^{d+2}\,.$$
Using the independence, we have
\begin{align*}
&\Prb\left(\exists f_n\in\cS_n(\Gamma^1,\Gamma^2,\Omega) : \,\dis\big(\amu_n(f_n),\ssigma\cL^d\big)\leq \ep\right)\\&\hspace{4cm}\leq \prod_{B\in\Delta^i_{\lambda_i}: B \subset \Omega}\Prb\left(\exists f_n\in\cS_n(B) : \,\dis\big(\amu_n(f_n)\ind_B,\ssigma\ind_B\cL^d\big)\leq  \beta_0\delta^{d+2}\right)\,.
\end{align*}
 Let $\eta_0>0$. We consider a small enough $\ep$ such that $2d\lambda_i2^{-i}\leq \eta_0$.
Hence, it yields
\begin{align}\label{eq:ind}
-&\frac{1}{n^d}\log\Prb\left(\exists f_n\in\cS_n(\Gamma^1,\Gamma^2,\Omega) : \,\dis\big(\amu_n(f_n),\ssigma\cL^d\big)\leq \ep\right)\nonumber\\
&\geq \sum_{B\in\Delta^i_ {\lambda_i}: B\subset \Omega}-\frac{1}{n^d}\log\Prb\left(\exists f_n\in\cS_n(B) : \,\dis\big(\amu_n(f_n)\ind_B,\ssigma\ind_B\cL^d\big)\leq  \beta_0 \delta^{d+2}\right)\nonumber\\
&\geq \int_{\Omega\setminus \cV_2(\partial \Omega,\eta_0)}-\frac{1}{n^d\cL^d(B_i(x))}\log\Prb\left(\exists f_n\in\cS_n(B_i(x)) : \,\dis\big(\amu_n(f_n)\ind_{B_i(x)},\ssigma\ind_{B_i(x)}\cL^d\big)\leq  \beta_0\delta^{d+2}\right)d\cL^d(x)
\end{align}
where $B_i(x)$ corresponds to the unique $B\in\Delta^i_{\lambda_i}$ such that $x\in B$. Of course, $B_i(x)$ depends on $\ep$ and $x$.  Using Fatou lemma twice and inequality \eqref{eq:ind}, we obtain 
\begin{align}\label{eq:fatou1}
&\liminf_{\ep\rightarrow 0}\liminf_{n\rightarrow\infty}-\frac{1}{n^d}\log\Prb\left(\exists f_n\in\cS_n(\Gamma^1,\Gamma^2,\Omega) : \,\dis\big(\amu_n(f_n),\ssigma\cL^d\big)\leq \ep\right)\nonumber\\
&\quad\geq \int_{\Omega\setminus \cV_2(\partial \Omega,\eta_0)}\liminf_{\ep\rightarrow 0}\liminf_{n\rightarrow\infty}-\frac{1}{n^d\cL^d(B_i(x))}\log\Prb\left(\begin{array}{c}\exists f_n\in\cS_n(B_i(x)) :\\ \,\dis\big(\amu_n(f_n)\ind_{B_i(x)},\ssigma\ind_{B_i(x)}\cL^d\big)\leq  \beta_0 \delta^{d+2}\end{array}\right)d\cL^d(x)\,.
\end{align}

\noindent{\bf Step 2.}
We now prove that for $\cL^d$-almost all $x\in\Omega$, we have
\begin{align*}
\liminf_{\ep\rightarrow 0}\liminf_{n\rightarrow\infty}-\frac{1}{n^d\cL^d(B_i(x))}\log\Prb\left(\exists f_n\in\cS_n(B_i(x)) : \,\dis\big(\amu_n(f_n)\ind_{B_i(x)},\ssigma\ind_{B_i(x)}\cL^d\big)\leq \beta_0 \delta^{d+2}\right)\\\hfill\geq I(\ssigma(x))\,.
\end{align*}
Let $x\in\Omega\setminus \cV_2(\partial \Omega,\eta_0)$. There exists a unique $w\in\sZ^d$ such that $B_i(x)=\lambda_i2^{-i}(\fC+w)$. Since $\lambda_i2^{-i}n \in\sN$, it follows that $\lambda_i2^{-i}w\in\sZ_n^d$, we will write $c(x)$ for $\lambda_i2^{-i}w$. Let $n_0=\lambda_i2^{-i}n$ and so $\delta=n_0/n$. We recall that $\delta \leq 4\ep ^{1/(d+5)}$. By lemma \ref{lem:scaling2}, we have
\begin{align}\label{eqnwn}
\Prb&\left(\exists f_n\in\cS_n(B_i(x)) : \,\dis\big(\amu_n(f_n)\ind_{B_i(x)},\ssigma\ind_{B_i(x)}\cL^d\big)\leq \beta_0\delta^{d+2} \right)\nonumber\\
&\qquad\leq\Prb\left(\exists f_n\in\cS_n(B_i(x)) : \,\dis\big(\amu_n(f_n)\ind_{B_i(x)},\ssigma\ind_{B_i(x)}\cL^d\big)\leq 4\beta_0\ep ^{1/(d+5)}\left(\frac{n_0}{n}\right)^{d+1} \right)\nonumber\\
&\qquad\leq \Prb\left(\exists f_{n_0}\in\cS_{n_0}(\fC) : \,\dis\big(\amu_{n_0}(f_{n_0}),\ssigma\circ\pi_{c(x),\delta} \ind_{\fC}\cL^d\big)\leq 8\beta_0\ep ^{1/(d+5)}\right).
\end{align}
By a change of variable, we get
\begin{align}\label{eqsigma}
\int_{\fC}\|\ssigma\circ\pi_{c(x),\delta}(y)-\ssigma(x)\|_2d\cL^d(y)=\frac{n^d}{n_0 ^d}\int_{B_i(x)}\|\ssigma(y)-\ssigma(x)\|_2d\cL^d(y)\,.
\end{align}
 According to Definition 7.9 in \cite{Rudinrc}, the set $B_i(x)$ shrinks nicely to $x$ as $\ep$ goes to $0$.  Indeed, we have $B_i(x)\subset B(x,d\delta)$ and \[\cL^d(B_i(x))=\delta ^d= \frac{1}{\alpha_d d^d}\alpha_d d^d\delta^d=\frac{1}{\alpha_d d^d}\cL ^d(B(x,d\delta))\,\]
where $\alpha_d$ is the volume of the $d$-dimensional unit Euclidean ball. Moreover when $\ep$ goes to $0$, $\delta$ goes to $0$.
Since $\ssigma\in L^1(\sR^d\rightarrow\sR^d,\cL^d)$, using Lebesgue differentiation theorem on $\sR ^d$, there exists a subset $\widehat{\Omega}$ of $\Omega$ such that $\cL^d(\Omega\setminus \widehat{\Omega})=0$ and
\begin{align}\label{eq:lebedifthm}
\forall y\in \widehat{\Omega}\qquad\lim_{\ep \rightarrow 0}\frac{1}{\cL^d(B_i(y))}\int_{B_i(y)}\|\ssigma(w)-\ssigma(y)\|_2 d\cL ^d(w)=0\,.
\end{align}
Using equality \eqref{eqsigma}, it follows that
\begin{align*}
\|\ssigma\circ\pi_{c(x),\delta}\ind_\fC-\ssigma(x)\ind_\fC\|_{L^1}&=\int_{\fC}\|\ssigma\circ\pi_{c(x),\delta}((y)-\ssigma(x)\|_2d\cL^d(y)\\
&=\frac{n^d}{n_0 ^d}\int_{B_i(x)}\|\ssigma(y)-\ssigma(x)\|_2d\cL^d(y)\\
&=\frac{1}{\cL^d(B_i(x))}\int_{B_i(x)}\|\ssigma(y)-\ssigma(x)\|_2d\cL^d(y)\,.
\end{align*}
Using lemma \ref{lem:propdis2}, it yields that
\begin{align}\label{eq:defh_x}
\dis(\ssigma\circ\pi_{x,\delta}\ind_\fC\cL^d,\ssigma(x)\ind_\fC\cL^d)\leq 2\|\ssigma\circ\pi_{x,\delta}\ind_\fC-\ssigma(x)\ind_\fC\|_{L^1}\leq \frac{2}{\cL^d(B_i(x))}\int_{B_i(x)}\|\ssigma(y)-\ssigma(x)\|_2d\cL^d(y)\,.
\end{align}
We set 
$$\forall x\in \Omega \quad\forall \ep>0\qquad h_x(\ep)=\frac{2}{\cL^d(B_i(x))}\int_{B_i(x)}\|\ssigma(y)-\ssigma(x)\|_2d\cL^d(y)\,.$$
Thanks to equality \eqref{eq:lebedifthm}, we have
$$\forall x\in\widehat{\Omega}\qquad\lim_{\ep\rightarrow 0}h_x(\ep)=0\,.$$ 
Finally, using inequalities \eqref{eqnwn} and \eqref{eq:defh_x}, we have
\begin{align*}
\liminf_{n\rightarrow\infty}-&\frac{1}{n^d\delta ^d}\log\Prb\left(\exists f_n\in\cS_n(B_i(x)) : \,\dis\big(\amu_n(f_n)\ind_{B_i(x)},\ssigma\ind_{B_i(x)}\cL^d\big)\leq \beta_0\delta^{d+2}\right)\\
&\geq \liminf_{n\rightarrow\infty}-\frac{1}{n_0^d}\log \Prb\left(\exists f_{n_0}\in\cS_{n_0}(\fC) : \,\dis\big(\amu_{n_0}(f_{n_0}),\ssigma(x)\ind_{\fC}\cL^d\big)\leq 8\beta_0\ep ^{1/(d+5)}+h_x(\ep)\right)\\
&\geq\liminf_{n_0\rightarrow\infty}-\frac{1}{n_0^d}\log \Prb\left(\exists f_{n_0}\in\cS_{n_0}(\fC) : \,\dis\big(\amu_{n_0}(f_{n_0}),\ssigma(x)\ind_{\fC}\cL^d\big)\leq 8\beta_0\ep ^{1/(d+5)}+h_x(\ep)\right)\,.
\end{align*}
Using theorem \ref{thmbrique} and \eqref{eq:lebedifthm}, by letting $\ep$ goes to $0$, we get for any $ x\in\widehat{\Omega}$
\[\liminf_{\ep\rightarrow 0}\liminf_{n\rightarrow\infty}-\frac{1}{n^d\delta ^d}\log\Prb\left(\exists f_n\in\cS_n(B_i(x)) : \,\dis\big(\amu_n(f_n)\ind_{B_i(x)},\ssigma\ind_{B_i(x)}\cL^d\big)\leq \beta_0 \delta(\ep)^{d+2}\right)\geq I(\ssigma(x))\,.\]
Using inequality \eqref{eq:fatou1} and the previous inequality, we obtain
\[\lim_{\ep\rightarrow 0}\liminf_{n\rightarrow\infty}-\frac{1}{n^d}\log\Prb\left(\exists f_n\in\cS_n(\Gamma^1,\Gamma^2,\Omega) : \,\dis\big(\amu_n(f_n),\ssigma\cL^d\big)\leq \ep\right)\geq\int_{\Omega\setminus \cV_2(\partial \Omega,\eta_0)}I(\ssigma(x))d\cL^d(x)\,.\]
Finally, since $I(\ssigma(x))\geq 0$ for any $x\in\Omega$, using the monotone convergence theorem, we obtain by letting $\eta_0$ go to $0$:
\[\lim_{\ep\rightarrow 0}\liminf_{n\rightarrow\infty}-\frac{1}{n^d}\log\Prb\left(\exists f_n\in\cS_n(\Gamma^1,\Gamma^2,\Omega) : \,\dis\big(\amu_n(f_n),\ssigma\cL^d\big)\leq \ep\right)\geq \int_\Omega I(\ssigma(x))d\cL^d(x)\,.\] 
This yields the result.
\end{proof}
 \begin{proof}[Proof of lemma \ref{lem:scaling2}]
Let $n\geq 1$ and $x\in\sZ_n ^d$. Let $n_0\leq n$. We set $\delta=n_0/n$. Let us consider $\omega\in(\sR_+) ^{\E_n ^d}$ a configuration for which there exists $f_n\in\cS_n(\pi _{x,\delta}(\fC))$ such that 
\begin{align}\label{eq:disint}
\dis\big(\amu_n(f_n),\ssigma\ind_{\pi_{x,\delta}(\fC)}\cL^d\big)\leq \frac{1}{2}\delta^{d+1}\ep\,.
\end{align}
Let $f_n=f_n(\omega)$ be such a stream in the configuration $\omega$ (if there are several such streams, we pick one according to a deterministic rule). We set $\amu_n=\amu_n(f_n)$. We aim to prove that on the configuration $\omega\circ \pi_{x,\delta}\in (\sR_+)^{\E_{n_0}^d}$ the stream $f_n\circ \pi_{x,\delta}$ is in $\cS_{n_0}(\fC)$ and 
 \[\dis\big(\amu_{n_0}(f_n\circ\pi_{x,\delta}),\ssigma\circ\pi_{x,\delta}\ind_{\fC}\cL^d\big)\leq \ep\,.\]
Set $g_{n_0}=f_n\circ\pi_{x,\delta}$ the stream in the lattice $\E^d_{n_0}$ defined as follows
\[\forall e\in\E_{n_0}^d\qquad g_{n_0}(e)=f_n(\pi_{x,\delta}(e))\,.\]
We also set
\[
\amu_{n_0}:=\amu_{n_0}(g_{n_0})=\frac{1}{n_0^d}\sum_{e\in\E_{n_0}^d}g_{n_0}(e)\delta_{c(e)}\,.
\]
It is clear that $g_{n_0}\in\cS_{n_0}(\fC)$ on the configuration $\omega\circ\pi_{x,\delta}$.

 Let us compute the distance $\dis\big(\amu_{n_0}(g_{n_0}),\ssigma\circ\pi_{x,\delta}\ind_{\fC}\cL^d\big)$.
Let $k\geq 1$, $\lambda\in[1,2]$, $y\in[-1,1[^d$ and $w\in\sZ^d$. Let $Q\in\Delta_\lambda^k$, we have
\[
\amu_{n_0}(Q+y)=\frac{1}{n_0^d}\sum_{\substack{e\in\E_{n_0}^d :\\ c(e)\in (Q+y)}}f_n(\pi_{x,\delta}(e))=\frac{1}{n_0^d}\sum_{\substack{e\in\E_{n}^d :\\ c(e)\in \pi_{x,\delta}(Q+y)}}f_n(e)=\frac{n^d}{n_0^d}\amu_n(\pi_{x,\delta}(Q+y))\,.
\]
We have 
\[
\pi_{x,\delta}(Q+y)=\frac{n_0}{n}(Q+y)+ x=\frac{n_0}{n}Q+\frac{n_0}{n}y+x\,.
\]
Write $z=\frac{n_0}{n}y + x$.
Let $i\geq 1$ be such that
\[\frac{1}{2^{i}}<\lambda\frac{n_0}{n}\leq\frac{1}{2^{i-1}}\,.\]
Let $\lambda'\in[1,2]$ such that
\[\lambda\frac{n_0}{n}=\lambda'2^{-i}\,.\]
It yields that \[\pi_{x,\delta}(Q+y)\in\left(z+\Delta^{k+i}_{\lambda'}\right)\,.\]
Let us compute $\cL^d(B)/\cL^d(\pi_{x,\delta}(B))$ for $B=Q+y$:
\[\frac{\cL^d(B)}{\cL^d(\pi_{x,\delta}(B))}=\frac{\cL^d(B)}{\delta ^d\cL^d(B)}=\frac{n ^d}{n_0^d}\,.\]
Write $\nu=\ssigma\ind_{\pi_{x,\delta}(\fC)}\cL^d$ and $\nu\circ\pi_{x,\delta}= \ssigma\circ\pi_{x,\delta}\ind_{\fC}\cL^d$ .
We have by change of variable
\begin{align*}
\nu\circ\pi_{x,\delta}(B)&=\int_B\ssigma(\pi_{x,\delta}(y))\ind_{y\in\fC}d\cL^d(y)=\int_{B\cap \fC}\ssigma(\pi_{x,\delta}(y))d\cL^d(y)\\
&=\frac{n^d}{n_0^d}\int _{\pi_{x,\delta}(B\cap \fC)}\ssigma(y)d\cL^d(y)=\frac{n^d}{n_0^d}\nu(\pi_{x,\delta}(B))\,.
\end{align*}
Using inequality \eqref{eq:disint}, we have
\begin{align*}
\sum_{k=0}^\infty\frac{1}{2^{k}}\sum_{Q\in\Delta^k_\lambda} \|\amu_{n_0}(Q+y)-\nu\circ\pi_{x,\delta}(Q+y)\|_2&=\frac{n^d}{n_0^d}\sum_{k=0}^\infty\frac{1}{2^{k}}\sum_{Q\in\Delta^{k+i}_{\lambda'}}  \|\amu_{n}(Q+z)-\nu(Q+z)\|_2 \\
&\leq 2^{i}\frac{n^d}{n_0^d}\sum_{k=0}^\infty\frac{1}{2^{k}}\sum_{Q\in\Delta^{k}_{\lambda'}}  \|\amu_{n}(Q+z)-\nu(Q+z)\|_2 \\
&\leq  \frac{\lambda'}{\lambda}\frac{n^{d+1}}{n_0^{d+1}}\dis(\amu_{n},\ssigma\ind_{\pi_{x,\delta}(\fC)}\cL^d)\leq \ep\,.
\end{align*}
On the configuration $\omega\circ \pi_{x,\delta}$, we thus get
\[\dis\big(\amu_{n_0}(g_{n_0}),\ssigma\circ\pi_{x,\delta}\ind_{\fC}\cL^d\big)\leq  \ep\,.\]
It yields the result.
\end{proof}

\subsubsection{Proof of proposition \ref{prop:step2uld}}
Let us explain the strategy of the proof of proposition \ref{prop:step2uld}. The general idea of this proof is to build a discrete stream that is close to $\ssigma$ by reconnecting constant streams in cubes. To do so, we need to work with regular continuous streams that are close to $\ssigma$.  However, we cannot use the regularization sequence directly on $\ssigma$ since
$\ssigma\star K_p$ do not have null divergence close to the sources and the sinks. To avoid this issue, we first need to build a prolongated version $\ssigma'$ of $\ssigma$ defined in an extended version of $\Omega$ where the sources and the sinks have been pushed away. Doing so ensures that $\ssigma_p=\ssigma'\star K_p$ has null divergence almost everywhere on $\Omega$. 

Next, we till $\Omega$ into small cubes $(B_i)_{i\in J}$ centered at $(x_i)_{i\in J}$ such that $\ssigma_p\approx \ssigma_p(x_i)$ on $B_i$ for any $i\in J$. We consider the family of elementary events: in each cube $B_i$ there exists a discrete stream $f_n^{(i)}$ close to the constant continuous stream $\ssigma_p(x_i)\ind_{B_i}$. We use again corridors to reconnect these streams altogether and create a stream $f_n$, where outside the cubes and their corridors in $\Omega\setminus\cup_{i\in J}B_i$, the stream $f_n$ coincides with the discretized version of the stream $\ssigma_p$ (defined as in the proof of lemma \ref{lem:preconv}). Note that unlike the proof of theorem \ref{thmbrique}, where the node law at the macroscopic level was straightforward (because the continuous stream was constant and so the flow through the adjacent faces always match), here the stream $\ssigma_p$ is not constant and so the node law at the macroscopic level is harder to get.
A major difficulty of this proof is to build $\ssigma'$ in such away we can build a discretized version of $\ssigma_p$ that belongs in $\cS_n^M (\Gamma^1,\Gamma^2,\Omega)$.

 The aim of the following proposition is to push away the sink and source for $\ssigma$. We postpone its proof until the section \ref{sec:5.3}.
\begin{prop}[Prolongation of a continuous stream]\label{prop:prolsigma}Let $\ssigma\in\Sigma(\Gamma^1,\Gamma^2,\Omega)\cap \Sigma^M(\Gamma^1,\Gamma ^2,\Omega)$ (we recall that $\Sigma^M(\Gamma^1,\Gamma ^2,\Omega)$ was defined in \eqref{eq:defSigmaM}) such that $\widehat{I}(\ssigma)<\infty$. For any $\eta>0$, there exist $\rho=\rho(\eta)>0$ and $\widetilde {\Omega}$, $\widetilde \Gamma^1$, $\widetilde{\Gamma}^2$ and  $\ssigma'\in \Sigma(\widetilde{\Gamma}^1,\widetilde{\Gamma}^{2},\widetilde{\Omega})$ such that 
\begin{itemize}
\item $\Omega \subset \widetilde{\Omega}$,  $\widetilde \Gamma^1\cup\widetilde{\Gamma} ^2\subset \partial \widetilde \Omega$,  $d_2( \widetilde \Gamma^1\cup \widetilde \Gamma^2,\Gamma)=\rho$
and  $(\widetilde \Omega\setminus \Omega )\cap \cV_2(\Gamma\setminus (\Gamma^1\cup\Gamma^2),\rho/2)=\emptyset$ (see figure \ref{extomega})
\item $\|\ssigma-\ssigma'\ind_{\Omega}\|_{L^1}\leq \eta$ 
 and $\widehat{I}(\ssigma'\ind_\Omega)\leq\widehat{I}(\ssigma)+ \eta\,.$
\end{itemize}
\end{prop}
\begin{figure}[H]
\begin{center}
\def\svgwidth{0.9\textwidth}
   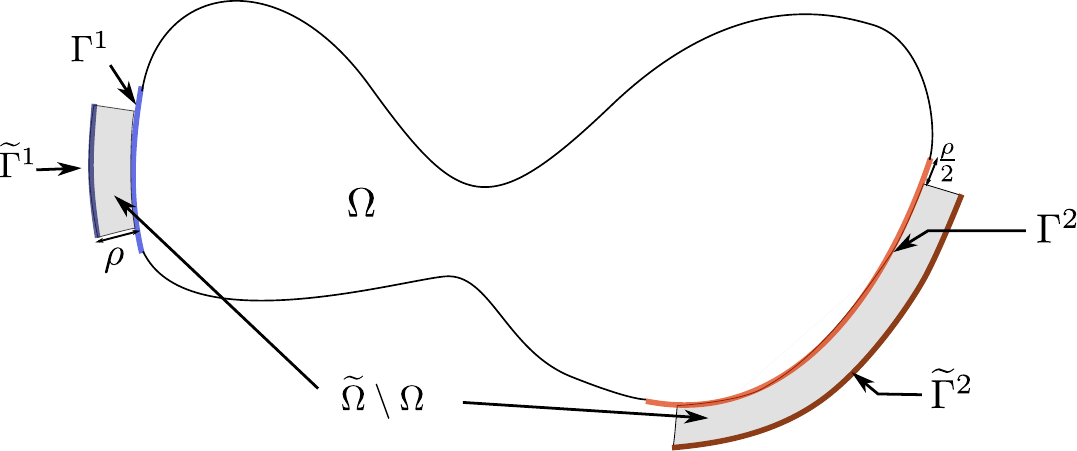
   \caption{\label{extomega}Example of possible $(\widetilde \Omega, \widetilde \Gamma^1,\widetilde \Gamma^2)$}
   \end{center}
\end{figure}
\begin{proof}[Proof of proposition \ref{prop:step2uld}]
Let $\ssigma\in\Sigma(\Gamma^1,\Gamma ^2,\Omega)\cap \Sigma^M(\Gamma^1,\Gamma ^2,\Omega)$. If $\widehat{I}(\ssigma)=\infty$, the result is straightforward. Let us assume that $\widehat{I}(\ssigma)<\infty$. 

Let $\eta>0$. Let $\rho$,  $\widetilde {\Omega}$, $\widetilde \Gamma^1$, $\widetilde{\Gamma}^2$ and $\ssigma'$ be as in the statement of proposition \ref{prop:prolsigma}.

\noindent {\bf Step 1. Approximation of $\ssigma'$ by a regular function.} 
Let $(K_p)_{p\geq 1}$ be the sequence of mollifiers (defined in \eqref{eq:defKn}).
Write $\ssigma_p=\ssigma' * K_p$. We have $\ssigma_p \in\sC^\infty (\sR^d,\sR^d)$. Let $p\geq 4/\rho$. We claim that
\[\forall x\in \cV_2(\Omega,\rho/2)\qquad\diver\ssigma_p(x)=0\,.\]
Since $\ssigma'\in \Sigma(\widetilde{\Gamma}^1,\widetilde{\Gamma}^{2},\widetilde{\Omega})$, we have (see remark 13 in \cite{CT1}) 
$$\diver\ssigma'= -(\ssigma'\cdot\nn_{\widetilde \Omega})\cH^{d-1}|_{\partial \widetilde\Omega}=-(\ssigma'\cdot\nn_{\widetilde \Omega})\cH^{d-1}|_{\widetilde\Gamma^1\cup \widetilde \Gamma ^2}\,.$$
Let $h\in\sC^\infty_c (\sR^d,\sR)$. We get
\begin{align*}
\int_{\sR^d}h\diver\ssigma_pd\cL^d&=-\int_{\sR^d}\ssigma_p\cdot \overrightarrow{\nabla}hd\cL^d\\
&=-\int_{\sR^d}\int_{\sR^d}K_p(y)\ssigma'(x-y)\cdot \overrightarrow{\nabla}h(x)d\cL^d(y)d\cL^d(x)\\
&=-\int_{\sR^d}\int_{\sR^d}K_p(y)\ssigma'(z)\cdot \overrightarrow{\nabla}h(z+y)d\cL^d(y)d\cL^d(z)\,.
\end{align*}
Since $h\in\sC^\infty_c (\sR^d,\sR)$, we have for $\cL^d$-almost every $y\in\sR^d$
$$\int_{\sR^d}K_p(y) \overrightarrow{\nabla}h(z+y)d\cL^d(y)=\overrightarrow{\nabla}\int_{\sR^d}K_p(y) h(z+y)d\cL^d(y)=\overrightarrow{\nabla}(K_p*h)(z)\,.$$
It follows that 
\begin{align*}
\int_{\sR^d}h\diver\ssigma_pd\cL^d&=\int_{\sR^d}K_p* h \diver \ssigma'd\cL^d=-\int_{\widetilde\Gamma^1\cup \widetilde \Gamma ^2}(\ssigma'\cdot\nn_{\widetilde \Omega}) K_p* h\, d\cH^{d-1}\,.
\end{align*}
Hence, if $h$ has its support included in $\cV_2(\Omega,\rho/2)$, the function $K_p* h$ has its support included in $\cV_2(\Omega,3\rho/4)$. For any $x\in \widetilde\Gamma^1\cup \widetilde \Gamma ^2$, we have $ K_p* h(x)=0$ and
 $$\int_{\sR^d}h\diver\ssigma_pd\cL^d=0\,.$$
It follows that $\diver \ssigma_p = 0$ on $\cV_2(\Omega,\rho/2)$.

Moreover, since  for $i=1,\dots,d$ we have $|\ssigma'(x)\cdot \overrightarrow{e_i}|\leq M$ for $\cL^d$-almost every $x$ in $\cV_2(\Omega,\rho)$, we have $|\ssigma_p(x)\cdot\overrightarrow{e_i}|\leq  M$ for $\cL^d$-almost every $x$ in $\cV_2(\Omega,\rho/2)$. For $p$ large enough, by proposition \ref{prop:continuityIhat}, we have $\widehat{I}(\ssigma_p)<\infty$.
 Let $n\geq 1$. Let $\ep>0$. The function $\ssigma_p $ is uniformly continuous on $\Omega$, that is there exists $\delta=\delta(\ep)>0$ such that 
\begin{align}\label{eq:contsigmap}
\forall x,y\in\Omega \qquad \|x-y\|_2\leq \delta \implies \|\ssigma_p(x)-\ssigma_p(y)\|_2\leq \ep\,.
\end{align}
In what follows, $m=\lfloor \ep ^{-\alpha}\rfloor $ where $\alpha=(2(3d+1))^{-1}$ was defined in \eqref{eq:defalpha}, $\kappa=\kappa(m,n,\delta)$ must satisfy $\kappa\leq\delta(\ep)/(2d)$, we set 
\[\kappa=\frac{2m}{n}\left\lfloor\frac{n\delta}{ 4dm}\right\rfloor\,.\]
We have $n\kappa\in\sN$ and $n\kappa(1+2d/m)/2\in\sN$. 
Besides, we get by definition of $\kappa$
\[\lim_{n\rightarrow \infty}\kappa(n)=\frac{\delta}{2d}\,.\]
 We divide $\Omega$ into small cubes of side-length $\kappa(1+2d/m)$. 
Write $\fM_\kappa$ the set of the centers of the cubes of side-length $\kappa(1+2d/m)$ included in $\Omega$, that is
\[\fM_\kappa=\left\{x\in\kappa \left(1+\frac{2d}{m}\right)\sZ^d\,:\,\pi_{x,\kappa(1+2d/m)}(\fC)\subset \Omega\setminus \cV_\infty(\partial \Omega,d\kappa) \right\}\,.\]
We define $\partial^{int}\fM_\kappa$ as the centers of the cubes in $\fM_\kappa $ that are in the boundary of $\fM_\kappa$, \textit{i.e.},
\[\partial^{int}\fM_\kappa=\left\{x\in\fM_\kappa : \,\exists y\in\kappa\left(1+\frac{2d}{m}\right)\sZ^d,\quad \|y-x\|_\infty =\kappa\left(1+\frac{2d}{m}\right), \,y\notin\fM_\kappa\right\}\,.\]
We denote by $\ssigma_p ^\kappa$ the approximation of $\ssigma_p$ at scale $\kappa$ defined as follow
\[\ssigma_p ^\kappa=\sum_{x\in\fM_\kappa}\ssigma_p(x)\ind_{\pi_{x,\kappa(1+2d/m)}(\fC)}\,.\]
Note that 
\begin{align}\label{eq:inclusionbordomega}
\Omega\setminus \cV_2 (\partial \Omega,d(d+2)\kappa)\subset \Omega\setminus \cV_\infty (\partial \Omega,(d+2)\kappa)\subset \bigcup_{x\in\fM_{\kappa}}\pi_{x,\kappa(1+2d/m)}(\fC)\,.\end{align}
Thanks to \eqref{eq:contsigmap} and proposition \ref{prop:minkowski}, we have for $\kappa$ small enough depending on $\Omega$
\begin{align}\label{eq:sigmad}
\|\ssigma_p^\kappa-\ssigma_p\ind_\Omega\|_{L^1}&\leq \ep\cL^d(\Omega)+2dM\cL^d\left(\cV_2(\partial \Omega,d(d+2)\kappa)\right)\leq  \ep\cL^d(\Omega)+10d^3M\cH^{d-1}(\Gamma)\kappa\,.
\end{align}
\noindent {\bf Step 2. Prove that $\ssigma_p\cdot \nn_{\widetilde{\Omega}}=0$ $\cH^{d-1}$- almost everywhere on $\partial \widetilde \Omega \cap\cV_2(\widetilde \Gamma^1\cup\widetilde\Gamma^2,\rho/2)^c$.}
 Let $ u\in\sC_c ^\infty (\sR ^d,\sR )$, by inequality \eqref{eq:caracterisationsigman}, we have
\begin{align*}
\int_{\partial \widetilde \Omega} (\ssigma_p\cdot\nn_{\widetilde{\Omega}})u\,d\cH ^{d-1}&=\int_{\sR ^d} \ssigma_p\cdot \overrightarrow{\nabla}ud\cL ^d\\
&=\int_{\sR ^d}\int_{\sR^d}\ssigma'(x-y)\cdot \overrightarrow{\nabla}u(x)\,K_p(y)d\cL^d(y)d\cL^d(x)\,.
\end{align*}
Since $u\in\sC_c ^\infty (\sR ^d,\sR )$, there exist a bounded subset $F_u$ of $\sR^d$ and a constant $C_u>0$ such that
\[\forall x\in\sR^d\qquad \| \overrightarrow{\nabla}u(x)\|_2\leq C_u\ind_{x\in F_u}\,.\]
We have
\begin{align*}
\int_{\sR ^d}\int_{\sR^d}|\ssigma'(x-y)\cdot \overrightarrow{\nabla}u(x)K_p(y)|d\cL^d(y)d\cL^d(x)&\leq \int_{\sR ^d}\int_{\sR^d}2dMC_uK_p(y)\ind_{x\in F_u}d\cL^d(y)d\cL^d(x)\\&=2dMC_u\cL^d(F_u)<\infty\,.
\end{align*}
Hence, we can apply Fubini Tonelli theorem
\begin{align*}
\int_{\partial\widetilde \Omega} (\ssigma_p\cdot\nn_{\widetilde{\Omega}})u\,d\cH ^{d-1}
&=\int_{\sR ^d}\int_{\sR^d}\ssigma'(x-y)\cdot \overrightarrow{\nabla}u(x)K_p(y)d\cL^d(x)d\cL^d(y)\\
&=\int_{\sR^d}K_p(y)\int_{\sR^d}\ssigma'(z)\cdot \overrightarrow{\nabla}u(z+y)d\cL^d(z)d\cL^d(y)\\
&=\int_{\sR^d}K_p(y)\int_{\partial \widetilde{\Omega}}(\ssigma'\cdot\overrightarrow{n}_{\widetilde \Omega})(x) u(x+y)d\cH^{d-1}(x)d\cL^d(y)
\end{align*}
Since $|(\ssigma'\cdot\overrightarrow{n}_{\widetilde\Omega})(x)|\leq 2dM$, we can again apply Fubini Tonelli theorem:
\begin{align}\label{eq:liensigmasigmap}
\int_{\partial \widetilde \Omega}(\ssigma_p\cdot\nn_{\widetilde{\Omega}})u\,d\cH ^{d-1}
&=\int_{\partial \widetilde \Omega}(\ssigma'\cdot\nn_{\widetilde{\Omega}})(x)\left(\int_{\sR^d} u(x+y)K_p(y)d\cL^d(y)\right)d\cH^{d-1}(x)\nonumber\\
&= \int_{\widetilde \Gamma ^1\cup\widetilde \Gamma ^2}(\ssigma'\cdot\nn_{\widetilde{\Omega}})(x)(K_p*u)(x)d\cH^{d-1}(x)\,.
\end{align}
We recall that $1/p\leq \rho/4$. Let $u\in\sC^\infty_c(\cV_2(\widetilde \Gamma^1\cup\widetilde\Gamma^2,\rho/2)^c,\sR)$. Then, $K_p* u$ has its support included in $\cV_2(\widetilde \Gamma^1\cup\widetilde\Gamma^2,\rho/4)^c$ and 
$$\int_{\partial\widetilde \Omega}  (\ssigma_p\cdot\nn_{\widetilde{\Omega}})u\,d\cH ^{d-1}=0\,.$$
It follows that
\begin{align}\label{eq:nodelawsigmadis}
\ssigma_p\cdot\nn_{\widetilde{\Omega}}=0\text{  $\cH^{d-1}$-almost everywhere on $\partial \widetilde \Omega \cap\cV_2(\widetilde \Gamma^1\cup\widetilde\Gamma^2,\rho/2)^c$}
\end{align}

\noindent {\bf Step 3. Construction of a stream in $\cup_{x\in\fM_\kappa}\pi_{x,\kappa(1+2d/m)}(\fC)$ close to $\ssigma$.} 
Set $$z_0=\left(-\frac{1}{2n},\dots,-\frac{1}{2n}\right)\,.$$
We denote by $\cE_\kappa(x)$ the following event:
\[\cE_\kappa(x)=\left\{ \exists f_n\in\cS_n(\pi_{x,\kappa}(\fC)):\begin{array}{c}\dis(\amu_n(f_n),\ssigma_p(x)\ind_ {\pi_{x,\kappa}(\fC)})\leq 12  \ep^{\alpha_0} \kappa ^d,\\\forall \diamond\in\{+,-\}\,\forall i\in\{1,\dots,d\}\,\forall A\in\cP_i^\diamond(m)\\\psi_i^\diamond(f_n,\pi_{x,\kappa}(A))=(1-\ep ^{\alpha/4})\left(\int_{\pi_{x+z_0,2\upsilon}(\fC_i ^\diamond)}\ssigma_p (y)\cdot \overrightarrow{e_i}d\cH ^{d-1}(y)\right)\,\frac{n ^{d-1} }{m^{d-1}}\end{array}\right\}\,\]
we recall that $\cP_i^-(m)$ and $\cP_i^+(m)$ were defined in \eqref{eq:defPi-} and \eqref{eq:defPi+} and $\alpha_0$ in \eqref{eq:defalpha0}. The choice of $z_0$ is to compensate the shift due to integrating over the plaquettes. This choice will be clear in the next step.
Let $f_n^x$ be a stream that satisfies the conditions of the event $\cE_\kappa (x)$ (if there are several possible choices, we choose according to a deterministic rule). Let $x\in\fM_\kappa$, we have on the event $\cE_\kappa(x)$ that for any $i\in\{1,\dots,d\}$ and $\diamond\in\{+,-\}$:
$$\psi_i^\diamond(f_n^x,\pi_{x,\kappa}(\fC_i^\diamond))=\sum_{A\in\cP_i ^\diamond(m)}\psi_i^\diamond(f_n^x,\pi_{x,\kappa}(A))=(1-\ep^{\alpha/4})\left(\int_{\pi_{x+z_0,2\upsilon}(\fC_i^\diamond)}\ssigma_p (y)\cdot \overrightarrow{e_i}d\cH ^{d-1}(y)\right) \,n ^{d-1}\,. $$ 
We define the corridor $\Cor$ as follows:
\[\Cor=\Omega\setminus\bigcup_{x\in\fM_\kappa}\pi_{x,\kappa}(\fC)\,.\]
We proceed similarly as in the proof of lemma \ref{lem:preconv}.
Let $\diamond\in\{+,-\}$, $x\in\fM_\kappa$, $i\in\{1,\dots,d\}$ such that $x+\kappa(1+d/m)\overrightarrow{e_i}\in\fM_\kappa$. For $\forall A\in\cP_i^\diamond(m)$, by lemma \ref{lem:mixing}, there exists a stream $\overline{f}_n^{x,A}$ in $\cyl(\pi_{x,\kappa}(A),\kappa d/m,\diamond\overrightarrow{e_i})$ such that 
\[\forall e\in \E_n^{i,\diamond}[\pi_{x,\kappa}(A)]\qquad \overline{f}_n^{x,A}\left(e\diamond\frac{\overrightarrow{e_i}}{n}\right)=f_n^x(e)\]
and	\[\forall e\in \E_n^{i,\diamond}\left[\pi_{x,\kappa} \left(A\diamond \frac{ d}{m}\overrightarrow{e_i}\right)\right]\qquad \overline{f}_n^{x,A}(e)=\frac{\psi_i^\diamond(f_n^x,\pi_{x,\kappa}( A))}{|\E_n^{i,\diamond}[\pi_{x,\kappa} (A)] |}\,.\]
If $x+\kappa(1+d/m)\overrightarrow{e_i}\notin\fM_\kappa$, by lemma \ref{lem:mixing}, there exists a stream $\overline{f}_n^{x,A}$ in $\cyl(\pi_{x,\kappa}(A),\kappa d/m-1/n,\overrightarrow{e_i})$ such that 
\[\forall e\in \E_n^{i,+}[\pi_{x,\kappa}(A)]\qquad \overline{f}_n^{x,A}\left(e+\frac{\overrightarrow{e_i}}{n}\right)=f_n^x(e)\]
and	\[\forall e\in \E_n^{i,+}\left[\pi_{x,\kappa} \left(A+\frac{d}{m}\overrightarrow{e_i}\right)-\frac{1}{n}\overrightarrow{e_i}\right]\qquad \overline{f}_n^{x,A}(e)=\frac{\psi_i^\diamond(f_n^x,\pi_{x,\kappa}( A))}{|\E_n^{+,\diamond}[\pi_{x,\kappa} (A)] |}\,.\]
This stream mixes the inputs in such a way the outputs are uniform.
We build $f_n^{prel}$ in $\cup_{x\in\fM_\kappa}\pi_{x,\kappa(1+2d/m)}(\fC)$ as follows 
\begin{align}\label{eq:deffnprel}
f_n^{prel}=\sum_{x\in\fM_\kappa}\left(f_n^{x}+\sum_{A\in\cup_{i=1,\dots,d}\cP_i^+(m)\cup\cP_i^-(m)}\overline{f}_n^{x,A}\right)\,.
\end{align}
We claim that on the event $\cap_{x\in\fM_\kappa}\cE_\kappa(x)$,the stream $f_n^{prel}$ satisfies the node law everywhere inside $\cup_{x\in\fM_\kappa}\pi_{x,\kappa(1+2d/m)}(\fC)$. Indeed, for any $x,y\in\fM_\kappa$ such that $\|x-y\|_1=\kappa(1+2d/m)$, we can write without loss of generality $y-x=\kappa (1+2d/m)\overrightarrow{e_i}$. On the event $\cE_\kappa(x)\cap \cE_\kappa(y)$, we have for any $ A\in\cP_i^+ (m)$
\[\psi_i^+(f_n^{x},\pi_{x,\kappa}(A))=(1-\ep^{\alpha/4})\left(\int_{\pi_{x+z_0,\kappa(1+2d/m)}(\fC_i ^+)}\ssigma_p (y)\cdot \overrightarrow{e_i}d\cH ^{d-1}(y)\right) \,\frac{n^{d-1}}{m ^{d-1}}=\psi_i^-\left(f_n^{y},\pi_{y,\kappa}(A-\overrightarrow{e_i})\right)\,.\]
The latter equality combined with the fact that 
$|\E_n^{i,+}[\pi_{x,\kappa} (A)] |=|\E_n^{i,-}[\pi_{y,\kappa} (A-\overrightarrow{e_i})] |$ (since $\E_n^{i,+}[\pi_{x,\kappa} (A)]=\E_n^{i,-}[\pi_{y,\kappa} (A-\overrightarrow{e_i})] +2\kappa d/m \overrightarrow{e_i}$ and $2\kappa d/m\in\sZ_n$)
 ensures that the node law is satisfied along the common face of $\pi_{x,\kappa(1+2d/m)}(\fC)$ and $\pi_{y,\kappa(1+2d/m)}(\fC)$.

At this stage, we have constructed the stream inside the cubes $\pi_{x,\kappa(1+2d/m)}(\fC)$, for $x\in\fM_\kappa$. The remaining part is the most technical part of the proof. The aim is to prolongate this stream in $\Omega \setminus \cup_{x\in\fM_\kappa}\pi_{y,\kappa(1+2d/m)}(\fC)$ in such a way the node law is respected everywhere except in $\Gamma_n^1\cup\Gamma_n^2$. To do so, we are going to build the discretized version  $\ssigma_p ^{disc}$ of $\ssigma_p$. Note that $\ssigma_p$ have been built in such a way  that its discretized version is in $\cS_n^M(\Gamma^1,\Gamma^2,\Omega)$. We have the stream $f_n^{prel}$ in the cubes and the stream $\ssigma_p ^{disc}$ outside the cubes. At this point the node law is not respected inside $\Omega$ along the common faces. For these common faces, the stream $f_n^{prel}$ has been built in such a way that its flow match with the flow of $\ssigma_p ^{disc}$. However, the inputs and the outputs do not perfectly match. We need to correct this difference by doing a mixing, but without corridor. This is the most technical part of the proof. 

\noindent {\bf Step 4. Construction of a discrete stream.}
We recall that $C$ is the cube of side-length $1/n$ centered at $0$, that is
\[C=\left[-\frac{1}{2n},\frac{1}{2n}\right]^d\,.\]
We consider the following stream $\ssigma_p ^{disc}$ that is the discretized version of $\ssigma_p$
defined as follows: for any $i\in\{1,\dots,d\}$, for any $e=\langle x,y\rangle \in\E_n^d$ such that $x,y\in\Omega_n$ and $\overrightarrow{xy}=\overrightarrow{e_i}/n$ ,
\[ \ssigma_p ^{disc}(e)=(1-\ep^{\alpha/4})n ^{d-1}\left(\int_{\cP(e)}\ssigma_p(u)\cdot \overrightarrow{e_i}\,\ind_{\widetilde\Omega}d\cH^{d-1}(u)\right)\,\overrightarrow{e_i}\,.\]


Let $x\in\Omega_n \setminus (\Gamma_n^1\cup\Gamma_n^2)$. We want to prove that $\ssigma_p ^{disc}$ satisfies the node law at $x$. We distinguish several cases.

\noindent{\bf Case 1.} We have $x+C\subset \Omega$.
Since $\diver\ssigma_p=0$ on $\Omega$, we obtain by applying Gauss-Green theorem to $\ssigma_p$ in $x+C$:
\[\int_{x+\partial C }\ssigma_p(u)\cdot \overrightarrow{n}_{x+C}(u)d\cH^{d-1}(u)=\sum_{\substack{y\in\sZ_n^d:\\\langle x,y\rangle\in\E_n^d}}\int_{\cP(\langle x,y\rangle)}\ssigma_p(u)\cdot (n\,\overrightarrow{xy})\,d\cH^{d-1}(u)=0\,.\]
It follows that $\ssigma_p ^{disc}$ satisfies the node-law at $x$.

\noindent{\bf Case 2.} We have $(x+C)\cap \Gamma\neq \emptyset$.
The amount of water $\di\ssigma_p ^{disc}(x)$ created at $x$ for the stream $\ssigma_p ^{disc}$ is equal to
\begin{align*}
\di\ssigma_p ^{disc}(x)&=(1-\ep ^{\alpha/4})n^{d-1}\sum_{y\in\Omega_n:\langle x, y \rangle\in\E_n^d}\int_{\cP(e)}\ssigma_p \cdot (n\,\overrightarrow{yx})\ind_{\widetilde \Omega}d\cH^{d-1}
\end{align*}
 We claim that for any $y\notin\Omega_n$ such that $\langle x,y\rangle\in\E_n^d$, we have $\cP(e)\cap \widetilde\Omega=\emptyset$. We distinguish two cases. 
\begin{itemize}
\item Let us assume $(x+C)\cap (\Gamma^1\cup\Gamma^2)\neq \emptyset$. If there exists $y\notin\Omega_n$ such that $\langle x,y\rangle\in\E_n^d$ then $x\in\Gamma_n^1\cup\Gamma_n^2$ and this is a contradiction. 
\item Let us assume that $(x+C)\cap (\Gamma^1\cup\Gamma^2)= \emptyset$. Since by construction $(\widetilde \Omega\setminus \Omega )\cap \cV_2(\Gamma\setminus (\Gamma^1\cup\Gamma^2),\rho/2)=\emptyset$, then we have $(x+C)\cap \widetilde \Omega=(x+C)\cap \Omega$. If there exists $y$ such that $e=\langle x,y\rangle\in\E_n^d$ and  $\cP(e)\cap \widetilde\Omega=\cP(e)\cap \Omega\neq \emptyset$ then $d_\infty(y,\Omega)\leq 1/2n$ and $y\in\Omega_n$.
\end{itemize}
It yields that
\begin{align*}
\di\ssigma_p ^{disc}(x)&=(1-\ep ^{\alpha/4})n^{d-1}\sum_{y\in\sZ_n^d:\langle x, y \rangle\in\E_n^d}\int_{\cP(e)}\ssigma_p \cdot (n\,\overrightarrow{yx})\ind_{\widetilde \Omega}d\cH^{d-1}\,.
\end{align*}
By applying the Gauss-Green theorem to $\ssigma_p$ in $(x+C)\cap \widetilde \Omega$, we have
\[-\sum_{y\in\sZ_n^d:\langle x, y \rangle\in\E_n^d}\int_{\cP(e)}\ssigma_p \cdot (n\,\overrightarrow{yx})\ind_{\widetilde \Omega}d\cH^{d-1}+ \int_{(\partial \widetilde \Omega\cap C)\setminus \partial C}\ssigma_p\cdot \nn_{\widetilde\Omega}d\cH^{d-1}=0\,.\]
Using equality \eqref{eq:nodelawsigmadis}, we get
\begin{align*}
\di\ssigma_p ^{disc}(x)=0\,.
\end{align*}
We conclude that $\ssigma_p ^{disc}$ satisfies the node law at $x$ for any $x\in\Omega_n\setminus (\Gamma_n^1\cup\Gamma_n^2)$.
 
%
%
%
\noindent {\bf Step 5. Correcting the stream} 
Let us now consider $x\in\partial ^{int}\fM_\kappa$. Let us denote by $E_\kappa(x)$ the set of faces of $\pi_{x,\kappa(1+2d/m)}(\fC)$ that are external, \textit{i.e.},
\[E_\kappa(x)=\left\{\kappa\left(1+\frac{2d}{m}\right)\fC_i^\diamond+x:\,x\diamond\kappa\left(1+\frac{2d}{m}\right)\overrightarrow{e_i}\notin\fM_\kappa ,\,i\in\{1,\dots,d\},\diamond\in\{-,+\}\right\}\,.\]  
For those faces, the stream $f_n^{prel}$ (defined in \eqref{eq:deffnprel}) does not perfectly coincide with the discretized version of $\ssigma_p$ but their flow match. To overcome this issue, we are going to build a stream that corrects these differences. We here want to mix, but without using a corridor. This means that we need to be particularly cautious that the stream we build does not exceed the capacity constraint. Let us first consider $F_0=x+\kappa(1+2d/m)\fC_i^-\in E_\kappa(x)\subset \Omega$. We recall that $\cP(e)$ denote the dual of the edge $e$. Set $\upsilon=\kappa(1+2d/m)/2$. Since $n\upsilon\in \sN$ and $x\in\sZ_n^d$, we have
\begin{align*}
F_0+z_0&= \left[-\upsilon-\frac{1}{2n},\upsilon-\frac{1}{2n}\right]^{i-1}\times\left\{-\upsilon-\frac{1}{2n}\right\}\times \left[-\upsilon-\frac{1}{2n},\upsilon-\frac{1}{2n}\right]^{d-i}+x\\
&=\bigcup_{x\in F_0 \cap\sZ_n^d}\cP\left(\left\langle x-\frac{1}{n}\overrightarrow{e_i},x\right\rangle\right)
=\bigcup_{e\in\E_n^{i,+}[F_0]}\cP(e)\,.
\end{align*}
It follows that
\begin{align}\label{eq:flowmatchprel}
\psi_i^-(f_n^{prel}, F_0)&=(1-\ep^{\alpha/4})\left(\int_{F_0+z_0}\ssigma_p (y)\cdot \overrightarrow{e_i}d\cH ^{d-1}(y)\right)\,n^{d-1}\nonumber\\
&=(1-\ep^{\alpha/4})\left(\int_{\underset{e\in\E_n^{i,+}[F_0]}{\cup}\cP(e)}\ssigma_p (y)\cdot \overrightarrow{e_i}d\cH ^{d-1}(y)\right)\,n^{d-1}=\psi_i^{+}(\ssigma_p^{disc},F_0)\,.
\end{align}
Let us now consider the case where $F_0=x+\kappa(1+2d/m)\fC_i^+\in E_\kappa(x)\subset \Omega$. We have
\begin{align*}
F_0+z_0&= \left[-\upsilon-\frac{1}{2n},\upsilon-\frac{1}{2n}\right]^{i-1}\times\left\{\upsilon-\frac{1}{2n}\right\}\times \left[-\upsilon-\frac{1}{2n},\upsilon-\frac{1}{2n}\right]^{d-i}\\&=\bigcup_{x\in \left(F_0-\frac{1}{n} \overrightarrow{e_i} \right)\cap\sZ_n^d}\cP\left(\left\langle x,x+\frac{1}{n}\overrightarrow{e_i}\right\rangle\right)=\bigcup_{e\in\E_n^{i,-}[F_0-\frac{1}{n}\overrightarrow{e_i}]}\cP(e)\,.
\end{align*}
It follows that
\begin{align}\label{eq:flowmatchprelbis}
\psi_i^+\left(f_n^{prel}, F_0-\frac{1}{n}\overrightarrow {e_i}\right)&=(1-\ep^{\alpha/4})\left(\int_{F_0+z_0}\ssigma_p (y)\cdot \overrightarrow{e_i}d\cH ^{d-1}(y)\right)\,n^{d-1}\nonumber\\
&=(1-\ep^{\alpha/4})\left(\int_{\underset{e\in\E_n^{i,-}[F_0-\frac{1}{n}\overrightarrow{e_i}]}{\cup}\cP(e)}\ssigma_p (y)\cdot \overrightarrow{e_i}d\cH ^{d-1}(y)\right)\,n^{d-1}\nonumber\\
&=\psi_i^{+}(\ssigma_p^{disc},F_0-\frac{1}{n}\overrightarrow {e_i})\,.
\end{align}
We refer to figure \ref{fig:defz0} for the illustration of the choice of $z_0$.
\begin{figure}[H]
\begin{center}
\def\svgwidth{0.6\textwidth}
   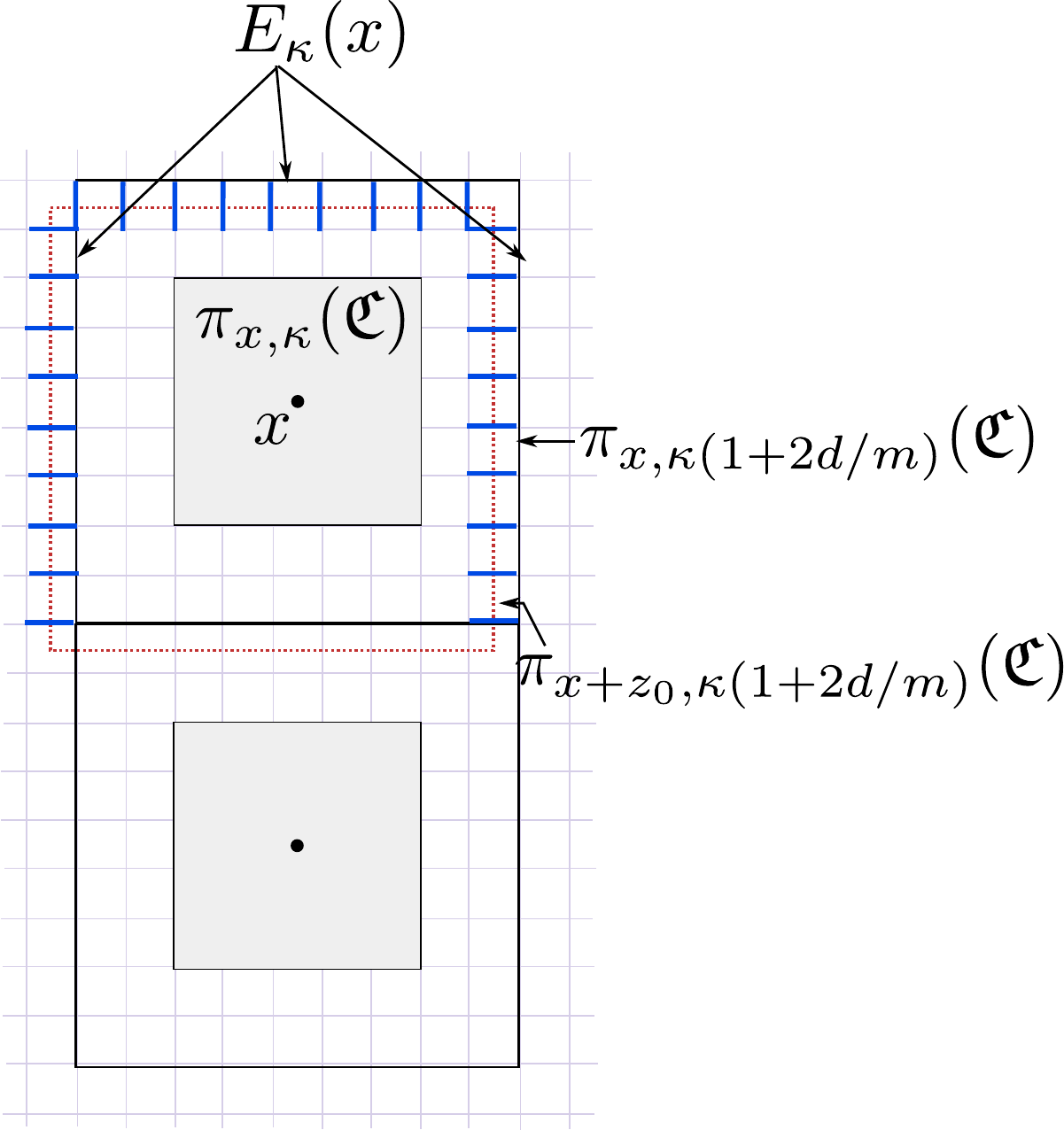
 
   \caption[figdefz0]{\label{fig:defz0} Choice of $z_0$. The edges in bold represent edges where we will affect the value given by $\ssigma_p ^{disc}$.
   }
   \end{center}
\end{figure}

$\bullet$ Let us first assume that $\ssigma_p(x)\cdot\overrightarrow{e_i}\geq 0$. Up to a translation of $-1/n \,\overrightarrow{e_i}$ the case $+$ is treated in the same way as the case $-$. To avoid cumbersome notations, we only treat the case where $\diamond=-$ but by seak of generality we do not replace $\diamond$ by $-$.
We have
$$\forall e\in\E_n^{i,\diamond}\left[\kappa\left(\fC_i^\diamond\diamond\frac{ d}{m}\overrightarrow{e_i}\right)+x\right] \qquad f_n ^{prel}(e)=\frac{1-\ep ^{\alpha/4}}{\kappa ^{d-1}} \left(\int_{F_0+z_0}\ssigma_p (y)\cdot \overrightarrow{e_i}d\cH ^{d-1}(y)\right)\,\overrightarrow{e_i}\,.$$
Besides, using \eqref{eq:contsigmap}, we have
$$\left|\frac{1}{\kappa ^{d-1}}\int_{F_0+z_0}\ssigma_p (y)\cdot \overrightarrow{e_i}d\cH ^{d-1}(y)-\left(1+\frac{2d}{m}\right) ^{d-1}\ssigma_p(x)\cdot \overrightarrow{e_i}\right|\leq \ep \left(1+\frac{2d}{m}\right) ^{d-1}\,$$
and
$$\left|n ^{d-1}\int_{\cP(e)}\ssigma_p(u)\cdot \overrightarrow{e_i}\,d\cH^{d-1}(u)-\ssigma_p(x)\cdot\overrightarrow{e_i}\right|\leq \ep\,.$$
It follows that for any $e\in\E_n^{i,\diamond}\left[\kappa\left(\fC_i^\diamond\diamond \frac{d}{m}\overrightarrow{e_i}\right)+x\right]$
\begin{align}\label{eq:controlefprel}
 (1-\ep ^{\alpha/4})\left(1+\frac{2d}{m}\right) ^{d-1}(\ssigma_p(x)\cdot \overrightarrow{e_i}-\ep) \leq f_n^{prel}(e)\cdot \overrightarrow{e_i}\leq (1-\ep ^{\alpha/4})\left(1+\frac{2d}{m}\right) ^{d-1}(\ssigma_p(x)\cdot \overrightarrow{e_i}+\ep)
\end{align}
and for any $e\in\E_n^{i,\diamond}[F_0]$
\begin{align}\label{eq:controlesigmapn}
(1-\ep ^{\alpha/4})(-\ssigma_p(x)\cdot \overrightarrow{e_i}-\ep) \leq- \ssigma_p ^{disc}(e)\cdot \overrightarrow{e_i}\leq (1-\ep ^{\alpha/4})(-\ssigma_p(x)\cdot \overrightarrow{e_i}+\ep)\,.
\end{align}
Combining the two previous inequalities for $e\in\E_n^{i,\diamond}\left[\kappa\left(\fC_i^\diamond\diamond \frac{d}{m}\overrightarrow{e_i}\right)+x\right]$, we obtain
\begin{align*}
(-\ssigma_p^{disc}(e)+f_n^{prel}(e))\cdot \overrightarrow{e_i}&\leq (1-\ep ^{\alpha/4})\left(\left (1 + \left(1+\frac{2d}{m}\right) ^{d-1}\right)\ep +\left ( \left(1+\frac{2d}{m}\right) ^{d-1}-1\right)\ssigma_p(x)\cdot \overrightarrow{e_i}\right)\\
&\leq 2^d\ep+2^{d-1}\frac{2d}{m}M\leq 2^{d+1}d M\ep^\alpha 
\end{align*}
for small enough $\ep$ depending on $d$ and $M$ where we recall that $m=\lfloor \ep^{-\alpha}\rfloor$.
Moreover, we have
\begin{align*}
(-\ssigma_p^{disc}(e)+f_n^{prel}(e))\cdot \overrightarrow{e_i}\geq (1-\ep ^{\alpha/4})\left(\left(\left(1+\frac{2d}{m}\right) ^{d-1}-1\right)\ssigma_p(x)\cdot \overrightarrow{e_i}-2^d\ep\right)\geq - 2^{d}\ep\geq -M\,
\end{align*}
for $\ep$ small enough depending on $M$.
For $e\in \E_n^{i,\diamond}[F_0]\setminus \E_n^{i,\diamond}\left[\kappa\left(\fC_i^\diamond\diamond \frac{d}{m}\overrightarrow{e_i}\right)+x\right]$, we have $f_n^{prel}(e)=0$,
\begin{align*}
(-\ssigma_p^{disc}(e)+f_n^{prel}(e))\cdot \overrightarrow{e_i}=-\ssigma_p^{disc}(e)\cdot \overrightarrow{e_i}=-(1-\ep^{\alpha/4})n ^{d-1}\int_{\cP(e)}\ssigma_p(u)\cdot \overrightarrow{e_i}\,d\cH^{d-1}(u)\geq -M\,
\end{align*}
and using inequality \eqref{eq:controlesigmapn}
$$(-\ssigma_p^{disc}(e)+f_n^{prel}(e))\cdot \overrightarrow{e_i}=-\ssigma_p^{disc}(e)\cdot \overrightarrow{e_i}\leq (1-\ep ^{\alpha/4})(-\ssigma_p (x)\cdot\overrightarrow{e_i}+\ep)\leq \ep\,.$$
We recall that we assume here that $\ssigma_p(x)\cdot \overrightarrow{e_i}\geq 0$.

We can index the edges of $\E_n^{i,\diamond}[F_0]$ by $\{1,\dots, \kappa(1+2d/m)n\}^{d-1}$.
We recall the definition of $\mathfrak{p}_i$ in \eqref{eq:defpi}. We set 
$$\forall e\in\E_n^{i,\diamond}[F_0]\qquad\zeta(e)=n\mathfrak{p}_i(c(e))+\left(\left\lfloor \frac{\kappa(1+2d/m)n}{2}\right\rfloor+1\right)\sum_{j\in\{1,\dots,d\}\setminus\{i\}}\mathfrak{p}_i(\overrightarrow{e_j})\,.$$
It is easy to check that $\zeta(e)\in\{1,\dots, \kappa(1+2d/m)n\}^{d-1}$ (we recall that $ \kappa(1+2d/m)n\in\sZ$).
 Set for any $e$ in $\E_n^{i,\diamond}[F_0]$ $$f_{in }(\zeta(e))=(-\ssigma_p^{disc}(e)+f_n^{prel}(e))\overrightarrow{e_i}\,.$$
If $e$ is such that $\zeta(e)\notin\{\kappa d n/ m+1, \kappa(1+d/m)n\}^{d-1}$, then 
$$f_{in}(\zeta(e))=-\ssigma_p ^{disc} (e)\cdot \overrightarrow{e_i}\,.$$
It follows that
$$\forall y \in \{1,\dots, \kappa(1+2d/m)n\}^{d-1}\qquad -M\leq f_{in}(y)\leq 2^{d+1}dM\ep ^\alpha\,.$$
To apply lemma \ref{lem:mixprecis}, we have to check that the sequence $(f_{in}(y),y\in\{1,\dots, \kappa(1+2d/m)n\}^{d-1})$ satisfies the conditions stated in this lemma.
First note that thanks to inequality \eqref{eq:flowmatchprel}, we have
$$\sum_{y\in\{1,\dots, \kappa(1+2d/m)n\}^{d-1}}f_{in}(y)=0\,.$$
By \eqref{eq:contsigmap}, we have 
\begin{align}\label{eq:difentrearetessigma}
\forall e_0,e_1\in\E_n^{i,\diamond}[F_0]\qquad |\ssigma_p ^{disc} (e_0)\cdot\overrightarrow{e_i}-\ssigma_p^{disc}(e_1)\cdot\overrightarrow{e_i}|\leq \ep\,.
\end{align}
We distinguish two cases.
 
$\triangleright$ We assume that $0\leq \ssigma_p(x)\cdot\overrightarrow{e_i}\leq 2^d\ep^{1-\alpha}$. In that case, for any $e_0\in\E_n ^{i,\diamond}[F_0]$, using \eqref{eq:controlefprel}, we have
$|f_n^{prel}(e_0)\cdot \overrightarrow{e_i}|\leq 2^{d-1}(2^d\ep^{1-\alpha}+\ep)$. It follows that for any $y,z\in\{1,\dots,\kappa(1+2d/m)n\}^{d-1}$, we have 
$$|f_{in}(y)-f_{in}(z)|\leq \ep + 2^{d}(2^d\ep^{1-\alpha}+\ep)\leq \ep ^{\alpha}\,$$
for $\ep$ small enough depending on $d$ where we recall that $\alpha <1/2$ (see \eqref{eq:defalpha}).

 $\triangleright$ We assume that $\ssigma_p(x)\cdot\overrightarrow{e_i}\geq 2^d\ep ^{1-\alpha}\geq 2^d\ep m$.
Let $l\in\{1,\dots,d-2\}$ and $u\in\{1,\dots, \kappa(1+2d/m)n\}^{l}$, if $u\notin \{\kappa d n/ m+1, \kappa(1+d/m)n\}^{l}$ then we have
$$\forall y\in \left\{1,\dots, \kappa\left(1+\frac{2d}{m}\right)n\right\}^{d-1-l}\qquad f_{in}(u,y)=-\ssigma_p ^{disc} (\zeta^{-1}(u,y)))\,$$ and using \eqref{eq:difentrearetessigma}, we have $$\forall y_0,y_1\in \left\{1,\dots, \kappa\left(1+\frac{2d}{m}\right)n\right\}^{d-1-l}\qquad |f_{in}(u,y_0)-f_{in}(u,y_1)|\leq \ep\,.$$
If $u\in \{\kappa d n/ m+1, \kappa(1+d/m)n\}^{l}$ using inequalities \eqref{eq:controlefprel} and \eqref{eq:controlesigmapn}
\begin{align*}
&\sum_{y\in \{1,\dots, \kappa(1+2d/m)n\}^{d-1-l}}f_{in}(u,y)\\
&\quad= -\sum_{y\in \{1,\dots, \kappa(1+2d/m)n\}^{d-1-l}}\ssigma_p^{disc}(\zeta ^{-1}(u,y))\cdot\overrightarrow{e_i} +\sum_{y\in \{\kappa dn/m,\dots, \kappa(1+d/m)n\}^{d-1-l}}f_n^{prel}(\zeta ^{-1}(u,y))\cdot\overrightarrow{e_i}\\
&\quad\geq (1-\ep ^{\alpha/4}) \left((-\ssigma_p(x)\cdot\overrightarrow{e_i}-\ep)\left(1+\frac{2d}{m}\right) ^{d-1-l}+\left(1+\frac{2d}{m}\right) ^{d-1}(\ssigma_p(x)\cdot\overrightarrow{e_i}-\ep)\right)(\kappa n) ^{d-1-l}\\
&\quad= (1-\ep ^{\alpha/4}) \left(\ssigma_p(x)\cdot\overrightarrow{e_i}\left(\left(1+\frac{2d}{m}\right) ^{l}-1\right)-\ep\left(1+\left(1+\frac{2d}{m}\right) ^{l}\right)\right)(\kappa n) ^{d-1-l}\left(1+\frac{2d}{m}\right) ^{d-1-l}\\ 
&\quad\geq (1-\ep ^{\alpha/4}) \left(\frac{2d}{m}\ssigma_p(x)\cdot \overrightarrow{e_i}-2^{d-1}\ep\right)(\kappa n) ^{d-1-l}\\
&\quad \geq (1-\ep ^{\alpha/4})(2^{d+1}d-2^{d-1})\ep\geq 0  \,.
\end{align*}
In both cases, the conditions to apply lemma \ref{lem:mixprecis} are fulfilled.
By lemma \ref{lem:mixprecis}, there exists a stream $g_n^{x,F_0}$ in $\cyl(F_0,(d-1)\kappa(1+2d/m),\diamond\overrightarrow{e_i})\subset \Omega$ such that 
\[\forall e\in\E_n^{i,\diamond}[F_0]\qquad g_n^{x,F_0}\left(e\diamond\frac{\overrightarrow{e_i}}{n}\right)=-\ssigma_p^{disc}(e)+f_n^{prel}(e)\,.\]
The stream $g_n^{x,F_0}$ satisfies the node law everywhere except for points in $\sZ_n^d\cap F_0$.
Moreover, we have for any edge $e\in\E_n^d\cap\cyl(F_0,(d-1)\kappa,\diamond\overrightarrow{e_i})$ parallel to $\overrightarrow{e_i}$:
\[g_n^{x,F_0}(e)\cdot\overrightarrow{e_i}\in[-M,2^{d+1}dM\ep ^{\alpha}]\,\]
and by \eqref{eq:contsigmap}
\[g_n^{x,F_0}(e)\cdot\overrightarrow{e_i}+\ssigma_p ^{disc}(e)\cdot \overrightarrow{e_i}\geq -M+ \ssigma_p(x)\cdot \overrightarrow{e_i}-\ep\geq -M-\ep\]
and 
\[g_n^{x,F_0}(e)\cdot\overrightarrow{e_i}+\ssigma_p ^{disc}(e)\cdot \overrightarrow{e_i}\leq M+\ep +2^{d+1}dM\ep^\alpha\,.\]
For an edge $e$ parallel to $\overrightarrow{e_j}$ with $j\neq i$:
\[\|g_n^{x,F_0}(e)\|_2\leq 2^{d+1}dM\ep ^{\alpha}\,.\]

$\bullet$ Let us assume that $\ssigma_p(x)\cdot\overrightarrow{e_i}< 0$. Hence, we have $(- \ssigma_p(x))\cdot\overrightarrow{e_i}\geq 0$.
We can apply the previous case for $-\ssigma_p ^{disc}$ and $-f_n^{prel}$. Then, we multiply by $-1$ the stream we obtained. We end up with a discrete stream $g_n^{x,F_0}$ in $\cyl(F_0,(d-1)\kappa(1+2d/m),\diamond\overrightarrow{e_i})\subset \Omega$ such that 
\[\forall e\in\E_n^{i,\diamond}[F_0]\qquad g_n^{x,F_0}\left(e\diamond\frac{\overrightarrow{e_i}}{n}\right)=-\ssigma_p^{disc}(e)+f_n^{prel}(e)\,.\]
Moreover, we have for any edge $e\in\E_n^d\cap\cyl(F_0,(d-1)\kappa,\diamond\overrightarrow{e_i})$ parallel to $\overrightarrow{e_i}$:
\[g_n^{x,F_0}(e)\cdot\overrightarrow{e_i}\in[-2^{d+1}dM\ep ^{\alpha},M]\,\]
and by \eqref{eq:contsigmap}
\[g_n^{x,F_0}(e)\cdot\overrightarrow{e_i}+\ssigma_p ^{disc}(e)\cdot \overrightarrow{e_i}\leq M+ \ssigma_p(x)\cdot \overrightarrow{e_i}+\ep\leq M+\ep\]
and 
\[g_n^{x,F_0}(e)\cdot\overrightarrow{e_i}+\ssigma_p ^{disc}(e)\cdot \overrightarrow{e_i}\geq -M-\ep -2^{d+1}dM\ep^\alpha\,.\]
For an edge $e$ parallel to $\overrightarrow{e_j}$ with $j\neq i$:
\[\|g_n^{x,F_0}(e)\|_2\leq 2^{d+1}dM\ep ^{\alpha}\,.\]

Finally, we build $f_n\in\cS_n(\Gamma^1,\Gamma^2,\Omega)$ as follows
\[\forall e=\langle w,z\rangle\in\Omega\cap\E_n^d\qquad f_n(e)=\left\{
    \begin{array}{ll}
       f_n^{prel} (e) & \mbox{if $w,z\in\cup_{Q\in\fM_\kappa}Q$} \\
       \ssigma_p^{disc}(e)+\sum_{x\in\partial ^{int}\fM_\kappa}\sum_{F\in E_\kappa(x)}g_n^{x,F}(e)& \mbox{otherwise.}
    \end{array}
\right.\]
The node law is satisfied everywhere in $\Omega_n$ for $f_n$. Note that by construction of $\fM_\kappa$ each $e\in \cyl(F_0,(d-1)\kappa(1+2d/m),\diamond\overrightarrow{e_i})$ belongs at most to $2d$ such cylinder (one for each direction): for each $j\in\{1,\dots,d\}$ there exists at most one $\circ\in\{+,-\}$ and $y\in\partial ^{int}\fM_\kappa$ such that $F_1=\pi_{y,\kappa(1+2d/m)}(\fC_i ^\circ)\in E_\kappa(y)$ and $e\in \cyl(F_1,(d-1)\kappa(1+2d/m),\circ\overrightarrow{e_i})$. Indeed, let us assume there exists $x,y\in\partial ^{int}\fM_\kappa$ with $\kappa(1+2d/m)\fC_i^+ +x \in E_\kappa(x)$ and $\kappa(1+2d/m)\fC_i^-+y \in E_\kappa(y)$ such that 
\begin{align*}
e\in \cyl\left(\kappa\left(1+\frac{2d}{m}\right)\fC_i^+ +x-\frac{1}{n}\overrightarrow{e_i},(d-1)\kappa\left(1+\frac{2d}{m}\right),\overrightarrow{e_i}\right)\\\hfill\cap\, \cyl\left(\kappa\left(1+\frac{2d}{m}\right)\fC_i^- +y,(d-1)\kappa\left(1+\frac{2d}{m}\right),-\overrightarrow{e_i}\right)\,.
\end{align*}
 It follows that $y-x=t\overrightarrow{e_i}$ with $t<\kappa (1+2d/m)+2(d-1)\kappa (1+2d/m)=2d\kappa(1+2d/m)$.
Since $x,y \in\fM_\kappa$, we have $d_\infty( \pi_{x,\kappa(1+2d/m)},\partial \Omega)\geq d\kappa$ and $d_\infty( \pi_{y,\kappa(1+2d/m)},\partial \Omega)\geq d\kappa$. It yields that
$\pi_{x,\kappa(2d+1+2d/m)}(\fC)\cup\pi_{y,\kappa(2d+1+2d/m)}(\fC)\subset \Omega $. Since $\|y-x\|_\infty \leq 2d\kappa(1+2d/m)\leq \kappa(2d+1+2d/m)$ (the last inequality holds for large enough $m$). It is easy to check that $d_\infty (\pi_{x+\kappa(1+2d/m)\overrightarrow{e_i},\kappa(1+2d/m)},\partial \Omega)\geq d\kappa$. Hence, it follows that $x+\kappa(1+2d/m)\overrightarrow{e_i} \in\fM_\kappa$, this contradicts the fact that $\kappa(1+2d/m)\fC_i^+ +x\in E_\kappa(x)$.

It follows that for any $e\in\Omega$, we have
$$\|f_n(e)\|_2\leq M+2^{d+1}d^2M\ep ^\alpha+\ep\leq M(1+\ep ^{\alpha/2})\,$$
for $\ep$ small enough depending on $d$ and $M$.

\noindent {\bf  Conclusion.}
By proposition \ref{prop:minkowski} and \eqref{eq:inclusionbordomega}, we have for $\kappa$ small enough depending on $\Omega$,
\begin{align*}
\cL^d(\Cor)&\leq \cL^d(\cV_2(\partial \Omega,2d^2\kappa))+\frac{\cL^d(\Omega)}{\kappa ^d}\kappa ^{d}\left(\left(1+\frac{2d}{m}\right)^d-1\right)\\
&\leq 8\cH ^{d-1}(\partial \Omega)d^2\kappa +\cL ^d(\Omega)2^{d+1}\frac{d}{m}\leq 8\cH ^{d-1}(\partial \Omega)d\delta(\ep) +\cL ^d(\Omega)2^{d+1}\frac{d}{m(\ep)}
\end{align*}
where we use \eqref{eq:binomenewton} in the second inequality.
We recall that $\delta(\ep)$ and $1/m(\ep)$ goes to $0$ when $\ep$ goes to $0$, we recall that $\delta(\ep)$ depends on $p$.
We set 
$$\mathrm{H}(\ep)=\inf\left\{a>0:G([M-a,M]\leq 8\cH ^{d-1}(\partial \Omega)d\delta(\ep) +\cL ^d(\Omega)2^{d+1}\frac{d}{m(\ep)}\right\}\,.$$
 We can prove as in the proof of theorem \ref{thmbrique} equality \eqref{eq:cortend0}, that $\lim_{\ep\rightarrow 0}\mathrm{H}(\ep)=0$ and for any $p\geq 1$
 \begin{align}\label{eq:corid}
\lim_{\ep\rightarrow 0}\limsup_{n\rightarrow\infty}\cL^d(\Cor)\log G([M-\mathrm{H}(\ep),M])=0\,.
\end{align}
Besides, we have
\begin{align*}
\cH^{d-1}(\partial \Cor)\leq |\fM_\kappa|2d\kappa^{d-1}+\cH^{d-1}(\partial \Omega)\leq \cL ^d(\Omega)\frac{2d}{\kappa}+\cH^{d-1}(\partial \Omega)\,.
\end{align*}
Using an inequality similar to \eqref{eq:contdis4}, it follows that for $n$ large enough (depending on $\ep$)
\begin{align}\label{eq:conccontrcardcor}
|\{e\in\E_n^d:e\in\Cor\}|\leq 3dn^d\cL^d(\Cor)\,.
\end{align}
We set $\widetilde{f}_n=(1-\ep ^{\alpha/2})(1-\mathrm{H}(\ep)/M)f_n$. Hence for any $e\in \Omega$, we have
$$\|\widetilde f _n (e)\|_2 \leq (1-\ep ^{\alpha/2})\left(1-\frac{\mathrm{H}(\ep)}{M}\right)M(1+\ep ^{\alpha/2})\leq M-\mathrm{H}(\ep)\,.$$
Hence, on the event $\cap_{x\in\fM_\kappa}\cE_\kappa(x)\cap \{\forall e\in\Cor\quad t(e)\geq M-\mathrm{H}(\ep)\}$, we have $\widetilde{f}_n\in\cS_n(\Gamma^1,\Gamma ^2,\Omega)$. 
 Using lemmas \ref{lem:propdis2} and  \ref{lem:propdis4}, we have for $n$ large enough
\begin{align*}
\dis(\amu_n(f_n),\ssigma_p^\kappa\cL^d)&\leq \sum_{x\in\fM_\kappa}\dis(\amu_n(f_n)\ind_{\pi_{x,\kappa}(\fC)},\ssigma_p(x)\ind_{\pi_{x,\kappa}(\fC)}\cL^ d)+2\|\ssigma_p^\kappa \ind_{\pi_{x,\kappa(1+2d/m)}(\fC)\setminus \pi_{x,\kappa}(\fC)} \|_{L^1}\\
&\qquad+\frac{2}{n^d}\sum_{e\in\E_n^d\cap \Cor}\|f_n(e)\|_2\\
&\leq 12 \ep^{\alpha_0}\cL^d(\Omega)+6d M\cL^d(\Cor)\,.
\end{align*}
Using inequality \eqref{eq:sigmad} and lemma \ref{lem:propdis2}, it follows that
\begin{align}\label{eq:disfn}
\dis(\ssigma_p^\kappa\cL^d,\ssigma\cL^d)&\leq \dis(\ssigma_p^\kappa\cL^d,\ssigma_p\ind_\Omega\cL^d)+ \dis(\ssigma_p\ind_\Omega\cL^d,\ssigma'\ind_\Omega\cL^d)+ \dis(\ssigma'\ind_\Omega\cL^d,\ssigma\cL^d)\nonumber\\
&\leq 2\|\ssigma_p^\kappa-\ssigma_p\ind_\Omega\|_{L^1}+2\|\ssigma'-\ssigma_p\|_{L^1}+2\|\ssigma'\ind_\Omega-\ssigma\|_{L^1}\nonumber\\
&\leq 2\ep \cL^d(\Omega)+20d^3M\cH^{d-1}(\Gamma)\kappa+2\|\ssigma'-\ssigma_p\|_{L^1}+2\eta\,.
\end{align}
Moreover, using inequality \eqref{eq:disfn}, we have
\begin{align*}
\dis(\amu(\widetilde{f}_n),\ssigma\cL^d)&\leq \dis(\amu(\widetilde{f}_n),\amu(f_n))+\dis(\amu(f_n),\ssigma_p ^ \kappa\cL^d)+\dis(\ssigma_p^\kappa\cL^d,\ssigma\cL^d)\\
&\leq 2d\cL^d(\cV_\infty(\Omega,1))\left(\frac{\mathrm{H}(\ep)}{M}(1-\ep ^{\alpha/2})+\ep ^{\alpha/2}\right)M +6dM\cL^d(\Cor)+(2\ep+12 \ep^{\alpha_0})\cL^d(\Omega)\\&\qquad+20d^3M\cH^{d-1}(\Gamma)\kappa+2\|\ssigma'-\ssigma_p\|_{L^1}+2\eta\,.
\end{align*}
Hence, using the independence, we have for $p$ large enough depending on $\eta$ for $\ep$ small enough depending on $p$, $d$ and $M$ and then $n$ large enough depending on $\ep$
\begin{align}\label{eq:prodind}
&\prod_{e\in\Cor\cap\E_n^d}\Prb(t(e)\geq M-\mathrm{H}(\ep) )\prod_{x\in\fM_\kappa}\Prb(\cE_\kappa(x))\leq \Prb(\exists \widetilde{f}_n\in\cS_n(\Gamma_1,\Gamma_2,\Omega):\, \dis(\amu_n(\widetilde{f}_n),\ssigma\cL^d)\leq 3\eta)
\end{align} 
where we recall that $\kappa$ goes to $0$ when $\ep$ goes to $0$.
We set $n_0= n\kappa $.
 Let $x\in \Omega \setminus \Cor$. Note that $\|\ssigma_p(x)\ind_\fC-\ssigma_p(c(x))\ind_{\fC}\|_{L^1}\leq \ep$ where $c(x)\in\fM_\kappa$ such that $x\in \pi_{c(x),\kappa}(\fC)$. By lemma \ref{lem:propdis2}, it yields 
 $$\dis(\amu_{n_0}(f_{n_0}),\ssigma_p(c(x))\ind_ {\fC}\cL^d)\leq \dis(\amu_{n_0}(f_{n_0}),\ssigma_p(x)\ind_ {\fC}\cL^d)+2\ep\,.$$
We can apply lemma \ref{lem:scaling1} and use the previous inequality
\begin{align}\label{eq:eqn2b}
\Prb(\cE_\kappa(c(x))&= \Prb\left( \begin{array}{c}\exists f_n\in\cS_n(\pi_{c(x),\kappa}(\fC)):\dis(\amu_n(f_n),\ssigma_p(c(x))\ind_ {\pi_{c(x),\kappa}(\fC)})\leq 12 \ep^{\alpha_0} \kappa ^d,\\\forall \diamond\in\{+,-\}\,\forall i\in\{1,\dots,d\}\,\forall A\in\cP_i^\diamond(m)\\\psi_i^\diamond(f_n,\pi_{c(x),\kappa}(A))=(1-\ep ^{\alpha/4})\left(\int_{\pi_{c(x)+z_0,2\upsilon}(\fC_i ^\diamond)}\ssigma_p (y)\cdot \overrightarrow{e_i}d\cH ^{d-1}(y)\right)\,\frac{n ^{d-1} }{m^{d-1}}\end{array}\right)\nonumber\\
&\geq \Prb\left(\begin{array}{c}\exists f_{n_0}\in\cS_{n_0}(\fC):\quad \dis(\amu_{n_0}(f_{n_0}),\ssigma_p(c(x))\ind_ {\fC}\cL^d)\leq 3 \ep^{\alpha_0} ,\\\forall \diamond\in\{+,-\}\,\forall i\in\{1,\dots,d\}\,\forall A\in\cP_i^\diamond(m)\\\psi_i^\diamond(f_{n_0},A)=(1-\ep ^{\alpha/4})\left(\int_{\pi_{c(x)+z_0,\kappa(1+2d/m)}(\fC_i ^\diamond)}\ssigma_p (y)\cdot \overrightarrow{e_i}\,d\cH ^{d-1}(y)\right)\frac{n_0 ^{d-1}}{(m\kappa)^{d-1}} \end{array}\right)\nonumber\\
&\geq\Prb\left(\begin{array}{c}\exists f_{n_0}\in\cS_{n_0}(\fC):\quad \dis(\amu_{n_0}(f_{n_0}),\ssigma_p(x)\ind_ {\fC}\cL^d)\leq  \ep^{\alpha_0} ,\\\forall \diamond\in\{+,-\}\,\forall i\in\{1,\dots,d\}\,\forall A\in\cP_i^\diamond(m)\\\psi_i^\diamond(f_{n_0},A)=(1-\ep ^{\alpha/4})\left(\int_{\pi_{c(x)+z_0,\kappa(1+2d/m)}(\fC_i ^\diamond)}\ssigma_p (y)\cdot \overrightarrow{e_i}\,d\cH ^{d-1}(y)\right)\frac{n_0 ^{d-1}}{(m\kappa)^{d-1}} \end{array}\right)\,.
\end{align}
We check that the conditions to apply lemma \ref{lem:toutefamille} are satisfied.
Since $\diver\ssigma_p=0$, we have by Gauss-Green theorem applied to $\ssigma_p$ in $\pi_{c(x)+z_0,\kappa(1+2d/m)}(\fC)$
\begin{align*}
\sum_{i=1}^d\int_{\pi_{c(x)+z_0,\kappa(1+2d/m)}(\fC_i ^-)}\ssigma_p (y)\cdot \overrightarrow{e_i}\,d\cH ^{d-1}(y)&=\sum_{i=1}^d\int_{\pi_{c(x)+z_0,\kappa(1+2d/m)}(\fC_i^+)}\ssigma_p(y)\cdot \overrightarrow{e_i}d\cH^{d-1}(y)
\end{align*}
Moreover, for all $\diamond\in\{+,-\}$, for all $i\in\{1,\dots,d\}$, and for all $ A\in\cP_i^\diamond(m)$, we have
\begin{align*}
&\left|\frac{1}{(m\kappa )^{d-1}}\int_{\pi_{c(x)+z_0,\kappa(1+2d/m)}(\fC_i ^\diamond)}\ssigma_p (y)\cdot \overrightarrow{e_i}d\cH ^{d-1}(y)-\cH^{d-1}(A)\ssigma_p(x)\cdot \overrightarrow{e_i}\right|\\
&\quad =\frac{1}{(m\kappa) ^{d-1}}\left|\int_{\pi_{c(x)+z_0,\kappa(1+2d/m)}(\fC_i ^\diamond)}\ssigma_p (y)\cdot \overrightarrow{e_i}d\cH ^{d-1}(y)-\cH^{d-1}(\pi_{c(x)+z_0,\kappa}(\fC_i ^\diamond))\ssigma_p(x)\cdot \overrightarrow{e_i}\right|\\
&\quad\leq \frac{1}{(m\kappa )^{d-1}}\int_{\pi_{c(x)+z_0,\kappa}(\fC_i ^\diamond)}\|\ssigma_p (y)-\ssigma_p(x)\|_2d\cH ^{d-1}(y)+\frac{1}{(m\kappa)^{d-1}}\kappa ^{d-1}\left(\left(1+\frac{2d}{m}\right) ^{d-1}-1\right)\|\ssigma_p\cdot\overrightarrow{e_i}\|_{L^\infty}\\
&\quad\leq \left( \ep +2^dM\frac{d}{m}\right)\cH^{d-1}(A)\,.
\end{align*}
where we use in the last inequality that $2d/m\leq 1$ and inequality \eqref{eq:binomenewton}. 
We recall that $m=\lfloor \ep ^{-\alpha}\rfloor$, hence we have for $\ep$ small enough
$$\left|\frac{1}{(m\kappa) ^{d-1}}\int_{\pi_{x+z_0,\kappa(1+2d/m)}(\fC_i ^\diamond)}\ssigma_p (y)\cdot \overrightarrow{e_i}d\cH ^{d-1}(y)-\cH^{d-1}(A)\ssigma_p(x)\cdot \overrightarrow{e_i}\right|\leq 2 ^{d+2}M d\ep ^{\alpha}\cH^{d-1}(A)\,.$$
It follows that the conditions to apply lemma \ref{lem:toutefamille} are fulfilled. Thanks to theorem \ref{thmbrique} and inequality \eqref{eq:eqn2b}, it yields
\begin{align}\label{eq:scaling1}
\limsup_{\ep\rightarrow 0}\limsup_{n\rightarrow\infty}-\frac{1}{n_0^d}\log \Prb(\cE_\kappa(c(x))\leq \limsup_{\ep\rightarrow 0}\limsup_{n_0\rightarrow\infty}-\frac{1}{n_0^d}\log \Prb(\cE_\kappa(c(x))\leq  I(\ssigma_p(x))\,.
\end{align}

Besides, we have 
\begin{align}\label{eq:eqn1}
\limsup_{\epsilon\rightarrow 0}\limsup_{n\rightarrow\infty}\sum_{x\in\fM_\kappa}-\frac{1}{n^d}\log\Prb (\cE_\kappa(x) )&=\limsup_{\ep\rightarrow 0}\limsup_{n\rightarrow\infty}\int_{\Omega}-\frac{1}{n_0^d}\log\Prb (\cE_\kappa(c(x)) \ind_{\underset{w\in\fM_\kappa}{\cup}\pi_{w,\kappa}(\fC)}(x)d\cL ^d(x)\,.
\end{align}
We would like to use the reverse Fatou Lemma. Fix $\ep>0$. To be able to use this lemma we need to upperbound the integrand uniformly on $\ep$ and $n$ by an integrable function. To do so, we need to use inequalities from the proofs of lemma \ref{lem:lemexistencefamille} and lemma \ref{lem:toutefamille}. We have using inequalities \eqref{eq:eqn2b}, \eqref{eq3.1:3} and \eqref{eq3.1:4} 
\begin{align*}
-\frac{1}{n_0^d}\log\Prb (\cE_\kappa(c(x)) &\leq -\frac{1}{n_0^d}\log \Prb\left( \exists f_{n_0}\in\cS_{n_0}(\fC) : \begin{array}{c}\,\forall \diamond\in\{+,-\}\,\forall i\in\{1,\dots,d\}\,\forall A\in\cP_i^\diamond(m) \quad \psi_i^\diamond(f_{n_0},A)=\lambda_A^\diamond\\ \text{and }\,\dis\big(\amu_{n_0}(f_{n_0}),\ssigma_p(x)\ind_{\fC}\cL^d\big)\leq \ep^{\alpha}\end{array}\right)\\
&\qquad-\kappa'_d \log G\left(\left[\frac{\|\ssigma_p(x)\|_2}{2d},+\infty\right[\right)
\end{align*}
where $(\lambda_A^+)_A$ and $(\lambda_A^-)$ are the families defined in lemma \ref{lem:lemexistencefamille} associated with $\ssigma(x)$ and $\ep$. Note that $K\geq 1$ in \eqref{eq3.1:4}.
Finally, using equality \eqref{eq3.1:2}, we obtain 
\begin{align}\label{eq:controlereversefatou}
-\frac{1}{n_0^d}\log\Prb (\cE_\kappa(c(x)) &\leq -\frac{1}{n_0^d}\log \Prb\left( \exists f_{n_0}\in\cS_{n_0}(\fC) : \dis\big(\amu_{n_0}(f_{n_0}),\ssigma_p(x)\ind_{\fC}\cL^d\big)\leq \ep\right)\nonumber\\
&\qquad-\frac{2dm^{d-1}}{n_0^d}\log\left(\frac{m ^{d-1}\sqrt{\ep}}{4\kappa_d\ep^\alpha n_0^{d-1}}\right)-\kappa'_d \log G\left(\left[\frac{M}{\sqrt {2d}},+\infty\right[\right)\,.
\end{align}
Besides, using lemma \ref{lem:propdis2}, we have
\begin{align*}
\Prb&\left( \exists f_{n_0}\in\cS_{n_0}(\fC) : \dis\big(\amu_{n_0}(f_{n_0}),\ssigma_p(x)\ind_{\fC}\cL^d\big)\leq \ep\right)\\&\hspace{3cm}\geq \Prb\left( \exists f_{n_0}\in\cS_{n_0}(\fC) : \dis\big(\amu_{n_0}(f_{n_0}),(1-\ep/(4dM))\ssigma_p(x)\ind_{\fC}\cL^d\big)\leq \ep/2\right)\,.
\end{align*}
It follows that using inequality \eqref{eq:controleIparG} for $n$ large enough
\begin{align*}
 -\frac{1}{n_0^d}\log \Prb\left( \exists f_{n_0}\in\cS_{n_0}(\fC) : \dis\big(\amu_{n_0}(f_{n_0}),\ssigma_p(x)\ind_{\fC}\cL^d\big)\leq \ep\right)&\leq -d\log G([(1-\ep/(4dM))\|\ssigma_p(x)\|_\infty,M])\\&\leq-d\log G([(1-\ep/(4dM))M,M]) \,.
 \end{align*}
We can therefore use reverse Fatou lemma for a fixed $\ep$, we obtain
\begin{align*}
\limsup_{n\rightarrow\infty}\sum_{x\in\fM_\kappa}-\frac{1}{n^d}\log\Prb (\cE_\kappa(x) )&\leq\int_{\Omega}\limsup_{n\rightarrow\infty}-\frac{1}{n_0^d}\log\Prb (\cE_\kappa(c(x)) \ind_{\underset{w\in\fM_\kappa}{\cup}\pi_{w,\kappa}(\fC)}(x)d\cL ^d(x)\\&\leq\int_{\Omega}\limsup_{n_0\rightarrow\infty}-\frac{1}{n_0^d}\log\Prb (\cE_\kappa(c(x)) \ind_{\underset{w\in\fM_\kappa}{\cup}\pi_{w,\kappa}(\fC)}(x)d\cL ^d(x)\,.
\end{align*}
Using inequality \eqref{eq:controlereversefatou}, we have
\begin{align*}
\limsup_{n_0\rightarrow\infty}&-\frac{1}{n_0^d}\log\Prb (\cE_\kappa(c(x))\\
& \leq\limsup_{n_0\rightarrow\infty} -\frac{1}{n_0^d}\log \Prb\left( \exists f_{n_0}\in\cS_{n_0}(\fC) : \dis\big(\amu_{n_0}(f_{n_0}),\ssigma_p(x)\ind_{\fC}\cL^d\big)\leq \ep\right)-\kappa'_d \log G\left(\left[\frac{M}{\sqrt{2d}},+\infty\right[\right)\\
&\leq I(\ssigma_p(x))-\kappa'_d \log G\left(\left[\frac{M}{2d},+\infty\right[\right)
\end{align*}
and the right hand side is integrable on $\Omega$, we can use again the reverse Fatou lemma, we obtain
\begin{align*}
\limsup_{\epsilon\rightarrow 0}\limsup_{n\rightarrow\infty}\sum_{x\in\fM_\kappa}-\frac{1}{n^d}\log\Prb (\cE_\kappa(x) )&\leq\int_{\Omega}\limsup_{\ep\rightarrow 0}\limsup_{n_0\rightarrow\infty}-\frac{1}{n_0^d}\log\Prb (\cE_\kappa(c(x)) \ind_{\underset{w\in\fM_\kappa}{\cup}\pi_{w,\kappa}(\fC)}(x)d\cL ^d(x)\,.
\end{align*}
Combining inequalities \eqref{eq:eqn1} and \eqref{eq:scaling1}, we obtain
\begin{align}\label{eq:eqn3b}
\limsup_{\epsilon\rightarrow 0}\limsup_{n\rightarrow\infty}\sum_{x\in\fM_\kappa}-\frac{1}{n^d}\log\Prb (\cE_\kappa(x)) \leq \int_{\Omega}I(\ssigma_p(x))d\cL^d(x)=\widehat{I}(\ssigma_p)\,.
\end{align}
Finally, using inequalities \eqref{eq:corid}, \eqref{eq:conccontrcardcor}, \eqref{eq:prodind} and \eqref{eq:eqn3b}, we obtain 
\begin{align*}
\limsup_{n\rightarrow\infty}-\frac{1}{n^d}\log \Prb(\exists f_n\in\cS_n(\Gamma_1,\Gamma_2,\Omega):\, \dis(\amu_n(f_n),\ssigma)\leq 3 \eta)\leq \widehat{I}(\ssigma_p)\,.
\end{align*}
By proposition \ref{prop:continuityIhat}, we have
\[\lim_{p\rightarrow\infty}\widehat{I}(\ssigma_p)=\widehat{I}(\ssigma')\,.\]
By the properties of $\ssigma'$, we have 
\begin{align*}
\limsup_{n\rightarrow\infty}-\frac{1}{n^d}\log \Prb(\exists f_n\in\cS_n(\Gamma_1,\Gamma_2,\Omega):\, \dis(\amu_n(f_n),\ssigma)\leq 3 \eta)\leq \widehat{I}(\ssigma)+\eta \,.
\end{align*}
The result follows by letting $\eta$ go to $0$.
\end{proof}
\subsection{Proof of proposition \ref{prop:prolsigma}\label{sec:5.3} }
\begin{proof}[Proof of proposition \ref{prop:prolsigma}] Let $\ssigma\in\Sigma(\Gamma^1,\Gamma ^2,\Omega)\cap \Sigma^M(\Gamma^1,\Gamma ^2,\Omega)$ such that $\widehat{I}(\ssigma)<\infty$.
By hypothesis \ref{hypo:omega}, there exist $\cS_1,\dots,\cS_l$ hypersurfaces of class $\sC^1$ such that $\Gamma\subset \cup_{i=1,\dots,l}\cS_i$.

\noindent {\bf Step 1. Decomposition of $\ssigma$.}
 Let $p\geq 1$.
 We denote by $\cN_p$ the following subset of $\Gamma$.
 $$\cN_p= (\Gamma \cap \cV_2(\partial_\Gamma \Gamma^1\cup\partial_\Gamma \Gamma^2\cup _{i=1,\dots,l}\partial _\Gamma (\cS_i\cap \Gamma),1/p))\cup\bigcup_{i=1,\dots ,d}\left\{x\in \Gamma : 0<|\nn_\Omega(x)\cdot \overrightarrow{e_i}|\leq 1/p\right\}\,.$$
We aim to decompose $\ssigma=\ssigma^{(p)}+\ssigma^{(p),res}$ such that 
\begin{itemize}
\item $\ssigma^{(p)},\ssigma^{(p),res}\in\Sigma(\Gamma^1,\Gamma ^2,\Omega)$;
\item $\ssigma^{(p)}\cdot \overrightarrow{n}_{\Omega}=0$ $\cH^{d-1}$-almost everywhere on $\cN_p$;
\item $\ssigma^{(p),res}$ is negligible in some sense.
\end{itemize}
We are going to build these continuous streams as the limit of discrete streams. 
Since $\ssigma\in\Sigma^M(\Gamma^1,\Gamma ^2,\Omega)$, there exist an increasing function $\psi:\sN\rightarrow \sN$ and  $f_{\psi(n)}\in\cS_{\psi(n)}^M (\Gamma^1,\Gamma^2,\Omega)$ for $n\in\sN$, such that 
$$\lim_{n\rightarrow \infty}\dis(\amu_{{\psi(n)}}(f_{{\psi(n)}}),\ssigma\cL^d)=0\,.$$ To lighten the notations, we will write $f_n$ instead of $f_{\psi(n)}$.
By lemma \ref{lem:res}, there exists a couple $(\oGam_n,(p(\ogam))_{\ogam\in\oGam_n})$ such that
$$f_n=\sum_{\overrightarrow{\gamma}\in\oGam}p(\ogam)\sum_{\langle\langle x,y\rangle\rangle\in\overrightarrow{\gamma}}n\,\overrightarrow{xy}\ind_{\langle x,y\rangle}\,$$
where $\oGam_n$ is a set of self-avoiding oriented path that have both extremities in $\Gamma_n^1\cup\Gamma_n ^2$. 
If there are several possible choices for this couple, we pick one according to a deterministic rule.
We can decompose the set $\oGam_n$ into two disjoint sets $\oGam ^{(p)}_n$ and $\oGam^{(p),res}_n$
where $$\oGam^{(p),res}_n=\left\{\ogam\in\oGam: \,\ogam \text{ has at least one extremity in $\cV_2(\cN_p,d/n)$}\right\}$$
and $\oGam ^{(p)}_n=\oGam\setminus \oGam_n^{(p),res}$.
We set $$f_n^{(p)}=\sum_{\ogam\in \oGam^{(p)}_n}p(\ogam)\sum_{\overrightarrow{e}=\langle \langle x,y\rangle\rangle \in \ogam}n\,\overrightarrow{xy}\ind_{e}\,$$
and $f_n^{(p),res}=f_n- f_n^{(p)}$.
It is easy to check that $$f_n^{(p)}\in\cS_n^M( \Gamma ^1\setminus \cN_p,\Gamma^2 \setminus \cN_p,\Omega)\,.$$
Let $N\geq 1$. 
By compactness and lemma \ref{lem:convdiscstream}, up to extractions, we can assume that for any $p\in\{1,\dots,N\}$ the measure $\amu_n(f_n ^{(p)})$ converges weakly towards a stream $\ssigma^{(p)}\cL^d$ where $\ssigma^{(p)}\in\Sigma(\Gamma^1,\Gamma^2,\Omega)$ and $\ssigma^{(p)}\cdot \overrightarrow{n}_{\Omega}=0$ $\cH^{d-1}$-almost everywhere on $(\Gamma \setminus (\Gamma^1\cup\Gamma^2))\cup\cN_p$.
Besides, we recall that by lemma \ref{lem:res}, for any $e_0\in\E_n^d$ and any $p\geq 1$, we have 
$$(f_n^{(p)}(e_0)-f_n^{(p+1)}(e_0))\cdot f_n(e_0)=\sum_{\ogam\in \oGam^{(p)}_n\setminus \oGam^{(p+1)}_n}p(\ogam)\sum_{\overrightarrow{e}=\langle \langle x,y\rangle\rangle \in \ogam}n\,\overrightarrow{xy}\cdot f_n(e_0)\ind_{e}(e_0)\geq 0\,.$$
It follows that
\begin{align*}
\frac{1}{n^d}\sum_{e\in \E_n^d}\sum_{p=1}^N\|f_n ^{(p+1)}(e)-f_n^{(p)}(e)\|_2 &=\frac{1}{n^d}\sum_{e\in \E_n^d}\sum_{p=1}^N(f_n ^{(p+1)}(e)-f_n^{(p)}(e))\cdot \frac{f_n(e)}{\|f_n(e)\|_2}\\
&\leq \frac{1}{n^d}\sum_{e\in \E_n^d}\|f_n (e)\|_2\leq 2d\cL^d(\cV_\infty(\Omega,1))M\,.
\end{align*}
By inequality \eqref{eq:sigmaL1mun}, it follows that 
\[\sum_{p=1}^N\|\ssigma^{(p+1)}-\ssigma^{(p)}\|_{L^1}\leq 2d\cL^d(\cV_\infty(\Omega,1))M\,.\]
By letting $N$ go to infinity, we obtain
\[\sum_{p=1}^\infty\|\ssigma^{(p+1)}-\ssigma^{(p)}\|_{L^1}\leq 2d\cL^d(\cV_\infty(\Omega,1))M\,.\]
It follows that there exists $\ssigma_0\in \Sigma(\Gamma^1,\Gamma ^2,\Omega)$ such that
$\lim_{p\rightarrow\infty}\|\ssigma^{(p)}-\ssigma_0\|_{L^1}=0$. Note that in general, we don't have necessarily $\ssigma_0=\ssigma$. However, we prove that the stream $\ssigma_0-\ssigma$ has null divergence on $\sR^d$ and null trace on $\Gamma$. We set $\ssigma^{(p),res}=\ssigma-\ssigma^{(p)}$. 
By following the arguments steps 3 and 4 in the proof of lemma \ref{lem:convdiscstream} and proposition \ref{prop:minkowski}, we have for any $u\in\sC^\infty_c(\sR^d,\sR)$,  
\begin{align*}
\left|\int_{\sR ^d} \ssigma^{(p),res}\cdot \overrightarrow{\nabla}ud\cL ^d\right|&= \left| \lim_{n\rightarrow\infty}-\frac{1}{n ^{d-1}}\sum_{x\in (\Gamma_n^1\cup \Gamma_n ^2)\cap \cV_2(\cN_p,d/n)}u(x)\,\di f_n(x)\right|\\
&\leq d M\|u\|_\infty \lim_{n\rightarrow\infty}n\frac{|\sZ_n^d\cap \cV_2(\cN_p,d/n)|}{n ^{d}}\\
&\leq d M\|u\|_\infty \lim_{n\rightarrow\infty}n\cL^d(\cV_2(\cN_p,2d/n))\\
&\leq 8d^2 M\|u\|_\infty\cH^{d-1}(\cN_p)\,.
\end{align*}
Besides, since the manifolds intersect transversally for $i\neq j\in\{1,\dots,l\}$ the intersection $\cS_i\cap\cS_j$ is a sub-manifold of codimension 2 (see for instance chapter 1 paragraph 5 in \cite{GuilleminPollack}). It follows that $\cH^{d-1}(\cS_i\cap\cS_j)=0$ and
$$\cH^{d-1}(\partial_{\Gamma}(\cS_i\cap \Gamma))\leq \sum_{j\neq i}\cH^{d-1}(\cS_i\cap\cS_j)=0\,.$$
Since $\cH^{d-1}(\partial_\Gamma \Gamma^1\cup\partial_\Gamma \Gamma^2\cup_{i=1,\dots l}\partial_\Gamma (\cS_i\cap\Gamma))=0$, we have
$\lim_{p\rightarrow\infty}\ind_ {\cV_2(\partial_\Gamma \Gamma^1\cup\partial_\Gamma \Gamma^2\cup_{i=1,\dots l}\partial_\Gamma (\cS_i\cap\Gamma),1/p)}(x)=0$ for $\cH^{d-1}$-almost every $x$ in $\Gamma$. Hence, thanks to the dominated convergence theorem
\begin{align*}
\lim_{p\rightarrow\infty}\cH^{d-1}&(\Gamma \cap\cV_2(\partial_\Gamma \Gamma^1\cup\partial_\Gamma \Gamma^2\cup_{i=1,\dots l}\partial_\Gamma (\cS_i\cap\Gamma),1/p))\\
&=\int_{\Gamma}\lim_{p\rightarrow\infty}\ind_ {\cV_2(\partial_\Gamma \Gamma^1\cup\partial_\Gamma \Gamma^2\cup_{i=1,\dots l}\partial_\Gamma (\cS_i\cap\Gamma),1/p)}(x)d\cH^{d-1}(x)=0\,.
\end{align*}
Let $i\in\{1,\dots,d\}$. For $\cH^{d-1}$-almost every $x\in\Gamma$, the normal exterior vector $\nn_{\Omega}(x)$ is well defined. For every $x\in\Gamma$ such that $\nn_\Omega(x)$ is well defined, we have
$$\lim_{p\rightarrow \infty}\ind_{0<|\nn_{\Omega}(x)\cdot \overrightarrow{e_i}|\leq 1/p}=0\,.$$
Thanks to dominated convergence theorem, we have
$$\lim_{p\rightarrow \infty}\int_{\Gamma}\ind_{0<|\nn_{\Omega}(x)\cdot \overrightarrow{e_i}|\leq 1/p}d\cH^{d-1}(x)=0\,.$$
Finally, we have that 
$$\lim_{p\rightarrow \infty}\cH^{d-1}(\cN_p)=0\,$$ and for any $u\in\sC^\infty_c(\sR^d,\sR)$, we have
$$\lim_{p\rightarrow\infty} \int_{\sR ^d} \ssigma^{(p),res}\cdot \overrightarrow{\nabla}ud\cL ^d=0\,.$$
We have
\begin{align*}
\left|\int_{\sR ^d} (\ssigma^{(p),res}-\ssigma+\ssigma_0)\cdot \overrightarrow{\nabla}ud\cL ^d\right|&= \left|\int_{\sR ^d} (\ssigma_0-\ssigma^{(p)})\cdot \overrightarrow{\nabla}ud\cL ^d\right|\leq \|\overrightarrow{\nabla}u\|_{L ^\infty}\|\ssigma_0-\ssigma^{(p)}\|_{L^1}\,.
\end{align*}
It follows that for any $u\in \sC^\infty _c(\sR^d,\sR)$,
$$\lim_{p\rightarrow\infty} \int_{\sR ^d} \ssigma^{(p),res}\cdot \overrightarrow{\nabla}ud\cL ^d=\int_{\sR ^d} (\ssigma-\ssigma_0)\cdot \overrightarrow{\nabla}ud\cL ^d=0\,.$$
Hence, we have $\diver (\ssigma-\ssigma_0)=0$ $\cL^d$-almost everywhere on $\Omega$.
Furthermore, by equality \eqref{eq:caracterisationsigman}, we get
\begin{align*}
\int_{\sR ^d} (\ssigma-\ssigma_0)\cdot \overrightarrow{\nabla}ud\cL ^d=\int_{\Gamma} (\ssigma-\ssigma_0)\cdot \nn_\Omega \, ud\cH ^{d-1}\,.
\end{align*}
Finally, for any $u\in \sC^\infty _c(\sR^d,\sR)$, we have
$$\int_\Gamma ((\ssigma-\ssigma_0)\cdot \overrightarrow{n}_\Omega)ud\cH^{d-1}=0\,.$$
{\bf Step 2. Prolongation of the discrete stream.}
Let $r_0>0$ we will choose later. We define the prolongated version $\widetilde f_n^{(p)}$ of $f_n^{(p)}$ (see figure \ref{fig:prolstream}) as follows
\begin{align*}
\widetilde f_n^{(p)}=f_n^{(p)}+\sum_{x\in\Gamma_n^1\cup\Gamma_n^2\setminus \cV_2(\cN_p,d/n)}\sum_{i=1,\dots,d}\Bigg(\ind _{\nn_{\Omega}(\pi_i(x))\cdot \overrightarrow{e_i}\geq 1/p}\sum_{k=1}^{\lfloor r_0 n\rfloor}f_n^{(p)}(\langle x,x-\overrightarrow{e_i}/n\rangle)\ind_{\langle x+(k-1)\overrightarrow{e_i}/n,x+k\overrightarrow{e_i}/n\rangle}\\
\hfill +\ind _{\nn_{\Omega}(\pi_i(x))\cdot \overrightarrow{e_i}\leq -1/p}\sum_{k=1}^{\lfloor r_0 n\rfloor}f_n^{(p)}(\langle x,x+\overrightarrow{e_i}/n\rangle)\ind_{\langle x-(k-1)\overrightarrow{e_i}/n,x-k\overrightarrow{e_i}/n\rangle}\Bigg)
\end{align*}
where $\pi_i(x)\in \Gamma$ is the intersection between $\Gamma$ and $\{x+\lambda \overrightarrow{e_i},\lambda\in\sR\}$, if there are several intersection points, we pick the closest from $x$. Note that it may exist two disjoint points $x$ and $y$ in $\Gamma_n^1\cup\Gamma_n^2$ such that 
$\pi_i(x)=\pi_i(y)$. However, this is not an issue since by definition of $\cS_n^M (\Gamma^1,\Gamma^2,\Omega)$, we have $f_n^{(p)}(\langle x,y\rangle)=0$. 
 Roughly speaking, we obtain the stream $\widetilde f_n^{(p)}$ by prolongating the stream $f_n^{(p)}$ through straight lines.
\begin{figure}[H]
\begin{center}
\def\svgwidth{0.9\textwidth}
   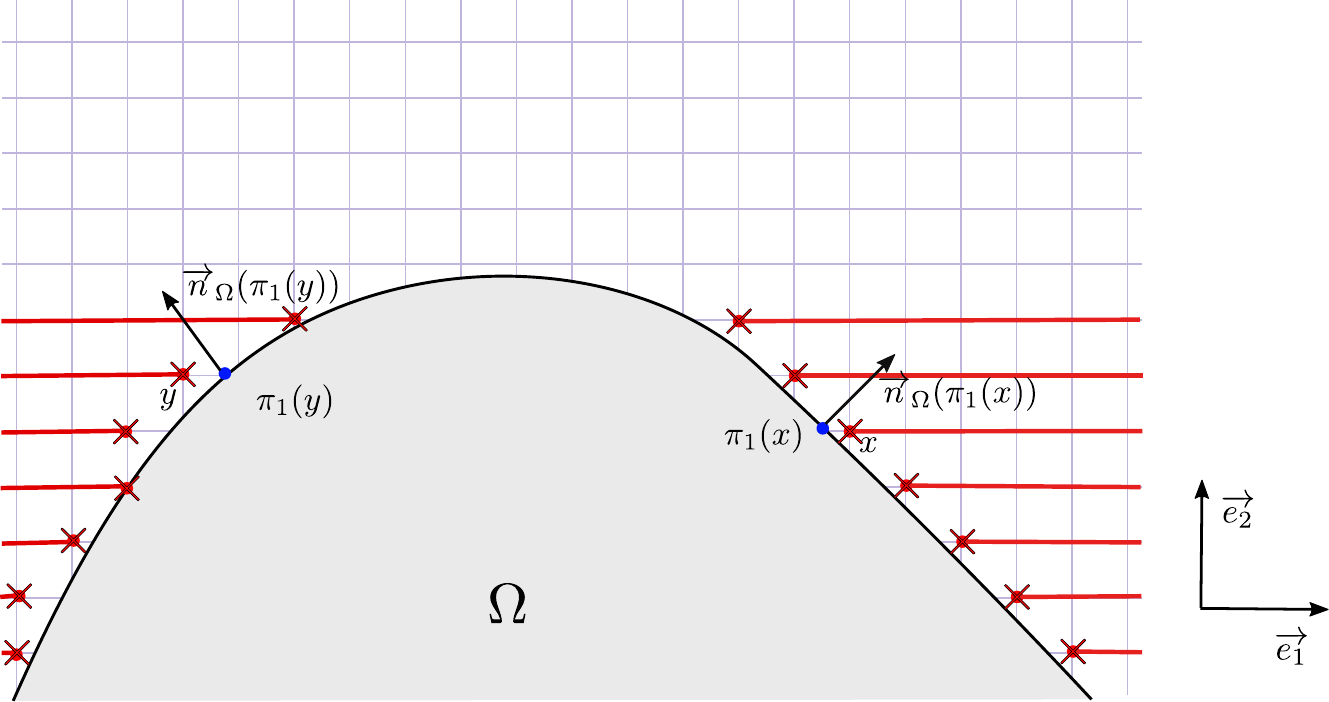
   \label{fig:prolstream}
   \caption[figprol]{The crosses correspond to points in $\Gamma_n^1$. By sake of clarity we only represent how we prolongate the stream in the direction $\overrightarrow{e_1}$ (represented by the bold lines). Note that in the figure, $\nn_\Omega(\pi_1(x))\cdot \overrightarrow{e_1}>0$ and  $\nn_\Omega(\pi_1(y))\cdot \overrightarrow{e_1}<0$. The corresponding $p$ is chosen big enough.}
   \end{center}
\end{figure}

{\bf Choice of $r_0$.} Let $x\in\overline{\Gamma^1\cup \Gamma^2 \setminus \cN_p}$. Since $\Omega$ is a Lipschitz domain, there exist $r>0$, an hyperplane $H_x$ containing $x$ of normal vector $\nn_x$ and $\phi_x: H\rightarrow \sR$ a Lipschitz function such that
\[B(x,r)\cap \Gamma = \left\{y+\phi_x(y)\nn_x: y\in H_x\cap B(x,r)\right\}\,.\]
Since $x\notin \cV_2(\cup_{j=1,\dots,l}\partial _\Gamma (\cS_j\cap \Gamma),1/2p)$, up to choosing a smaller $r$, we can assume that $B(x,r)\cap\Gamma \subset \cS_j$ for some $j\in\{1,\dots,l\}$ and for any $i\in\{1,\dots,d\}$, if $|\nn_\Omega(x)\cdot \overrightarrow{e_i}|\geq 1/p$ then for any $y\in \Gamma \cap B(x,r)$, we have $|\nn_\Omega(y)\cdot \overrightarrow{e_i}|\geq 1/(2p)$. If $\nn_\Omega(x)\cdot \overrightarrow{e_i}=0$ then for any $y\in \Gamma \cap B(x,r)$, we have $|\nn_\Omega(y)\cdot \overrightarrow{e_i}|\leq 1/(2p)$. We used the fact that the hypersurface $\cS_j$ is of class $\sC^1$. To each $x$ in $\Gamma^1\cup \Gamma^2 \setminus \cN_p$, we associate $r_x>0$ as above.
We can extract from the family $(B(x,r_x/2),x\in\overline{\Gamma^1\cup \Gamma^2\setminus \cN_p})$ a finite covering $(B(x_i,r_{x_i}/2),i\in I)$ of the compact set $\overline{\Gamma^1\cup\Gamma^2\setminus \cN_p}$. We set $$r_0=\frac{1}{4}\min\left(\min_{i\in I}r_{x_i},1\right)\,.$$

{\bf We prove that the stream does not exceed the capacity constraint.}
Let $e$ be an edge in the support of $\widetilde f _n ^{(p)}-f_n^{(p)}$. Let us assume $e$ has its two endpoints in $\Omega_n$. Let $i\in\{1,\dots,d\}$ be such that $e$ is colinear to $\overrightarrow{e_i}$. Then, there exists $x\in (\Gamma^1\cup\Gamma^2)\setminus \cN_p$ such that $|\nn_\Omega(x)\cdot \overrightarrow{e_i}|\geq 1/p$ and $y\in \Gamma$ such that $x-y=t\overrightarrow{e_i}$ with $|t|\leq r_0$.
\begin{figure}[H]
\begin{center}
\def\svgwidth{0.9\textwidth}
   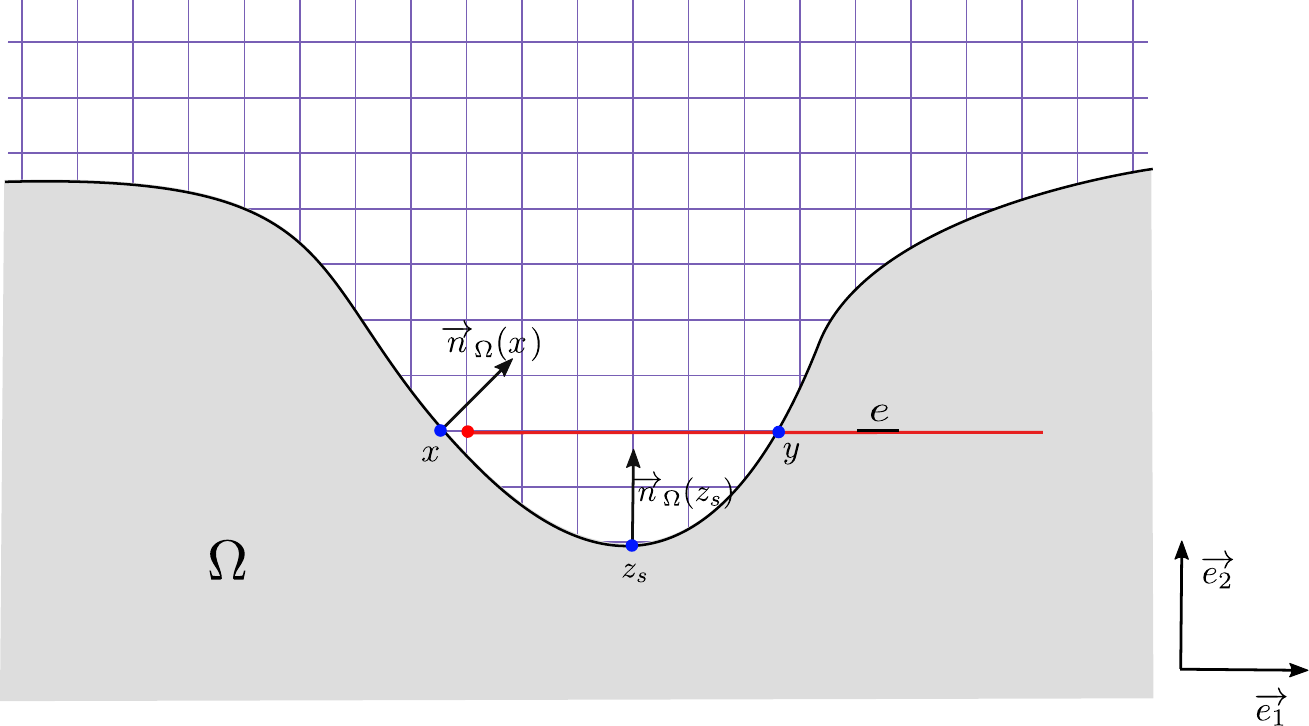
   \caption{Illustration of the case where there exists $e$ with its two endpoints in $\Omega_n$ in the support of $\widetilde f_n^{(p)}-f_n^{(p)}$.}
   \label{fig:caspb1}
   \end{center}
\end{figure}
 We claim that there does not exists any $y\in\Gamma$ such that $x-y=t\overrightarrow{e_i}$ with $|t|\leq 2r_0$. Since we prolongate the stream by the exterior  the existence of such $y$ implies that by prolongating the stream we have crossed $\Omega$ again. Indeed, assume such a $y$ exists. Let $j\in I$ such that $x\in B(x_j,r_{x_j}/2)$. Since $\|x-y\|_2\leq r_{x_j}/2$, we have that $y\in B(x_j,r_{x_j})$. Moreover, since $|\nn_\Omega(x)\cdot \overrightarrow{e_i}|\geq 1/p$, by definition of $r_{x_j}$, we have 
\begin{align}\label{eq:defzdansB} 
 \forall z \in \Gamma\cap B(x_i,r_{x_j})\qquad |\nn_\Omega(z)\cdot \overrightarrow{e_i}|\geq 1/2p\,.
 \end{align}
 Let us denote by $x'$ and $y'$ the points in $H_{x_j}$ (the hyperplane associated to $x_j
 $) such that $x=x'+\phi_{x_j}(x')\nn_{x_j}$ and $y=y'+\phi_{x_j}(y')\nn_{x_j}$. Let us denote by $\phi_0$ the following mapping
$$\forall s\in[0,1]\qquad \phi_0(s)=\phi_{x_j}((1-s)x'+sy')\,.$$
Since $x_j\notin\cN_p$ and $\Gamma$ is locally $\sC^1$ around $x$ we know that $\phi_0$ is of class $\sC^1$.
By the mean value theorem, there exists $s\in]0,1[$ such that $\phi_0'(s)=\phi_0(1)-\phi_0(0)=\phi_{x_j}(y')-\phi_{x_j}(x')$. In other words, the vector $y'-x'+(\phi_{x_j}(y')-\phi_{x_j}(x'))\nn_{x_j}=t\overrightarrow{e_i}$ belongs to the tangent space at the point $z_s=((1-s)x'+sy'+\phi_0(s)\nn_{x_j})\in \Gamma$ (see figure \ref{fig:caspb1}). Consequently, we have $$\nn_{\Omega}(z_s)\cdot \overrightarrow{e_i}=0\,.$$
This contradicts \eqref{eq:defzdansB}.
The support of $\widetilde f _n ^{(p)}-f_n^{(p)}$ does not contain edges with two endpoints in $\Omega_n$.

Let us assume that there exists $e\in\E_n^d$ such that $\|\widetilde f_n ^{(p)}(e)\|_2>M$. Necessarily $e$ is not in $\Omega$. Roughly speaking, in this situation, by prolongating the stream two disjoint points of $\Gamma$ use the same edge. Let $i\in\{1,\dots,d\}$ be such that $e$ is colinear to $\overrightarrow{e_i}$. Then there exists $x$ and $y$ in $\Gamma^1\cup\Gamma^2\setminus \cN_p$ such that 
 $|\nn_\Omega(x)\cdot \overrightarrow{e_i}|\geq 1/p$ and 
 $x-y=t\overrightarrow{e_i}$ with $|t|\leq 2r_0$ (see figure \ref{fig:caspb2}. By the same reasoning than above, this is excluded.
 \begin{figure}[H]
\begin{center}
\def\svgwidth{0.9\textwidth}
   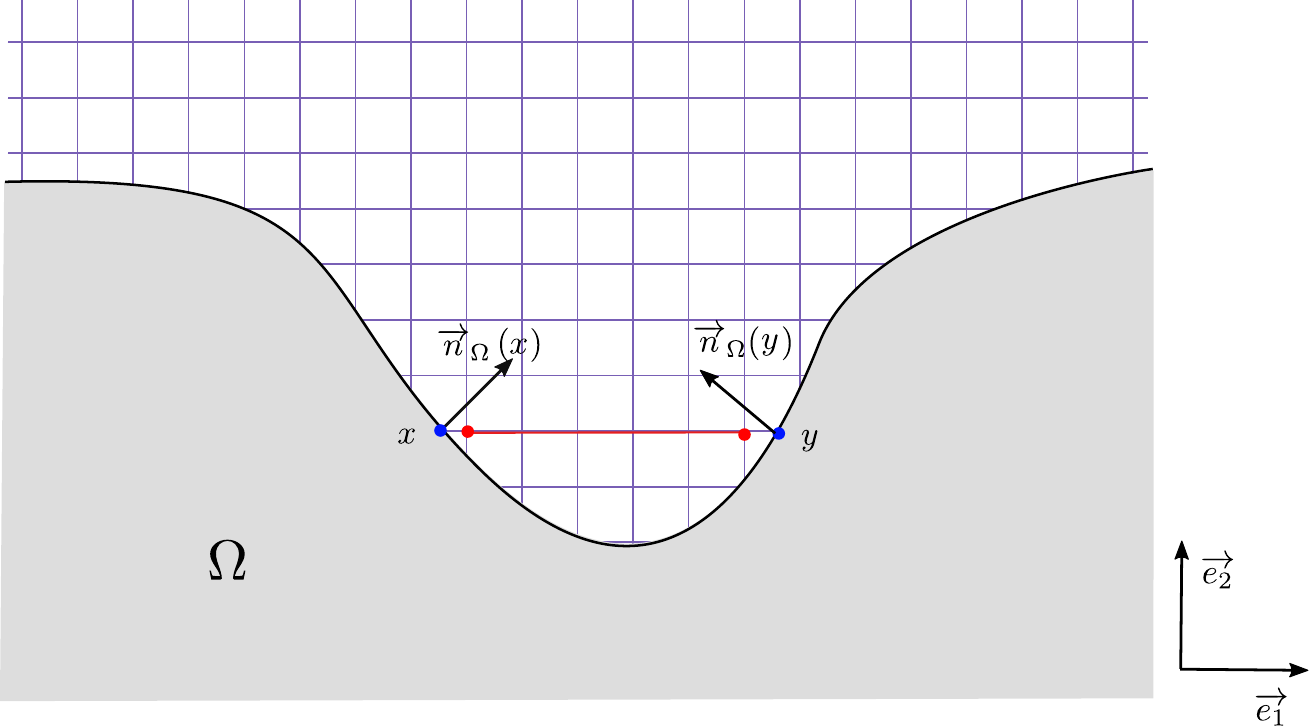
   \caption{Illustration of the case where there exists $e$ such that $\|\widetilde f _n (e)\|_2>M$.}
   \label{fig:caspb2}
   \end{center}
\end{figure}
{\bf We prove that the stream $\widetilde f_n^{(p)}$ satisfies the node law at $\Gamma_n^1\cup\Gamma_n^2$.}
Let $x\in\Gamma_n^1\cup\Gamma_n^2$. If $x\in \cV_2(\cN_p,d/n)$, then for any edge $e$ that has $x$ for endpoint, we have by construction $f_n^{(p)}(e)=\widetilde f_n ^{(p)}(e)=0$.
Let us now assume that $x\notin \cV_2(\cN_p,d/n)$. Assume there exists $i\in\{1,\dots,d\}$ such that $\nn_{\Omega}(\pi_i(x))\cdot\overrightarrow{e_i}=0$. Since $x\notin \cV_2(\cN_p,d/n)$, we have that for any $y\in B(x,d/n)\cap\Gamma $, $\nn_\Omega(y)\cdot \overrightarrow{e_i}=0$. In other words, $\Gamma$ is locally flat near $\pi_i(x)$. Consequently, it yields that
 $x+\overrightarrow{e_i}/n$ and $x-\overrightarrow{e_i}/n$ belong to $\Gamma_n^1\cup \Gamma_n^2$. By definition of $\cS_n^M(\Gamma_1,\Gamma_2,\Omega)$, we have $$f_n^{(p)}(\langle x,  x-\overrightarrow{e_i}/n\rangle)=f_n^{(p)}(\langle x,  x+\overrightarrow{e_i}/n\rangle)=f_n^{(p)}(\langle x,  x+\overrightarrow{e_i}/n\rangle)=0\,.$$
 Finally, we have
\begin{align*}
\di \widetilde f_n^{(p)}(x)&=\sum_{i=1,\dots,d}\widetilde f_n^{(p)}(\langle x,x-\overrightarrow{e_i}/n\rangle)\cdot \overrightarrow{e_i}-\widetilde f_n^{(p)}(\langle x,x+\overrightarrow{e_i}/n\rangle)\cdot \overrightarrow{e_i}\\
&=\sum_{i=1,\dots,d}\ind_{\nn_{\Omega}(\pi_i(x))\cdot \overrightarrow{e_i}\geq 1/p}(f_n^{(p)}(\langle x,x-\overrightarrow{e_i}/n\rangle)-f_n^{(p)}(\langle x,x-\overrightarrow{e_i}/n\rangle)\cdot \overrightarrow{e_i}\\
&\qquad\qquad+\ind_{\nn_{\Omega}(\pi_i(x))\cdot \overrightarrow{e_i}\leq -1/p}(f_n^{(p)}(\langle x,x+\overrightarrow{e_i}/n\rangle)-f_n^{(p)}(\langle x,x+\overrightarrow{e_i}/n\rangle)\cdot \overrightarrow{e_i}\\
&=0\,.
\end{align*}

\noindent{\bf Step 3. We prove that the prolongated discrete stream converges towards a continuous stream in an extended version of $\Omega$.}
We claim that the node law is satisfied for $\widetilde f_n ^{(p)}$ at any point in $\cV_2(\Omega, r_0/2p)$. We prove this result by contradiction. Assume there exists a point $w\in \cV_2(\Omega,r_0/2p)$ where the node law is not satisfied (see figure \ref{fig:caspb3}). Necessarily, $w\notin \Omega_n$ and there
exists $y\in \Gamma$ such that $w=y+t_0\nn_{\Omega}(y)$ with $t_0\leq r_0/2p$ (see figure \ref{fig:caspb3}. Since the node law is not respected only at the end of a prolongated line, \textit{i.e}, for $w\in\sZ_n^d$ such that there exists $x\in \Gamma$ that satisfies $w=x+ t\overrightarrow {e_i}$ with $|t|\geq r_0-1/n$ and $|\nn_\Omega(x)\cdot \overrightarrow{e_i}|\geq 1/p$. There exists $j\in I$ such that $x\in B(x_j,r_{x_j}/2)$ and $|\nn_\Omega(x_j)\cdot \overrightarrow{e_i}|\geq 1/p$. Since $\|x-y\|_2\leq 2r_0\leq r_{x_j}/2$ we have $y\in B(x_j,r_{x_j})$.
 \begin{figure}[H]
\begin{center}
\def\svgwidth{0.9\textwidth}
   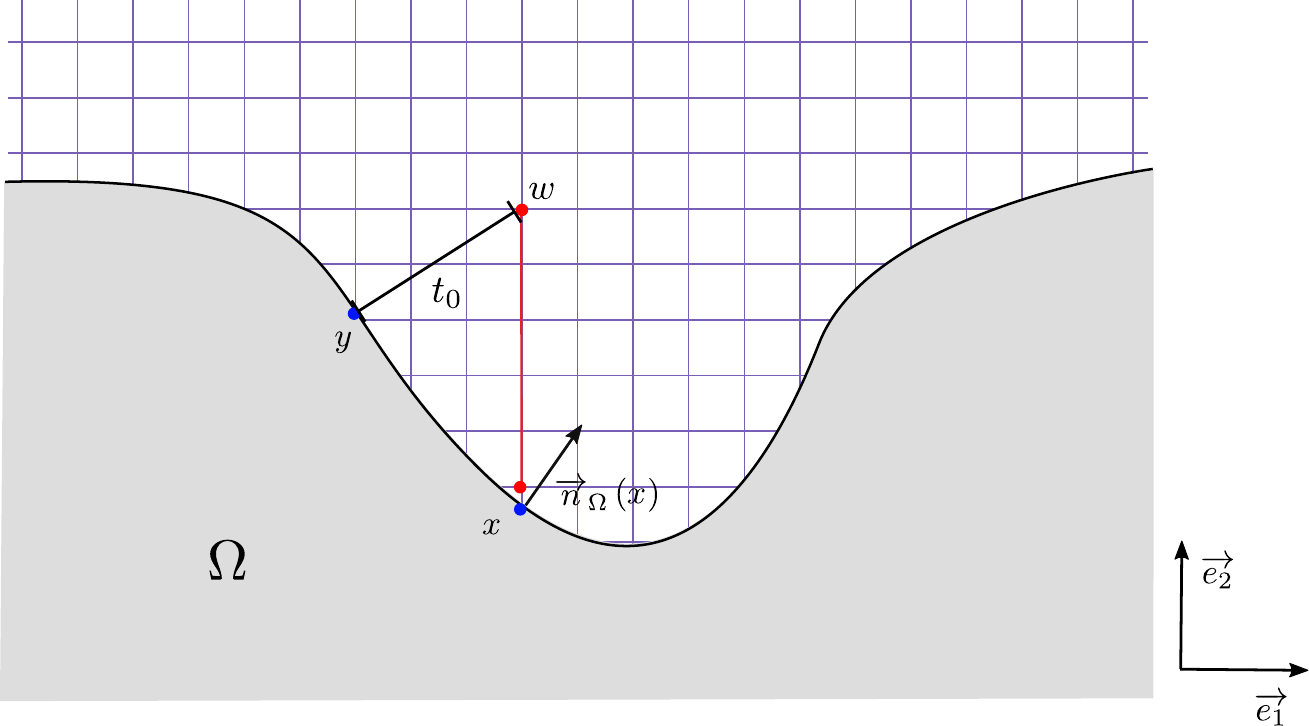
   \caption{Illustration of the case where the node law is not respected in $\cV_2(\Omega,r_0/2p)$.}
   \label{fig:caspb3}
   \end{center}
\end{figure}
 By the mean value theorem, there exists $z\in B(x_j,r_{x_j})\cap \Gamma$ such that $y-x$ is in the tangent space of $z$. Hence, we have $(y-x)\cdot \nn_\Omega(z)=0$.
And so $t\overrightarrow {e_i}\cdot \nn_\Omega(z)=t_0 \nn_\Omega(y)\cdot \nn_\Omega(z)$ and $|t\overrightarrow {e_i}\cdot \nn_\Omega(z)|\leq t_0\leq r_0/2p\leq 1/2p$.
This contradicts the fact that $|\overrightarrow {e_i}\cdot \nn_\Omega(z)|\geq 1/2p$ by definition of $r_{x_j}$.

Write $$\widetilde{\Omega}=\Omega\cup\cV_2(\Gamma^1\cup\Gamma^2\setminus \cN_p,r_0/2p)$$ and 
$$\widetilde{\Gamma}^1=\{x\in\widetilde \Omega: d_2(x,\Gamma^1)=r_0/2p\}$$
$$\widetilde{\Gamma}^2=\{x\in\widetilde \Omega: d_2(x,\Gamma^2)=r_0/2p\}\,.$$
By lemma \ref{lem:convdiscstream}, we can prove that we can extract from $\amu_n(\widetilde f_n^{(p)})\ind_{\widetilde{\Omega}}$ a sequence that converges towards a continuous stream $\widetilde{\sigma}^{(p)}\cL^d$ in $\{\widetilde{\sigma}\cL^d:\,\widetilde{\sigma}\in\Sigma(\widetilde{\Gamma}^1,\widetilde{\Gamma}^1,\widetilde \Omega)\}$ that coincides with $\ssigma ^{(p)}\cL^d$ on $\Omega$.
It is easy to check that since $\ssigma-\ssigma_0$ has null divergence on $\sR^d$ and null trace on $\Gamma$ then 
$(\widetilde{\sigma}^{(p)}+\ssigma-\ssigma_0)\cL^d$ in $\{\widetilde{\sigma}\cL^d:\,\widetilde{\sigma}\in\Sigma(\widetilde{\Gamma}^1,\widetilde{\Gamma}^1,\widetilde \Omega)\}$ and coincides with $(\ssigma ^{(p)}+\ssigma-\ssigma_0)\cL^d$ on $\Omega$.
%

{\bf Continuity of $\widehat{I}$.}
Note that $\widehat{I}((\widetilde{\sigma}^{(p)}+\ssigma-\ssigma_0)\ind_\Omega)=\widehat{I}(\ssigma^{(p)}+\ssigma-\ssigma_0)$.
Since the sequence $(\ssigma^{(p)}+\ssigma-\ssigma_0)_{p\geq 1}$ converges in $L^1$ towards $\ssigma$, up to extraction, we can assume that $(\ssigma^{(p)}+\ssigma-\ssigma_0)(x)$ converges towards $\ssigma(x)$ for $\cL^d$-almost every $x$ on $\Omega$. Let $\eta>0$. Set $\delta = \eta/ (4d\cL^d(\Omega) M)$.
We have
$$\widehat{I}((1-\delta)(\ssigma^{(p)}+\ssigma-\ssigma_0))=\int_{\Omega}I((1-\delta)(\ssigma^{(p)}+\ssigma-\ssigma_0)(x))d\cL^d(x)\,.$$
Since $(1-\delta)(\ssigma^{(p)}+\ssigma-\ssigma_0)(x)\in\mathring{\mathcal{D}}_I$, by proposition \ref{prop:continuityI}, the function $I$ is continuous at this point:
$$\lim_{p\rightarrow\infty}I((1-\delta)(\ssigma^{(p)}+\ssigma-\ssigma_0)(x))=I((1-\delta)\ssigma(x))\,.$$
Moreover, by proposition \ref{prop:controleI}, we have
$I((1-\delta)(\ssigma^{(p)}+\ssigma-\ssigma_0)(x))\leq -d\log G([(1-\delta)M,M])$ for $\cL^d$-almost every $x$.
We can therefore apply the dominated convergence theorem:
$$\lim_{p\rightarrow\infty}\widehat{I}((1-\delta)(\ssigma^{(p)}+\ssigma-\ssigma_0))=\widehat{I}((1-\delta)\ssigma)\,.$$
Using the convexity of $\widehat{I}$, we obtain
$$\lim_{p\rightarrow\infty}\widehat{I}((1-\delta)(\ssigma^{(p)}+\ssigma-\ssigma_0))\leq \widehat{I}((1-\delta)\ssigma )\leq (1-\delta)\widehat{I}(\ssigma)+\delta\widehat I (\overrightarrow{0})\leq \widehat{I}(\ssigma)\,.$$ 
Let $p$ be large enough such that 
$$\widehat{I}((1-\delta)(\ssigma^{(p)}+\ssigma-\ssigma_0))\leq \widehat{I}(\ssigma)+\eta$$
and
$$\|\ssigma-(\ssigma^{(p)}+\ssigma-\ssigma_0)\|_{L^1}\leq \eta/2\,.$$
We have
\begin{align*}
\|\ssigma-(1-\delta)(\widetilde \sigma^{(p)}+\ssigma-\ssigma_0)\ind_{\Omega}\|_{L^1}&=\|\ssigma-(1-\delta)(\ssigma^{(p)}+\ssigma-\ssigma_0)\|_{L^1}\\&\leq \|\ssigma-\ssigma^{(p)}\|_{L^1}+ \delta\cL^d(\Omega)2dM\leq \eta\,.
\end{align*}
We set $\ssigma'=(1-\delta)(\widetilde \sigma ^{(p)}+\ssigma-\ssigma_0)$ and $\rho=r_0/2p$. This concludes the proof.


\end{proof}

\section{Upper large deviation principle\label{sect:goodtaux}}
The following little lemma will appear several times in what follows. We refer to lemma 6.7 in \cite{Cerf:StFlour} for a proof of this lemma.
\begin{lem}[Lemma 6.7 in \cite{Cerf:StFlour}] \label{lem:estimeanalyse}
Let $f_1,\dots,f_r$ be $r$ non-negative functions defined on $]0,1[$. Then,
$$\limsup_{\ep\rightarrow 0}\ep\log \left(\sum_{i=1}^rf_i(\ep)\right)=\max_{1\leq i\leq r}\limsup_{\ep\rightarrow0}\ep\log f_i(\ep)\,.$$
\end{lem}
We recall that we endow $\cM(\overline{\cV_\infty(\Omega,1)})^d$ with the topology $\mathcal{O}$ associated with the distance $\dis$ and the Borelian $\sigma$-field $\cB$ and that $\Prb_n$ denotes the following probability:
$$\forall A\in\cB\qquad \Prb_n(A)=\Prb(\exists f_n\in\cS_n(\Gamma^1,\Gamma^2,\Omega):\amu_n(f_n)\in A)\,.$$
Let $\ssigma\in\Sigma(\Gamma ^1,\Gamma ^2,\Omega)$ and $\delta>0$. We denote by $\cB_{\dis}(\ssigma,\delta)$ the open ball centered at $\ssigma\cL^d$ of radius $\delta$:
$$\cB_{\dis}(\ssigma,\delta)=\left\{ \nu \in\cM(\overline{\cV_\infty(\Omega,1)})^d: \dis(\ssigma\cL^d,\nu)<\delta\right \}\,.$$
We denote $\cU$ the following basis of neighborhood of the null element of $\cM(\overline{\cV_\infty(\Omega,1)})^d$:
$$\cU=\left\{\cB_{\dis}\left(0,\frac{1}{p}\right):p\geq 1\right\}\,.$$
To prove theorem \ref{thm:pgd}, it is sufficient to prove that $\widetilde{I}$ is a good rate function, that the sequence of measures $(\Prb_n)_{n\geq1}$ is $\widetilde{I}$-tight and that the following local estimates are satisfied (see section 6.2 in \cite{Cerf:StFlour}). This is the purpose of the following proposition.

\begin{prop}\label{prop:localestimates}

The function $\widetilde{I}$ is a good rate function. The sequence of measures $(\Prb_n)_{n\geq1}$ is $\widetilde{I}$-tight, \textit{i.e.}, there exists positive constants $c$ and $\lambda_0$ such that
\[\forall \lambda\geq \lambda_0 \quad\forall U \in\mathcal{U}\qquad \limsup_{n\rightarrow \infty}\frac{1}{n^d}\log \Prb(\exists f_n\in\cS_n(\Gamma^1,\Gamma ^2,\Omega):\,\amu_n(f_n)\notin \widetilde{I}^{-1} ([0,\lambda])+U)\leq -c\lambda\,.\]
 Moreover, the local estimates are satisfied:
 $$\forall \,\nu\in\cM(\overline{\cV_\infty(\Omega,1)})^d \quad\forall \ep>0\qquad \liminf_{n\rightarrow \infty}\frac{1}{n^d}\log\Prb\left(\exists f_n\in\cS_n(\Gamma^1,\Gamma^2,\Omega) : \,\dis\big(\amu_n(f_n),\nu\big)\leq \ep\right)\geq -\widetilde{I}(\nu)\,.$$
\begin{align*}
\forall \,\nu\in\cM(\overline{\cV_\infty(\Omega,1)})^d & \quad{s.t.}\quad \widetilde{I}(\nu)<\infty, \quad\forall \ep>0,\quad \exists \delta=\delta(\ep)>0\\
&\limsup_{n\rightarrow \infty}\frac{1}{n^d}\log\Prb(\exists f_n\in\cS_n(\Gamma^1,\Gamma^2,\Omega) : \,\dis\big(\amu_n(f_n),\nu\big)\leq \delta(\ep))\leq -\widetilde{I}(\nu)(1-\ep)\,.
\end{align*}
\end{prop}
Before proving this proposition, we prove the following lemma.
\begin{lem}\label{lem:Sigmacompact} The set $\{\ssigma\cL^d:\ssigma\in\Sigma(\Gamma^1,\Gamma^2,\Omega)\cap \Sigma^M(\Gamma^1,\Gamma^2,\Omega)\}$ is compact for the topology associated with the distance $\dis$.
\end{lem}

\begin{proof}[Proof of lemma \ref{lem:Sigmacompact}]
Let $\ssigma\in\Sigma(\Gamma^1,\Gamma ^2,\Omega)\cap \Sigma^M(\Gamma^1,\Gamma^2,\Omega)$. Set $\nu=\ssigma\cL^d$. We have 
\[|\nu|(\Omega ^c)=0\qquad\text{and}\qquad |\nu|(\Omega)\leq 2dM\cL^d(\Omega)\,.\]
It follows by Prohorov theorem that the set $\{\ssigma\cL^d:\ssigma\in\Sigma(\Gamma^1,\Gamma^2,\Omega)\cap \Sigma^M(\Gamma^1,\Gamma ^2,\Omega)\}$ is relatively compact for the weak topology in the sense that for any sequence $(\ssigma_n\cL^d)_{n\geq 1}$, we can extract a subsequence $(\ssigma_{\psi(n)}\cL^d)_{n\geq 1}$ such that there exists $\nu_0 \in\cM(\overline{\cV_\infty(\Omega,1)})^d$ such that 
\[\forall f\in\sC_b(\sR^d,\sR)\qquad \lim_{n\rightarrow \infty}\int_{\sR ^d}f\ssigma_{\psi(n)}d\cL ^d=\int_{\sR^d}fd\nu_0 \,.\]
By lemma \ref{lem:weakconvautresens}, it follows that 
$$\lim_{n\rightarrow \infty}\dis(\ssigma_{\psi(n)}\cL ^d,\nu_0)=0\,.$$
Since $\ssigma_{\psi(n)}\in\Sigma ^M$ it is itself the weak limit of  a sequence of discrete streams: there exists $\phi:\sN\rightarrow\sN$ an increasing function such that for all $m\geq 1$  there exists $f_{\psi(n),\phi(m)}\in\cS_{\phi(m)}^M(\Gamma^1,\Gamma^2,\Omega)$ and
$$\lim_{m\rightarrow\infty}\dis(\ssigma_{\psi(n)}\cL^d,\amu_{\phi(m)}(f_{\psi(n),\phi(m)}))=0\,.$$
For any $n\geq 1$, we define $\phi_0(n)$ to be
$$\phi_0(n)=\inf\left\{\phi(m):\dis(\ssigma_{\psi(n)}\cL^d,\amu_{\phi(m)}(f_{\psi(n),\phi(m)}))\leq \frac{1}{n}\right\}\,.$$
It follows that 
$$\lim_{n\rightarrow\infty}\dis(\ssigma\cL^d,\amu_{\phi_0(n)}(f_{\psi(n),\phi_0(n)}))=0\,.$$
By lemma \ref{lem:convdiscstream}, we have that $\nu_0 =\ssigma_0\cL^d$ where $\ssigma_0\in \Sigma(\Gamma^1,\Gamma ^2,\Omega)\cap \Sigma^M(\Gamma^1,\Gamma ^2,\Omega)$.
Hence, the set $\{\ssigma\cL^d:\ssigma\in\Sigma(\Gamma^1,\Gamma^2,\Omega)\cap \Sigma^M(\Gamma^1,\Gamma ^2,\Omega)\}$ is compact for the topology associated with the distance $\dis$.
\end{proof}
\begin{proof}[Proof of proposition \ref{prop:localestimates}]

{\noindent \bf Step 1. We prove that $\widetilde{I}$ is lower semi-continuous.} 
Let $\nu\in \cM(\overline{\cV_\infty(\Omega,1)})^d$ and $(\nu_p)_{p\geq1}$ be a sequence such that $\dis(\nu_p,\nu)$ goes to $0$ when $p$ goes to infinity.\\

\noindent $\bullet$ If $\nu \in\{\ssigma\cL^d:\ssigma\in\Sigma(\Gamma^1,\Gamma^2,\Omega)\}$ and $\liminf_{p\rightarrow\infty} \widetilde{I}(\nu_p)=+\infty$, then $$\liminf_{p\rightarrow\infty} \widetilde{I}(\nu_p)\geq \widetilde{I}(\nu)\,.$$

\noindent $\bullet$ If $\nu=\ssigma\cL^d$ with $\ssigma\in\Sigma(\Gamma^1,\Gamma^2,\Omega)$ and $\liminf_{p\rightarrow\infty} \widetilde{I}(\nu_p)<\infty$, then we can extract from the sequence $(\nu_p)_{p\geq1}$ the measures such that $\widetilde{I}(\nu_p)<\infty$.
We will denote this subsequence sequence by $(\ssigma_p\cL^d)_{p\geq 1}$ where $\ssigma_p\in\Sigma(\Gamma^1,\Gamma ^2,\Omega)$ and $\widehat{I}(\ssigma_p)=\widetilde{I}(\ssigma_p\cL^d)$.
We use the same arguments as in proposition \ref{prop:lscI}. However, we cannot use this proposition because we do not have the almost sure convergence of $\ssigma_p$ towards $\sigma$.
The function 
$$\ep\mapsto \liminf_{n\rightarrow \infty}-\frac{1}{n^d}\log\Prb\left(\exists f_n\in\cS_n(\Gamma^1,\Gamma^2,\Omega) : \,\dis\big(\amu_n(f_n),\ssigma\cL^d\big)\leq \ep\right)$$
is clearly non-increasing.
Moreover, by theorem \ref{thm:ULDpatate}, we have
$$\lim_{\ep\rightarrow0}\liminf_{n\rightarrow \infty}-\frac{1}{n^d}\log\Prb\left(\exists f_n\in\cS_n(\Gamma^1,\Gamma^2,\Omega) : \,\dis\big(\amu_n(f_n),\ssigma\cL^d\big)\leq \ep\right)=\widehat{I}(\ssigma)\,.$$
Let $\delta>0$, there exists $\ep_0=\ep_0(\delta)>0$ such that 
$$\forall\ep\leq \ep_0\qquad\liminf_{n\rightarrow\infty}-\frac{1}{n^d}\log\Prb(\exists f_n\in\cS_n(\fC):\dis(\amu_n(f_n),\ssigma\cL^d)\leq \ep)\geq \widehat{I}(\ssigma)-\delta\,.$$
Let $p_0\geq 1$ be such that for any $p\geq p_0$, $\dis(\ssigma_p\cL^d,\ssigma\cL^d)\leq \ep_0/4$.
It yields that
\begin{align*}
\forall \ep\leq \ep_0\quad\forall p\geq p_0\qquad \liminf_{n\rightarrow\infty}-\frac{1}{n^d}&\log\Prb(\exists f_n\in\cS_n(\fC):\dis(\amu_n(f_n),\ssigma\cL^d)\leq \ep_0)\\&\leq  \liminf_{n\rightarrow\infty}-\frac{1}{n^d}\log\Prb(\exists f_n\in\cS_n(\fC):\dis(\amu_n(f_n),\ssigma_p\cL^d)\leq \ep/4)\,.
\end{align*}
It follows that
$$\forall \ep\leq \ep_0\quad\forall p\geq p_0\qquad  \liminf_{n\rightarrow\infty}-\frac{1}{n^d}\log\Prb(\exists f_n\in\cS_n(\fC):\dis(\amu_n(f_n),\ssigma_p\cL^d)\leq \ep/4) \geq I(\vv)-\delta\,.$$
By letting first $\ep$ goes to $0$ and then taking the liminf in $p$, we obtain
$$\liminf_{p\rightarrow\infty} \widehat{I}(\ssigma_p)\geq \widehat{I}(\ssigma)-\delta\,.$$
Since the previous inequality holds for any $\delta>0$, it follows that 
\begin{align*}
\liminf_{p\rightarrow\infty}  \widehat{I}(\ssigma_p)\geq\widehat{I}(\ssigma)\,.
\end{align*}

\noindent $\bullet$ If $\nu\in \cM(\overline{\cV_\infty(\Omega,1)})^d \setminus\{\ssigma\cL^d:\ssigma\in\Sigma(\Gamma^1,\Gamma^2,\Omega)\}$, then by definition of $\widetilde{I}$, we have $\widetilde{I}(\nu)=+\infty$. Let $M>0$. By proposition \ref{propadmissiblestream}, there exists $\ep_0>0$ such that 
$$\limsup_{n\rightarrow \infty}-\frac{1}{n^d}\log\Prb(\exists f_n\in\cS_n(\Gamma^1,\Gamma^2,\Omega) : \,\dis\big(\amu_n(f_n),\nu\big)\leq \ep_0)\geq M\,.$$
For any $p\geq 1$ such that $\dis(\nu_p,\nu)\leq \ep_0/2$, we have 
\begin{align*}
\forall \ep\leq \frac{\ep_0}{2}\qquad \limsup_{n\rightarrow \infty}-\frac{1}{n^d}&\log\Prb(\exists f_n\in\cS_n(\Gamma^1,\Gamma^2,\Omega) : \,\dis\big(\amu_n(f_n),\nu_p\big)\leq \ep)\\
&\geq \limsup_{n\rightarrow \infty}-\frac{1}{n^d}\log\Prb(\exists f_n\in\cS_n(\Gamma^1,\Gamma^2,\Omega) : \,\dis\big(\amu_n(f_n),\nu_p\big)\leq \ep_0/2)\\
&\geq \limsup_{n\rightarrow \infty}-\frac{1}{n^d}\log\Prb(\exists f_n\in\cS_n(\Gamma^1,\Gamma^2,\Omega) : \,\dis\big(\amu_n(f_n),\nu\big)\leq \ep_0)\geq M\,.
\end{align*}
By first taking the limit when $\ep$ goes to $0$ and then the liminf in $p$, we obtain 
\begin{align*}
\liminf_{p\rightarrow\infty}\widetilde{I}(\nu_p)\geq M\,.
\end{align*}
By letting $M$ go to infinity, we obtain 
\begin{align*}
\liminf_{p\rightarrow\infty}\widetilde{I}(\nu_p)= \widetilde{I}(\nu)=+\infty\,.
\end{align*}
Finally, we have in any case
$$\liminf_{p\rightarrow\infty} \widetilde{I}(\nu_p)\geq \widetilde{I}(\nu)\,.$$
Since the space is metric, this implies that
$\widetilde{I}$ is lower semi-continuous.

{\noindent \bf Step 2. We prove that $\widetilde{I}$ is a good rate function.} 
 Let us prove that its level sets are compact for the distance $\dis$. Let $\lambda>0$. We have
\begin{align*}
\{\ssigma\cL ^d:\,\ssigma\in\Sigma(\Gamma^1,\Gamma ^2,\Omega)\cap \Sigma^M(\Gamma^1,\Gamma ^2,\Omega),\,\widehat{I}(\ssigma)\leq \lambda\}& =\{\nu\in\cM(\overline{\cV_\infty(\Omega,1)})^d:\,\widetilde{I}(\nu)\leq \lambda\}\\
&\subset \{\ssigma\cL^d:\ssigma\in\Sigma(\Gamma^1,\Gamma^2,\Omega)\cap \Sigma^M(\Gamma^1,\Gamma ^2,\Omega)\}\,.
\end{align*}
Since $\widehat{I}$ is lower semi-continuous, its level sets are closed for the topology associated with $\dis$. Moreover by lemma \ref{lem:Sigmacompact}, the set $\{\ssigma\cL^d:\ssigma\in\Sigma(\Gamma^1,\Gamma^2,\Omega)\cap \Sigma^M(\Gamma^1,\Gamma ^2,\Omega)\}$ is compact. Hence, the level sets are compact for the topology associated with $\dis$. This implies that $\widetilde{I}$ is a good rate function. 

{\noindent  \bf Step 3.
We prove that $\widetilde{I}$ satisfies the local estimates.} Let $\ssigma\in \Sigma(\Gamma^1,\Gamma ^2,\Omega)\cap \Sigma^M(\Gamma^1,\Gamma^2,\Omega)$ such that $\widehat{I}(\ssigma)<\infty$.
The function 
$$\ep\mapsto \liminf_{n\rightarrow \infty}\frac{1}{n^d}\log\Prb\left(\exists f_n\in\cS_n(\Gamma^1,\Gamma^2,\Omega) : \,\dis\big(\amu_n(f_n),\ssigma\cL^d\big)\leq \ep\right)$$
is clearly non-decreasing.
Hence, using theorem \ref{thm:ULDpatate}, we have
\begin{align*}
\forall  \ep>0\qquad &\liminf_{n\rightarrow \infty}\frac{1}{n^d}\log\Prb\left(\exists f_n\in\cS_n(\Gamma^1,\Gamma^2,\Omega) : \,\dis\big(\amu_n(f_n),\ssigma\cL^d\big)\leq \ep\right)\\
&\geq \lim_{\ep'\rightarrow 0}\liminf_{n\rightarrow \infty}\frac{1}{n^d}\log\Prb\left(\exists f_n\in\cS_n(\Gamma^1,\Gamma^2,\Omega) : \,\dis\big(\amu_n(f_n),\ssigma\cL^d\big)\leq \ep'\right)=-\widehat{I}(\ssigma)=-\widetilde{I}(\ssigma\cL^d)\,.
\end{align*}
For measures $\nu$ such that $\widetilde{I}(\nu)=+\infty$, we have
\begin{align*}
\forall  \ep>0\qquad &\liminf_{n\rightarrow \infty}\frac{1}{n^d}\log\Prb\left(\exists f_n\in\cS_n(\Gamma^1,\Gamma^2,\Omega) : \,\dis\big(\amu_n(f_n),\nu\big)\leq \ep\right)\geq-\widetilde{I}(\nu)=-\infty \,.
\end{align*}
Let $\ssigma\in \Sigma(\Gamma^1,\Gamma^2,\Omega)\cap \Sigma^M(\Gamma^1,\Gamma^2,\Omega)$ such that $\widehat{I}(\ssigma)<\infty$.
The function 
$$\ep\mapsto -\limsup_{n\rightarrow \infty}\frac{1}{n^d}\log\Prb\left(\exists f_n\in\cS_n(\Gamma^1,\Gamma^2,\Omega) : \,\dis\big(\amu_n(f_n),\ssigma\cL^d\big)\leq \ep\right)$$
is non-increasing and thanks to theorem \ref{thm:ULDpatate}, the function goes to $\widehat{I}(\ssigma)=\widetilde{I}(\ssigma\cL^d)$ when $\ep$ goes to $0$. There exists $\delta=\delta(\ep)>0$ such that
\begin{align}\label{eq:localestimatelimsup}
-\limsup_{n\rightarrow \infty}\frac{1}{n^d}\log\Prb\left(\exists f_n\in\cS_n(\Gamma^1,\Gamma^2,\Omega) : \,\dis\big(\amu_n(f_n),\ssigma\cL^d\big)\leq \delta(\ep)\right)\geq \widehat{I}(\ssigma)(1-\ep)=\widetilde{I}(\ssigma\cL^d)(1-\ep)\,.
\end{align}

{\noindent \bf Step 4. We prove the $\widetilde{I}$-tightness.} Let $\lambda \geq 0$. First note that 
$\widetilde{I}^{-1}([0,\lambda])=\{\ssigma \cL^d:\,\ssigma\in\widehat{I}^{-1}([0,\lambda])\cap \Sigma^M(\Gamma^1,\Gamma ^2,\Omega)\}$. Let $U \in\mathcal{U}$. Let $K>0$. Let $\ep_0>0$.
For $\ssigma\in \Sigma^M(\Gamma^1,\Gamma ^2,\Omega)$ such that $\widehat{I}(\ssigma)<+\infty$, there exists $\delta_{\ssigma}(\ep_0)>0$, such that the local estimates \eqref{eq:localestimatelimsup} is satisfied for $\ep_0$. 
For $\ssigma$ such that $\widehat{I}(\ssigma)\leq \lambda$, up to choosing a smaller $\delta_{\ssigma}$, we can assume that $\cB_{\dis}(\ssigma,\delta_{\ssigma})\subset (\widetilde{I}^{-1}([0,\lambda])+U)$.
For $\ssigma$ such that $\widehat{I}(\ssigma)=+\infty$, there exists $\delta_{\ssigma}$ such that
$$\limsup_{n\rightarrow\infty}\frac{1}{n^d}\log\Prb( \exists f_n\in\cS_n(\Gamma^1,\Gamma^2,\Omega):\,\dis(\amu_n(f_n),\ssigma\cL^d)\leq \delta_{\ssigma})\leq -K\,.$$
Since by lemma \ref{lem:Sigmacompact}, the set $\{\ssigma\cL^d:\ssigma\in\Sigma(\Gamma^1,\Gamma^2,\Omega)\cap \Sigma^M(\Gamma^1,\Gamma^2,\Omega)\}$ is compact for $\dis$, we can extract from $(\cB_{\dis}(\ssigma,\delta_{\ssigma}),\ssigma\in \Sigma(\Gamma^1,\Gamma^2,\Omega)\cap \Sigma^M(\Gamma^1,\Gamma^2,\Omega))$ a finite covering 
$(\cB_{\dis}(\ssigma_i,\delta_{\ssigma_i}),i=1,\dots,N)$ of $\{\ssigma\cL^d:\ssigma\in\Sigma(\Gamma^1,\Gamma^2,\Omega)\cap \Sigma^M(\Gamma^1,\Gamma^2,\Omega)\}$.
We have
\begin{align*}
\Prb&\left(\exists f_n\in\cS_n(\Gamma^1,\Gamma ^2,\Omega):\,\amu_n(f_n)\notin \widetilde{I}^{-1} ([0,\lambda])+U\right)\\&\leq \sum_{i=1}^N\Prb\left(\exists f_n\in\cS_n(\Gamma^1,\Gamma ^2,\Omega):\,\amu_n(f_n)\notin \widetilde{I}^{-1} ([0,\lambda])+U, \amu_n(f_n)\in\cB_{\dis}(\ssigma_i,\delta_{\ssigma_i})\right)\\
&\quad+\Prb\left( \exists f_n\in\cS_n(\Gamma^1,\Gamma^2,\Omega):\,\amu_n(f_n)\notin \bigcup_{i=1,\dots,N}\cB_{\dis}(\ssigma_i,\delta_{\ssigma_i}) \right)\,.
\end{align*}
Thanks to proposition \ref{propadmissiblestream}, we have
$$\limsup_{n\rightarrow\infty}\frac{1}{n^d}\log\Prb\left( \exists f_n\in\cS_n(\Gamma^1,\Gamma^2,\Omega):\,\amu_n(f_n)\notin \bigcup_{i=1,\dots,N}\cB_{\dis}(\ssigma_i,\delta_{\ssigma_i}) \right)=-\infty\,.$$
Moreover, if $\widehat{I}(\ssigma_i)\leq\lambda$, then $(\widetilde{I}^{-1} ([0,\lambda])+U)^c\cap\cB_{\dis}(\ssigma_i,\delta_{\ssigma_i})=\emptyset$ and
$$\Prb\left(\exists f_n\in\cS_n(\Gamma^1,\Gamma ^2,\Omega):\,\amu_n(f_n)\notin (\widetilde{I}^{-1} ([0,\lambda])+U), \amu_n(f_n)\in\cB_{\dis}(\ssigma_i,\delta_{\ssigma_i})\right)=0\,.$$
By lemma \ref{lem:estimeanalyse}, it follows that
\begin{align*}
\limsup_{n\rightarrow\infty}\frac{1}{n^d}&\log\Prb(\exists f_n\in\cS_n(\Gamma^1,\Gamma ^2,\Omega):\,\amu_n(f_n)\notin \widetilde{I}^{-1} ([0,\lambda])+U)\\
&\leq -\min\left((1-\ep_0)\min\left\{\widehat{I}(\ssigma_i):\,i=1,\dots,N,\,\widehat{I}(\ssigma_i)>\lambda\right\},K\right)\leq -\min((1-\ep_0)\lambda,K)\,.
\end{align*}
By letting first $K$ go to infinity and then $\ep_0$ go to $0$, we obtain 
\begin{align*}
\limsup_{n\rightarrow\infty}\frac{1}{n^d}\log\Prb&(\exists f_n\in\cS_n(\Gamma^1,\Gamma ^2,\Omega):\,\amu_n(f_n)\notin \widetilde{I}^{-1} ([0,\lambda])+U)\leq -\lambda\,.
\end{align*}
This concludes the proof.
\end{proof}
Before proving theorem \ref{thm:uldmf}, we need to prove that the map $\ssigma\cL^d\mapsto \flow^{cont}(\ssigma)$ is continuous.
\begin{prop}\label{prop:flowcont}
The function $\ssigma\cL^d\mapsto \flow^{cont}(\ssigma)$ is continuous where $\{\ssigma\cL^d:\,\ssigma\in\Sigma(\Gamma^1,\Gamma ^2,\Omega)\}$ is endowed with the topology associated with the distance $\dis$.
\end{prop}
\begin{proof}[Proof of proposition \ref{prop:flowcont}]
Let $\ssigma\in\Sigma(\Gamma^1,\Gamma ^2,\Omega) $ and $(\ssigma_p)_{p\geq 1}$ be a sequence in $\Sigma(\Gamma^1,\Gamma ^2,\Omega)$ such that \[\lim_{p\rightarrow\infty }\dis(\ssigma_p\cL^d,\ssigma\cL^d)=0\,.\] 
Let $\eta>0$.
Thanks to the proof of proposition 4.7. in \cite{CT1} (see (4.28)), we know that there exists a finite family of hyperrectangles $\cA_1,\dots,\cA_{\cN}$ of disjoint interiors and $h\geq 0$ such that 
\begin{align}\label{eq:approxflow}
\forall \ssigma\in\Sigma(\Gamma^1,\Gamma ^2,\Omega)\qquad \left|\flow ^{cont}(\ssigma)-\frac{1}{2h}\sum_{i=1}^{\cN}\int_{\cyl(\cA_i,h)}\ssigma\cdot \overrightarrow{v_i} d\cL ^d\right|\leq C_0\eta M\,
\end{align}
where $\overrightarrow{v_i}$ is normal to $\cA_i$ and $C_0$ is a constant depending on $\Omega$ and the $\cA_i$.
By lemma \ref{lem:weakconv}, $\ssigma_p\cL^d$ weakly converges to $\ssigma\cL^d$. By Portmanteau theorem, we have that for any $i\in\{1,\dots,\cN\}$
$$\lim_{p\rightarrow\infty}\int_{\cyl(\cA_i,h)}\ssigma_p\cdot \overrightarrow{v_i} d\cL ^d=\int_{\cyl(\cA_i,h)}\ssigma\cdot \overrightarrow{v_i} d\cL ^d\,.$$
Using inequality \eqref{eq:approxflow}, it follows that
$$\limsup_{p\rightarrow\infty}\left|\flow ^{cont}(\ssigma)-\flow ^{cont}(\ssigma_p)\right|\leq 2C_0\eta M\,.$$
Finally, by letting $\eta$ goes to $0$, we obtain
\[\lim_{p\rightarrow\infty}\flow ^{cont}(\ssigma_p)=\flow ^{cont}(\ssigma)\,.\]
The result follows.
\end{proof}
We recall that 
$$J(\lambda)=\inf\left\{\widehat{I}(\ssigma):\ssigma\in\Sigma(\Gamma^1,\Gamma^2,\Omega)\cap \Sigma^M(\Gamma^1,\Gamma^2,\Omega), \flow^{cont}(\ssigma)=\lambda\right\}\,.$$
\begin{prop} \label{prop:prelthmuldmf}The function $J$ is convex on $\sR_+$. There exists $\lambda_{max}>0$ such that $J$ is finite on $[0,\lambda_{max}[$ and infinite on $]\lambda_{max},+\infty[$. Moreover, $J$ is increasing on $[\phi_{\Omega},\lambda_{max}[$, $J(\lambda)=0$ on $[0,\phi_{\Omega}]$ and
\[\forall \lambda\in]0,\lambda_{max}[\qquad\lim_{n\rightarrow\infty}\frac{1}{n^d}\log\Prb(\phi_n(\Gamma^1,\Gamma^2,\Omega)\geq\lambda n ^{d-1})=-J(\lambda)\,.\]
\end{prop}
Note that we did not study the behavior of the function at $\lambda_{max}$ since eventually we will replace the value of $J( \lambda_{max})$ by the value of its left limit at $\lambda_{max}$.
\begin{proof}[Proof of proposition \ref{prop:prelthmuldmf}]
{\noindent \bf Step 1. We prove that the infimum in the definition of $J$ is attained.}
Since by proposition \ref{prop:flowcont}, the function $\ssigma\mapsto\flow^{cont}(\ssigma)$ is continuous then the set $\{  \ssigma\in\Sigma(\Gamma^1,\Gamma ^2,\Omega)\cap \Sigma^M(\Gamma^1,\Gamma^2,\Omega),\,\flow^{cont}(\ssigma)=\lambda\}$ is closed. Besides, by lemma \ref{lem:Sigmacompact}, the set $\{\ssigma\cL^d:\ssigma\in\Sigma(\Gamma^1,\Gamma ^2,\Omega)\cap \Sigma^M(\Gamma^1,\Gamma^2,\Omega)\}$ is compact and so the set $\{  \ssigma\in\Sigma(\Gamma^1,\Gamma ^2,\Omega)\cap \Sigma^M(\Gamma^1,\Gamma^2,\Omega),\,\flow^{cont}(\ssigma)=\lambda\}$ is also compact. Consequently, the lower semi-continuous function $\widehat{I}$ attains its minimum over this set: there exists $\ssigma\in \Sigma(\Gamma^1,\Gamma ^2,\Omega)\cap \Sigma^M(\Gamma^1,\Gamma^2,\Omega)$ such that $\flow^{cont}(\ssigma)=\lambda$ and $J(\lambda)=\widehat{I}(\ssigma)$.

{\noindent \bf Step 2. We prove a lower bound.}
Let $\lambda\geq 0$ such that $J(\lambda)<\infty$. 
Note that we have 
\begin{align*}
J(\lambda)&=\inf\left\{\widehat{I}(\ssigma):\ssigma\in\Sigma(\Gamma^1,\Gamma^2,\Omega)\cap \Sigma^M(\Gamma^1,\Gamma^2,\Omega), \flow^{cont}(\ssigma)\geq \lambda\right\}\,.
\end{align*}
Indeed, for any $\ssigma\in\Sigma(\Gamma^1,\Gamma^2,\Omega)$ such that $\flow^{cont}( \ssigma)=(1+\delta)\lambda$ with $\delta\geq 0$, we have using the convexity of $I$ (theorem \ref{thm:convexity})
$$J(\lambda)\leq \widehat{I}\left(\frac{1}{1+\delta}\ssigma\right)\leq \frac{1}{1+\delta}\widehat{I}(\ssigma)+\left(1-\frac{1}{1+\delta}\right)\widehat{I}(\overrightarrow{0})\leq \widehat{I}(\ssigma)\,.$$
Let $\ep>0$. Thanks to proposition \ref{prop:flowcont}, there exists $\delta>0$ such that for any $\ssigma'\in\Sigma(\Gamma^1,\Gamma ^2,\Omega)$ we have 
\[\dis(\ssigma,\ssigma')\leq \delta \implies |\flow^{cont}(\ssigma)-\flow^{cont}(\ssigma')|\leq \frac{\ep}{2}\,.\]
We have
\begin{align*}
&\Prb(\exists f_n\in\cS_n(\Gamma^1,\Gamma^2,\Omega):\,\dis(\amu_n(f_n),\ssigma\cL^d)\leq \delta)\\&\hspace{2cm}\leq \Prb (\phi_n(\Gamma^1,\Gamma^2,\Omega)\geq(\lambda-\ep)n ^{d-1}, \exists f_n\in\cS_n(\Gamma^1,\Gamma^2,\Omega):\,\dis(\amu_n(f_n),\ssigma\cL^d)\leq \delta)\\
&\hspace{2cm}\quad+\Prb (\phi_n(\Gamma^1,\Gamma^2,\Omega)<(\lambda-\ep)n ^{d-1}, \exists f_n\in\cS_n(\Gamma^1,\Gamma^2,\Omega):\,\dis(\amu_n(f_n),\ssigma\cL^d)\leq \delta)\,.
\end{align*}
Let us prove that 
\begin{align}\label{eq:probatend0}
\exists n_0\geq 1\quad \forall n \geq n_0\qquad\Prb (\phi_n(\Gamma^1,\Gamma^2,\Omega)<(\lambda-\ep)n ^{d-1}, \exists f_n\in\cS_n(\Gamma^1,\Gamma^2,\Omega):\,\dis(\amu_n(f_n),\ssigma\cL^d)\leq \delta)=0\,.
\end{align}
If not, there exists a sequence $(\psi(n))_{n\geq 1}$ such that 
\[\Prb (\phi_{\psi(n)}(\Gamma^1,\Gamma^2,\Omega)<(\lambda-\ep)\psi(n) ^{d-1}, \exists f_{\psi(n)}\in\cS_{\psi(n)}(\Gamma^1,\Gamma^2,\Omega):\,\dis(\amu_{\psi(n)}(f_{\psi(n)},\ssigma\cL^d)\leq \delta)>0\,.\]
By the same arguments as in the proof of Proposition \ref{propadmissiblestream}, we can choose a sequence of realizations $(\omega_{\psi(n)})_{n\geq 1}$ and extract a subsequence from $(\amu_{\psi(n)}(f_{\psi(n)}(\omega_{\psi(n)}))_{n\geq 1}$ that converges weakly towards $\ssigma'\cL^d$ with $\ssigma'\in\Sigma(\Gamma^1,\Gamma ^2,\Omega)$. To lighten the notation we will denote by $(\amu_n(f_n))_{n\geq 1}$ this subsequence.
By Proposition 4.7. in \cite{CT1}, we have that
\[\lim_{n\rightarrow\infty}\frac{\flow_n(f_n)}{n^{d-1}}=\flow ^{cont}(\ssigma')\,.\]
By lemma \ref{lem:weakconvautresens}, we obtain that $$\lim_{n\rightarrow\infty}\dis(\amu_n(f_n),\ssigma'\cL^d)=0\,.$$
It follows that $\dis(\ssigma\cL^d,\ssigma'\cL^d)\leq \delta$ and so
\[\lim_{n\rightarrow\infty}\frac{\flow_n(f_n)}{n^{d-1}}=\flow ^{cont}(\ssigma')\geq \lambda-\frac{\ep}{2}\]
This contradicts the fact that $$\flow_n(f_n)\leq \phi_{n}(\Gamma^1,\Gamma^2,\Omega)<(\lambda-\ep)n^{d-1}\,.$$
It yields that
\begin{align}\label{eq:pgdflowliminf}
-J(\lambda)=-\widehat{I}(\ssigma)&\leq \liminf_{n\rightarrow\infty}\frac{1}{n^d}\log\Prb(\exists f_n\in\cS_n(\Gamma^1,\Gamma^2,\Omega):\,\dis(\amu_n(f_n),\ssigma\cL^d)\leq \delta)\nonumber\\&\leq \liminf_{n\rightarrow\infty}\frac{1}{n^d}\log\Prb (\phi_n(\Gamma^1,\Gamma^2,\Omega)\geq(\lambda-\ep)n ^{d-1})\,.
\end{align}

{\noindent \bf Step 3. We prove an upper bound.}
Let $K>0$. Let $\ep_0>0$.
Thanks to proposition \ref{prop:flowcont}, to each $\ssigma\in \Sigma(\Gamma^1,\Gamma^2,\Omega)$, we can associate a real number $\delta_{\ssigma}>0$ such that
\[\forall \ssigma' \in\Sigma(\Gamma^1,\Gamma^2,\Omega)\qquad \dis(\ssigma\cL^d,\ssigma'\cL^d)\leq \delta_{\ssigma} \implies |\flow^{cont}(\ssigma)-\flow^{cont}(\ssigma')|\leq \frac{\ep}{2}\,.\]
For $\ssigma$ such that $\widehat{I}(\ssigma)<+\infty$, up to choosing a smaller $\delta_{\ssigma}$, we can assume that the local estimate given by the inequality \eqref{eq:localestimatelimsup} is satisfied for $\ep_0$:
\begin{align}\label{eq:localestimecf1}
\limsup_{n\rightarrow \infty}\frac{1}{n^d}\log\Prb\left(\exists f_n\in\cS_n(\Gamma^1,\Gamma^2,\Omega) : \,\dis\big(\amu_n(f_n),\ssigma\cL^d\big)\leq \delta_{\ssigma}\right)\leq -\widehat{I}(\ssigma)(1-\ep_0)\,.
\end{align}
For $\ssigma$ such that $\widehat{I}(\ssigma)=+\infty$, up to choosing a smaller $\delta_{\ssigma}$, we can assume that
\begin{align}\label{eq:localestimecf2}
\limsup_{n\rightarrow\infty}\frac{1}{n^d}\log\Prb( \exists f_n\in\cS_n(\Gamma^1,\Gamma^2,\Omega):\,\dis(\amu_n(f_n),\ssigma\cL^d)\leq \delta_{\ssigma})\leq -K\,.
\end{align}
Since by lemma \ref{lem:Sigmacompact}, the set $\{\ssigma\cL^d:\,\ssigma\in\Sigma(\Gamma^1,\Gamma^2,\Omega)\cap \Sigma^M(\Gamma^1,\Gamma^2,\Omega)\}$ is compact for $\dis$, we can extract from $(\cB_{\dis}(\ssigma,\delta_{\ssigma}),\ssigma\in \Sigma(\Gamma^1,\Gamma^2,\Omega)\cap \Sigma^M(\Gamma^1,\Gamma^2,\Omega))$ a finite covering 
$(\cB_{\dis}(\ssigma_i,\delta_{\ssigma_i}),i=1,\dots,N)$.
We have
\begin{align}\label{eq:boundsupphi}
\Prb(\phi_n(\Gamma^1,\Gamma^2,\Omega)>(\lambda+\ep)n ^{d-1})&\leq \sum_{i=1}^N\Prb\left(\phi_n(\Gamma^1,\Gamma^2,\Omega)\geq(\lambda+\ep)n ^{d-1}, \amu_n^{max}\in\cB_\dis(\ssigma_i,\delta_{\ssigma_i})\right)\nonumber\\&\quad+\Prb\left( \exists f_n\in\cS_n(\Gamma^1,\Gamma^2,\Omega):\,\amu_n(f_n)\notin \bigcup_{i=1,\dots,N}\cB_{\dis}(\ssigma_i,\delta_{\ssigma_i}) \right)\,.
\end{align}
We claim that
\[\exists n_0\geq 1\quad \forall n \geq n_0\qquad\Prb\left (\exists f_n\in\cS_n(\Gamma^1,\Gamma^2,\Omega):\,\amu_n(f_n)\notin \bigcup_{i=1,\dots,N}\cB_{\dis}(\ssigma_i,\delta_{\ssigma_i}) \right)=0\,.\]
Indeed, if it was not true, by the same arguments we used to prove \eqref{eq:probatend0} and lemma \ref{lem:convdiscstream}, we can prove that there exists $\ssigma\cL^ d \notin \bigcup_{i=1,\dots,N}\cB_{\dis}(\ssigma_i,\delta_{\ssigma_i}) $ with $\ssigma\in\Sigma^M(\Gamma^1,\Gamma^2,\Omega)$. This contradicts the fact that $(\cB_{\dis}(\ssigma_i,\delta_{\ssigma_i}),i=1,\dots,N)$ is a covering of 
 $\{\ssigma\cL^d:\,\ssigma\in\Sigma(\Gamma^1,\Gamma^2,\Omega)\cap \Sigma^M(\Gamma^1,\Gamma^2,\Omega)\}$.
Let $i\in\{1,\dots,N\}$ such that $\flow^{cont}(\ssigma_i)<\lambda$. By the same arguments we used to prove \eqref{eq:probatend0}, we can prove that
\[\exists n_0\geq 1\quad \forall n \geq n_0\qquad\Prb\left (\phi_n(\Gamma^1,\Gamma^2,\Omega)\geq(\lambda+\ep)n ^{d-1},  \amu_n^{max}\in\cB_\dis(\ssigma_i,\delta_{\ssigma_i})\right)=0\,.\]
Using lemma \ref{lem:estimeanalyse}, inequality \eqref{eq:boundsupphi} and the local estimates \eqref{eq:localestimecf1}, \eqref{eq:localestimecf2}, it follows that 
\begin{align*}
\limsup_{n\rightarrow\infty}&\frac{1}{n^d}\log\Prb(\phi_n(\Gamma^1,\Gamma^2,\Omega)\geq(\lambda+\ep)n ^{d-1})\\&\leq \limsup_{n\rightarrow\infty}\frac{1}{n^d}\log\left(\sum_{i=1}^N\Prb( \exists f_n\in\cS_n(\Gamma^1,\Gamma^2,\Omega):\,\dis(\amu_n(f_n),\ssigma_i\cL^d)\leq \delta_{\ssigma_i})\ind_{\flow^{cont}(\ssigma_i)\geq\lambda}\right)\\
&\leq -\min\left((1-\ep_0)\min\{\widehat{I}(\ssigma_i):\flow^{cont}(\ssigma_i)\geq\lambda, i=1,\dots,N\} ,K\right)\leq -\min((1-\ep_0)J(\lambda),K)\,.
\end{align*}
By letting $K$ go to infinity and then $\ep_0$ go to $0$, we obtain
\begin{align}\label{eq:pgdfluxlimsup}
\limsup_{n\rightarrow\infty}\frac{1}{n^d}\log\Prb(\phi_n(\Gamma^1,\Gamma^2,\Omega)\geq(\lambda+\ep)n ^{d-1})\leq -J(\lambda).
\end{align}

{\noindent \bf Step 4. We prove that $J$ is convex and conclude.}
Let us prove that the function $J$ is convex. Let $x,y>0$ such that $J(x)<\infty$ and $J(y)<\infty$. Let $\alpha\in[0,1]$. Let $\ssigma_x\in\Sigma(\Gamma^1,\Gamma^2,\Omega)\cap \Sigma^M(\Gamma^1,\Gamma^2,\Omega)$ (respectively $\ssigma_y$) such that $\widehat{I}(\ssigma_x)=J(x)$ and $\flow^{cont}(\ssigma_x)=x$ (respectively $\widehat{I}(\ssigma_y)=J(y)$ and $\flow^{cont}(\ssigma_y)=y$).
We have
\begin{align*}
\flow^{cont}(\alpha\ssigma_x+(1-\alpha)\ssigma_y)&=-\int_{\Gamma^1}(\alpha\ssigma_x+(1-\alpha)\ssigma_y)\cdot \overrightarrow{n}_\Omega d\cH^{d-1}\\&=\alpha \flow^{cont}(\ssigma_x)+(1-\alpha) \flow^{cont}(\ssigma_y)\\&=\alpha x+(1-\alpha)y\,.
\end{align*}
Using the convexity of $\widehat{I}$, it follows that
\begin{align*}
J(\alpha x+(1-\alpha)y)&\leq \widehat{I}(\alpha\ssigma_x+(1-\alpha)\ssigma_y)\\&\leq \alpha\widehat{I}(\ssigma_ x)+(1-\alpha)\widehat{I}(\ssigma_ y)\\
&\leq \alpha J(x)+(1-\alpha)J(y)\,.
\end{align*}
Thus, $J$ is convex.
Let $\lambda_{max}=\sup\{\lambda>0:J(\lambda)<\infty\}$.
By convexity, for any $\lambda \in ]0,\lambda_{max}[$, $J$ is continuous at $\lambda$. Using inequalities \eqref{eq:pgdflowliminf} and \eqref{eq:pgdfluxlimsup}, we have
\[\limsup_{n\rightarrow\infty}\frac{1}{n^d}\log\Prb(\phi_n(\Gamma^1,\Gamma^2,\Omega)\geq\lambda n ^{d-1})\leq -\lim_{\ep \rightarrow 0}J(\lambda-\ep)=-J(\lambda)\] and
\[\liminf_{n\rightarrow\infty}\frac{1}{n^d}\log\Prb(\phi_n(\Gamma^1,\Gamma^2,\Omega)\geq\lambda n ^{d-1})\geq -\lim_{\ep \rightarrow 0}J(\lambda+\ep)=-J(\lambda)\,.\]
As a result we have
\[\lim_{n\rightarrow\infty}\frac{1}{n^d}\log\Prb(\phi_n(\Gamma^1,\Gamma^2,\Omega)\geq\lambda n ^{d-1})=-J(\lambda)\,.\]

Using theorems \ref{thm:CerfTheret} and \ref{thm:CT2}, and the continuity of $J$, we have that $J(\lambda)>0$ if and only if $\lambda>\phi_\Omega$.
Let us prove that $J$ is increasing on $[\phi_{\Omega},\lambda_{max}[$.
Let $\lambda,\lambda'\in[\phi_{\Omega},\lambda_{max}[$ such that $\lambda<\lambda'$.
There exists $\alpha\in]0,1[$ such that
\[\lambda=(1-\alpha)\phi_{\Omega}+\alpha\lambda'\,.\]
Using the convexity of $J$, we have
\[J(\lambda)\leq (1-\alpha) J(\phi_{\Omega})+\alpha J(\lambda')=\alpha J(\lambda')<J(\lambda')\,.\]
The result follows.
\end{proof}

\begin{proof}[Proof of theorem \ref{thm:uldmf}]
We recall that $\widetilde J_u$ was defined in \eqref{eq:defJu}.

{\bf$\bullet$ Lower bound.}
We prove the local lower bound:
$$\forall \lambda\geq 0 \quad\forall\ep>0\qquad \liminf_{n\rightarrow\infty}\frac{1}{n^d}\log \Prb\left(\frac{\phi_n(\Gamma^1,\Gamma^2,\Omega)}{n^{d-1}}\in]\lambda-\ep,\lambda+\ep[\right)\geq-\widetilde J_u(\lambda)\,.$$
Let $\lambda>0$ and $\ep>0$. If $\widetilde J_u(\lambda)=+\infty$, there is nothing to prove. If $\lambda_{min}\leq \lambda<\phi_{\Omega}$ such that $\widetilde J_u (\lambda)<\infty$, we have by theorem \ref{thm:lldmf}
\begin{align*}
\liminf_{n\rightarrow\infty}\frac{1}{n^{d-1}}\log \Prb\left(\frac{\phi_n(\Gamma^1,\Gamma^2,\Omega)}{n^{d-1}}\in]\lambda-\ep,\lambda+\ep[\right)\geq-\widetilde J_l(\lambda)\,.
\end{align*} 
It follows that 
\begin{align*}
\liminf_{n\rightarrow\infty}\frac{1}{n^{d}}\log \Prb\left(\frac{\phi_n(\Gamma^1,\Gamma^2,\Omega)}{n^{d-1}}\in]\lambda-\ep,\lambda+\ep[\right)\geq 0=\widetilde J _u (\lambda)\,.
\end{align*} 
Let us now assume that $\lambda\geq \phi_{\Omega}$. We have
\begin{align*}
\Prb\left(\frac{\phi_n(\Gamma^1,\Gamma^2,\Omega)}{n^{d-1}}\in]\lambda-\ep,\lambda+\ep[\right)\geq \Prb\left(\frac{\phi_n(\Gamma^1,\Gamma^2,\Omega)}{n^{d-1}}\geq \lambda \right)-\Prb\left(\frac{\phi_n(\Gamma^1,\Gamma^2,\Omega)}{n^{d-1}}\geq \lambda+\ep \right)\,.
\end{align*}
Since $J(\lambda)<J(\lambda+\ep)$, by lemma \ref{lem:estimeanalyse}, it leads to
$$\liminf_{n\rightarrow\infty}\frac{1}{n^d}\log \Prb\left(\frac{\phi_n(\Gamma^1,\Gamma^2,\Omega)}{n^{d-1}}\in]\lambda-\ep,\lambda+\ep[\right)\geq-J(\lambda)=-\widetilde J_u (\lambda)\,.$$

{\bf$\bullet$ Upper bound.} We have to prove that for all closed subset $\cF$ of $\sR^+$
\begin{align}\label{eq:upbdfin}
\limsup_{n\rightarrow\infty}\frac{1}{n^d}\log \Prb\left(\frac{\phi_n(\Gamma^1,\Gamma^2,\Omega)}{n^{d-1}}\in\cF\right)\leq-\inf_{\cF}\widetilde J_u\,.
\end{align}
Let $\cF$ be a closed subset of $\sR^+$. 
%
We consider $\cF_1=\cF\cap[0,\phi_\Omega]$ and $\cF_2=\cF\cap[\phi_{\Omega},+\infty[$.
Let us first assume that $\cF_1\neq\emptyset$. Let $f_1=\sup \cF_1$. We distinguish two cases.
\begin{itemize}
\item We assume $f_1\geq \lambda_{min}$. Then by proposition \ref{prop:prelthmuldmf}, we have $\inf_{\cF }\widetilde J_u=0$ and inequality \eqref{eq:upbdfin} is trivially satisfied.
\item We assume $f_1<\lambda_{min}$. Then, $\inf_{\cF }\widetilde J_u=\inf_{\cF _2}\widetilde J_u$ and for $n$ large enough, by definition of $\lambda_{min}$
$$\Prb\left(\frac{\phi_n(\Gamma^1,\Gamma^2,\Omega)}{n^{d-1}}\in \cF_1\right)\leq \Prb\left(\frac{\phi_n(\Gamma^1,\Gamma^2,\Omega)}{n^{d-1}}\leq f_1\right) =0\,.$$
It follows, that have
$$\limsup_{n\rightarrow\infty}\frac{1}{n^d}\log \Prb\left(\frac{\phi_n(\Gamma^1,\Gamma^2,\Omega)}{n^{d-1}}\in\cF\right)=\limsup_{n\rightarrow\infty}\frac{1}{n^d}\log \Prb\left(\frac{\phi_n(\Gamma^1,\Gamma^2,\Omega)}{n^{d-1}}\in\cF_2\right)\,.$$
\end{itemize}
Hence, to prove inequality \eqref{eq:upbdfin}, it remains to prove that
$$\limsup_{n\rightarrow\infty}\frac{1}{n^d}\log \Prb\left(\frac{\phi_n(\Gamma^1,\Gamma^2,\Omega)}{n^{d-1}}\in\cF_2\right) \leq - \inf_{\cF _2}\widetilde J_u\,.$$
Let us assume that $\cF_1=\emptyset$. If $\cF_2=\emptyset$ then 
$$\Prb\left(\frac{\phi_n(\Gamma^1,\Gamma^2,\Omega)}{n^{d-1}}\in\cF\right)=0$$
 and the inequality \eqref{eq:upbdfin} follows.
If $\cF_2\neq\emptyset$, we set $f_2=\inf\cF_2$. 
We have
\[\Prb\left(\frac{\phi_n(\Gamma^1,\Gamma^2,\Omega)}{n^{d-1}}\in\cF_2\right)\leq  \Prb\left(\frac{\phi_n(\Gamma^1,\Gamma^2,\Omega)}{n^{d-1}}\geq f_2\right)\,.\]
Using theorem \ref{thm:ULDpatate}, it yields that
$$\limsup_{n\rightarrow\infty}\frac{1}{n^d}\log\Prb\left(\frac{\phi_n(\Gamma^1,\Gamma^2,\Omega)}{n^{d-1}}\in\cF\right)\leq \limsup_{n\rightarrow\infty}\frac{1}{n^d}\log\Prb\left(\frac{\phi_n(\Gamma^1,\Gamma^2,\Omega)}{n^{d-1}}\geq f_2\right)=-J(f_2)=-\inf_{\cF}\widetilde J_u$$
since $J$ is increasing on $[\phi_{\Omega},+\infty[$.
This concludes the proof.
\end{proof}

\paragraph{Acknowledgment} The first author would like to thank Raphaël Cerf for helpful discussions. The first author would also like to thank Ziad Kobeissi and Marc Pegon for their help.

\bibliographystyle{plain}

\end{document}